\documentclass[11pt]{article}
\usepackage{amssymb,amsmath,bm}
\usepackage{amsthm}
\usepackage{xcolor}
\usepackage{textcomp}
\usepackage{enumerate}      
\usepackage{graphicx}        
\usepackage{caption, subcaption}

\usepackage{url}

\usepackage{tabu}

\usepackage[export]{adjustbox}

\usepackage{mathrsfs}

\usepackage{url}

\usepackage{array}
\newcommand{\PreserveBackslash}[1]{\let\temp=\\#1\let\\=\temp}
\newcolumntype{C}[1]{>{\PreserveBackslash\centering}p{#1}}
\newcolumntype{R}[1]{>{\PreserveBackslash\raggedleft}p{#1}}
\newcolumntype{L}[1]{>{\PreserveBackslash\raggedright}p{#1}}

\renewcommand{\setminus}{{\smallsetminus}}

\usepackage{amssymb,amsmath,bm}
\usepackage{amsthm}
\usepackage{hyperref}
\usepackage{mathrsfs}
\usepackage{textcomp}
\usepackage{enumerate}
\usepackage{graphicx}
\usepackage{tikz, tikz-cd}
\usepackage{verbatim}

\newtheorem{theorem}{Theorem}[section]
\newtheorem{lemma}[theorem]{Lemma}
\newtheorem{proposition}[theorem]{Proposition}
\newtheorem{definition}[theorem]{Definition}
\newtheorem{corollary}[theorem]{Corollary}

\newtheorem{prop}[theorem]{Proposition}
\theoremstyle{remark}
\newtheorem{remark}[theorem]{Remark}
\newcommand{\R}{\mathbb{R}}

\theoremstyle{remark}

\numberwithin{equation}{section}

\begin{document}

\title{\bf Turaev-Viro invariant from the modular double of $\mathrm {U}_{q}\mathfrak{sl}(2;\mathbb R)$}
\date{}

\author{Tianyue Liu} 
\author{Shuang Ming} 
\author{Xin Sun} 
\author{Baojun Wu}
\author{Tian Yang}
\author{Tianyue Liu, Shuang Ming, Xin Sun, Baojun Wu, and Tian Yang}

\maketitle

\begin{abstract} 
We define  a family of Turaev-Viro type  invariants of  hyperbolic $3$-manifolds with totally geodesic boundary from the $6j$-symbols of the modular double of $\mathrm U_{q}\mathfrak{sl}(2;\mathbb R)$, and prove that these invariants decay exponentially  with the  rate the hyperbolic volume of the manifolds and with the $1$-loop term the adjoint twisted Reidemeister torsion of the double of the manifolds. 
\end{abstract}

\tableofcontents

\section{Introduction} 
There are two very different approaches to 3-dimensional topology. The first approach follows Thurston's Geometrization Program (as completed by Perelman), and uses rigid homogeneous geometries to address purely topological problems. Hyperbolic geometry plays a central role in this framework. The second approach arose from Jones' discovery of completely new knot invariants, the Jones polynomial and its colored versions. This purely combinatorial invariant was later provided with a physical interpretation by Witten\,\cite{Wit89} using the quantum SU(2)-Chern-Simons theory, and expanded to give invariants of 3-dimensional manifolds. These quantum invariants of knots and 3-manifolds were  placed  by Reshetikhin-Turaev\,\cite{RT90,RT91} and Turaev-Viro\,\cite{turaev1992hq} in a rigorous mathematical setting through the representation theory of the quantum group $\mathrm U_q\mathfrak{sl}(2;\mathbb C)$.

The conjectural connection between these very different points of view arises when one considers the asymptotic behavior of quantum invariants. The most famous example of such a connection is Kashaev's Volume Conjecture\,\cite{Kas97} which, as rephrased by Murakami-Murakami\,\cite{MM01}, asserts that the value of the colored Jones polynomials of a hyperbolic knot at a certain root of unity grows exponentially with the growth rate the hyperbolic volume of the complement of the knot.  Another appealing conjecture involving exponential growth of quantum invariants and hyperbolic volumes emerged through work of the Chen-Yang\,\cite{CY}, where they observed that, for many hyperbolic 3-manifolds, their Reshetikhin-Turaev invariants and the Turaev-Viro invariants at the suitable root of unity grow exponentially with the growth rate again the hyperbolic volume of the manifolds.  These conjectures, backed up by experimental data, heuristic arguments, and rigorous proofs in a number of families  of examples, have attracted much attention in the past thirty years.

The goal of  this paper  is to initiate a mathematical study  of topological invariants associated to 
a quantum group  of significant interest in mathematical physics \cite{Fad95, Fad16, ponsot1999liouville} and representation theory \cite{FI14}, the modular double of $\mathrm U_{q}\mathfrak{sl}(2;\mathbb R)$ which we denote by $\mathrm U_{q\tilde{q}}\mathfrak{sl}(2;\mathbb R)$.  As demonstrated by Witten~\cite{Wit89}, the SU(2)-Chern-Simons theory is closely related to a two-dimensional conformal field theory (CFT), namely the SU(2)-Wess-Zumino-Witten (WZW) model. The fusion kernel of conformal blocks for the SU(2)-WZW model is the so-called $6j$-symbol for $\mathrm U_q\mathfrak{sl}(2;\mathbb C)$, which can be used to define the Jones polynomial and other $\mathrm U_q\mathfrak{sl}(2;\mathbb C)$ quantum invariants. This is an example of the correspondence between  3D topological quantum field theory (TQFT) and 2D CFT \cite{Wit89}.  In analogy to the relation between $\mathrm U_q\mathfrak{sl}(2;\mathbb C)$ and the SU(2)-WZW model,  the quantum group U$_{q\tilde{q}}\mathfrak{sl}(2;\mathbb R)$ is closely related to an important 2D CFT called the Liouville theory.   
As conjectured by Teschner \cite{ponsot2002boundary} and proved in \cite{GRSS}, the $6j$-symbol of  $\mathrm U_{q\tilde{q}}\mathfrak{sl}(2;\mathbb R)$ gives the fusion kernel of the Liouville conformal blocks. Liouville CFT has been intensively studied in physics since 1980s~\cite{Pol81,DO94,ZZ96,Teschner95}   and more recently in mathematics~\cite{DKRV16, KRV19b, GKRV20, GKRV21}. Among its various facets, it is closely related to the physical origin of the Volume Conjecture, namely the quantization of Teichmuller theory, complex Chern-Simons theory, and 3D gravity \cite{Witten:1989ip,Dimofte:2011gm,CEZ}.

In this paper, we focus on the Turaev-Viro type invariants of hyperbolic $3$-manifolds with totally geodesic boundary. These invariants are given by state-integrals associated with an ideal triangulation of the manifold, whose integrand consists of a product of $\mathrm U_{q\tilde{q}}\mathfrak{sl}(2;\mathbb R)$ $6j$-symbols, and 
the integral is over the continuous and non-compact spectrum of the so-called positive representations of $\mathrm U_{q\tilde{q}}\mathfrak{sl}(2;\mathbb R)$.
We prove the convergence of the state-integral and the independence of the choice of the ideal triangulation, showing that the construction indeed gives  a family of topological invariants of the manifold. (See Theorem \ref{Converge},  Theorem \ref{WD3} and Definition \ref{VTV}.) In contrast with other previously defined state-integral $3$-manifold invariants  (e.g. \cite{AK,KLV}), our invariants do not rely on the choice or the existence of any additional data or structures.

More importantly, we studied the asymptotic behavior of these invariants and proved that,  as the quantum parameter   tends to $0$, these invariants decay exponentially with the decay rate the hyperbolic volume of the manifold; furthermore, the sub-leading term in the asymptotic expansion of these invariants equals the adjoint twisted Reidemeister torsion of the double of the manifold. (See Theorem~\ref{VC}.) This is the first family of quantum invariants that settles the above-mentioned connection between hyperbolic geometry and quantum topology, while for the other invariants, such a relationship has only been proved in examples \cite{ohtsuki2016asymptotic, ohtsuki2018asymptotic,belletti2022growth} 
or under additional constraints\cite{wong2023relative, BGP, aribi2024andersen}. The asymptotic expansion of our Turaev-Viro type invariants relies on the saddle point approximation of the state integral. The proof relies on a delicate asymptotic analysis of the integrand made of $\mathrm U_{q\tilde q}\mathfrak{sl}(2;\mathbb R)$ $6j$-symbols. 
In fact, our approach yields a thorough understanding of the asymptotics of a single $\mathrm U_{q\tilde q}\mathfrak{sl}(2;\mathbb R)$ $6j$-symbol, which is of independent interest. (See Theorems \ref{cov} and \ref{cov2}.) Other key ingredients of our proof include the theory of Luo-Yang\,\cite{L,LY} on hyperbolic polyhedral metrics, the Murakami-Yano volume formula~\cite{MY,U,MU,BY2}, angle structures and Kojima ideal triangulations of hyperbolic $3$-manifolds with totally geodesic boundary, and the computation of Wong-Yang\,\cite{WY3} of the adjoint twisted Reidemeister torsion of $3$-manifolds with hyperbolic polyhedral metrics. In particular, to overcome a major difficulty caused by the possible existence of flat tetrahedra in a Kojima ideal triangulation, we establish a fine estimate comparing the quantum  dilogarithm functions and the dilogarithm function. (See Proposition~\ref{EST}).

The rest of the paper is organized as follows. In Sections~\ref{subsec:1.1} and~\ref{subsec:1.2}, we state our main results. 
In Section~\ref{subsec:outlook}, we discuss the relevance of our invariants to subjects in mathematical physics and geometry, and outline some future directions.
In Section~\ref{sec:pre}, we provide preliminary background and results. 
In Section ~\ref{sec:tetrahedron}, we prove  Theorems \ref{cov} and \ref{cov2} on the asymptotics of a single $\mathrm U_{q\tilde q}\mathfrak{sl}(2;\mathbb R)$  $6j$-symbol. 
In Section ~\ref{sec:manifold}, we prove the convergence of our state integral (Theorem~\ref{Converge}) and establish its asymptotic behavior (Theorem~\ref{VC}). 
In Section ~\ref{sec:invariance}, we prove Theorem~\ref{WD3} that our state-integral indeed defines a topological invariant.

\subsection{\texorpdfstring{$\mathrm{U}_{q\tilde q}\mathfrak{sl}(2;\mathbb R)$ $6j$}{Uq(\mathfrak{sl}(2;R)) 6j}-symbol and its asymptotics}\label{subsec:1.1}

Let $S_b$ be the  \emph{double sine function} defined for $z\in \mathbb C$ with $0<\mathrm{Re}(z)<Q$ by
\begin{equation}\label{eq:def-S}
S_b(z)=\exp\Bigg(\int_\Omega\frac{\sinh\Big(\big(\frac{Q}{2}-z\big)t\Big)}{4t\sinh(\frac{bt}{2})\sinh(\frac{t}{2b})}dt\Bigg)
\end{equation}
where $b\in (0,1)$ and $Q=b+b^{-1}$, 
and  the contour $\Omega$ goes along the real line and passes above the pole at the origin. By the functional equation (see e.g \cite[A.15]{Teschner:2012em})
\begin{equation}\label{FE1}
S_b(z+b^{\pm 1})=2\sin (\pi b ^{\pm 1} z)S_b(z),
\end{equation}
$S_b$ extended to a meromorphic function on $\mathbb C$ with the set of poles $\{ -nb-mb^{-1} \ |\ m, n\in \mathbb Z_{\geqslant 0}\}$ and the set of  zeros  $\{ Q + nb + mb^{-1} \ |\ m, n\in \mathbb Z_{\geqslant 0}\}.$

For $(a_1,\dots,a_6)\in\mathbb C^6,$ let  
\begin{align*}
&t_1=a_1+a_2+a_3,  &&   t_2=a_1+a_5+a_6, \\
&t_3=a_2+a_4+a_6, &&t_4=a_3+a_4+a_5,\\
&q_1=a_1+a_2+a_4+a_5, &&q_2=a_1+a_3+a_4+a_6, \\
&q_3=a_2+a_3+a_5+a_6, &&q_4=2Q.
\end{align*}

\begin{figure}[htbp]
\centering
\includegraphics[scale=0.4]{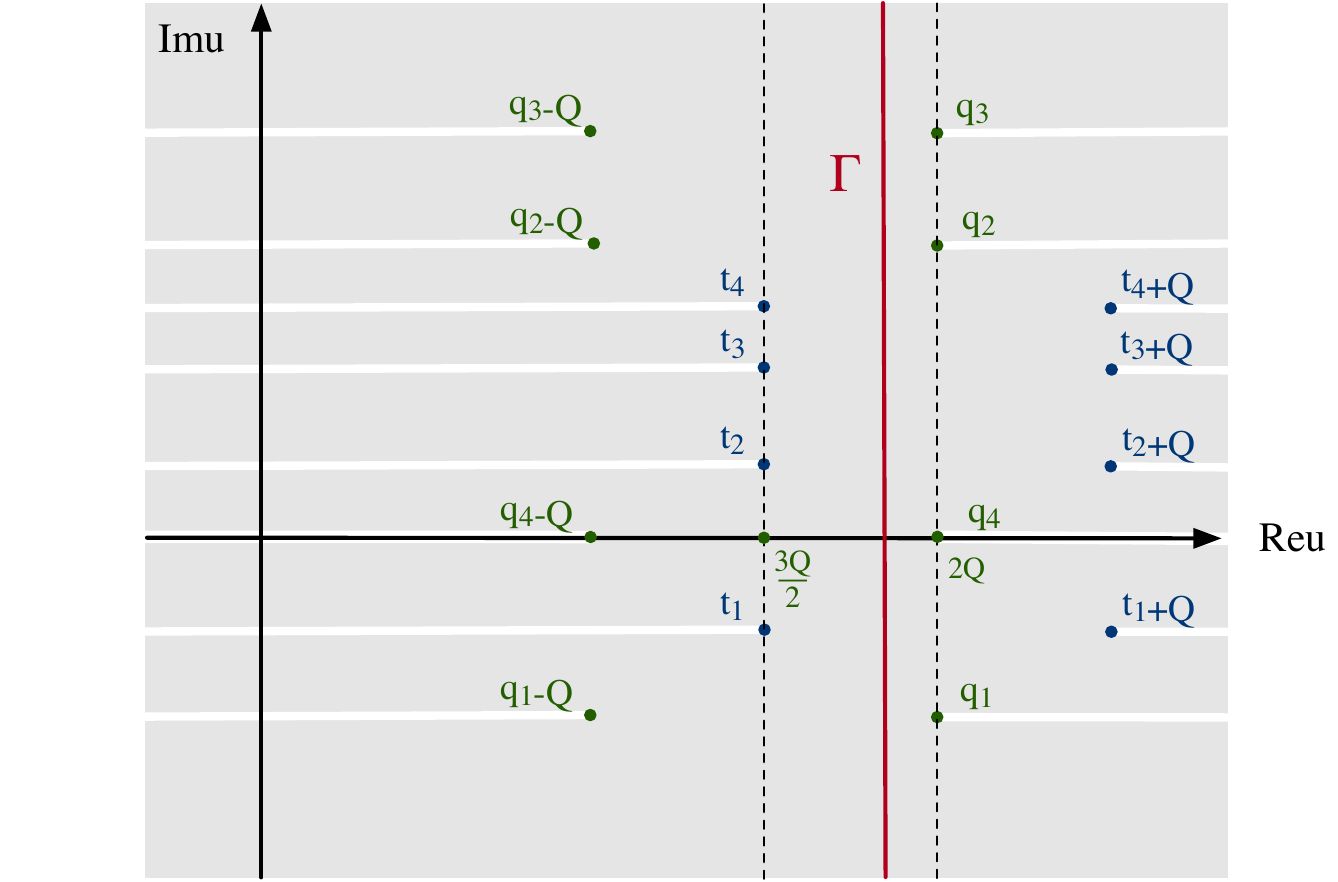}
\caption{Contour $\Gamma$ and possible zeros and poles (located in the white rays) of the integrand in~\eqref{b-6j}.}
\label{Da}
\end{figure}

\begin{definition}\label{def：6j}
Suppose $a_k \in \frac{Q}{2} + \mathbf i \mathbb R_{> 0}$ for all $k\in\{1,\dots,6\}$. 
The  $\mathrm U_{q\tilde q}\mathfrak{sl}(2;\mathbb R)$ $6j$-symbol, or the $b$-$6j$ symbol, with parameter $\boldsymbol a=(a_1,\dots,a_6)$ is given by
\begin{equation}\label{b-6j}
\bigg\{\begin{matrix} a_1 & a_2 & a_3 \\ a_4 & a_5 & a_6 \end{matrix} \bigg\}_b=\Bigg(\frac{1}{\prod_{i=1}^4\prod_{j=1}^4S_b(q_j-t_i)}\Bigg)^{\frac{1}{2}}\int_\Gamma \prod_{i=1}^4S_b(u-t_i)\prod_{j=1}^4S_b(q_j-u)d{\mathrm{Im}} u,
\end{equation}
where the contour $\Gamma$ is any vertical line passing the interval $(\frac{3Q}{2},2Q)$. See Figure \ref{Da}. 
\end{definition}
We will explain in Section~\ref{subsec:contour} that the integral in~\eqref{b-6j} converges absolutely and does not depend on the choice of the contour $\Gamma$.

We recall the origin of the $b$-$6j$ symbols, although it is not needed for understanding our results. Let
\[
q=e^{\pi i b^2} \quad\textrm{and}\quad \tilde{q}=e^{\pi i b^{-2}}.
\]The quantum group $\mathrm U_q\mathfrak{sl}(2;\mathbb C)$, as introduced in \cite{drinfeld1985rx}, is a Hopf algebra that arises as a deformation of the universal enveloping algebra of $\mathfrak{sl}(2;\mathbb C)$. Equipping $\mathrm U_q\mathfrak{sl}(2;\mathbb C)$ with a $*$-structure using the split real form $\mathfrak{sl}(2; \mathbb{R})\subset \mathfrak{sl}(2; \mathbb{C})$, we obtain the $*$-Hopf algebra $\mathrm U_q\mathfrak{sl}(2;\mathbb R)$;
and there is a family of $*$-representations indexed by the spectrum $\frac{Q}{2}+\mathbf{i}\mathbb R_{> 0}$ which have been intensively studied in~\cite{Schmudgen2004OperatorRO,PT01}.
The vector space for these representations is the Hilbert space $L^2(\R)$, and elements in $\mathrm U_q\mathfrak{sl}(2;\mathbb R)$ act as self-adjoint unbounded operators.
It turns out that all these representations are simultaneously representations of $\mathrm U_{\tilde{q}}\mathfrak{sl}(2;\mathbb R)$, hence are representations of the modular double $\mathrm U_{q\tilde{q}}\mathfrak{sl}(2;\mathbb R)\doteq\mathrm U_{q}\mathfrak{sl}(2;\mathbb R)\otimes \mathrm U_{\tilde{q}}\mathfrak{sl}(2;\mathbb R)$. 
For these induced \emph{positive representations} of $\mathrm U_{q\tilde{q}}\mathfrak{sl}(2;\mathbb R)$, one can define the $\mathrm U_{q\tilde{q}}\mathfrak{sl}(2;\mathbb R)$ $6j$-symbols following a standard recipe in the representation theory of quantum groups, giving the $b$-$6j$ symbols. The formula~\eqref{b-6j} comes from \cite[(2.26)]{Teschner:2012em}. See Section~\ref{sec:5.2} for more background.

Our first result provides the asymptotics of the $b$-$6j$ symbols. The limiting expression is in terms of the co-volume function and the Gram matrix for truncated hyperideal tetrahedra, which are the building blocks of hyperbolic 3-manifolds with totally geodesic boundary.  
A truncated hyperideal tetrahedron $\Delta$
is completely determined by the lengths  of its six edges $(l_1,\dots,l_6)$, or its six dihedral angles $(\theta_1,\dots,\theta_6)$. 
The co-volume function ${\mathrm{Cov}}(l_1,\dots,l_6)$  is defined by $\mathrm{Vol}(\Delta)+\sum_{k=1}^6\frac{\theta_k l_k}{2}$, first considered by Luo in \cite{L}
as the Legendre transform of the volume function, which satisfies $\frac{\partial \mathrm{Cov}(l_1,\dots,l_6)}{\partial l_k}=\frac{\theta_k}{2}$. See Section~\ref{extcov} for more details on truncated hyperideal tetrahedra. The definition of the Gram matrix $\mathrm{Gram}(l_1,\cdots,l_6)$ of $\Delta$ can also be found in Section~\ref{extcov}.

\begin{theorem}\label{cov} Let $(l_1,\dots, l_6)\in{\mathbb R_{>0 }^6}$ be the lengths of the edges of a truncated hyperideal tetrahedron $\Delta.$ Then  as $b\to 0,$
$$\bigg\{\begin{matrix} \frac{Q}{2}+ \mathbf i\frac{l_1}{2\pi b} & \frac{Q}{2}{+} \mathbf i\frac{l_2}{2\pi b} & \frac{Q}{2}{+} \mathbf i\frac{l_3}{2\pi b} \\ \frac{Q}{2}{+} \mathbf i\frac{l_4}{2\pi b} & \frac{Q}{2}{+} \mathbf i\frac{l_5}{2\pi b} & \frac{Q}{2}{+} \mathbf i\frac{l_6}{2\pi b} \end{matrix} \bigg\}_b=\frac{1}{2}\frac{e^{\frac{-\mathrm{Cov}(l_1,\dots,l_6)}{\pi b^2}}}{\sqrt[4]{-\det\mathrm{Gram}(l_1,\dots,l_6)}}\Big(1+O\big(b^2\big)\Big).$$
In particular, 
$$\lim_{b\to 0}\pi b^2\log\bigg\{\begin{matrix} \frac{Q}{2}{+} \mathbf i\frac{l_1}{2\pi b} & \frac{Q}{2}{+} \mathbf i\frac{l_2}{2\pi b} & \frac{Q}{2}{+} \mathbf i\frac{l_3}{2\pi b} \\ \frac{Q}{2}{+} \mathbf i\frac{l_4}{2\pi b} & \frac{Q}{2}{+} \mathbf i\frac{l_5}{2\pi b} & \frac{Q}{2}{+} \mathbf i\frac{l_6}{2\pi b} \end{matrix} \bigg\}_b=-\mathrm{Cov}(l_1,\dots,l_6).$$
\end{theorem} 
When $(l_1,\dots, l_6)\in{\mathbb R_{>0 }^6}$ are not the edge lengths of a truncated hyperideal tetrahedron,  they can be the edge lengths of a so-called flat tetrahedron, or even worse. 
Our next theorem considers both of the two cases, and relies on a natural extension $\widetilde{\mathrm{Cov}}$  of the co-volume function $\mathrm{Cov}$ introduced by Luo-Yang~\cite{LY}.
See Section~\ref{extcov} for the precise definitions.

\begin{theorem}\label{cov2} 
\begin{enumerate}[(a)]
    \item If $(l_1,\dots, l_6)\in{\mathbb R_{>0}^6}$ are the lengths of the edges of a flat tetrahedron, then
$$\lim_{b\to 0}\pi b^2\log\bigg\{\begin{matrix} \frac{Q}{2}{+} \mathbf i\frac{l_1}{2\pi b} & \frac{Q}{2}{+} \mathbf i\frac{l_2}{2\pi b} & \frac{Q}{2}{+} \mathbf i\frac{l_3}{2\pi b} \\ \frac{Q}{2}{+} \mathbf i\frac{l_4}{2\pi b} & \frac{Q}{2}{+} \mathbf i\frac{l_5}{2\pi b} & \frac{Q}{2}{+} \mathbf i\frac{l_6}{2\pi b} \end{matrix} \bigg\}_b=-\widetilde{\mathrm{Cov}}(l_1,\dots,l_6).$$

\item If $(l_1,\dots, l_6)\in{\mathbb R_{>0}^6}$ are neither the lengths of the edges of a truncated hyperideal tetrahedron nor the lengths of the edges a flat tetrahedron, then
$$\limsup_{b\to 0}\pi b^2\log\Bigg|\bigg\{\begin{matrix} \frac{Q}{2}{+} \mathbf i\frac{l_1}{2\pi b} & \frac{Q}{2}{+} \mathbf i\frac{l_2}{2\pi b} & \frac{Q}{2}{+} \mathbf i\frac{l_3}{2\pi b} \\ \frac{Q}{2}{+} \mathbf i\frac{l_4}{2\pi b} & \frac{Q}{2}{+} \mathbf i\frac{l_5}{2\pi b} & \frac{Q}{2}{+} \mathbf i\frac{l_6}{2\pi b} \end{matrix} \bigg\}_b\Bigg|= -\widetilde{\mathrm{Cov}}(l_1,\dots,l_6).$$
\end{enumerate}
\end{theorem}

\begin{remark}\label{rmk:ads}
In a subsequent paper \cite{ADS25}, as a refinement of Theorem \ref{cov2} (b), we will prove that if $(l_1,\dots, l_6)\in{\mathbb R_{>0}^6}$ are neither the edge lengths of a truncated hyperideal tetrahedron nor the edge lengths of a flat tetrahedron, then they are the edge lengths of a truncated hyperideal tetrahedron $\Delta$ in the anti-de Sitter space $\mathbb{A}\mathrm d\mathbb{S}^3$, and
\begin{equation}\label{AdS/CFT}
\bigg\{\begin{matrix} \frac{Q}{2}{+} \mathbf i\frac{l_1}{2\pi b} & \frac{Q}{2}{+} \mathbf i\frac{l_2}{2\pi b} & \frac{Q}{2}{+} \mathbf i\frac{l_3}{2\pi b} \\ \frac{Q}{2}{+} \mathbf i\frac{l_4}{2\pi b} & \frac{Q}{2}{+} \mathbf i\frac{l_5}{2\pi b} & \frac{Q}{2}{+} \mathbf i\frac{l_6}{2\pi b} \end{matrix} \bigg\}_b=\frac{e^{\frac{-\widetilde{\mathrm{Cov}}(\boldsymbol l)}{\pi b^2}} }{\sqrt[4]{\det\mathrm{Gram}(\boldsymbol l)}} \Bigg(\cos\bigg(\frac{\mathrm{Cov}(\Delta)}{\pi b^2}{-}\frac{\pi}{4}\bigg) +O\big(b^2\big)\Bigg).
\end{equation}
Here, $\mathrm{Cov}(\Delta)$ is the co-volume of $\Delta$ defined by $\mathrm{Cov}(\Delta)=\mathrm{Vol}(\Delta)+\sum_{k=1}^6\frac{\theta_kl_k}{2}$ and satisfies 
$\frac{\partial\mathrm{Cov(\Delta)}}{\partial l_k}=\frac{\theta_k}{2},$
where $\mathrm{Vol}(\Delta)$ is the anti-de Sitter volume of $\Delta$ and $\theta_1,\dots,\theta_6$ are the anti-de Sitter dihedral angles of $\Delta$. {See also Remark~\ref{rmk:ads2}.}
\end{remark}

Our proof of Theorems~\ref{cov} and~\ref{cov2}, as carried out in Section ~\ref{sec:tetrahedron}, relies on a fine estimate  of the double sine function $S_b(x)$ by the dilogarithm function, and the saddle point approximation. The key ingredient is the concavity of  the leading term when we expand the integrand in the  $b$-$6j$ symbol 
in~$b$.

The relation between tetrahedral volume and the asymptotic of  $6j$ symbols were previously considered in several works. 
Inspired by the work of Wigner \cite{Wigner}~, Ponzano and Regge \cite{Ponzano} proposed a relation between the {Wigner} $6j$ symbol of $\mathfrak{sl}(2;\mathbb C)$ and Euclidean tetrahedron, which was proved in \cite{Roberts99}. The $6j$ symbols of  $\mathrm U_q\mathfrak{sl}(2;\mathbb C)$  was considered in~\cite{taylor_6j_2006,C,CM, BY2}.
Previously in physics,  the $6j$ symbols for $\mathrm U_{q\tilde q}\mathfrak{sl}(2;\mathbb R)$ were considered in \cite{Teschner:2012em} and \cite{CHJL} with inputs on the real line and on the spectrum $\frac{Q}{2}+ \mathbf i {\mathbb R_{>0}}$ respectively interpreted as the dihedral angles and the edge lengths of a hyperbolic tetrahedron, with the exponential decay rates respectively the volume and the co-volume of the tetrahedron.
Our Theorems~\ref{cov} and~\ref{cov2} are the first mathematical results on the asymptotic of the $6j$ symbol for $\mathrm U_{q\tilde q}\mathfrak{sl}(2;\mathbb R)$. 

{\begin{remark} It is obvious from (\ref{def：6j}) that the $b$-$6j$ symbols satisfy the following \emph{tetrahedral symmetry}:
\begin{align}
\begin{Bmatrix} 
      a_1 & a_2 & a_3 \\
      a_4 & a_5 & a_6
   \end{Bmatrix}_b=\begin{Bmatrix} 
      a_2 & a_1 & a_3 \\
      a_5 & a_4 & a_6
   \end{Bmatrix}_b=\begin{Bmatrix} 
      a_1 & a_3 & a_2 \\
      a_4 & a_6 & a_5
   \end{Bmatrix}_b=\begin{Bmatrix} 
      a_1 & a_5 & a_6 \\
      a_4 & a_2 & a_3
   \end{Bmatrix}_b.
\end{align}
In addition, the $b $-$6j$ symbols also satisfy a \emph{reflection symmetry} that if we replace any number of $a_i=\frac{Q}{2}+\mathbf i\frac{l_i}{2\pi b} $
by its complex conjugate $\frac{Q}{2}- \mathbf i\frac{l_i}{2\pi b} $, then the value of the $b$-$6j$ symbol does not change. 
This fact is not obvious from Definition~\ref{def：6j}, but can be derived from~\cite[Equations (2.17) and (2.24)]{Teschner:2012em}. 
The reflection symmetry implies that the $b$-$6j$ symbols are real-valued functions.
\end{remark}}


\subsection{$\mathrm U_{q\tilde q}\mathfrak{sl}(2;\mathbb R)$ Turaev-Viro invariants: state-integral convergence, topological invariance and asymptotics} \label{subsec:1.2}
We now define a Turaev-Viro type invariant for $\mathrm U_{q\tilde q}\mathfrak{sl}(2;\mathbb R)$  using the $b$-$6j$ symbols. 
Let  $M$ be a compact $3$-manifold with nonempty boundary, and let $\mathcal T$ be an ideal triangulation of $M,$ with the set of edges $E$ and the set of tetrahedra $T.$
A $b$-coloring of $(M,\mathcal T)$ is an assignment of a complex number $a_e$ of the form $\frac{Q}{2}+\mathbf i\frac{l_e}{2\pi b}$ with $l_e\in\mathbb R_{{>0}}$ to each edge $e$ of $\mathcal T.$ Let $\boldsymbol a=\big(a_e\big)_{e\in E}.$ For each $e\in E,$ define
\begin{equation}
  |e|_{\boldsymbol a}=\big|S_b\big(2a_e\big)\big|^2=4\sinh \big(l_e\big)\sinh\bigg(\frac{l_e}{b^2}\bigg),
\end{equation}
{where the second equality can be found in~\cite[Equation (2.19) and (2.21)]{Ivan2021}.} For each $\Delta\in T,$ define
$$|\Delta|_{\boldsymbol a}=\bigg\{\begin{matrix} a_{e_1} & a_{e_2} & a_{e_3} \\ a_{e_4} & a_{e_5} & a_{e_6} \end{matrix} \bigg\}_b=\bigg\{\begin{matrix} \frac{Q}{2}+\mathbf i\frac{l_{e_1}}{2\pi b} & \frac{Q}{2}+\mathbf i\frac{l_{e_2}}{2\pi b} & \frac{Q}{2}+\mathbf i\frac{l_{e_3}}{2\pi b} \\ \frac{Q}{2}+\mathbf i\frac{l_{e_4}}{2\pi b} & \frac{Q}{2}+\mathbf i\frac{l_{e_5}}{2\pi b} & \frac{Q}{2}+\mathbf i\frac{l_{e_6}}{2\pi b} \end{matrix} \bigg\}_b,$$
where $\{e_1,\dots,e_6\}$ are the edges of $\mathcal T$ adjacent to $\Delta$ so that $e_1,e_2,e_3$ are the edges of a face of $\Delta$ and, for $i\in\{1,2,3\},$  $e_i$ and $e_{i+3}$ are opposite to each other.
Consider the following integral 
\begin{equation}\label{int}
\mathrm{TV}_b(M,\mathcal T) \doteq \int_{\big(\frac{Q}{2}+\mathbf i\mathbb R_{{>0}}\big)^E} \prod_{e\in E}|e|_{\boldsymbol a}\prod_{\Delta\in T}|\Delta|_{\boldsymbol a}d\mathrm{Im}\boldsymbol a,
\end{equation}
where $d\mathrm{Im}\boldsymbol a= \prod_{e\in E} d\mathrm{Im}a_e.$

There are two issues for this state-integral \eqref{int}  to define a topological invariant of $M$.
The first one is the convergence of the integral. 
To resolve this issue, we restrict to the class of Kojima ideal triangulations, which always exist for 
 hyperbolic 3-manifold with totally geodesic boundaries. See Section ~\ref{sec:2.8} for the precise construction. We prove the following convergence result.

\begin{theorem}\label{Converge} Let $M$ be a hyperbolic $3$-manifold with totally geodesic boundary, and let $\mathcal T$ be a Kojima ideal triangulation of $M.$  Then  for all $b\in(0,1),$ the integrand in $\mathrm{TV}_b(M,\mathcal T)$ is absolutely integrable.
\end{theorem}

The second issue is in general the independence of the state-integral with respective to the triangulations. 
In our case, it is the independence of $\mathrm{TV}_b(M,\mathcal T) $ with respective to the Kojima ideal triangulation $\mathcal T$, which is resolved in the next theorem.

\begin{theorem}\label{WD3} Let $M$ be a hyperbolic $3$-manifold with totally geodesic boundary. Then   for all $b\in (0,1)$ and  any  two Kojima ideal triangulations $\mathcal T_1$ and $\mathcal T_2$ of $M,$
$$\mathrm{TV}_b(M,\mathcal T_1)=\mathrm{TV}_b(M,\mathcal T_2).$$
\end{theorem}

\begin{definition}\label{VTV} Let $M$ be a hyperbolic $3$-manifold with totally geodesic boundary. For $b \in (0,1),$ the $b$-th \emph{$\mathrm U_{q\tilde q}\mathfrak{sl}(2;\mathbb R)$ Turaev-Viro invariant} of $M$ is defined by 
$$\mathrm{TV}_b(M)=\mathrm{TV}_b(M,\mathcal T)$$
for any Kojima ideal triangulation $\mathcal T$ of $M.$ 
\end{definition}

We are now ready to state our main result on the asymptotic of $\mathrm{TV}_b(M)$.

\begin{theorem}\label{VC} Let $M$ be a hyperbolic $3$-manifold with totally geodesic boundary. Then as $b\to 0,$
\begin{equation}\label{eq:1.7}
\mathrm{TV}_b(M)=(2\mathbf i)^{\frac{\chi(M)}{2}}  \frac{e^{\frac{-\mathrm{Vol}(M)}{\pi b^2}}}{\sqrt{\pm \mathrm{Tor}(DM, \mathrm{Ad}_{\rho_{DM}})}}\Big(1+O\big(b^2\big)\Big),
\end{equation}
where $\mathrm{Vol}(M)$ is the hyperbolic volume of $M,$ $\mathrm{Tor}(DM, \mathrm{Ad}_{\rho_{DM}})$ is the Reidemeister torsion of the double $DM$ of $M$ twisted by the adjoint action of  the double $\rho_{DM}$ of the holonomy representation of the hyperbolic structure on $M$, and $\chi(M)$ is the Euler characteristic of $M.$  
As a consequence, 
$$\lim_{b\to 0}\pi b^2\log \mathrm{TV}_b(M)=-\mathrm{Vol}(M).$$
\end{theorem}
We refer to Section~\ref{sec:2.9} for further background on the adjoint twisted Reidemeister torsion.
\bigskip

\begin{remark}
The verification of the convergence and the topological invariance presents novel difficulties when studying invariants defined by a state-integral. On the contrary, 
for the quantum invariants constructed from finite-dimensional irreducible representations of $\mathrm U_{q}\mathfrak{sl}(2;\mathbb C)$ such as the colored Jones polynomials, the Reshetikhin-Turaev invariants and the Turaev-Viro invariants which are defined by a finite sum, the convergence holds automatically  
and the topological invariance follows clearly from the properties of the building blocks such as the Jones-Wenzl projectors, the Kirby coloring and the 6j symbols. However, the study of their asymptotics and proving the corresponding Volume Conjectures appear much more challenging. 
\end{remark}

Our proof of Theorem \ref{Converge} is as follows. In Section~\ref{sec:2.8}, we prove that Kojima ideal triangulations admit angle structures; 
and in Section~\ref{sec:convergence} we prove the more general Theorem~\ref{angleconv} that the state-integral~\eqref{int} absolutely converges for all ideal triangulations that admit angle structures. As a special case, Theorem \ref{Converge} follows.

Theorem~\ref{WD3} is proved in Section \ref{sec:invariance}.
Recall  for the Turaev-Viro invariant for finite-dimensional irreducible representations of  $\mathrm U_{q}\mathfrak{sl}(2;\mathbb C)$  that the topological invariance follows directly from the orthogonality and  the pentagon equation of the $6j$ symbol, corresponding respectively to the invariance under  the $0$-$2$ and $2$-$3$ Pachner moves. 
In our case, although the orthogonality and  the pentagon equation still hold, the state-integral may not absolutely converge for intermediate ideal triangulations, ruining the invariance under the Pachner moves. 
We resolve this issue  by carefully choosing a sequence of Pachner moves between Kojima ideal triangulations such that each intermediate ideal triangulation supports angle structures, and the absolute convergence follows from Theorem~\ref{angleconv}.
We tend to believe that for all triangulations where the state-integral converges absolutely, the value should agree with our invariant.    

The proof of Theorem \ref{VC}, as carried out in Section \ref{sec:4.3}, uses a saddle point approximation as in the proof of Theorems~\ref{cov} and~\ref{cov2}. To verify the conditions of the saddle point approximation, we use the results by Luo-Yang \cite{LY} to identify the critical point for the integrand with the hyperbolic structure (see Section \ref{anglestr} for a detailed background), the Murakami-Yano volume formula~\cite{MY, U, MU, BY} to identify the critical value with the negative of the hyperbolic volume, 
and Wong-Yang formula~\cite{WY3} to identify the Hessian determinant for the integrand at the critical point with the adjoint twisted Reidemeister torsion. 
Finally, by carefully choosing an integral contour, we obtain Theorem \ref{VC}. It is worth noticing that two key ingredients of our proof, the Murakami-Yano and the Wong-Yang formulae originate from the asymptotics of the quantum invariants constructed from the finite-dimensional irreducible representations of $\mathrm U_{q}\mathfrak{sl}(2;\mathbb C)$.


\subsection{Outlook and perspectives}\label{subsec:outlook}

\noindent{\bf Towards the Teichm\"uller/Virasoro TQFT.}  Our invariant is closely related to several established $3$-manifold invariants, including those defined by Andersen-Kashaev~\cite{AK} and Kashaev-Luo-Vartanov~\cite{KLV} in the mathematics literature, Collier-Eberhardt-Zhang~\cite{CEZ} and Hartman~\cite{Hartman1,Hartman2} in  the physics literature. We believe that these invariants come from a single 3D topological quantum field theory, known as the Teichm\"uller TQFT~\cite{AK} or the Virasoro TQFT~\cite{CEZ}. The significance of this TQFT in studying 3D  quantum gravity, non-compact Chern-Simons theory, and the quantum Teichm\"uller theory is detailed in~\cite{AK,CEZ} and references therein. We plan to extend our invariants to cusped and closed hyperbolic $3$-manifolds, completing the construction of the TQFT. We also plan to clarify their relationship with the invariants defined in~\cite{AK,KLV,CEZ}, and to establish their connection with $3$-dimensional hyperbolic geometry in analogy to Theorem \ref{VC}.  

\bigskip

\noindent{\bf Connection to Liouville conformal field theory.} Under the 3D TQFT/2D CFT correspondence~\cite{Wit89},
the 2D conformal field theory corresponding to our invariant is the Liouville CFT~\cite{Verlinde:1989ua}, which originated in Polyakov's path integral formulation of bosonic string theory~\cite{Pol81}. 
A key datum in the conformal bootstrap formulation of Liouville CFT, the Virasoro conformal blocks,  is supposed to provide a canonical basis for the Teichmuller/Virasoro TQFT.
It was conjectured by Ponsot and Teschner~\cite{ponsot1999liouville} that our $b$-$6j$ symbol is the fusion kernel of the Virasoro conformal blocks.
As a recent major breakthrough in mathematical physics, Liouville CFT has been rigorously constructed via probabilistic methods and the conformal bootstrap is established in this framework~\cite{DKRV16,KRV19b,GKRV20,GKRV21}.
Furthermore, Ponsot-Teschner's conjecture on the $b$-$6j$ symbol is proved in~\cite{GRSS}. 
We plan to investigate the relation to the Teichmuller/Virasoro TQFT and the Liouville CFT, and use this connection and tools from Liouville CFT to better study the Teichmuller/Virasoro TQFT.

\bigskip

\noindent{\bf Volume Conjectures for state-sum invariants.} 
The $6j$-symbols for the finite-dimensional irreducible representations of $\mathrm U_{q}\mathfrak{sl}(2,\mathbb C)$~\cite{Kirillov:191317} arise as residues of the meromorphic extension of the $b$-$6j$ symbols~\cite{Pawelkiewicz:2013wga}. 
We plan to use this fact as a starting point to investigate the relationship between our invariants and those constructed from the finite-dimensional irreducible representations of $\mathrm U_{q}\mathfrak{sl}(2,\mathbb C)$ such as the colored Jones polynomials, the Reshetikhin-Turaev and the Turaev-Viro invariants, with the hope that our 
Theorem~\ref{VC} can shed light on the Volume Conjectures for those invariants.

\bigskip

\noindent{\bf Casson Conjecture for angle structures.} 
As proved in Theorem \ref{angleconv}, the state-integral $\mathrm{TV}_b(M,\mathcal T)$ converges for any angled ideal triangulation $\mathcal T$ of $M$; and by the same argument as in the proof of Theorem \ref{VC}, we can prove that the exponential decay rate of $\mathrm{TV}_b(M,\mathcal T)$ is at least the volume of any angle structure on $(M,\mathcal T).$ We believe that for any two angled ideal triangulations $\mathcal T_1$ and $\mathcal T_2,$ the two state-integrals  $\mathrm{TV}_b(M,\mathcal T_1)$ and  $\mathrm{TV}_b(M,\mathcal T_2)$ should coincide, which is stronger than the current Theorem~\ref{WD3}. If this is the case, then it will confirm the long-standing Casson Conjecture that the volume of any angle structure on $(M,\mathcal T)$ is at most the hyperbolic volume of the $M.$ Such a relationship between invariants that exponentially decay and the Casson Conjecture was first observed by Ben Aribi-Wong\,\cite{aribi2024andersen} in their study of the asymptotics of the Andersen-Kashaev Teichm\"uller TQFT invariants. 
\bigskip

\noindent\textbf{Acknowledgments.} We thank Lorentz Eberhardt,  Igor Frenkel,  Ling-Yan Hung, Rinat Kashaev, Feng Luo, Nicolai Reshetikhin, Joerg Teschner, and Ka Ho Wong for helpful discussions. We also thank Thomas Hartman for bringing our attention to his work~\cite{Hartman1,Hartman2}. T.L., X.S., and B.W.\ are supported by National Key R\&D Program of China (No.\ 2023YFA1010700). S.M.\ is supported by National Natural Science Foundation of China (No.12371124). T.Y.\ is supported by NSF Grants DMS-2203334 and DMS-2505908.


\section{Preliminaries}\label{sec:pre}

In this section, we provide background on various special functions and on hyperbolic 3-manifolds that are needed for a precise understanding of our main results and their proofs. The results in the first three subsections provide a fine asymptotic estimate of the double sine functions, which will be used in Section \ref{sec:tetrahedron}. Section~\ref{subsec:contour}  clarifies the choice of the contour in the definition of the $b$-$6j$ symbol. Section \ref{extcov} lists relevant geometric properties of the truncated hyperbolic tetrahedra. Sections \ref{HPM} and \ref{sec:2.8} introduce generalized hyperbolic polyhedral metrics, which will be used in Section \ref{sec:4.3}. Section \ref{anglestr} contains necessary results about angle structures for proving the Theorem \ref{VC} in Section \ref{sec:4.3}. Section \ref{sec:2.9} recalls the twisted adjoint Reidemeister torsions, and Section~\ref{subsec:saddle} recalls the saddle point approximations.


\subsection{Double sine function}
Let $S_b$ be the double sine function  defined in~\eqref{eq:def-S}. 
In the rest of this paper, we will intensively use the function $\log S_b\Big(\frac{x}{\pi b}+\frac{b}{2}\Big)$. 
For  $-\frac{\pi b^2}{2}<\mathrm{Re}x<\pi + \frac{\pi b^2}{2}$ so that $0<\mathrm{Re}\Big(\frac{x}{\pi b}+\frac{b}{2}\Big)<Q,$ we have 
$$\log S_b\Big(\frac{x}{\pi b}+\frac{b}{2}\Big)=\int_\Omega\frac{\sinh\Big(\big(\frac{1}{2b}-\frac{x}{\pi b}\big)t\Big)}{4t\sinh(\frac{bt}{2})\sinh(\frac{t}{2b})}dt.$$
The function $\log S_b\Big(\frac{x}{\pi b}+\frac{b}{2}\Big)$ can be holomorphically extended to the region $\mathbb C\setminus (-\infty,0]\cup[\pi,\infty)$. 
To achieve the extension, we use (\ref{FE1}) to get
$$S_b\Big(\frac{x+\pi }{\pi b} +\frac{b}{2}\Big)=2\cos\Big(\frac{x}{b^2}\Big)S_b\Big(\frac{x}{\pi b} +\frac{b}{2}\Big).$$
If $\mathrm{Re}\big(\frac{x}{b^2}\big) \in (-\pi, \pi)$ and $\mathrm{Im}x>0,$  then 
$$ \log\Big(2\cos \big(\frac{x}{b^2}\big)\Big)=-\frac{\mathbf ix}{b^2}+\log\Big(1+e^{\frac{2\mathbf ix}{b^2}}\Big),$$
which is holomorphic for all $x\in\mathbb C$ with $\mathrm{Im}x>0.$ Similarly, if $\mathrm{Re} \big(\frac{x}{b^2} \big)\in (-\pi, \pi)$ and $\mathrm{Im}x<0,$  then 
$$ \log\Big(2\cos \big(\frac{x}{b^2}\big)\Big)=\frac{\mathbf ix}{b^2}+\log\Big(1+e^{-\frac{2\mathbf ix}{b^2}}\Big),$$
which is holomorphic for all $x\in \mathbb C$ with $\mathrm{Im}x <0.$ Therefore, for $x\in\mathbb C\setminus (-\infty,0]\cup[\pi,\infty)$ with $\mathrm{Im}x>0,$ we can extend $\log S_b\Big(\frac{x}{\pi b}+\frac{b}{2}\Big)$ holomorphically  by the functional equation 
\begin{equation}\label{FE3}
\log S_b\Big(\frac{x+\pi}{\pi b}+\frac{b}{2}\Big)=\log S_b\Big(\frac{x}{\pi b}+\frac{b}{2}\Big) - \frac{\mathbf ix}{b^2}+\log\Big(1+e^{\frac{2\mathbf ix}{b^2}}\Big);
\end{equation}
and  for $x\in\mathbb C\setminus (-\infty,0]\cup[\pi,\infty)$ with $\mathrm{Im}x<0,$ we can extend $\log S_b\Big(\frac{x}{\pi b}+\frac{b}{2}\Big)$ holomorphically by the functional equation 
\begin{equation}\label{FE4}
\log S_b\Big(\frac{x+\pi }{\pi b}+\frac{b}{2}\Big)=\log S_b\Big(\frac{x}{\pi b}+\frac{b}{2}\Big) + \frac{\mathbf ix}{b^2}+\log\Big(1+e^{-\frac{2\mathbf ix}{b^2}}\Big).
\end{equation}


\subsection{Dilogarithm function}\label{dilog}

The double sine function $S_b$ is closely related to the dilogarithm function $\mathrm{Li}_2,$ which we will recall in this subsection.  Let $\log:\mathbb C\setminus (-\infty, 0]\to\mathbb C$ be the standard logarithm function defined by
$$\log z=\log|z|+\mathbf i\arg z$$
with $-\pi<\arg z<\pi.$ The \emph{dilogarithm function} $\mathrm{Li}_2: \mathbb C\setminus (1,\infty)\to\mathbb C$ is defined by
$$\mathrm{Li}_2(z)=-\int_0^z\frac{\log (1-u)}{u}du$$
where the integral is along any path in $\mathbb C\setminus (1,\infty)$ connecting $0$ and $z,$ which is holomorphic in $\mathbb C\setminus [1,\infty)$ and continuous in $\mathbb C\setminus (1,\infty).$ 
The dilogarithm function satisfies the follow properties (see eg. Zagier\,\cite{Z})
\begin{equation}\label{Li2}
\mathrm{Li}_2\Big(\frac{1}{z}\Big)=-\mathrm{Li}_2(z)-\frac{\pi^2}{6}-\frac{1}{2}\big(\log(-z)\big)^2.
\end{equation} 
In the unit disk $\big\{z\in\mathbb C\,\big|\,|z|<1\big\},$ 
$$\mathrm{Li}_2(z)=\sum_{n=1}^\infty\frac{z^n}{n^2}.$$

Next we recall some properties of the function $\mathrm{Li}_2\big(e^{2\mathbf ix}\big)$ which we will make an intensive use in the rest of the paper.  For $x\in \mathbb C$  either with $\mathrm{Re}x\in (0,\pi)$ or with $\mathrm{Re}x=0$ or $\pi$ and $\mathrm{Im}x>0$ so that $e^{2\mathbf ix}\in\mathbb C\setminus [1,+\infty),$ the value $\mathrm{Li}_2\big(e^{2\mathbf ix}\big)$ can be calculated by the following integral 
\begin{equation}\label{ID}
\mathrm{Li}_2\big(e^{2\mathbf ix}\big)=2\pi \mathbf i \int_\Omega \frac{e^{(2x-\pi)t}}{4\pi t^2\sinh(\pi t)}dt,
\end{equation}
where $\Omega$ is the same contour as before (see eg. \cite[Lemma 2.3]{WY2}). Since $\mathrm{Li}_2(z)$ is analytic for $z$ in  the unit disk, we have 
for $x\in\mathbb C$ with $\mathrm{Im}x>0$ so that $e^{2\mathbf i x}$ is in the unit disk that 
\begin{equation}\label{period1}
\mathrm{Li}_2\big(e^{2\mathbf i(x+\pi)}\big)=\mathrm{Li}_2\big(e^{2\mathbf ix}\big).
\end{equation}
By (\ref{Li2}), we have for $x\in\mathbb C$ with $\mathrm{Re}x\in (0,\pi) $ that 
$$\mathrm{Li}_2\big(e^{2\mathbf ix}\big)=\mathrm{Li}_2\big(e^{-2\mathbf ix}\big)+2x^2-2\pi x +\frac{\pi ^2}{3},    $$
and together with (\ref{period1}) we have for $x\in\mathbb C$ with $\mathrm{Re}x\in (0,\pi)$ and $\mathrm{Im}x<0$ that 
\begin{equation}\label{period2}
\mathrm{Li}_2\big(e^{2\mathbf i(x+\pi)}\big)=\mathrm{Li}_2\big(e^{2\mathbf ix}\big)+4\pi x.
\end{equation}
Using (\ref{period1}) and (\ref{period2}), we can extend $\mathrm{Li}_2\big(e^{2\mathbf ix}\big)$ to a continuous function on $\mathbb C\setminus (-\infty,0)\cup(\pi,\infty)$ which is holomorphic on $\mathbb C\setminus (-\infty,0]\cup[\pi,\infty).$  

The following closely related will be used intensively in the rest of the paper:
\begin{equation}\label{eq:Lx}
L(x)=x^2-\pi x +\frac{\pi^2}{6}-\mathrm{Li}_2\big(e^{2\mathbf ix}\big).
\end{equation}
This function is continuous on $\mathbb C\setminus (-\infty,0)\cup(\pi,\infty)$ and  holomorphic on $\mathbb C\setminus (-\infty,0]\cup[\pi,\infty),$ with the derivative \begin{equation}\label{derivative}
L'(x)= 2x-\pi+2\mathbf i\log(1-e^{2\mathbf i x}).
\end{equation}

By  (\ref{period1}) and (\ref{period2}), we have for $x\in\mathbb C\setminus (-\infty,0]\cup[\pi,\infty)$ with $\mathrm{Im}x>0$ that 
\begin{equation}\label{period3}
L(x+\pi)=L(x)+2\pi x,
\end{equation}
and for $x\in \mathbb C\setminus (-\infty,0]\cup[\pi,\infty)$ with $\mathrm{Im}x< 0$ that 
\begin{equation}\label{period4}
L(x+\pi)=L(x)-2\pi x.
\end{equation}


\subsection{Relationship between double sine function and dilogarithm function}\label{SbLi}
We  now approximate $S_b\Big(\frac{x}{\pi b} +\frac{b}{2}\Big)$ using $L(x)$. Let $\nu_b(x)$ be such that
$$2\pi \mathbf i b^2 \log S_b\Big(\frac{x}{\pi b} +\frac{b}{2}\Big)-L(x)=\nu_b(x)b^4.$$ For $\delta>0$ and $K>0,$  as depicted in Figure~\ref{DdK}, let 
$$H_{\delta,K}=\big\{ x\in \mathbb C\ |\ \text{either }\delta \leqslant \mathrm{Re}x\leqslant \pi-\delta, \text{ or } |\mathrm{Re}x|\leqslant K \text{ and } |\mathrm{Im}x| \geqslant \delta \big\}.$$

\begin{figure}[htbp]
\centering
\includegraphics[scale=0.4]{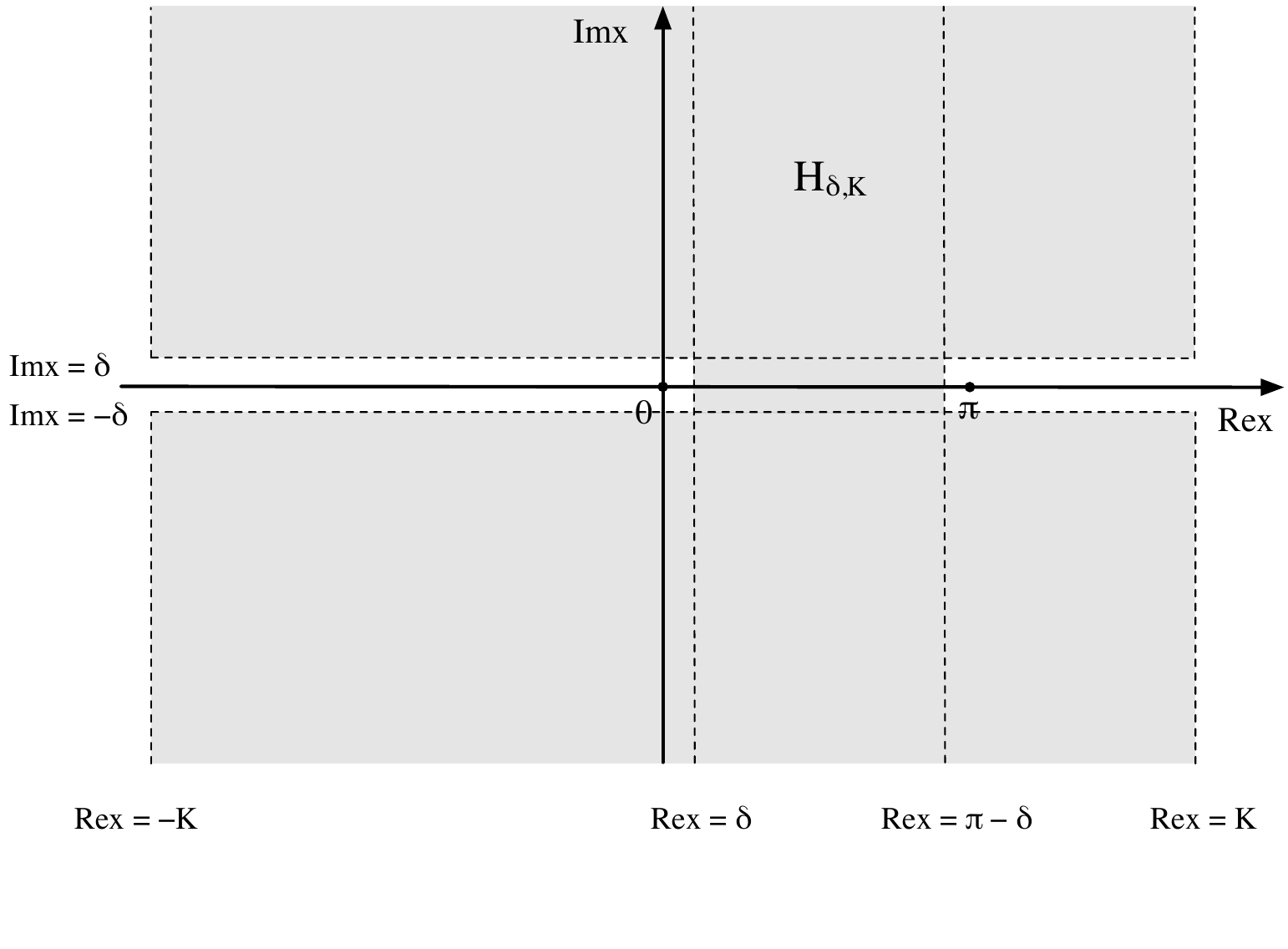}
\caption{Region $H_{\delta,K}$}
\label{DdK}
\end{figure}

\begin{proposition}\label{EST}  There exists a $B=B_{\delta,K}>0$  independent of $b$ such that for all $b\in(0,1)$ and for all  $x$ in $H_{\delta, K},$ 
$$|\nu_b(x)|< B,$$
i.e., 
$$\bigg|2\pi \mathbf i b^2 \log S_b\Big(\frac{x}{\pi b} +\frac{b}{2}\Big)-L(x)\bigg|<Bb^4.$$
\end{proposition}

For the proof of Proposition \ref{EST}, we 
will need \emph{Faddeev's quantum dilogarithm function}
$$\Phi_b(z)\doteq\exp\bigg(\int_\Omega \frac{e^{-2\pi\mathbf ibzt}}{4t\sinh(\pi t)\sinh(\pi t b^2)}dt\bigg).$$
By \cite[A4, A5, A13]{Teschner:2012em}, we have $$\Phi_b(z)=S_b\Big(\mathbf iz+\frac{Q}{2}\Big)e^{\frac{\pi \mathbf i z^2}{2}+\frac{\pi \mathbf i}{24}(b^2+b^{-2})}.$$
We need the following Lemma \ref{BGP} and Lemma \ref{Wu}.

\begin{lemma}\label{BGP} For $\delta>0,$ there exists a $B_{\delta}>0$ independent of $b$ such that for all $b\in(0,1)$ and all $y\in\mathbb R+ \mathbf i[-\pi + \delta,\pi-\delta],$ 
\begin{equation}\label{norm}
\Big|\log \Phi_b\big(\frac{y}{2\pi b}\big)-\frac{1}{2\pi \mathbf i b^2}\mathrm{Li}_2\big(-e^y\big)\Big|<B_\delta b^2.
\end{equation}
\end{lemma}

\begin{proof}  The proof relies on the proof of \cite[Lemma 7.13 and Lemma 7.14]{BGP} that for $\delta>0,$ there exists a $B_\delta>0$  independent of $b$ such that for all $b\in(0,1)$ and  for all $y\in\mathbb R\pm \mathbf i[\delta,\pi-\delta],$ 
$$\Big|\mathrm{Re}\Big(\log \Phi_b\big(\frac{y}{2\pi b}\big)-\frac{1}{2\pi \mathbf i b^2}\mathrm{Li}_2\big(-e^y\big)\Big)\Big|<B_\delta b^2,$$
where one sees that this stronger statement in Lemma \ref{BGP} actually holds. Indeed, in the proof, the authors of \cite{BGP} bounded the real part of the difference by bounding its norm, and they only used the condition that $ |\mathrm{Im}y| \leqslant \pi-\delta.$
\end{proof}

\begin{lemma}\label{Wu} For $\delta>0,$ there exists a $B_{\delta}>0$ independent of $b$ such that for all $b\in(0,1)$ and all $y\in\mathbb R\setminus (-\delta,\delta)+\mathbf i[-\pi,\pi],$ 
\begin{equation}\label{norm2}
\Big|\log \Phi_b\big(\frac{y}{2\pi b}\big)-\frac{1}{2\pi \mathbf i b^2}\mathrm{Li}_2\big(-e^y\big)\Big|<B_\delta b^2.
\end{equation}\end{lemma}

\begin{proof} We first consider the case that $\mathrm{Re}y\leqslant -\delta$ and $|\mathrm{Im}y|\leqslant \pi.$ In this case $\mathrm{Im}\frac{y}{2\pi b}\in[-\frac{1}{2b},\frac{1}{2b}]\subset[-\frac{Q}{2},\frac{Q}{2}],$ and by the change of variable $w=\pi b t,$ 
$$\log\Phi_b\big(\frac{y}{2\pi b}\big)=\int_\Omega \frac{e^{-\mathbf i\frac{yw}{ \pi b}}}{4w\sinh(wb)\sinh(wb^{-1})}dw=\int_\Omega \frac{e^{-\mathbf iyt}}{4t\sinh(\pi t)\sinh(\pi t b^2)}dt.$$
If we let $x=\frac{\pi}{2}-\mathbf i\frac{y}{2},$ then $\mathrm{Re}x\in[0,\pi]$ and $\mathrm{Im}x >\frac{\delta}{2},$ and (\ref{ID})  implies that 
$$\frac{1}{2\pi \mathbf i b^2}\mathrm{Li}_2\big(-e^y\big)=\int_\Omega \frac{e^{-\mathbf iyt}}{4\pi b^2 t^2\sinh(\pi t)}dt,$$
and hence 
$$\log\Phi_b\big(\frac{y}{2\pi b}\big)-\frac{1}{2\pi \mathbf i b^2}\mathrm{Li}_2\big(-e^y\big)=\int_\Omega \frac{e^{-\mathbf iyt}}{4t\sinh(\pi t)}\Big(\frac{1}{\sinh(\pi b^2 t)}-\frac{1}{\pi b^2 t}\Big)dt.$$
Now we specify the integral contour $\Omega$ to  $(-\infty,-1]\cup C \cup [1,\infty),$ where $C=\{ e^{\mathbf i\theta} \ |\ \theta\in [0,\pi]\} $ is the upper semi-unitcircle centered at $0.$ 

We first estimate the integral on $C.$ By \cite[Proof of Lemma 3]{AH} with $\zeta = -\mathbf iy,$ $\gamma = \pi b^2$ and $R=1$ therein,  we have 
\begin{equation}\label{IC}
\bigg|\int_{C} \frac{e^{-\mathbf iyt}}{4t\sinh(\pi t)}\Big(\frac{1}{\sinh(\pi b^2 t)}-\frac{1}{\pi b^2 t}\Big)dt\bigg|\leqslant B_1 \Big(1+e^{\mathrm{Re}y }\Big)b^2\leqslant 2 B_1 b^2,
\end{equation}
where $B_1$ is the constant  $B_R$ in \cite{AH} with $R=1$ and the last inequality comes from the assumption that $\mathrm{Re}y\leqslant -\delta <0.$ 

Next we estimate the integral on $[1,\infty).$ Let 
$$\mathrm I^+ (y)= \int _ 1^\infty \frac{e^{-\mathbf iyt}}{4t\sinh(\pi t)}\Big(\frac{1}{\sinh(\pi b^2 t)}-\frac{1}{\pi b^2 t}\Big)dt.$$ 
Then using the identity 
$$\frac{e^{-\mathbf iyt}}{4\sinh(\pi t)}=\bigg(\frac{e^{-\pi t}}{4\sinh(\pi t)}+\frac{1}{2}\bigg)e^{(-\mathbf iy -\pi)t},$$
we can write $\mathrm I^ + (y) =\mathrm I_1^+ (y)+\mathrm I_2^+ (y) ,$
where
$$\mathrm I_1^+(y)=\int_1^\infty \frac{e^{-\pi t}}{4\sinh(\pi t)}e^{(-\mathbf iy -\pi)t}\Big(\frac{1}{\sinh(\pi b^2 t)}-\frac{1}{\pi b^2 t}\Big)\frac{dt}{t}$$
and
$$\mathrm I_2^+ (y)= \int _1^\infty \frac{1}{2}e^{(-\mathbf iy-\pi)t}\Big(\frac{1}{\sinh(\pi b^2 t)}-\frac{1}{\pi b^2 t}\Big)\frac{dt}{t}.$$
For $\mathrm I_1^+(y),$ we use the following facts: (1) $|\big(\frac{1}{\sinh x}-\frac{1}{x})\frac{1}{x}|\leqslant \frac{1}{6}$ for all $x\in \mathbb R,$ (2) under the assumption that 
$|\mathrm{Im}y|\leqslant \pi,$ $\big|e^{(-\mathbf iy -\pi)t}\big| = e^{(\mathrm{Im}y-\pi)t}\leqslant 1$ for $t>0,$ and (3) $\int_1^\infty \frac{e^{-\pi t}}{\sinh(\pi t)}dt \leqslant  \frac{1}{\pi}\int_1^\infty e^{-\pi t}dt =\frac{e^{-\pi}}{\pi^2 }.$ Then we have
\begin{equation}\label{I1}
|\mathrm I_1^+(y)|\leqslant \pi b^2\int_1^\infty \frac{e^{-\pi t}}{4\sinh(\pi t)}\Big|e^{(-\mathbf iy -\pi)t}\Big|\Big|\Big(\frac{1}{\sinh(\pi b^2 t)}-\frac{1}{\pi b^2 t}\Big)\frac{1}{\pi b^2 t}\Big|dt \leqslant  \frac{e^{-\pi}}{24 \pi} b^2.
\end{equation}
For $\mathrm I_2^+(y),$ using the change of variable $x=\pi b^2 t,$ we have 
$$\mathrm I_2^+(y) = \int _{\pi b^2}^\infty \frac{e^{(-\mathbf iy-\pi)\frac{x}{\pi b^2}}}{2}\Big(\frac{1}{\sinh x}-\frac{1}{x}\Big)\frac{dx}{x}.$$ To simplify the notations, we let $\lambda = \lambda(y)= (-\mathbf iy-\pi)\frac{1}{\pi b^2}$ and $f(x)=\big(\frac{1}{\sinh x}-\frac{1}{x})\frac{1}{x}.$ By the assumption that $\mathrm{Re}y\leqslant -\delta,$ we have $|\lambda|\geqslant |\mathrm{Im}\lambda| =\frac{|\mathrm{Re}y|}{\pi b^2}\geqslant \frac{\delta}{\pi b^2};$ and by the assumption that 
$|\mathrm{Im}y|\leqslant \pi,$ we have $\mathrm{Re}\lambda \leqslant 0.$ By the integration by parts, we have 
$$\mathrm I_2^+(y)= \int _{\pi b^2}^\infty \frac{e^{\lambda x}}{2}f(x)dx = \frac{e^{\lambda x}f(x)}{2\lambda} \bigg | _{\pi b^2}^\infty -  \frac{1}{2\lambda}\int _{\pi b^2}^\infty e^{\lambda x}f'(x)dx.$$
By a direct computation, we see that  $f'(x)>0$ for all $x >0.$ Also, as $\mathrm{Re}\lambda\leqslant 0$ and  $\lim _{x \to \infty} f(x)=0,$ we have $\big|e^{\lambda x}\big|\leqslant 1$  for $x>0$ and $\lim _{x\to\infty} e^{\lambda x}f(x)=0.$  As a consequence, we have
\begin{equation}\label{Il2}
|\mathrm I_2^+(y)|\leqslant \frac{|f(\pi b^2) |}{2|\lambda|}+\frac{1}{2|\lambda|}\bigg|\int_{\pi b^2}^\infty f'(x)dx\bigg|=\frac{|f(\pi b^2)|}{|\lambda|} \leqslant \frac{\pi}{6\delta} b^2,
\end{equation}
where the last inequality comes from the facts that $|f(x)|\leqslant \frac{1}{6}$ for all $x\in \mathbb R$ and $|\lambda|\geqslant \frac{\delta}{\pi b^2}.$  Putting (\ref{I1}) and (\ref{Il2}) together, we have 
\begin{equation}\label{I+}
\mathrm I^+ (y)\leqslant \bigg(\frac{e^{-\pi}}{24\pi} +  \frac{ \pi}{6\delta} \bigg)b^2.
\end{equation}

Finally, for the estimate of the integral on $(-\infty,-1],$ let 
$$\mathrm I^- (y)= \int _{-\infty}^{-1} \frac{e^{-\mathbf iyt}}{4t\sinh(\pi t)}\Big(\frac{1}{\sinh(\pi b^2 t)}-\frac{1}{\pi b^2 t}\Big)dt.$$ Then by the change of variable $t \mapsto -t$ we see that  $\mathrm I^- (y) = \mathrm I^+(-y).$ Since $|\mathrm{Im}y|\leqslant \pi$ and $|\mathrm {Re}y|\geqslant \delta,$ we have $\big| e^{\lambda(-y)x}\big|=e^{(-\mathrm {Im}y-\pi)x}\leqslant 1$ for $x>0$  and $|\lambda(-y)| \geqslant  \frac{|\mathrm{Re}y|}{\pi b^2}\geqslant \frac{\delta}{\pi b^2}.$ 
As a consequence, we still have $|\mathrm I_1^+(-y) | \leqslant  \frac{e^{-\pi}}{24 \pi} b^2$ and $|\mathrm I_2^+(-y)|\leqslant \frac{\pi}{6\delta} b^2,$ which implies that 
\begin{equation}\label{I-}
|\mathrm I^- (y)| = |\mathrm I^+ (-y)|  =   |\mathrm I_1^+ (-y)|  +  |\mathrm I_2^+ (-y)|   \leqslant \bigg(\frac{e^{-\pi}}{24 \pi} +  \frac{ \pi}{6\delta} \bigg)b^2.
\end{equation}
Putting (\ref{IC}), (\ref{I+}) and (\ref{I-}) together, we have
$$\Big|\log \Phi_b\big(\frac{y}{2\pi b}\big)-\frac{1}{2\pi \mathbf i b^2}\mathrm{Li}_2\big(-e^y\big)\Big|<\bigg(2B_1+\frac{e^{-\pi}}{12\pi}+\frac{\pi}{3\delta}\bigg) b^2.$$ 
\\

Next, we consider the case that $\mathrm{Re}y\geqslant \delta$ and $|\mathrm{Im}y|\leqslant \pi.$ By \cite[Proof of Lemma 7.12]{BGP}, for $y\in \mathbb C$ with $\mathrm{Re}y\geqslant \delta$ and $|\mathrm{Im}y|<\pi,$ one has 
$$\log \Phi_b\big(\frac{y}{2\pi b}\big)-\frac{1}{2\pi \mathbf i b^2}\mathrm{Li}_2\big(-e^y\big)=\overline{\log \Phi_b\big(\frac{-\overline y}{2\pi b}\big)-\frac{1}{2\pi \mathbf i b^2}\mathrm{Li}_2\big(-e^{-\overline y}\big)-\frac{\mathbf i\pi}{12}b^2},$$
which by continuity of $\Phi_b$ and $\mathrm{Li}_2$ also holes for $y$ with $\mathrm{Re}y\geqslant \delta$ and $|\mathrm{Im}y|\leqslant \pi.$ As $-\overline y $ is in the previous case, we have 
$$\Big|\log \Phi_b\big(\frac{y}{2\pi b}\big)-\frac{1}{2\pi \mathbf i b^2}\mathrm{Li}_2\big(-e^y\big)\Big|<\bigg(2B_1+\frac{e^{-\pi}}{12\pi}+\frac{\pi}{3\delta}+\frac{\pi}{12}\bigg) b^2.$$

Letting $B_\delta=2B_1+\frac{e^{-\pi}}{12\pi}+\frac{\pi}{3\delta}+\frac{\pi}{12},$ we have the result.
\end{proof}

\begin{proof}[Proof of Proposition \ref{EST}]
Let
$$ \varphi^b(z) = 2\pi \mathbf ib^2 \int_\Omega \frac{e^{(2z-\pi)t}}{4t\sinh(\pi t)\sinh(\pi tb^2)}dt$$
so that
\begin{equation}\label{Psiphi}
\Phi_b(z)=\exp\bigg(\frac{1}{2\pi \mathbf ib^2}\varphi^b(\frac{\pi}{2}-\mathbf{i}\pi bz)\bigg);
\end{equation}

By (\ref{norm}) and (\ref{norm2}),  
and letting $x=\frac{\pi}{2}-\mathbf i\frac{y}{2},$  we have  for all  $x$ with either  $\delta \leqslant \mathrm{Re}x\leqslant \pi-\delta,$ or $0\leqslant \mathrm{Re}x\leqslant \pi$ and $|\mathrm{Im}x| \geqslant \delta,$ that
$$\Big|\varphi^b(x)-\mathrm{Li}_2\big(e^{2\mathbf ix}\big)\Big|<2\pi B_{2\delta}b^4.$$
Then by 
(\ref{eq:Lx}),
\begin{equation}\label{C1}
\bigg|2\pi \mathbf i b^2 \log S_b\Big(\frac{x}{\pi b} +\frac{b}{2}\Big)-L(x)\bigg|<\bigg(2\pi B_{2\delta}+\frac{\pi^2}{12}\bigg)b^4.
\end{equation}

For the  case that $|\mathrm{Re}x| \leqslant K$ and $|\mathrm{Im}x|\geqslant \delta,$ by  the functional equations (\ref{FE3}), (\ref{FE4}), (\ref{period3}) and (\ref{period4}), we have 
\begin{equation*}
\bigg|2\pi \mathbf i b^2 \log S_b\Big(\frac{x+\pi }{\pi b} +\frac{b}{2}\Big)-L(x+\pi)\bigg|<\bigg|2\pi \mathbf i b^2 \log S_b\Big(\frac{x}{\pi b} +\frac{b}{2}\Big)-L(x)\bigg|+2\pi b^2\bigg|\log\Big(1+e^{\pm \frac{2\mathbf ix}{b^2}}\Big)\bigg|
\end{equation*}
where the exponent of the last term is $\frac{\mathbf ix}{b^2}$ if $\mathrm{Im}x>0$ and is $-\frac{\mathbf ix}{b^2}$ if $\mathrm{Im}x<0.$ Now for $\mathrm{Im}x\geqslant \delta,$ we have
\begin{equation*}
\bigg|\log\Big(1+e^{ \frac{2\mathbf ix}{b^2}}\Big)\bigg|< e^{-\frac{2\delta}{b^2}}<\frac{b^2}{2\delta};
\end{equation*}
and for $\mathrm{Im}x\leqslant -\delta,$ we have
\begin{equation*}
\bigg|\log\Big(1+e^{ -\frac{2 \mathbf i x}{b^2}}\Big)\bigg|<e^{-\frac{2\delta}{b^2}}<\frac{b^2}{2\delta}.
\end{equation*}
Together with  (\ref{C1}) and by letting $B=2\pi B_{2\delta}+\frac{\pi^2}{12}+\frac{N}{2\delta}$  with an $N$ such that $N\delta>K,$  we have 
\begin{equation*}
\bigg|2\pi \mathbf i b^2 \log S_b\Big(\frac{x}{\pi b} +\frac{b}{2}\Big)-L(x)\bigg|< Bb^4.
\end{equation*}
This completes the proof.
\end{proof}

\subsection{Contour of integration in the $b$-$6j$ symbol}\label{subsec:contour}
The following lemma gives us more flexibility in choosing the contour for computing the $b$-6$j$ symbols. 
\begin{lemma}\label{lem:contour}
Given $L>0$ and  $c\in [\frac{3Q}{2}, 2Q]$, let 
$\Gamma^*$ be a contour satisfying the following condition:  $\mathrm {Re} u=c$ for $u\in \Gamma^*$ with $|\mathrm {Im} u|\geqslant L$; and $\mathrm{Re} u \in (\frac{3Q}{2}, 2Q)$ for $u\in \Gamma^*$ with $|\mathrm {Im} u|< L$. Then for $(a_1,\dots, a_6)  \in \big(\frac{Q}{2} + \mathbf i {\mathbb R_{> 0}}\big)^6$ and for $L$ sufficiently large, we have 
\begin{equation*}
\bigg\{\begin{matrix} a_1 & a_2 & a_3 \\ a_4 & a_5 & a_6 \end{matrix} \bigg\}_b={{-\mathbf i}}\Bigg(\frac{1}{\prod_{i=1}^4\prod_{j=1}^4S_b(q_j-t_i)}\Bigg)^{\frac{1}{2}}\int_{\Gamma^*} \prod_{i=1}^4S_b(u-t_i)\prod_{j=1}^4S_b(q_j-u)du.
\end{equation*} 
In particular, the integral absolutely converges and does not depend on the choices of $\Gamma^*$.
\end{lemma}
\begin{proof}
By \cite[Eq. (2.22)]{Ivan2021}, for all $z$ with $\text{Re}z=\frac{Q}{2}$, we have $|S_b(z)|=1$. Since $\mathrm{Re}(q_j-t_i)=\frac{Q}{2}$ for all $i,j\in\{1,\dots,4\}$, we have
\begin{equation}\label{=1}
    \bigg|\prod_{i=1}^4\prod_{j=1}^4S_b(q_j-t_i)^{-1}\bigg|=1.
\end{equation}
On the other hand, since the integrand is meromorphic over $\mathbb C$ with  $t_i$'s and $q_j$'s the only poles in the region $\{u \in \mathbb C\ |\ \mathrm {Re}u\in[\frac{3Q}{2},2Q]\}$, which lie on the boundary of the region (see Figure \ref{Da}), it suffices to show that the integrand is bounded by a function which decays exponentially  as $|\mathrm{Im}u|\to \infty$. 
   
To this end, by \cite[Formula (B.53)]{eberhardt2023}, we have
$$\lim_{|\mathrm{Im}z|\to\infty}e^{\text{sgn}(\mathrm{Im}z)(\frac{\pi \mathbf{i} }{2}z(z-Q)+\frac{\pi\mathbf{i}}{12}(Q^2+1))}S_b(z)=1;$$
and by taking the norm, we have 
$$\lim_{|\mathrm{Im}z|\to\infty}e^{-\pi|\mathrm{Im}z|(\mathrm{Re}z-\frac{Q}{2})}|S_b(z)|=1.$$
In particular, by the continuity of the double sine function, for $z$ with $\mathrm{Re}z< \frac{Q}{2}$, there is a constant $C$ depending continuously on $b$ and $\mathrm{Re}z$, and being independent of $\mathrm{Im}z$, such that the double sine function decays exponentially in the sense that 
\begin{equation}\label{Sbestimate}
|S_b(z)|\leqslant Ce^{-\pi |\mathrm{Im}z|(\frac{Q}{2}-\mathrm{Re}z)};
\end{equation}
and for $z$ with $\mathrm{Re}z= \frac{Q}{2}$, there is a constant $C$ continuously depending on $b$ such that (\ref{Sbestimate}) holds for $z$ with $|\mathrm{Im}z|\geqslant \frac{1}{2}$.
Therefore, for $u$ with $\mathrm {Re}u\in[\frac{3Q}{2},2Q]$  and $|\mathrm {Im}u|\geqslant L\doteq\max\{\text{Im} t_i,\text{Im} q_j\}+\frac{1}{2},$ we have
 \begin{equation}\label{Integrandestimate}
 \bigg|\prod_{i=1}^4S_b(u-t_i)\prod_{j=1}^4S_b(q_j-u)\bigg|\leqslant C^8e^{-\pi\sum_{i=1}^4|\mathrm{Im}(u-t_i)|(2Q-\mathrm{Re}u)-\pi \sum_{j=1}^4|\mathrm{Im}(u-q_j)|(\mathrm{Re}u-\frac{3Q}{2})},
 \end{equation}
hence the integrand is dominated by a function which  decays exponentially as $|\mathrm{Im}u|\to\infty$. Together with \ref{=1}, this completes the proof. 
\end{proof}
As an immediate consequence of Lemma  \ref{lem:contour}, we have the following proposition.
\begin{prop}\label{prop:6j-cont}
    The $b$-$6j$ symbol $\bigg\{\begin{matrix} a_1 & a_2 & a_3 \\ a_4 & a_5 & a_6 \end{matrix} \bigg\}_b$ is a continuous function in the variables $\{a_1,\dots,a_6\}$.
\end{prop}

 \subsection{Truncated hyperideal tetrahedra and co-volume function}\label{extcov}

Following \cite{BB}, a truncated
\emph{hyperideal tetrahedron} $\Delta$ in $\mathbb{H}^3$ is a
compact convex polyhedron that is diffeomorphic to a truncated
tetrahedron in $\mathbb{E}^3$ with four hexagonal faces $\{H_1,H_2,H_3,H_4\}$ isometric to
right-angled hyperbolic
 hexagons and four triangular faces $\{T_1,T_2,T_3,T_4\}$ isometric to hyperbolic triangles. We will call $H_i$'s the \emph{faces}  and $T_i$'s the \emph{triangles of truncation} of $\Delta.$ 
An \emph{edge} in a hyperideal tetrahedron is the
 intersection of two hexagonal faces and the \emph{dihedral angle} at an edge is the angle between the two hexagonal faces
   adjacent to it. The angle between a hexagonal face and
   a triangular face is always $\frac{\pi}{2}.$ See (1) of  Figure \ref{hyperideal}. 
   
\begin{figure}[htbp]
\centering
\includegraphics[scale=0.3]{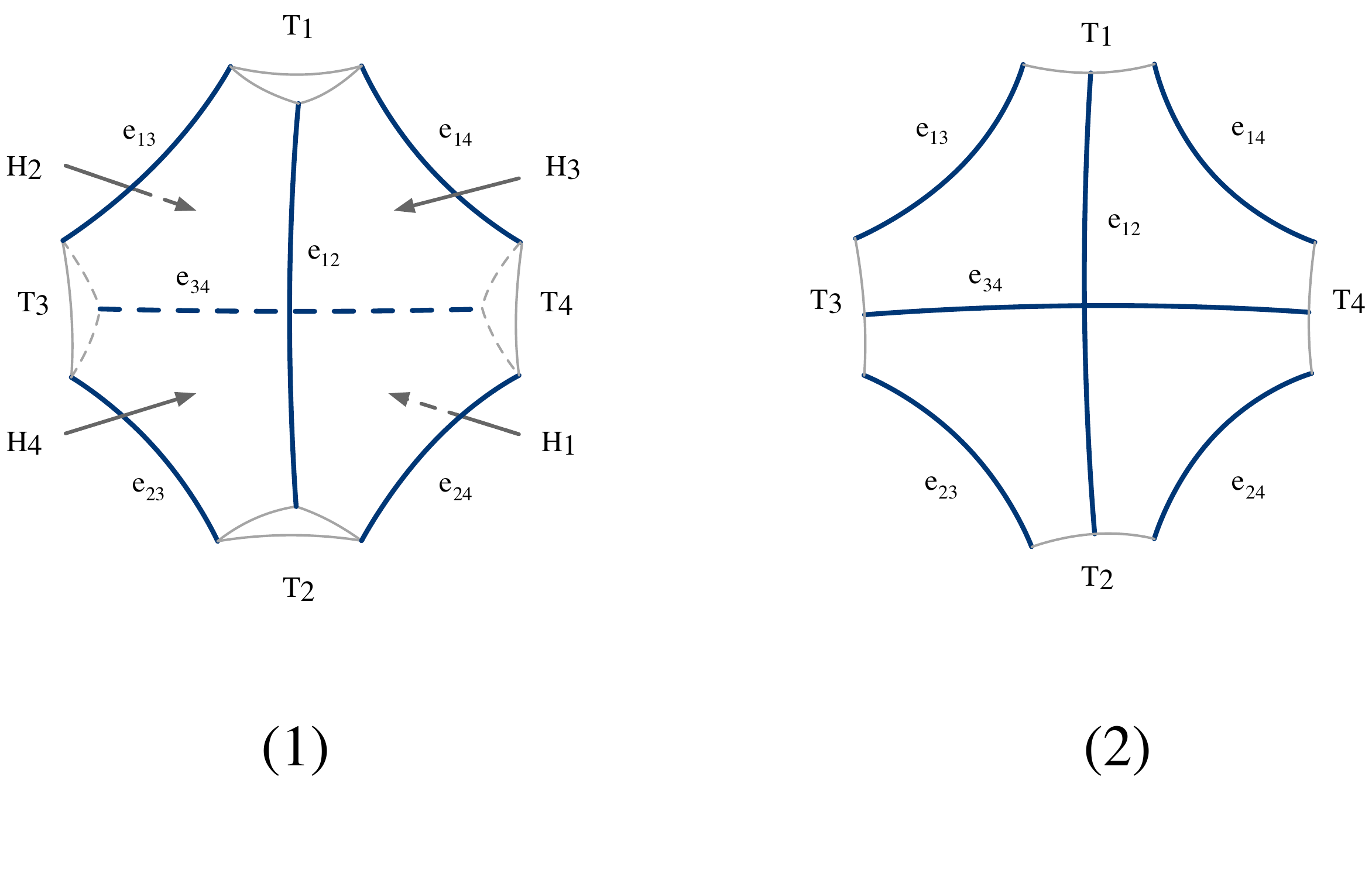}
\caption{(1) truncated hyperideal tetrahedron, (2) flat tetrahedron}
\label{hyperideal} 
\end{figure}

Let $e_{ij}$ be the edge connecting the triangular faces $T_i$ and $T_j,$ and let $\theta_{ij}$ and $l_{ij}$ respectively be the dihedral angle at and edge length of $e_{ij}.$  Then by the Cosine Law of hyperbolic triangles and right-angled hyperbolic hexagons, we have\begin{equation}\label{cos2}
\cos\theta_{kl}=\frac{ch_{kl}+ch_{ik}ch_{jk}+ch_{il}ch_{jl}+ch_{ik}ch_{jl}ch_{kl}+ch_{il}ch_{jk}ch_{kl}-ch_{ij}ch_{kl}^2}{\sqrt{-1+ch_{kl}^2+ch_{ik}^2+ch_{il}^2+2ch_{kl}ch_{ik}ch_{il}}\sqrt{-1+ch_{kl}^2+ch_{jk}^2+ch_{jl}^2+2ch_{kl}ch_{jk}ch_{jl}}},
\end{equation}
where  $\{i,j,k,l\}=\{1,2,3,4\}$ and $ch_{ij}=\cosh l_{ij}.$ By \cite{BB}, a truncated hyperideal tetrahedron is up to isometry determined by its six dihedral angles $\{\theta_1,\dots, \theta_6\},$ and by (\ref{cos2}) is determined by its six edge lengths $(l_1,l_2,l_6,l_3,l_5,l_4)=(l_{12},l_{13},l_{14},l_{23},l_{24},l_{34}).$

\begin{proposition}\label{degeneration}(\cite[Proposition 4.5]{LY})
Let $\mathcal L\subset \mathbb R^6_{\geqslant 0}$ be the space of  truncated hyperideal tetrahedra parametrized by the edge lengths, and let $\partial \mathcal L$ be the boundary of $\mathcal L$ in $\mathbb{R}_{{\geqslant } 0}^6.$ Then $\partial \mathcal L=X_1\sqcup X_2 \sqcup X_3,$ where each $X_i,$ $i=1,2,3,$ is a real analytic codimension-$1$ submanifold of $\mathbb{R}_{{\geqslant } 0}^6.$ The complement $\mathbb{R}_{{\geqslant } 0}^6\setminus (\mathcal L\cup \partial \mathcal L)$ is a disjoint union of three manifolds $\Omega_i$  so that $\overline \Omega_i\cap\partial\mathcal L=X_i,$ $i=1,2,3.$
\end{proposition}

The spaces in Proposition \ref{degeneration} are defined using the \emph{Gram matrix}. For $\boldsymbol l= (l_1,\dots,l_6)\in\mathbb R^6_{\geqslant 0},$ the Gram matrix of $\boldsymbol l$ is defined by
\begin{equation*}
\mathrm{Gram}(\boldsymbol l)=\begin{bmatrix}
1 & -\cosh l_1 & -\cosh l_2
    & -\cosh l_6 \\
-\cosh l_1 & 1 & -\cosh l_3
    & -\cosh l_5\\
    -\cosh l_2 & -\cosh l_3 & 1
    & -\cosh l_4 \\
 -\cosh l_6 & -\cosh l_5 & -\cosh l_4
    & 1
  \end{bmatrix}.
\end{equation*}
Note that if $(l_1,\dots,l_6)$ are the edges lengths of a hyperideal tetrahedron $\Delta,$ then $\mathrm{Gram}(\boldsymbol l)$ is the Gram matrix $\mathrm{Gram}(\Delta)$ of $\Delta$ in the edges lengths. For $i,j \in \{1,2,3,4\},$ let $G_{ij}$ be the $ij$-th cofactor of $\mathrm{Gram}(\boldsymbol l).$ 
Then the spaces in Proposition \ref{degeneration} can be described as follows.
\begin{enumerate}[(1)]
\item $\boldsymbol l\in\mathcal L$ if and only if 
$\frac{G_{ij}}{\sqrt{G_{ii}G_{jj}}}\in(-1,1)$ for any $\{i,j\}\subset\{1,2,3,4\}.$ We note that in this case, $\frac{G_{ij}}{\sqrt{G_{ii}G_{jj}}}$ coincides with the right hand side of (\ref{cos2}).

\item $\boldsymbol l\in\partial \mathcal L$ if and only if 
$\frac{G_{ij}}{\sqrt{G_{ii}G_{jj}}}\in\{-1,1\}$ for any $\{i,j\}\subset\{1,2,3,4\}.$ 
An $\boldsymbol l\in \partial \mathcal L$ can be considered as the edge lengths of a \emph{truncated hyperideal flat tetrahedron}, or a \emph{flat tetrahedron} for short, as depicted in (2) of  Figure \ref{hyperideal}. Namely, take a right-angled hyperbolic octagon $\Delta$ with eight edges
cyclically labelled as $T_1, $$e_{14},$$T_4,$$ e_{24},
$$T_2, $ $e_{23}, $ $T_3,$  $e_{13}.$  Let $e_{12}$
be the shortest geodesic arc in $\Delta$ joining
$T_1$ to $T_2$  and let $e_{34}$ be the shortest geodesic arc joining $T_3$ and $T_4.$ Then 
$\Delta$  is a  flat  tetrahedron with the six
edges $e_{ij}.$   The \emph{dihedral
angles}  at $e_{12}$ and $e_{34}$ are defined to be $\pi,$ and at all the other edges are defined to be  $0,$ corresponding to $\cos\theta_{kl}=\frac{G_{ij}}{\sqrt{G_{ii}G_{jj}}}=1$ or $-1.$

\item  $\boldsymbol l\in {\mathbb R_{\geqslant 0}}^6\setminus (\mathcal L\cup\partial \mathcal L)$ if and only if 
$\frac{G_{ij}}{\sqrt{G_{ii}G_{jj}}}\notin [-1,1]$  for any $\{i,j\}\subset\{1,2,3,4\}.$ 
\end{enumerate}

We call an $\boldsymbol l\in {\mathbb R_{\geqslant 0}}^6\setminus \mathcal L$  the edge lengths of a \emph{generalized hyperideal tetrahedra}.

As a consequence of Proposition \ref{degeneration}, we have the following trichotomy:

\begin{proposition}\label{tri}
\begin{equation*}
\det\mathrm{Gram}(\boldsymbol l) 
\left\{\begin{array}{ccl}
      <0 & \text{iff } & \boldsymbol l\in\mathcal L,\\
      =0 & \text{iff} & \boldsymbol l\in\partial \mathcal L,\\
       >0 & \text{iff} & \boldsymbol l\in {\mathbb R_{\geqslant  0}}^6\setminus (\mathcal L\cup\partial \mathcal L).\\
    \end{array} \right.
\end{equation*}
\end{proposition}

\begin{proof} We first consider the special case $\boldsymbol l_0=(0,\dots,0).$ In this case, $\boldsymbol l_0\in\mathcal L$ is the edge lengths of the regular ideal octahedron (considered as a truncated hyperideal tetrahedron) and by a direct computation $\det\mathrm{Gram}(\boldsymbol l_0)=-16<0.$ 

Next, suppose one of the components of $\boldsymbol l,$ say, $l_1\neq 0.$ Then by \cite[2.5.1. Theorem (Jacobi)]{P},
\begin{equation}\label{J}
\det\mathrm{Gram}(\boldsymbol l)=-\frac{G_{34}^2-G_{33}G_{44}}{1-\cosh^2l_1}=\frac{G_{34}^2-G_{33}G_{44}}{\sinh^2l_1}.
\end{equation}
We also observe that $G_{ii}<0$ for all $i.$ Then:

\begin{enumerate}[(i)]
\item If $\boldsymbol l\in\mathcal L,$ then $G_{34}^2-G_{33}G_{44}<0,$ and as a consequence of (\ref{J}),
$\det\mathrm{Gram}(\boldsymbol l)<0.$

\item If $\boldsymbol l\in\mathcal \partial L,$ then  $G_{34}^2-G_{33}G_{44}=0,$ and as a consequence of (\ref{J}), 
$\det\mathrm{Gram}(\boldsymbol l)=0.$

\item If $\boldsymbol l\in {\mathbb R_{\geqslant 0}}^6\setminus (\mathcal L\cup\partial \mathcal L),$ then  $G_{34}^2-G_{33}G_{44}>0,$ and as a consequence of (\ref{J}),
$\det\mathrm{Gram}(\boldsymbol l)>0.$
\end{enumerate}
\end{proof}

The following Lemma \ref{df} will be used in the proof of Proposition \ref{KC}, which will be needed in the proof of Theorem \ref{VC}.

\begin{lemma}\label{df} Let $\boldsymbol l=(l_1,\dots,l_6)\in \partial \mathcal L$ be the edge lengths of a flat hyperideal tetrahedron with dihedral angles  $\theta_1=\theta_4=\pi$ and $\theta_2=\theta_3=\theta_5=\theta_6=0.$ Let $K_0$ be a subset of $\{2,3,5,6\},$ and for $s>0$ let 
\begin{equation*}
l_{s,k}=
\left\{\begin{array}{lcl}
      l_k+s & \text{if} &k\in  K_0,\\
           l_k  & \text{if } & k\in \{1,\dots, 6\}\setminus   K_0.\\
    \end{array} \right.
\end{equation*}
Then there exists an $s_0>0$ such that  for all $s\in (0,s_0),$  $\boldsymbol l_s=(l_{s,1},\dots, l_{s,6})\in\mathcal L.$ 
\end{lemma}

\begin{proof} In a truncated tetrahedron $\Delta,$ we let $T_1,\dots,T_4$ be the triangles of truncation, let $e_{ij}$ be the edge of $\Delta$ connecting $T_i$ and $T_j,$ and let $e_{jk}^i$ be the edge of $T_i$ adjacent to the edges $e_{ij}$ and $e_{ik}.$ Then $e^i_{jk},e^i_{jl},e^i_{kl}$ are the edges of $T_i.$  For $(l_{12},\dots, l_{34})\in\mathbb R^6_{\geqslant 0},$ let 
$$x_{ij}^k=\cosh^{-1}\bigg(\frac{\cosh l_{ij}+\cosh l_{ik}\cosh l_{jk}}{\sinh l_{ik}\sinh l_{jk}}\bigg).$$
Then by the Law of Cosine of right-angled hyperbolic hexagons, if $l_{ij}$'s are the  lengths of the edges $e_{ij}$'s of  a  flat or truncated hyperideal tetrahedron, then $x_{ij}^k$'s are the lengths of the short edges $e_{ij}^k$'s. Recall from Proposition \ref{degeneration} that  $\partial \mathcal L = X_1\cup X_2\cup X_3$ and $\mathbb R^6_{\geqslant 0}\setminus (\mathcal L \cup \partial \mathcal L) =\Omega_1\cup\Omega_2\cup\Omega_3.$  In \cite[Proposition 4.4 and Lemma 4.6]{LY}, it shows that $\mathcal L$ consists of $(l_{ij})$'s such that the strict triangular inequalities  $x_{ij}^k< x_{jk}^i+x_{ik}^j$ hold for all $\{i,j,k\}\subset \{1,2,3,4\};$ $X_1=X_{12}=X_{34},$ $X_2=X_{32}=X_{24}$ and  $X_3=X_{14}=X_{23},$ where $X_{ij}$ consists of $(l_{ij})$'s such that $x^i_{kl}= x^i_{jk}+x^i_{il},$ $x^j_{kl}= x^j_{ik}+x^j_{il},$ $x^k_{ij}= x^k_{il}+x^k_{jl}$ and $x^l_{ij}= x^l_{ik}+x^l_{jk}$ for $\{i,j,k,l\}=\{1,2,3,4\};$ and $\Omega_1=\Omega_{12}=\Omega_{34},$ $\Omega_2=\Omega_{32}=\Omega_{24}$ and $\Omega_3=\Omega_{14}=\Omega_{23},$ where $\Omega_{ij}$ consists of $(l_{ij})$'s such that $x^i_{kl}> x^i_{jk}+x^i_{il},$ $x^j_{kl}> x^j_{ik}+x^j_{il},$ $x^k_{ij}>  x^k_{il}+x^k_{jl}$ and $x^l_{ij}>  x^l_{ik}+x^l_{jk}$ for $\{i,j,k,l\}=\{1,2,3,4\}.$ As a consequence, if for at least one $i\in\{1,2,3,4\},$ the quantities $x^i_{jk},$ $x^i_{jl},$ $x^i_{kl}$ satisfy the strict triangle inequalities, then $(l_{12},\dots, l_{34})\in \mathcal L.$

Now for a flat tetrahedron with dihedral angles  $\theta_{12}=\theta_{34}=\pi$ and $\theta_{13}=\theta_{14}=\theta_{23}=\theta_{24}=0$ or for
 a truncated hyperideal tetrahedron with  the dihedral angles $\theta_{12},\theta_{34}$ sufficiently close to $\pi$ and $\theta_{13},\theta_{14},\theta_{23},\theta_{24}$ sufficiently close to $0,$ we call $e_{12}, e_{34}$ the $\pi$-edges and $e_{13}, e_{14}, e_{23}, e_{24}$ the $0$-edges. We claim that: there is an $s_0>0$ such that if we increase the length of one $0$-edge by $s\in(0,s_0)$ and keep the lengths of all the other edges unchanged, then the tetrahedron is non-flat. Indeed, suppose that the length $l_{13}$ of  $e_{13}$ is increased by $t$ and all the other edge lengths remain the same. Then by the Law of Cosine of right-angled hyperbolic hexagons, $x^2_{13}$ increases and $x^2_{14}$ and $x^2_{34}$ remain the same. As a consequence, the quantity $x^2_{13}+ x^2_{14} - x^2_{34}$ (which is positive when the tetrahedron is non-flat and  equals $0$ when the tetrahedron is flat)  increases. Therefore, as long as $x^2_{13}$ does not exceed $x^2_{14} + x^2_{34},$ the quantities  $x^2_{13}, x^2_{14}, x^2_{34}$ satisfy the strict triangle inequalities, and the tetrahedron is non-flat. 
 
Finally,  by the claim, for any subset $ K_0$ of the $0$-edges, we can increase their lengths by $s\in (0,s_0)$ one edge at a time by another, while keeping the lengths of the other edges unchanged; and the result is a non-flat tetrahedron, and $(l_{s,12},\dots,l_{s,34})\in\mathcal L.$ 
\end{proof}

Following \cite{L}, the \emph{co-volume function} $\mathrm{Cov}:\mathcal L \to \mathbb R$ is defined by
\begin{equation}\label{covfun}
\mathrm{Cov}(\boldsymbol l)=\mathrm{Vol}(\Delta)+\frac{1}{2}\sum_{k=1}^6\theta_kl_k
\end{equation}
where $\Delta$ is the  truncated hyperideal tetrahedron with edge lengths $\{l_1,\dots,l_6\}$ and $\{\theta_1,\dots,\theta_6\}$ are the dihedral angles of $\Delta$ considered as functions of $\boldsymbol l,$  which is locally  strictly concave up on $\mathcal L,$ and satisfies
\begin{equation}\label{Sch}
\frac{\partial \mathrm{Cov}(\boldsymbol l)}{\partial l_k}=\frac{\theta_k}{2}
\end{equation}
for each $k\in\{1,\dots,6\}.$ It is shown in \cite[Lemma 4.7 and Corollary 4.9]{LY} that for each $k,$ the dihedral angle function $\theta_k:\mathcal L\to\mathbb R$ extends  continuously  to a $\widetilde \theta_k: \mathbb R^6_{>0}\to\mathbb R$ which is a constant function on each component of $\mathbb R^6_{\geqslant 0}\setminus \mathcal L$ by
\begin{equation*}
\widetilde\theta_k(\boldsymbol l)=\widetilde\theta_{k+3}(\boldsymbol l)=\left\{\begin{array}{ccl}
\pi & \text{if} & \boldsymbol l\in X_k\cup \Omega_k,\\
0 & \text{if} & \boldsymbol l\in \big(X_i\cup \Omega_i\big)\cup\big(X_j\cup \Omega_j\big),\\
\end{array}\right.
\end{equation*}
where $\{i,j,k\}=\{1,2,3\}.$ Define the following continuous $1$-form $\mu$ on $\mathbb R^6_{\geqslant 0}$ by
$$\mu(\boldsymbol l)=\frac{1}{2}\sum_{k=1}^6\widetilde\theta_k(\boldsymbol l)dl_k.$$
Then we have the following propositions.

\begin{proposition}\cite[Proposition 4.10]{LY} The $1$-form $\mu$ is closed in $\mathbb R^6_{\geqslant  0},$ that is, for any triangle $\Delta$ in $\mathbb R^6_{\geqslant  0},$ $\int_{\partial \Delta}\mu=0.$
\end{proposition}

\begin{proposition}\cite[Corollary 4.12]{LY} \label{ccu}
 The \emph{extended co-volume function}
\begin{equation}\label{extcovfun}
\widetilde{\mathrm{Cov}}(\boldsymbol l)=\int_{(0,\dots,0)}^{\boldsymbol l}\mu+\mathrm{Cov}(0,\dots,0)
\end{equation}
is a well-defined $C^1$-smooth concave up function, which satisfies 
$$\frac{\partial \widetilde{\mathrm{Cov}}(\boldsymbol l)}{\partial l_k}=\frac{\widetilde\theta_k(\boldsymbol l)}{2},$$
and coincides with $\mathrm{Cov}$ on $\mathcal L.$
\end{proposition}

\subsection{Hyperbolic polyhedral metrics and co-volume function} \label{HPM}

Let $M$ be a $3$-manifold with non-empty boundary, and let $\mathcal T$ be an ideal triangulation of $M,$ i.e., a realization of $M$ as the quotient of a finite set $T$ of  truncated Euclidean tetrahedra along  homeomorphisms between pairs of the faces of the truncated Euclidean  tetrahedra. We call the quotients of the edges of the truncated Euclidean tetrahedra the \emph{edges} of $\mathcal T,$ and let $E$ be the set of edges of $\mathcal T.$

A \emph{hyperbolic polyhedral metric} of hyperideal type  on $(M,\mathcal T)$ assigns each edge $e\in E$ an edge lengths $l_e\in\mathbb R_{{\geqslant }  0},$ so that for each tetrahedron $\Delta\in T$ with edges $e_1,\dots,e_6,$ the six numbers $l_{e_1},\dots,l_{e_6}$ are the edges lengths of a truncated hyperideal tetrahedron.  Denote by $\mathcal L(M,\mathcal T)$  the space of hyperbolic polyhedral metrics on $(M,\mathcal T).$ Then by \cite[Proposition 4.4]{LY}, $\mathcal L(M,\mathcal T)$ is an open subset of ${\mathbb R_{\geqslant 0}}^E.$
The \emph{cone angle} $\theta_e$ of a hyperbolic polyhedral metric $\boldsymbol l$ at an edge $e$ is the sum of the dihedral angles adjacent to $e$ of the truncated hyperideal tetrahedra determined by $\boldsymbol l.$ The \emph{discrete curvature} of $\boldsymbol l$ at $e$ is $K_e=2\pi-\theta_e,$ and the \emph{discrete curvature} of $\boldsymbol l$ is the vector $K(\boldsymbol l)=(K_{e})_{e\in E}\in\mathbb R^E.$

 In \cite{L}, to study the rigidity property of hyperbolic polyhedral metrics, the \emph{co-volume function} $\mathrm{Cov}: \mathcal L(M,\mathcal T) \to \mathbb R$ is defined by 
$$\mathrm{Cov}(\boldsymbol l)=\sum_{\Delta\in T} \mathrm{Cov}(\boldsymbol l_\Delta) - \pi\sum_{e\in E}l_e,$$
which  is locally strictly concave up, as each summand $\mathrm{Cov}(\boldsymbol l_\Delta)$  is locally strictly concave up on $\mathcal L$ and $- \pi\sum_{e\in E}l_e$ is a linear function; and by (\ref{Sch}) for each $\boldsymbol l\in \mathcal L(M, \mathcal T)$  the gradient 
$$\nabla \mathrm{Cov}(\boldsymbol l) = - \frac{K(\boldsymbol l)}{2},$$
and a critical point $\boldsymbol l^*$ of $\mathrm{Cov},$ if exists, defines a hyperbolic metric on $M$ with totally geodesic boundary, with the critical value
$$\mathrm {Cov} (\boldsymbol l^*)= \mathrm{Vol}(M).$$
In \cite{LY},  an $\boldsymbol l\in {\mathbb R_{\geqslant 0}}^E$ is called  a \emph{generalized hyperbolic polyhedral metric},  and  the \emph{extended co-volume function}  $\widetilde{\mathrm{Cov}}: {\mathbb R_{\geqslant 0}}^E \to \mathbb R$ is defined by 
$$\widetilde{\mathrm{Cov}}(\boldsymbol l)=\sum_{\Delta\in T}\widetilde{\mathrm{Cov}}(\boldsymbol l_\Delta) - \pi\sum_{e\in E}l_e.$$
Then $\widetilde{\mathrm{Cov}}$ coincides with $\mathrm{Cov}$ on $\mathcal L(M,\mathcal T),$ which is locally strictly concave up; and as $\widetilde{\mathrm{Cov}}(\boldsymbol l_\Delta)$ is concave up on ${\mathbb R_{\geqslant 0}}^6$ for each tetrahedra $\Delta,$ $\widetilde{\mathrm{Cov}}$ is concave up on the convex set  $\mathbb R_{{\geqslant } 0 }^E.$ 
Putting these facts together, it is proved in \cite{LY} that hyperbolic polyhedral metrics are rigid in the sense that they are uniquely determined by their discrete curvatures.

\begin{theorem}\label{rigidity}\cite[Theorem 5.1]{LY} For an ideally triangulated $3$-manifold  $(M,\mathcal T),$ the \emph{discrete curvature map}  $K: \mathcal L(M,\mathcal T) \to \mathbb R^E$ defined by sending $\boldsymbol l$ to its discrete curvature $K(\boldsymbol l)$  is a continuous injection. 
\end{theorem}

 \subsection{Angle structures 
 }\label{anglestr}

Let $(M,\mathcal{T})$ be an ideally triangulated 3-manifold with non-empty boundary, and let 
$E$ and $T$ respectively be the sets of edges and tetrahedra. We call a pair $(\Delta,e),$ $\Delta\in T$ and $e\in E,$  a corner of $\mathcal T$ if $\Delta$ is adjacent to  $e.$ We also denote by $\Delta\sim e$ when $(\Delta,e)$ is a corner. According to \cite[Definition 6.1]{LY}, an \emph{angle assignment} on $(M, \mathcal T)$ assigns each corner a number $\theta_{(\Delta,e)}$ in $(0,\pi),$  called
the dihedral angle of $e$ in $\Delta,$ so that sum of
the dihedral angles at three edges in the same tetrahedron
$\Delta$ adjacent to each vertex is strictly less than $\pi.$ By \cite{BB}, this is exactly the conditions for six numbers in $(0,\pi)$ to be the dihedral angles of a truncated hyperideal tetrahedron. 
 The {\it cone angle} of
an angle assignment is the assignment  $ \Theta \in \mathbb R^E$ that assigns  each edge $e$ to the sum of
dihedral angles at $e.$ An \emph{angle structure} is an angle assignment with the cone angle $\Theta=(2\pi,\dots,2\pi).$

The volume function $\mathrm{Vol}: \mathcal A_{\Theta}(M,\mathcal T)\to \mathbb R$ is defined by
$$\mathrm{Vol}(\boldsymbol \theta) = \sum_{\Delta\in T} \mathrm{Vol} (\Delta_{\boldsymbol \theta _ \Delta}),$$
where $\Delta_{\boldsymbol \theta _ \Delta}$ is the truncated  hyperideal tetrahedron with dihedral angles $\boldsymbol \theta _ \Delta$ assigned by $\boldsymbol\theta$ to the corners of $\Delta.$ By \cite{R}, the volume function can be extended continuously to the closure $\overline{\mathcal A_{\Theta}(M,\mathcal T)}$ of $\mathcal A_{\Theta}(M,\mathcal T).$

For a generalized hyperbolic polyhedral metric  $\boldsymbol l \in \mathbb R_{>0}^E,$ we define the \emph{extended dihedral angle} $\theta(\boldsymbol l)_{(\Delta,e)}$ of $ \boldsymbol  l$ at the corner $(\Delta,e)$  to be the corresponding extended
dihedral angle of the generalized hyperideal tetrahedron with
edge lengths $\boldsymbol l_\Delta.$ Let $C$ be the set of corners of $\mathcal T,$ and let $\boldsymbol \theta(\boldsymbol l)=\big(\theta(\boldsymbol l)_{(\Delta,e)}\big)_{(\Delta,e)\in C}.$  We define the \emph{extended cone angle} of $\boldsymbol l$ at $e$ as the sum of the extended dihedral angles of $\boldsymbol l$ at the corners  around $e.$

\begin{theorem}\cite[Theorem 6.3]{LY}\label{rigidity2}
 If $\mathcal A_{\Theta}(M, \mathcal T)\neq\emptyset,$ then
 \begin{enumerate}[(a)]
 \item there exists a unique $\boldsymbol\theta \in  \overline{\mathcal A_{\Theta}(M,\mathcal T)}$ that achieves the maximum volume, and
 \item for each generalized polyhedral metric $\boldsymbol l \in\mathbb{R}_{{>} 0}^E$ with cone angles $\Theta,$ the extended dihedral angles $\boldsymbol\theta(\boldsymbol l)$ of $\boldsymbol l$ achieves the maximum volume on $\overline{\mathcal A_{\Theta}(M,\mathcal T)},$ 
 \item if $\boldsymbol\theta \in  \overline{\mathcal A_{\Theta}(M,\mathcal T)}$ achieves the maximum volume, then there exists a generalized hyperbolic polyhedral metric $\boldsymbol l\in\mathbb{R}_{{>} 0}^E$  whose  extended dihedral angles $\boldsymbol \theta(\boldsymbol l) =\boldsymbol\theta.$
  \end{enumerate}
\end{theorem}

The following Corollary \ref{Cor3} will play a crucial role in the proof of Theorems \ref{WD3}.

\begin{corollary}\label{Cor3}
Let $M$ be a hyperbolic $3$-manifold with totally geodesic boundary and let $\mathcal T$ be an angled ideal triangulation of $M$ such that  the space of angle structures $\mathcal A(M,\mathcal T)\neq \emptyset.$   Then there exists a $t_0>0$ and a generalized hyperbolic polyhedral metric $\boldsymbol l_{t_0}\in {\mathbb R_{>0}}^E$ on $(M,\mathcal T)$ whose cone angle equals $2\pi+t_0$ at each edge of $\mathcal T.$
\end{corollary}

\begin{proof} 
We first show that if $t_0$ is sufficiently small, then $\mathcal A_{\boldsymbol { 2\pi+ t_0}}(M, \mathcal T)\neq\emptyset.$ Let $\boldsymbol \Theta :\mathbb (0,\pi)^C \to \mathbb R^E$ be cone angle map defined for $\boldsymbol \theta=(\theta_{(\Delta,e)})\in \mathbb (0,\pi)^C$ and $e\in E$  by 
$$\boldsymbol\Theta (\boldsymbol \theta)(e)=\sum_{\Delta\sim e} \theta_{(\Delta,e)},$$
which computes the sum of the dihedral angles around $e.$ 
Let $\mathcal A$ be the space of all possible dihedral angles of a truncated hyperideal tetrahedron. Then by \cite{BB}, $\mathcal A$ is the set of $6$-tuples of numbers in $(0,\pi)$ such that the sum of the threes numbers around each vertex is strictly less than $\pi,$ which  is a convex open polytope in $\mathbb (0,\pi)^6.$ As a consequence $\mathcal A^T$ is a convex open polytope in $(0,\pi)^C;$ and as the cone angle map $\boldsymbol \Theta$ is linear, the image $\boldsymbol\Theta(\mathcal A^T)$ is a convex open polytope in $\mathbb R^E.$ We observe that for $\Theta\in\mathbb R^E,$ $\mathcal A_{\Theta}(M, \mathcal T)\neq\emptyset$ if and only if $\Theta \in \boldsymbol\Theta (\mathcal A ^T).$ Since $\mathcal A(M,\mathcal T)\neq\emptyset,$ we know that $(2\pi,\dots,2\pi)\in  \boldsymbol \Theta (\mathcal A ^T);$ and since $\boldsymbol \Theta(\mathcal A^T)$ is open in $\mathbb R^E,$ there exists a $t_0>0$ such that $(2\pi+t_0, \dots, 2\pi+t_0)\in \boldsymbol \Theta (\mathcal A ^T).$ As a consequence,  $\mathcal A_{\boldsymbol { 2\pi+ t_0}}(M, \mathcal T)\neq\emptyset.$ Guaranteed to exist by Theorem \ref{rigidity2} (1), we let  $\boldsymbol\theta_{t_0}\in \overline{\mathcal A_{\boldsymbol { 2\pi+ t_0}}(M, \mathcal T)}$ be the unique angle assignment  that achieves the  maximum volume; and by Theorem \ref{rigidity2} (3),  we let $\boldsymbol l^*_{t_0}$ be a generalized polyhedral metric with $\boldsymbol \theta(\boldsymbol l^*_{t_0})=\boldsymbol \theta_{t_0}.$ 
\end{proof}

\subsection{Kojima ideal triangulations}\label{sec:2.8}

By Kojima \cite{Ko}, every hyperbolic $3$-manifold $M$ with totally geodesic boundary can be  canonically decomposed into a union of  convex truncated hyperideal polyhedra (see \cite{BB} for the definition.) 
This is the analogue in our setting of the Epstein-Penner decomposition\,\cite{EP} of cusped hyperbolic $3$-manifolds. The uniqueness comes from the fact that for a hyperbolic $3$-manifold $M$ with totally geodesic boundary, we can start the Voronoi process  from the boundary of $M$ whereas in the cusped case one has to make extra choices of the horo-spheres, which creates non-uniqueness. The polyhedral decomposition can be further decomposed into  a triangulation of $M,$ called a \emph{Kojima ideal triangulation} of $M,$  in the following steps:
\begin{enumerate}[Step 1.]
    \item For each polyhedron $P_i$ of the decomposition, choose a hyperideal vertex $p_i$ of $P_i$ and connect it to all the other hyperideal vertices  of $P$ by geodesic arcs to obtain a set of cones with bases the faces of $P_i$ not adjacent to $p_i$ and the tip the hyperideal vertex $p_i.$
    
    \item For each face $F_{ik}$ of  $P_i$ that is not adjacent to $p_{i},$ choose a hyperideal vertex $p_{ik}$ of $F_{ik}$  and connect it to all the other hyperideal vertices  of $F_{ik}$ to obtain an ideal triangulation of $F_{ik}$ by geodesic arcs,  which in turn gives an ideal triangulation of the polyhedron $P_i$ by taking the cone over the  hyperideal triangles of the faces with the tip the hyperideal vertex $p_i.$
    
    \item By doing so, the triangulations on the two sides of a face $F$ of the Kojima polyhedral decomposition are both cones over a point, and may not agree (when the two tips  of the cones from the faces $F_{ik}$ and $F_{jl}$ in the two polyhedra $P_i$ and $P_j$ on the two sides of $F$ are distinct). When this happens, as shown in Figure \ref{Kojima}, we insert layered flat tetrahedra to connect the ideal triangulations of the faces on the two sides.
\end{enumerate}

\begin{figure}[htbp]
\centering
\includegraphics[scale=0.35]{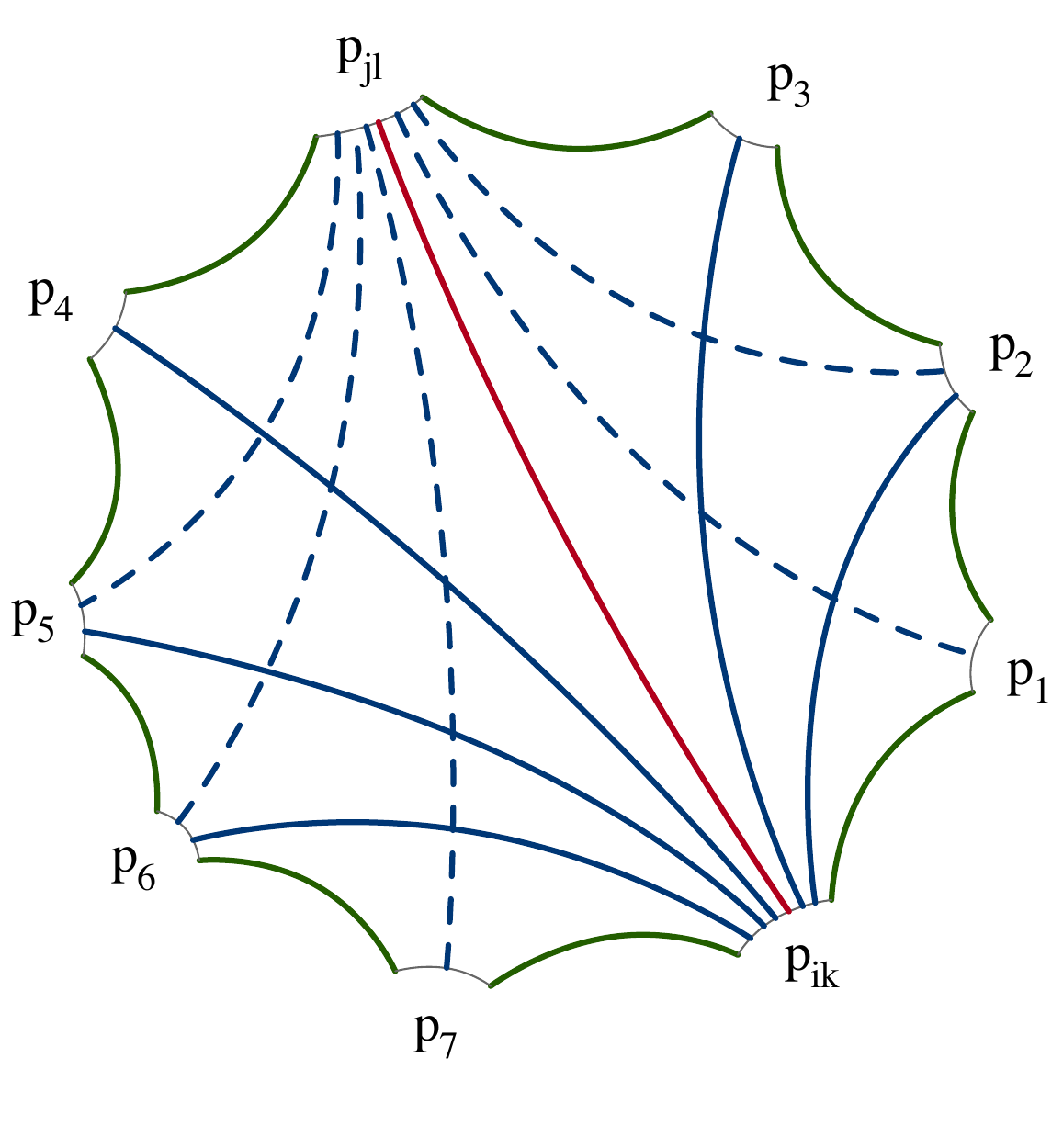}
\caption{Inserted layered flat tetrahedra: the flat tetrahedra respectively with hyperideal vertices $\{p_{ik},p_{jl}, p_1,p_2\},$   $\{p_{ik},p_{jl}, p_2,p_3\},$  $\{p_{ik},p_{jl}, p_4,p_5\},$  $\{p_{ik},p_{jl}, p_5,p_6\}$ and  $\{p_{ik},p_{jl}, p_6,p_7\}$ are inserted to connect the ideal triangulations of the hyperideal polygon on the top with hyperideal triangles of vertices $\{p_{ik},p_1,p_2\},$ $\{p_{ik},p_2,p_3\},$ $\{p_{ik},p_3,p_{jl}\},$ $\{p_{ik},p_{jl},p_4\},$ $\{p_{ik},p_4,p_5\},$$\{p_{ik},p_5,p_6\}$ and $\{p_{ik},p_6,p_7\}$ and on the bottom with hyperideal triangles of vertices  $\{p_{jl},p_3,p_2\},$ $\{p_{jl},p_2,p_1\},$ $\{p_{jl},p_1,p_{ik}\},$ $\{p_{jl},p_{ik},p_7\},$ $\{p_{jl},p_7,p_6\},$$\{p_{jl},p_6,p_5\}$ and $\{p_{jl},p_5,p_4\}.$  There are two flat tetrahedra on the right of the edge connecting $p_{ik}$ and $p_{jl}$ (colored in red) and three on the left of it.}
\label{Kojima} 
\end{figure}

A more detailed description of this construction can be found in \cite{HRS} for the cusped case, which is combinatorially equivalent to our setting; and here we highlight some of the key features of this construction. In a face $F$ of the Kojima polyhedral decomposition, we call the edge connecting the two tips of the cones from the two sides of $F$ \emph{central} (colored in red in Figure \ref{Kojima}), and call the edges of the polygon $F$ \emph{side} (colored in green in Figure \ref{Kojima}). Notice that  the central edge of a face can be a side edge when the two cone points are adjacent in $F.$ Then we have the following observations: 
\begin{enumerate}[(1)]
\item every flat tetrahedra is adjacent to a central edge;
\item both the central and the side edges are the $0$-edges in any flat tetrahedron adjacent to them; equivalently, none of the central edges and the side edges are the $\pi$-edges of any flat tetrahedron adjacent to them.
\end{enumerate}

\begin{proposition}\label{KC}
Let $\mathcal T$ be a Kojima ideal triangulation of a hyperbolic $3$-manifold $M$ with totally geodesic boundary and let    $\boldsymbol l ^* $ be its edge lengths. Then $\boldsymbol l ^* \in \partial \mathcal L(M,\mathcal T),$ i.e.,  there is  a one-parameter family of hyperbolic polyhedral metrics $\{\boldsymbol l^*_s\}$ on $(M,\mathcal T)$ that converges to $\boldsymbol l^*.$ 
\end{proposition}

\begin{proof}

For a flat tetrahedron $\Delta$ in $\mathcal T,$ let $s_{0,\Delta}$ be the constant in Lemma \ref{df} for $\Delta,$ and let $s_0$ be the minimum of $\{s_{0,\Delta}| \Delta
\in T\text{ flat}\}. $ Let $\boldsymbol l^*$ be the edge lengths of $\mathcal T;$ and for $s\in(0,s_0),$ let $\boldsymbol l^*_s\in{\mathbb R_{>0}}^E$ be defined by
\begin{equation*}
l^*_{s,e}=
\left\{\begin{array}{ll}
      l^*_e+s & \text{if }  e \text{ is central},\\
           l^*_e   & \text{if } e \text{ is not central}.\\
    \end{array} \right.
\end{equation*}
Then by the observations (1) and (2) above, for each flat tetrahedron in $\boldsymbol l^*,$ the lengths of some of the $0$-edges increase and the lengths of the other edges remain the same, and by Lemma \ref{df} it becomes non-flat; and for each non-flat tetrahedron in $\boldsymbol l^*,$ when $s_0$ is small enough, it remains non-flat.  Therefore, $\boldsymbol l^*_s\in \mathcal L(M,\mathcal T)$ for all $s\in (0,s_0)$ and $\boldsymbol l^*_s$ converges to $\boldsymbol l^*$ as $s$ tends to $0.$ 
\end{proof}

Recall that the extended co-volume function $\widetilde{\mathrm{Cov}}$ is locally strictly concave up on $\mathcal L(M,\mathcal T)$ and is concave up on $\mathbb R_{> 0 }^E.$  In general, at an $\boldsymbol l\in \partial L(M,\mathcal L),$ there is no guarantee that  $\widetilde{\mathrm{Cov}}$ is locally strictly concave up, but for a Kojima ideal triangulation, we have the following

\begin{proposition} \label{Covconcave} Let $\mathcal T$ be a Kojima ideal triangulation of a hyperbolic $3$-manifold $M$ with totally geodesic boundary and let    $\boldsymbol l ^* \in \partial \mathcal L(M,\mathcal T)$ be its edge lengths.  Then the extended co-volume function $\widetilde{\mathrm{Cov }}$ is strictly concave up near $\boldsymbol l^*.$ As a consequence, $\boldsymbol l^*$ is the unique minimum  point of  $\widetilde{\mathrm{Cov}}.$
\end{proposition}

\begin{proof} Let $T_{\text{non-flat}}$ be the set of non-flat tetrahedra in $\boldsymbol l^*$ and let $T_{\text{flat}} $ be the set of flat tetrahedra in $\boldsymbol l^*.$ Then 
$$\widetilde{\mathrm{Cov}}(\boldsymbol l)=\sum_{\Delta\in T_{\text{non-flat}}}\mathrm{Cov}(\boldsymbol l_\Delta)+\sum_{\Delta\in T_{\text{flat}}}\widetilde{\mathrm{Cov}}(\boldsymbol l_\Delta)-\pi\sum_{e\in E}l_e$$
near $\boldsymbol l^*.$
Since in a Kojima ideal triangulation, each edge $e$ is adjacent to some non-flat tetrahedron $\Delta,$ $\sum_{\Delta\in T_{\text{non-flat}}}\mathrm{Cov}(\boldsymbol l_\Delta)$ is strictly concave up in $\boldsymbol l$ near $\boldsymbol l^*.$ Together with the facts that $\sum_{\Delta\in T_{\text{flat}}}\widetilde{\mathrm{Cov}}(\boldsymbol l_\Delta)$ is concave up and $\pi\sum_{e\in E}l_e$ is linear in $\boldsymbol l,$ we have the result. 
\end{proof}

\begin{proposition} \label{KA}
Let $M$ be a hyperbolic $3$-manifold with totally geodesic boundary and let $\mathcal T$ be a Kojima ideal triangulation of $M.$  Then  $\mathcal A(M, \mathcal T)\neq\emptyset.$ 
\end{proposition}

\begin{proof} The edge lengths of $\mathcal T$ give a generalized hyperbolic polyhedral metric on $(M,\mathcal T)$ with dihedral angles in $[0,\pi],$ and with cone angles $2\pi$ around all the edges. We will deform these dihedral angles into $(0,\pi)$ while remaining the cone angles unchanged. We observe from the construction of the Kojima ideal triangulations that around each edge there is at least one positive dihedral angle. Let $m(e),$ $n(e)$ and $k(e)$ respectively  be the numbers of the $0$-dihedral angles, $\pi$-dihedral angles and dihedral angles in $(0,\pi)$ around $e.$ Then for a sufficiently small $\epsilon>0,$ we assign $\epsilon$ to each $0$-dihedral angle, $\pi-3\epsilon$ to each $\pi$-dihedral angle, and $\theta-\frac{m(e)-3n(e)}{k(e)}\epsilon$ to each dihedral angle $\theta\in(0,\pi)$ around the edge $e.$ In this way, each dihedral angle is in $(0,\pi),$ and the cone angles remained $2\pi,$  hence gives an angle structure on $(M,\mathcal T).$
\end{proof}

Inspired by Proposition \ref{KA}, we call an ideal triangulation $\mathcal T$ of  $M$  \emph{deformable} if one can assign each corner of $\mathcal T$ a number in $[0,\pi]$ such that \begin{enumerate}[(1)]
\item for each tetrahedron, the six assigned numbers are either the extended dihedral angles of  a  flat tetrahedron or the dihedral angles of a truncated hyperideal tetrahedron,
\item the extended cone angles at all the edges are equal to $2\pi,$ and 
\item each edge is adjacent to at least one non-flat tetrahedron. 
\end{enumerate}
Then by exactly the same argument as in the proof of Proposition \ref{KA}, one obtain the following Proposition \ref{DF} which will be intensively used in the proof of Theorem \ref{WD3}.

\begin{proposition}\label{DF}
Let $M$ be  a hyperbolic $3$-manifold with totally geodesic boundary and let $\mathcal T$ be a deformable ideal triangulation of $M.$  Then $\mathcal A(M, \mathcal T)\neq\emptyset.$
\end{proposition}

\subsection{Adjoint twisted Reidemeister torsions}\label{sec:2.9}

We first  recall results of Porti\,\cite{P} for the Reidemeister torsions of hyperbolic $3$-manifolds twisted by the adjoint action $\mathrm {Ad}_\rho=\mathrm {Ad}\circ\rho$ of an irreducible $\mathrm {PSL}(2;\mathbb C)$-representation $\rho.$ Here $\mathrm {Ad}$ is the adjoint action of $\mathrm {PSL}(2;\mathbb C)$ on its Lie algebra $\mathfrak{sl}(2;\mathbb C)\cong \mathbb C^3.$

For a closed orientable hyperbolic $3$-manifold  $M$ with the holonomy representation $\rho,$ by the Weil local rigidity theorem and the Mostow rigidity theorem,
$\mathrm H_k(M;\mathrm{Ad}_\rho)=0$ for all $k.$ Then the adjoint twisted Reidemeister torsion 
$$\mathrm{Tor}(M;\mathrm{Ad}_\rho)\in\mathbb C^*/\{\pm 1\}$$
 is defined without making any additional choice.

Now suppose  $M$ is a compact, orientable  $3$-manifold with boundary consisting of $n$ disjoint tori $T_1 \dots,  T_n$ whose interior admits a complete hyperbolic structure with  finite volume. Let $\mathrm X(M)$ be the $\mathrm{PSL}(2;\mathbb C)$-character variety of $M,$ and let $\mathrm X^n(M)=\cup \mathrm X_k(M)$ the union of the irreducible components $\{\mathrm X_k(M)\}$ of $\mathrm X(M)$ that have dimension equal to $n.$ If $M$ is hyperbolic, then $\mathrm X^n(M)$ is non-empty because it contains the distinguished component $ \mathrm X_0(M)$ containing the character of the holomony representation of the complete hyperbolic structure of $M$\,\cite{T, NZ}. 

Below we recall two fundamental results (Theorem \ref{HM} and Theorem \ref{funT}) of Porti\,\cite{P}. Theorem \ref{funT} was originally proved for characters in $\mathrm X_0(M),$ and was generalized to characters in $\mathrm X^n(X)$ by essentially the same argument    in \cite{WY3}. We denote by $\mathrm X^{\text{irr}}(M)$ the Zariski-open subset of $\mathrm X(M)$ consisting of the irreducible characters.

\begin{theorem}\cite[section 3.3.3]{Po}\label{HM} For a system of simple closed curves $\boldsymbol\alpha=(\alpha_1,\dots,\alpha_n)$ on $\partial M$ with $\alpha_i\subset T_i,$ $i\in\{1,\dots,n\},$  and a character $[\rho]$ in a Zariski open subset of $\mathrm X_0(M)\cap\mathrm X^{\text{irr}}(M),$  we have:
\begin{enumerate}[(i)]
\item For $k\neq 1,2,$ $\mathrm H_k(M;\mathrm{Ad}\rho)=0.$
\item  For $i\in\{1,\dots,n\},$ up to scalar $\mathrm Ad_\rho(\pi_1(T_i))^T$ has a unique invariant vector $\mathbf I_i\in \mathbb C^3;$ and
$$\mathrm H_1(M;\mathrm{Ad}\rho)\cong \mathbb C^n$$ 
with a basis
$$\mathbf h^1_{(M,\alpha)}=\{\mathbf I_1\otimes [\alpha_1],\dots, \mathbf I_n\otimes [\alpha_n]\}$$
where $([\alpha_1],\dots,[\alpha_n])\in \mathrm H_1(\partial M;\mathbb Z)\cong 
\bigoplus_{i=1}^n\mathrm H_1(T_i;\mathbb Z).$
 
\item Let $([T_1],\dots,[T_n])\in \bigoplus_{i=1}^n\mathrm H_2(T_i;\mathbb Z)$ be the fundamental classes of $T_1,\dots, T_n.$ Then 
 $$\mathrm H_2(M;\mathrm{Ad}\rho)\cong\bigoplus_{i=1}^n\mathrm H_2(T_i;\mathrm{Ad}\rho)\cong \mathbb C^n$$ 
with  a basis 
$$\mathbf h^2_M=\{\mathbf I_1\otimes [T_1],\dots, \mathbf I_n\otimes [T_n]\}.$$
\end{enumerate}
\end{theorem}

\begin{remark}[\cite{Po,NZ,HK}]\label{rm} Important examples of the characters in Theorem \ref{HM} include the character of the holonomy representation of the complete hyperbolic structure on the interior of $M,$
 the restriction of the holonomy representation of the closed $3$-manifold $M_{\boldsymbol\mu}$ obtained from $M$ by doing the hyperbolic Dehn-filling along the system of simple closed curves $\boldsymbol\mu$ on $\partial M,$
 and the holonomy representation of a hyperbolic structure on the interior of $M$ whose completion is a conical manifold with cone angles less than $2\pi.$
\end{remark}

Let $\boldsymbol\alpha=(\alpha_1,\dots,\alpha_n)$ be a system of simple closed curves on $\partial M$ with $\alpha_i\subset T_i,$ $i\in\{1,\dots,n\}.$  A character $[\rho]$ in $\mathrm X^n(M)\cap\mathrm X^{\text{irr}}(M)$ is \emph{$\boldsymbol\alpha$-regular} if condition (ii) in Theorem \ref{HM} is satisfied. Then there is an \emph{adjoint twisted Reidemeister torsion function $\mathbb T_{(M,\boldsymbol\alpha)}$  with respect to $\boldsymbol\alpha$} defined  on $\mathrm X^n(M)\cap\mathrm X^{\text{irr}}(M)$ such that 
$$\mathbb T_{(M,\boldsymbol\alpha)}([\rho])\in \mathbb C^*/ \{\pm 1\}$$
if $\rho$ is $\boldsymbol \alpha$-regular, and equals  $0$ if otherwise.  

\begin{theorem}\cite[Theorem 4.1]{Po}\label{funT}
Let $M$ be a compact, orientable  $3$-manifold with boundary consisting of $n$ disjoint tori $T_1 \dots,  T_n$ whose interior admits a complete hyperbolic structure with  finite volume. Let  $\mathbb C(\mathrm X^n(M)\cap\mathrm X^{\text{irr}}(M))$ be the ring of rational functions over $\mathrm X^n(M)\cap\mathrm X^{\text{irr}}(M).$ Then
\begin{equation*}
\begin{split}
\mathrm H_1(\partial M;\mathbb Z)&\to \mathbb C(\mathrm X^n(M)\cap\mathrm X^{\text{irr}}(M))\\
\quad\quad\boldsymbol\alpha\quad\ \  &\mapsto \quad\quad\quad\mathbb T_{(M,\boldsymbol\alpha)}
\end{split}
\end{equation*}
 up to sign defines  a  function which is a $\mathbb Z$-multilinear homomorphism with respect to the direct sum $\mathrm H_1(\partial M;\mathbb Z)\cong 
\bigoplus_{i=1}^n\mathrm H_1(T_i;\mathbb Z)$ satisfying the following properties:
\begin{enumerate}[(i)]
\item For a system of simple closed curves $\boldsymbol\alpha$ on $\partial M,$ if the component $\mathrm X_k(M)$ contains an $\boldsymbol\alpha$-regular character, then the support of $\mathbb T_{(M,\boldsymbol\alpha)}$
contains a Zariski-open subset of $\mathrm X_k(M)\cap\mathrm X^{\text{irr}}(M)$ consisting of all the $\boldsymbol\alpha$-regular characters in $\mathrm X_k(M).$ 

\item \emph{(Change of Curves Formula).} Let $\boldsymbol\beta=\{\beta_1,\dots,\beta_n\}$ and $\boldsymbol\gamma=\{\gamma_1,\dots,\gamma_n\}$ be two systems of simple closed curves on $\partial M.$ Let $( u_{\beta_1},\dots,  u_{\beta_n})$ and $( u_{\gamma_1},\dots, u_{\gamma_n})$ respectively be the logarithmic holonomies of the curves in $\boldsymbol\beta$ and $\boldsymbol\gamma.$ Then we have the equality of rational functions
\begin{equation}\label{coc}
\mathbb T_{(M,\boldsymbol\beta)}
=\pm\det\bigg( \frac{\partial ( u_{\beta_1},\dots, u_{\beta_n})}{\partial ( u_{\gamma_1},\dots, u_{\gamma_n})}\bigg)\mathbb T_{(M,\boldsymbol\gamma)}
\end{equation}
on $\mathrm X_k(M)\cap\mathrm X^{\text{irr}}(M)$ for the component $\mathrm X_k(M)$  containing a $\boldsymbol \gamma$-regular character, where $\frac{\partial ( u_{\beta_1},\dots, u_{\beta_n})}{\partial ( u_{\gamma_1},\dots, u_{\gamma_n})}$ is the Jocobian matrix of $( u_{\beta_1},\dots, u_{\beta_n})$ with respect to $( u_{\gamma_1},\dots, u_{\gamma_n}).$

\item \emph{(Surgery Formula).} Let $[\rho_{\boldsymbol\mu}]\in \mathrm X^n(M)\cap \mathrm X^{\text{irr}}(M)$ be the character induced by the holonomy of the closed $3$-manifold $M_{\boldsymbol\mu}$ obtained from $M$ by doing the hyperbolic Dehn filling along the system of simple closed curves $\boldsymbol \mu$ on $\partial M.$ If $\mathrm H(\gamma_1),\dots,\mathrm H(\gamma_n)$ are the logarithmic holonomies of the core curves $\gamma_1,\dots,\gamma_n$ of the solid tori added. Then
\begin{equation*}\label{sf}
\mathrm{Tor}(M_{\boldsymbol\mu};\mathrm{Ad}_{\rho_{\boldsymbol\mu}})=\pm\mathbb T_{(M,\boldsymbol \mu)}([\rho_{\boldsymbol\mu}])\prod_{i=1}^n\frac{1}{4\sinh^2\frac{\mathrm H(\gamma_i)}{2}}.
\end{equation*}
\end{enumerate}
\end{theorem}

Next we recall a  result on the computation of the adjoint twisted Reidemeister torsions from \cite{WY3}.

\begin{theorem}\cite[Theorem 1.6 (2)]{WY3}\label{Torsion2}
Let  $(M,\mathcal T)$ be a hyperbolic polyhedral $3$-manifold with the sets of edges and tetrahedra  respectively   $E$ and $T,$ and let $N$ be the $3$-manifold obtained from  the double of $M$ by removing the double of the edges of $M.$ 
Let  $(l_e)_{e\in E}$ be the lengths of the edges of $M;$ and for each $\Delta\in T$ with edges  $e_1,\dots,e_6,$  let $\mathbb G_\Delta=\mathrm{Gram}( l_{e_1},\dots, l_{e_6})$ be the Gram matrix of $( l_{e_{1}},\dots, l_{e_{6}}).$ Let $\rho$ be the holonomy representation of the hyperbolic cone metric on $N$ obtained by doubling  the hyperbolic polyhedral metric of $M,$ let $\boldsymbol m$ be the system of the meridians of a tubular neighborhood of the double of the edges, and let $(\theta_e)_{e\in E}$ be the cone angle functions in terms of the edge lengths of $M.$ 
Then  the value of the adjoint twisted  Reidemeister torsion function $ \mathbb T_{(N,\boldsymbol m)}$ at $[\rho]$ can be computed as 
 $$ \mathbb T_{(N,\boldsymbol m)}([\rho])=\pm(-1)^{\frac{3|E|}{2}}2^{3|T|-|E|}\bigg(\frac{\partial \theta_{e_i}}{\partial l_{e_j}}\bigg)_{_{e_i,e_j\in E}}\prod_{\Delta\in T} \sqrt{\det\mathbb G_\Delta},$$
where $\Big(\frac{\partial \theta_{e_i}}{\partial l_{e_j}}\Big)_{e_i,e_j\in E}$ is the Jacobian determinant of the cone angle functions $\theta_e$'s with respect to the edge lengths $l_e$'s evaluated at the edge lengths of $M.$
\end{theorem}

\subsection{Saddle point approximations}\label{subsec:saddle}

The proof of Theorems \ref{cov}, \ref{cov2} and\ \ref{VC} relies on the following Saddle Point Approximations.

\begin{proposition}\cite[Proposition 6.1]{WY2}\label{saddle}
Let $D$ be a region in $\mathbb C^n$ and let $f(z_1,\dots, z_n)$ and $g(z_1,\dots, z_n)$ be holomorphic functions on $D$ independent of $\hbar$. Let $f_\hbar(z_1,\dots,z_n)$ be a holomorphic function of the form
$$ f_\hbar(z_1,\dots, z_n) = f(z_1,\dots, z_n) + \upsilon_\hbar(z_1,\dots,z_n)\hbar^2.$$
Let $S$ be an embedded real $n$-dimensional closed disk in $D$ and let $(c_1,\dots, c_n)$ be a point on $S.$ If
\begin{enumerate}[(i)]
\item $(c_1, \dots, c_n)$ is a critical point of $f$ in $D,$
\item $\mathrm{Re}(f)(c_1,\dots,c_n) > \mathrm{Re}(f)(z_1,\dots,z_n)$ for all $(z_1,\dots,z_n) \in S\setminus \{(c_1,\dots,c_n)\},$
\item the Hessian matrix $\mathrm{Hess}(f)(c_1,\dots,c_n)$ of $f$ at $(c_1,\dots,c_n)$ is non-singular,
\item $g(c_1,\dots,c_n) \neq 0,$  and
\item $|\upsilon_\hbar(z_1,\dots,z_n)|$ is bounded from above by a constant independent of $\hbar$ in $D,$
\item $S$ is smoothly embedded around $(c_1,\dots,c_n).$ 
\end{enumerate}
then
\begin{equation*}
\begin{split}
 \int_S g(z_1,\dots, z_n) &e^{\frac{f_\hbar(z_1,\dots,z_n)}{\hbar}} dz_1\dots dz_n\\
 &= \Big(2\pi \hbar\Big)^{\frac{n}{2}}\frac{g(c_1,\dots,c_n)}{\sqrt{\det\big(-\mathrm{Hess}(f)(c_1,\dots,c_n)\big)}} e^{\frac{f(c_1,\dots,c_n)}{\hbar}} \Big( 1 + O \big(\hbar\big)\Big).
 \end{split}
 \end{equation*}
\end{proposition}

\begin{remark} In \cite[Proposition 6.1]{WY2}, condition (3) is stated as:  the domain $\{\boldsymbol z\in D\ |\ \mathrm{Re}{f}(\boldsymbol z) <\mathrm{Re}{f}(\boldsymbol c)\}$ deformation retracts to $S\setminus\{\boldsymbol c\}. $ By shrinking $D$ if necessary, this is satisfied when $S$ is smoothly embedded around $\boldsymbol c$ as stated in condition  (vi) here.

\end{remark}

\begin{proposition}\cite[Proposition 5.2]{M}\label{saddle2}
Let $D$ be a region in $\mathbb C$ and let $f(z)$ and $g(z)$ be holomorphic functions on $D$ independent of $\hbar$. Let $f_\hbar(z)$ be a holomorphic function of the form
$$ f_\hbar(z) = f(z) + \upsilon_\hbar(z)\hbar^2.$$
Let $S$ be a curve in $D$ and let $c$ be a point on $S.$ If
\begin{enumerate}[(i)]
\item $f'(c)=0,$
\item $\mathrm{Re}(f)(c) > \mathrm{Re}(f)(z)$ for all $z \in S \setminus \{c\},$
\item  $f''(c)=0,$ 
\item  $f'''(c)\neq 0,$ 
\item $g(c) \neq 0,$ and
\item $|\upsilon_\hbar(z)|$ is bounded from above by a constant independent of $\hbar$ in $D,$
\item $S$ connects two  components of the region $\{z\in D\ |\ \mathrm{Re}{f}(z)<\mathrm{Re}{f}(c)\},$ 
\end{enumerate}
then there is a nonzero constant $C$ such that 
\begin{equation*}
 \int_S g(z) e^{\frac{f_\hbar(z)}{\hbar}} dz=C (2\pi  \hbar) ^{\frac{1}{3}}\frac{g(c)}{\sqrt[3]{-f'''(c)}}e^{\frac{f(c)}{\hbar}} \Big( 1 + O \big(\hbar\big)\Big).
 \end{equation*}
\end{proposition}

The precise value of the constant $C$ can be found in \cite[Proposition 5.2]{M}.


\section{Asymptotics of the \texorpdfstring{$\mathrm U_{q\tilde q}\mathfrak{sl}(2;\mathbb R)$ $6j$}{U_q(\mathfrak{sl}(2;R)) 6j}-symbol}\label{sec:tetrahedron}
In  this section we prove Theorem~\ref{cov} and Theorem~\ref{cov2} using the saddle point approximations. Recall the variables  $a_k, u, t_i, q_j,$ in Definition \ref{def：6j} of $\bigg\{\begin{matrix} a_1 & a_2 & a_3 \\ a_4 & a_5 & a_6 \end{matrix} \bigg\}_b$. We introduce a new set of variables $\alpha_k$, $\xi$, $\tau_i$, $\eta_j$ as follows:
\begin{align}
&\alpha_k=\pi b a_k-\frac{\pi b^2}{2} \textrm{ for }k\in\{1,\dots, 6\},\quad \xi=\pi b u, \label{eq:alpha-a}\\
 &\tau_i=\pi b t_i-\frac{3\pi b^2}{2}\textrm{ for }i\in\{1,2,3,4\},\quad\text{and}\quad \eta_j=\pi b q_j-2\pi b^2 \textrm{ for }j\in\{1,2,3,4\}.
\end{align}
Then 
\begin{align*}
&\tau_1=\alpha_1+\alpha_2+\alpha_3,  &&   \tau_2=\alpha_1+\alpha_5+\alpha_6, \\
&\tau_3=\alpha_2+\alpha_4+\alpha_6, &&\tau_4=\alpha_3+\alpha_4+\alpha_5,\\
&\eta_1=\alpha_1+\alpha_2+\alpha_4+\alpha_5, &&\eta_2=\alpha_1+\alpha_3+\alpha_4+\alpha_6, \\
&\eta_3=\alpha_2+\alpha_3+\alpha_5+\alpha_6, &&\eta_4=2\pi.
\end{align*}
Let $\boldsymbol\alpha=(\alpha_1,\dots,\alpha_6)$ and let
\begin{equation*}
\begin{split}
U_{\boldsymbol \alpha,b}(\xi)= &\,{\pi^2b^2} -\pi \mathbf i b^2 \sum_{i=1}^4\sum_{j=1}^4 \log S_b(q_j-t_i) + 2\pi \mathbf i b^2 \sum_{i=1}^4  \log S_b(u-t_i) + 2\pi \mathbf i b^2 \sum _{j=1}^4\log S_b(q_j-u).
\end{split}
\end{equation*}
Recall $L(x)$ from (\ref{eq:Lx}). Define
\begin{align}
&U_{\boldsymbol \alpha}(\xi)=-\frac{1}{2}\sum_{i=1}^4\sum_{j=1}^4L(\eta_j-\tau_i)+\sum_{i=1}^4L(\xi-\tau_i)+\sum_{j=1}^4L(\eta_j-\xi).\label{U}\\
&\kappa_{\boldsymbol \alpha}(\xi)={9}\pi^2+14\pi\sum_{k=1}^6\alpha_k-28\pi\xi-4\pi \mathbf i\sum_{i=1}^4\log\Big(1-e^{2\mathbf i(\xi-\tau_i)}\Big)+3\pi \mathbf i\sum_{j=1}^4\log\Big(1-e^{2\mathbf i(\eta_j-\xi)}\Big).  \label{kappa}\\
&\nu_{\boldsymbol\alpha,b}(\xi)=\frac{U_{\boldsymbol \alpha,b}(\xi)-U_{\boldsymbol \alpha}(\xi)-\kappa_{\boldsymbol \alpha}(\xi)b^2}{b^4}.
\label{nuU}
\end{align} 
Then by (\ref{b-6j}), (\ref{nuU}) and Lemma \ref{lem:contour}, after change of variables and a deformation of the contour, we have  
\begin{equation}\label{6jint}
\begin{split}
\bigg\{\begin{matrix} a_1 & a_2 & a_3 \\ a_4 & a_5 & a_6 \end{matrix} \bigg\}_b=&\frac{1}{\pi b}\int_{\Gamma}\exp\bigg(\frac{U_{\boldsymbol \alpha,b}(\xi)}{2\pi \mathbf ib^2} \bigg)d\xi\\
=&\frac{1}{\pi b}\int_{\Gamma}\exp\bigg(\frac{U_{\boldsymbol \alpha}(\xi)+\kappa_{\boldsymbol \alpha}(\xi)b^2+
\nu_{\boldsymbol \alpha,b}(\xi)b^4}{2\pi \mathbf ib^2} \bigg)d\xi,\\
\end{split}
\end{equation}
where $\Gamma$ is a contour approaching $2\pi +\mathbf i\mathbb R$   near infinity, and passes the real axis in the interval  $\big(\frac{3\pi}{2}, 2\pi\big)$.

We prove Theorems~\ref{cov} and~\ref{cov2} by applying the saddle point approximations to the integral in~\eqref{6jint}. To meet the conditions, we need to choose a suitable contour in~\eqref{6jint} passing through the critical points. 
We also need to choose a proper neighborhood of the critical points so that the integral inside this neighborhood dominates the asymptotics, and the integral outside this neighborhood is negligible. This requires the functions $U_{\boldsymbol \alpha}(\xi)$, $\kappa_{\boldsymbol \alpha}(\xi)$, and $\nu_{\boldsymbol \alpha,b}(\xi)b^4$ to satisfy certain  desired properties, which we will prove in Section~\ref{subsec:U-property}. We then prove Theorem~\ref{cov} in Section~\ref{subsec: thm2} by applying Proposition~\ref{saddle}; and prove Theorem~\ref{cov2} in Section~\ref{subsec:thm3} by applying Proposition~\ref{saddle2}.

\subsection{Properties of relevant functions}\label{subsec:U-property}

Recall from section~\ref{extcov} that $\mathcal L\subset {\mathbb R_{\geqslant 0}}^6$ is the space of all truncated hyperideal tetrahedra parametrized by the edge lengths. Let 
$$D=\Big\{ \xi\in\mathbb C\ \Big|\ \frac{3\pi}{2}< \mathrm{Re}(\xi) < 2\pi \Big\} \quad\textrm{and hence}\quad \partial D=\Big\{ \xi\in\mathbb C\ \Big|\ \mathrm{Re}(\xi)=\frac{3\pi}{2}\ \text{or}\ \mathrm{Re}(\xi)=2\pi \Big\}.$$
Set $\alpha_k=\frac{\pi}{2}+\mathbf i \frac{l_k}{2}$ in~\eqref{eq:alpha-a} so that $a_k=\frac{Q}{2}+ \mathbf i \frac{l_k}{2\pi b}$ as in Theorems~\ref{cov} and~\ref{cov2}.

\begin{proposition}\label{critical2} For $k\in\{1,\dots,6\},$ let $\alpha_k=\frac{\pi}{2}\pm \mathbf i \frac{l_k}{2}$ with $l_k\in{\mathbb R_{\geqslant 0}},$ and let $\boldsymbol l = (l_1,\dots, l_6).$ 
\begin{enumerate}[(1)]  
\item If $\boldsymbol l\in\mathcal L,$ then the function $U_{\boldsymbol \alpha}(\xi)$ has a unique critical point $\xi^*$ (i.e., $U'_{\boldsymbol \alpha}(\xi)=0$) in the domain $D$. 
Furthermore, $\xi^*$ is non-degenerate (i.e., $U''_{\boldsymbol \alpha}(\xi)\neq 0$). Finally
$$U_{\boldsymbol \alpha}(\xi^*)=-2\mathbf i\mathrm{Cov}(l_1,\dots,l_6),$$
where $\mathrm{Cov}$ is the co-volume function as in (\ref{covfun}).

\item  If $\boldsymbol l\in\mathcal \partial L,$ then
there exists $\xi^*\in \partial D$ such that $\xi^*$ is the unique critical point of $U_{\boldsymbol \alpha}(\xi)$   in a neighborhood of $D\cup \partial D$. Furthermore, $\xi^*$ is degenerate with $U'''_{\boldsymbol \alpha}(\xi^*)\neq 0$, and
$$U_{\boldsymbol \alpha}(\xi^*)=-2\mathbf i\widetilde{\mathrm{Cov}}(l_1,\dots,l_6),$$
where $\widetilde{\mathrm{Cov}}$ is the extended co-volume function as in (\ref{extcovfun}).  Moreover, if $\alpha_k=\frac{\pi}{2}+ \mathbf i \frac{l_k}{2}$ for all $k\in\{1,\dots,6\},$ then 
$$\mathrm{Re}\xi^*=2\pi.$$

\item  If $\boldsymbol l\in {\mathbb R_{\geqslant 0}}^6\setminus (\mathcal L\cup\partial \mathcal L),$ then
there exist $\xi_1^*$ and $\xi_2^*$ lying on the same component of $\partial D$ such that they are the two critical points of $U_{\boldsymbol \alpha}$ in a neighborhood of $D\cup \partial D$. Furthermore, they are non-degenerate, and 
$$\mathrm{Im}U_{\boldsymbol \alpha}(\xi_1^*)=\mathrm{Im}U_{\boldsymbol \alpha}(\xi_2^*)=-2\widetilde{\mathrm{Cov}}(l_1,\dots,l_6)$$
and 
$$\mathrm{Re}U_{\boldsymbol \alpha}(\xi_1^*)=-\mathrm{Re}U_{\boldsymbol \alpha}(\xi_2^*).$$
Moreover, if $\alpha_k=\frac{\pi}{2} +  \mathbf i \frac{l_k}{2}$ for all $k\in\{1,\dots,6\},$ then 
$$\mathrm{Re}\xi_1^*=\mathrm{Re}\xi_2^*=2\pi.$$
\end{enumerate}
\end{proposition}

\begin{proof}
We first show that in Case (1) $U_{\boldsymbol \alpha}$ has a unique  critical point in $D,$  in Case (2) $U_{\boldsymbol \alpha}$ has a unique critical point on $\partial D,$ and in Case (3) $U_{\boldsymbol \alpha}$ has two critical points on $\partial D,$ following  arguments adapted from \cite{MY, U, BY}.  (In fact our $U_\alpha$ only differs from~\cite[Eq. (3.2)]{BY} by a sign.)
For $k\in\{1,\dots,6\},$ let $u_k=e^{2\mathbf i\alpha_k},$ and let $z=e^{-2\mathbf i\xi}.$ Then by a direct computation or~\cite[Eq.\ (3.9)]{BY} we have
\begin{equation}\label{eq:mod4}
\frac{\partial U_{\boldsymbol \alpha}}{\partial \xi}=2\mathbf i\log\frac{(1-zu_1u_2u_3)(1-zu_1u_5u_6)(1-zu_2u_4u_6)(1-zu_3u_4u_5)}{(1-z)(1-zu_1u_2u_4u_5)(1-zu_1u_3u_4u_6)(1-zu_2u_3u_5u_6)}\quad\quad (\mathrm{mod}\ 4\pi).
\end{equation} Here $\mathrm{mod}\ 4\pi$ is due to $2\mathbf i\log$ in~\eqref{eq:Lx}. Then, as a necessary condition, $U'_{\boldsymbol \alpha}(\xi)=0$ implies
\begin{equation*}
\frac{(1-zu_1u_2u_3)(1-zu_1u_5u_6)(1-zu_2u_4u_6)(1-zu_3u_4u_5)}{(1-z)(1-zu_1u_2u_4u_5)(1-zu_1u_3u_4u_6)(1-zu_2u_3u_5u_6)}=1,
\end{equation*}
which is equivalent to the following quadratic equation  (see~\cite[Eq. (3.11)]{BY})
$$Az^2+Bz+C=0$$
with 
\begin{equation*}
\begin{split}
A=&u_1u_4+u_2u_5+u_3u_6-u_1u_2u_6-u_1u_3u_5-u_2u_3u_4-u_4u_5u_6+u_1u_2u_3u_4u_5u_6,\\
B=&-\Big(u_1-\frac{1}{u_1}\Big)\Big(u_4-\frac{1}{u_4}\Big)-\Big(u_2-\frac{1}{u_2}\Big)\Big(u_5-\frac{1}{u_5}\Big)-\Big(u_3-\frac{1}{u_3}\Big)\Big(u_6-\frac{1}{u_6}\Big),\\
C=&\frac{1}{u_1u_4}+\frac{1}{u_2u_5}+\frac{1}{u_3u_6}-\frac{1}{u_1u_2u_6}-\frac{1}{u_1u_3u_5}-\frac{1}{u_2u_3u_4}-\frac{1}{u_4u_5u_6}+\frac{1}{u_1u_2u_3u_4u_5u_6}.\\
\end{split}
\end{equation*}
Since $\alpha_k=\frac{\pi}{2}\pm \mathbf i \frac{l_k}{2}$ for each $k,$  $u_k^{\pm 1} =e^{\pm 2\mathbf i\alpha_k}=-e^{\pm l_k}<0.$ As a consequence, $A>0,$ $C>0$ and $B$ is real. Furthermore, $B^2-4AC=16\det\mathrm{Gram}(\boldsymbol l)$ as explained in~\cite[Page 15]{BY}. 
Therefore by Proposition~\ref{tri} we have
\begin{equation}\label{eq:gram}
B^2-4AC
\left\{\begin{array}{ccl}
      <0 & \text{iff } & \boldsymbol l\in\mathcal L,\\
      =0 & \text{iff} & \boldsymbol l\in\partial \mathcal L,\\
       >0 & \text{iff} & \boldsymbol l\in {\mathbb R_{\geqslant 0}^6}\setminus (\mathcal L\cup\partial \mathcal L).\\
    \end{array} \right.
\end{equation}
Let 
\begin{equation}\label{xi}
z^*=e^{-2\mathbf i\xi^*}=\frac{-B+\sqrt{B^2-4AC}}{2A}\quad\text{and}\quad {z^{**}}=e^{-2\mathbf i{\xi^{**}}}=\frac{-B-\sqrt{B^2-4AC}}{2A}
\end{equation}
be the two roots of this quadratic equation. 
Then by~\eqref{eq:mod4} we have 
$$\frac{d U_{\boldsymbol \alpha}}{d \xi}\Big|_{\xi=\xi^*}=4k\pi\quad\text{and}\quad\frac{d U_{\boldsymbol \alpha}}{d \xi}\Big|_{\xi={\xi^{**}}}=4k'\pi$$
for some integers $k$ and $k'.$ We have the following case by case discussion:
\begin{enumerate}[(1)]

\item   For  $\boldsymbol l\in\mathcal L,$ we claim that, among the two, only $\xi^*$ is in $D$ and is a critical point of $U_{\boldsymbol \alpha}.$  Indeed, at the special case $l_1=\dots=l_6=0,$ ie., $\alpha_1=\dots=\alpha_6=\frac{\pi}{2},$ we have 
$A=C=8$ and $B=0.$ Then $z^*=\mathbf i$ and ${z^{**}}=-\mathbf i.$ Hence we can choose $\xi^*=\frac{7\pi}{4}\in (\frac{3\pi}{2},2\pi);$ and there is no ${\xi^{**}}\in (\frac{3\pi}{2},2\pi)$ making $e^{-2\mathbf i{\xi^{**}}}=-\mathbf i.$ A direct computation also shows that at this special case $k=0,$ ie.,
$$\frac{d U_{\boldsymbol \alpha}}{d\xi}\Big|_{\xi=\frac{7\pi}{4}}=0.$$

Now for general $\alpha_k$'s,  since $A>0,$ $B$ is real, and $B^2-4AC<0$,  we have $\mathrm{Im}z^*>0$. Moreover, we can choose a unique $\xi^*$ with $\mathrm{Re}\xi^*\in (\frac{3\pi}{2},2\pi).$  
Since $k$ is an integer valued continuous function on the connected space on $\mathcal L,$ it is constantly $0,$ ie., 
\begin{equation}\label{=0}
\frac{d U_{\boldsymbol \alpha}}{d \xi}\Big|_{\xi=\xi^*}=0
\end{equation}
and $\xi^*$ is the unique critical point of $U_{\boldsymbol \alpha}$ in  $D.$  

\item For $\boldsymbol l\in\partial \mathcal L,$  since $A>0,$ $C>0,$ $B$ is real, and $B^2-4AC=0$, we have $z^*\in \mathbb R\setminus \{0\},$ hence we can choose a unique $\xi^*$ with $\mathrm{Re}\xi^*=\frac{3\pi}{2}$ or $2\pi.$ Again by the continuity of the integer valued function $k$ on $\overline {\mathcal L},$ $k=0$ and $\xi^*$ is the unique critical point of $U_{\boldsymbol \alpha}$ on  $\partial D.$ 

Moreover, if $\alpha_k=\frac{\pi}{2}+\mathbf i \frac{l_k}{2}$ for each $k,$  then $u_k-\frac{1}{u_k}=e^{l_k}-e^{-l_k}>0$ for each $k.$ As a consequence, $B<0$ and $e^{-2\mathbf i\xi^*}=z^*=\frac{-B}{2A}>0$ and hence $\mathrm{Re}\xi^*=2\pi.$
\\

Here we also observe that as ${z^{**}}=z^*$ are double roots of the quadratic equation $Az^2+Bz+C=0,$ we can choose ${\xi^{**}}=\xi^*\in\partial D.$ Then we have  $k'=k=0$ on $\partial \mathcal L,$ which will be used in Case (3).

\item For $\boldsymbol l\in \mathbb R_{\geqslant 0}^6\setminus (\mathcal L\cup\partial \mathcal L),$ since $A>0,$ $C>0,$ $B$ is real and $B^2-4AC>0$,  we have $z^*, {z^{**}}\in \mathbb R\setminus \{0\}$ with $z^*\neq {z^{**}},$ hence we can choose $\xi_1^*=\xi^*,$ $\xi_2^*={\xi^{**}}$ with $\mathrm{Re}\xi^*,$ ${\mathrm{Re}\xi^{**}}$ equal to  $\frac{3\pi}{2}$ or $2\pi.$ By the continuity of the integer valued functions $k$ and $k'$ on $\mathbb R^6_{>0}\setminus\mathcal L$ and the fact that $k'=k=0$ on $\partial \mathcal L$ as discussed  at the end of Case (2), we have $k=k'=0$ on $\mathbb R^6_{>0}\setminus (\mathcal L\cup \partial \mathcal L)$ and $\xi_1^*$ and $\xi_2^*$ are both critical points of $U_{\boldsymbol \alpha}$ on  $\partial D.$ Next, if $B>0,$ then both $\frac{-B\pm\sqrt{B^2-4AC}}{2A}<0$  and hence both $\xi_1^*$ and $\xi_2^*$ are on $\Gamma_{\frac{3\pi}{2}};$ 
and if $B<0,$ then both $\frac{-B\pm \sqrt{B^2-4AC}}{2A}>0$  and hence both $\xi_1^*$ and $\xi_2^*$ are on $\Gamma_{2\pi}.$  

Moreover, if  $\alpha_k=\frac{\pi}{2}+\mathbf i \frac{l_k}{2}$  for each $k,$  then   $u_k-\frac{1}{u_k}=e^{l_k}-e^{-l_k}>0$ for each $k.$ As a consequence, $A>0,$ $B<0$ and $e^{-2\mathbf i\xi_1^*}=z^*=\frac{-B+\sqrt{B^2-4AC}}{2A}>0,$  $e^{-2\mathbf i\xi_2^*}={z^{**}}=\frac{-B-\sqrt{B^2-4AC}}{2A}>0$ and hence $\mathrm{Re}\xi_1^*=\mathrm{Re}\xi_2^*=2\pi.$
\end{enumerate}

Next, we compute the value of  $U_{\boldsymbol \alpha}(\xi^*)$ in Cases (1) and  (2) and the value of $\mathrm{Im}U_{\boldsymbol \alpha}(\xi^*)$ in Case (3) following the same argument as in \cite{BY}. For $\boldsymbol{\alpha}=(\frac{\pi}{2}\pm \mathbf i\frac{l_1}{2},\dots,\frac{\pi}{2}\pm \mathbf i\frac{l_6}{2}),$ define 
\begin{equation}\label{W}
W(\boldsymbol\alpha)=U_{\boldsymbol \alpha}(\xi^*).
\end{equation}
By a direct computation, we have
$$W\Big(\frac{\pi}{2},\dots,\frac{\pi}{2}\Big)=-16\mathbf i\Lambda\Big(\frac{\pi}{4}\Big)=-2\mathbf i\mathrm{Cov}(0,\dots,0);$$
and it suffices to prove that 
\begin{equation}\label{dCovdlin}
    \frac{\partial W(\boldsymbol l)}{\partial l_k}=-2\mathbf i \frac{\partial \mathrm{Cov}(\boldsymbol l)}{\partial l_k}=-\mathbf i\theta_k \quad \text{for } \boldsymbol l\in\mathcal L\cup\partial \mathcal L,
\end{equation}
and 
\begin{equation}\label{dCovdlout}
\frac{\partial\mathrm{Im} W(\boldsymbol l)}{\partial l_k}=-2\frac{\partial \widetilde{\mathrm{Cov}}(\boldsymbol l)}{\partial l_k}=-\widetilde \theta_k \quad \text{for } \boldsymbol l\in\mathbb R^6_{ \geqslant  0}\setminus (\mathcal L\cup\partial\mathcal L).
\end{equation}
Without loss of generality, let us verify it for  $k=1.$ By the same algebraic computation and argument in the proof of \cite[Theorem 3.5]{BY}, one has
\begin{equation}\label{pW}
{\frac{\partial W(\boldsymbol l)}{\partial l_1}=\frac{1}{2}\log\frac{G_{34}-\sqrt{G_{34}^2-G_{33}G_{44}}}{G_{34}+\sqrt{G_{34}^2-G_{33}G_{44}}},}
\end{equation}
where $G_{ij}$ is the $ij$-th cofactor of the Gram matrix $\mathrm{Gram}(\boldsymbol l).$ We have the following case by case discussion:
\begin{enumerate}[(1)]
\item If $\boldsymbol l\in\mathcal L,$ then 
$$\cos\theta_1=\frac{G_{34}}{\sqrt{G_{33}G_{44}}}$$
and 
$$\frac{\partial W(\boldsymbol l)}{\partial l_1}={\frac{1}{2}\log e^{-2\mathbf i\theta_1}}=-\mathbf i\theta_1.$$

\item If $\boldsymbol l\in\partial \mathcal L=X_1\cup X_2\cup X_3,$ 
then 
\begin{equation*}
\theta_1=\widetilde\theta_1=\left\{\begin{array}{ccl}
\pi & \text{if} & \boldsymbol l \in X_1,\\
0 & \text{if} & \boldsymbol l \in X_2\cup X_3,\\
\end{array}
\right.
\end{equation*}
 and as $\boldsymbol l$ is a limit point of $\mathcal L$ and by the continuity of $\frac{\partial W}{\partial l_1}$ and Case (1),  
$$\frac{\partial W(\boldsymbol l)}{\partial l_1}=-\mathbf i\theta_1=-\mathbf i\widetilde\theta_1.$$

\item If $\boldsymbol l\in \mathbb R^6_{\geqslant  0}\setminus (\mathcal L\cup \partial \mathcal L)=\Omega_1\cup \Omega_2\cup \Omega_3,$ then by Proposition \ref{tri},  
$G_{34}^2-G_{33}G_{44}=\sinh^2l_1\det\mathrm{Gram}(\boldsymbol l)>0,$
and hence
$\frac{G_{34}-\sqrt{G_{34}^2-G_{33}G_{44}}}{G_{34}+\sqrt{G_{34}^2-G_{33}G_{44}}}$ is real. Then by (\ref{pW}),  the continuity of  $\frac{\partial W}{\partial l_1},$ Case (2) and Proposition \ref{degeneration}, 
\begin{equation*}
\frac{\partial \mathrm{Im}W(\boldsymbol l)}{\partial l_1}=-\widetilde\theta_1=\left\{\begin{array}{ccl}
-\pi & \text{if} & \boldsymbol l \in \Omega_1,\\
0 & \text{if} & \boldsymbol l \in \Omega_2\cup \Omega_3.\\
\end{array}
\right.
\end{equation*}
\end{enumerate}
Putting (1)(2)(3) together, we verified that 
$$\frac{\partial W(\boldsymbol l)}{\partial l_1}=-\mathbf i\theta_1 \quad \text{for } \boldsymbol l\in \mathcal L\cup \partial \mathcal L,$$
and 
$$\frac{\partial\mathrm{Im} W(\boldsymbol l)}{\partial l_1}=-\widetilde \theta_1 \quad \text{for } \boldsymbol l\in\mathbb R^6_{\geqslant 0}\setminus (\mathcal L\cup\partial\mathcal L).$$

Moreover, define 
$$W^*(\boldsymbol\alpha)=U_{\boldsymbol\alpha}(\xi^{**}).$$
Then by the same algebraic computation and argument in the proof of \cite[Theorem 3.5]{BY}, one has
$$\frac{\partial}{\partial l_k}\big(\mathrm{Re}W(\boldsymbol\alpha)+\mathrm{Re}W^*(\boldsymbol\alpha)\big) = 0$$
for all $k\in\{1,\dots,6\}.$ As a consequence, $\mathrm{Re}W(\boldsymbol\alpha)+\mathrm{Re}W^*(\boldsymbol\alpha)=c$ for some constant $c$ and all $\boldsymbol\alpha$ in $(\frac{\pi}{2}+\mathbf i \mathbb R)^6;$ and a direct computation at $\boldsymbol \alpha=(\frac{\pi}{2},\dots,\frac{\pi}{2})$ shows that $c=0.$ This implies that 
$$\mathrm{Re}U_{\boldsymbol\alpha}(\xi^*)=-\mathrm{Re}U_{\boldsymbol\alpha}(\xi^{**}).$$

Finally, we show that the critical points in Cases (1) and (3) are non-degenerate, and in Case (2) the critical point is degenerate with a nonzero the third derivative. By Proposition \ref{tri} and  Proposition \ref{Hess} below, we have: 
\begin{enumerate}[(1)]

\item   For  $\boldsymbol l\in\mathcal L,$
$$U''_{\boldsymbol \alpha}(\xi^*)=- 16 \exp\Big(\frac{\kappa_{\boldsymbol \alpha}(\xi^*)}{\pi \mathbf i}\Big) \sqrt{\det\mathrm{Gram}(\boldsymbol l)} \neq 0,$$
and $\xi^*$ is a non-degenerate critical point.

\item  For  $\boldsymbol l\in\mathcal \partial L,$
$$U''_{\boldsymbol \alpha}(\xi^*)=- 16\exp\Big(\frac{\kappa_{\boldsymbol \alpha}(\xi^*)}{\pi \mathbf i}\Big)   \sqrt{\det\mathrm{Gram}(\boldsymbol l)} =0,$$
and $\xi^*$ is a degenerate critical point.

By a direct computation, one sees that 
$$U'''_{\boldsymbol \alpha}(\xi^*)=8\mathbf i\bigg(\sum_{i=1}^4\frac{e^{2\mathbf i(\xi^*-\tau_i)}}{\big(1-e^{2\mathbf i(\xi^*-\tau_i)}\big)^2}-\sum_{j=1}^4\frac{e^{2\mathbf i(\eta_j-\xi^*)}}{\big(1-e^{2\mathbf i(\eta_j-\xi^*)}\big)^2}\bigg).$$
We observe that all the summands in the parentheses are real when $\mathrm{Re}\xi^*=\frac{3\pi}{2}$ or $2\pi.$ 
If $\mathrm{Re}\xi^*=\frac{3\pi}{2},$ then $\frac{e^{2\mathbf i(\xi^*-\tau_i)}}{(1-e^{2\mathbf i(\xi^*-\tau_i)})^2}>0$ and $\frac{e^{2\mathbf i(\eta_j-\xi^*)}}{(1-e^{2\mathbf i(\eta_j-\xi^*)})^2}<0$ for all $i, j\in\{1,2,3,4\}.$ Hence $\mathrm{Im}  U'''_{\boldsymbol \alpha}(\xi^*) >0$ and $U'''_{\boldsymbol \alpha}(\xi^*)\neq 0.$ Similarly, if $\mathrm{Re}\xi^*=2\pi,$ then $\frac{e^{2\mathbf i(\xi^*-\tau_i)}}{(1-e^{2\mathbf i(\xi^*-\tau_i)})^2}<0$ and $\frac{e^{2\mathbf i(\eta_j-\xi^*)}}{(1-e^{2\mathbf i(\eta_j-\xi^*)})^2}>0$ for all $i, j\in\{1,2,3,4\},$ and hence $\mathrm{Im}  U'''_{\boldsymbol \alpha}(\xi^*) < 0$ and $U'''_{\boldsymbol \alpha}(\xi^*)\neq 0.$ 

\item For $\boldsymbol l\in {\mathbb R_{\geqslant 0}}^6\setminus (\mathcal L\cup\partial \mathcal L)$ and  $i=1,2,$ 
$$U''_{\boldsymbol \alpha}(\xi_i^*)=- 16\exp\Big(\frac{\kappa_{\boldsymbol \alpha}(\xi_i^*)}{\pi \mathbf i}\Big)\big(\pm \sqrt{\det\mathrm{Gram}(\boldsymbol l)} \big)\neq 0,$$
and $\xi_1^*$ and $\xi_2^*$ are  non-degenerate critical points.
\end{enumerate} 
This completes the proof. 
\end{proof}

\begin{proposition}\label{Hess} For $\boldsymbol l=(l_1,\dots,l_6)\in{\mathbb R_{\geqslant 0}}^6,$ let $\boldsymbol \alpha=\big(\frac{\pi}{2}\pm \mathbf i\frac{l_1}{2},\dots,\frac{\pi}{2}\pm \mathbf i\frac{l_6}{2}\big).$ When $\boldsymbol l\in\mathcal L\cup\partial \mathcal L,$ let $\xi^*$ be the critical point of $U_{\boldsymbol \alpha}$ in $\overline D=D\cup \partial D;$ and when  $\boldsymbol l\in {\mathbb R_{\geqslant 0}}^6\setminus (\mathcal L\cup\partial \mathcal L),$ let $\xi^*$ be the critical point $\xi^*_1$ of $U_{\boldsymbol \alpha}$ and  let ${\xi^{**}}$ be the other critical point $\xi^*_1$ of $U_{\boldsymbol \alpha}$ on  $\partial D.$ 
  Then 
\begin{equation*}
\frac{-U''_{\boldsymbol \alpha}(\xi^*)}{\exp\big(\frac{\kappa_{\boldsymbol \alpha}(\xi^*)}{\pi \mathbf i}\big)}=16\sqrt{\det\mathrm{Gram}(\boldsymbol l)},
\end{equation*}
and
\begin{equation*}
\frac{-U''_{\boldsymbol \alpha}({\xi^{**}})}{\exp\big(\frac{\kappa_{\boldsymbol \alpha}({\xi^{**}})}{\pi \mathbf i}\big)}=-16\sqrt{\det\mathrm{Gram}(\boldsymbol l)}.
\end{equation*}
\end{proposition}

\begin{proof} The proof follows the argument adapted from that of \cite[Lemma 3]{CM}. Namely, from $U'_{\boldsymbol \alpha}(\xi^*)=0,$
we have 
$$\sum_{i=1}^4\log\big(1-e^{2\mathbf i(\xi^*-\tau_i)}\big)=\sum_{j=1}^4\log\big(1-e^{2\mathbf i(\eta_j-\xi^*)}\big)-8\mathbf i\xi^*+4\mathbf i\sum_{k=1}^6\alpha_k+2\pi \mathbf i;$$
and plugging in  (\ref{kappa}), we have
$$\kappa_{\boldsymbol \alpha}(\xi^*)={3}\pi^2+2\pi\sum_{k=1}^6\alpha_k-4\pi\xi^*-\pi \mathbf i\sum_{i=1}^4\log\big(1-e^{2\mathbf i(\xi^*-\tau_i)}\big).$$
For $k\in\{1,\dots,6\},$ let $u_k=e^{2\mathbf i\alpha_k},$ and let $z^*=e^{-2\mathbf i \xi^*}$ be the root of the quadratic equation as in (\ref{xi}).  Then
\begin{equation*}
\begin{split}
{U''_{\boldsymbol \alpha}(\xi^*)}=4\bigg(&\frac{z^*}{1-z^*}+\frac{z^*u_1u_2u_4u_5}{1-z^*u_1u_2u_4u_5}+\frac{z^*u_1u_3u_4u_6}{1-z^*u_1u_3u_4u_6}+\frac{z^*u_2u_3u_5u_6}{1-z^*u_2u_3u_5u_6}\\
& -\frac{z^*u_1u_2u_3}{1-z^*u_1u_2u_3}-\frac{z^*u_1u_5u_6}{1-z^*u_1u_5u_6}-\frac{z^*u_2u_4u_6}{1-z^*u_2u_4u_6}-\frac{z^*u_3u_4u_5}{1-z^*u_3u_4u_5}\bigg).
\end{split}
\end{equation*}
As a consequence, 
\begin{equation*}
\begin{split}
\frac{-U''_{\boldsymbol \alpha}(\xi^*)}{\exp\big({\frac{\kappa_{\boldsymbol \alpha}(\xi^*)}{\pi \mathbf i}}\big)}=&\frac{4(1-z^*u_1u_2u_3)(1-z^*u_1u_5u_6)(1-z^*u_2u_4u_6)(1-z^*u_3u_4u_5)}{z^{*2}u_1u_2u_3u_4u_5u_6}\\
&\bigg(\frac{z^*}{1-z^*}+\frac{z^*u_1u_2u_4u_5}{1-z^*u_1u_2u_4u_5}+\frac{z^*u_1u_3u_4u_6}{1-z^*u_1u_3u_4u_6}+\frac{z^*u_2u_3u_5u_6}{1-z^*u_2u_3u_5u_6}\\
&\quad-\frac{z^*u_1u_2u_3}{1-z^*u_1u_2u_3}-\frac{z^*u_1u_5u_6}{1-z^*u_1u_5u_6}-\frac{z^*u_2u_4u_6}{1-z^*u_2u_4u_6}-\frac{z^*u_3u_4u_5}{1-z^*u_3u_4u_5}\bigg);
\end{split}
\end{equation*}
and by a direct computation, this equals
$$4\bigg(3Az^*+2B+\frac{C}{z^*}\bigg)=4\bigg(Az^*-\frac{C}{z^*}\bigg),$$
which in turn equals,
$$4A(z^*-{z^{**}})=4\sqrt{B^2-4AC}=16\sqrt{\det\mathrm{Gram}(\boldsymbol l)},$$
where ${z^{**}}$ is the other root of the quadratic equation as in (\ref{xi}).  

By a similar computation, we have
$$\frac{-U''_{\boldsymbol \alpha}(\xi^{**})}{\exp\big({\frac{\kappa_{\boldsymbol \alpha}({\xi^{**}})}{\pi \mathbf i}}\big)}=4A({z^{**}}-z^*)=-16\sqrt{\det\mathrm{Gram}(\boldsymbol l)}.$$
This completes the proof.
 \end{proof}

\begin{proposition}\label{concave}
\begin{enumerate}[(1)]
\item
In the domain $D,$ the function $\mathrm{Im}U_{\boldsymbol \alpha}(\xi)$ is {strictly} concave down in $\mathrm{Im}\xi.$
 \item 
On $\partial D,$  the function $\mathrm{Im}U_{\boldsymbol \alpha}(\xi)$ is piecewise linear and concave down in $\mathrm{Im}\xi.$

\end{enumerate}
\end{proposition}

\begin{proof}  By a direct computation, we have
\begin{equation}\label{2nd}
\begin{split}
\frac{\partial ^2 \mathrm{Im}U_{\boldsymbol \alpha}(\xi)}{\partial \mathrm{Im}\xi^2}&\\
=-\sum_{i=1}^4&\frac{{\sin\big(2\mathrm{Re}(\xi-\tau_i)\big)}}{\cosh^2\mathrm{Im}(\xi-\tau_i)-\cos^2\mathrm{Re}(\xi-\tau_i)}-\sum_{j=1}^4\frac{{\sin\big(2\mathrm{Re}(\eta_j-\xi)\big)}}{\cosh^2\mathrm{Im}(\eta_j-\xi)-\cos^2\mathrm{Re}(\eta_j-\xi)}.
\end{split}
\end{equation}

To see (1), for  $\xi\in D,$ we have $\mathrm{Re}(\xi-\tau_i) \in (0,\frac{\pi}{2})$ for $i\in\{1,\dots,4\},$ and $\mathrm{Re}(\eta_j-\xi) \in (0,\frac{\pi}{2})$ for $j\in\{1,\dots,4\}.$ As a consequence, $\sin\big(2\mathrm{Re}(\xi-\tau_i)\big)>0$ for each $i,$ and $\sin\big(2\mathrm{Re}(\eta_j-\xi)\big)>0$  for each $j.$ Then by (\ref{2nd}),  
\begin{equation*}
\frac{\partial ^2 \mathrm{Im}U_{\boldsymbol \alpha}(\xi)}{\partial \mathrm{Im}\xi^2}<0,\\
\end{equation*}
and $\mathrm{Im}U_{\boldsymbol \alpha}$ is  {strictly}  concave down in $\mathrm{Im}\xi.$
\\

To see (2), for the concavity, we observe that $U_{\boldsymbol \alpha}(\xi)$  is a pointwise limit of a sequence  of concave down functions, each of which is the restriction of $U_{\boldsymbol \alpha}$ on the vertical lines $\{x_n+\mathbf iy\ |\ y\in \mathbb R\}$ with $x_n\in (\frac{3\pi}{2},2\pi)$ approaching $\frac{3\pi}{2}$ or $2\pi.$ For the piecewise linearity, we have $\mathrm{Re}(\xi-\tau_i) \in \{0,\frac{\pi}{2}\}$ for $i\in\{1,\dots,4\},$ and $\mathrm{Re}(\eta_j-\xi) \in \{0,\frac{\pi}{2}\}$ for $j\in\{1,\dots,4\}.$ As a consequence, $\sin\big(2\mathrm{Re}(\xi-\tau_i)\big)=0$ for each $i,$ and $\sin\big(2\mathrm{Re}(\eta_j-\xi)\big)=0$  for each $j.$ Then by (\ref{2nd}),
\begin{equation*}
\frac{\partial ^2 \mathrm{Im}U_{\boldsymbol \alpha}(\xi)}{\partial \mathrm{Im}\xi^2}=0
\end{equation*}
unless $\xi=\tau_i$ for some $i\in\{1,2,3,4\}$ or $\xi=\eta_j$ for some $j\in\{1,2,3,4\},$ which are the singular points. 
\end{proof}

To understand  the function $\mathrm{Im}U_{\boldsymbol \alpha},$ in the following Propositions \ref{limder}, \ref{PL}, \ref{limder2} and Corollaries \ref{cor3}, \ref{cor4} we study various behaviors of its partial derivatives. For $c\in[\frac{3\pi}{2},2\pi],$ we consider the contour 
$$\Gamma_c=\Big\{\xi\in \overline D\ \Big|\ \mathrm{Re}\xi =c  \Big\}.$$

\begin{proposition} \label{limder}
On $\Gamma_c,$ we have
\item $$\lim_{\mathrm{Im}\xi\to +\infty}\frac{\partial \mathrm{Im}U_{\boldsymbol \alpha}(\xi)}{\partial \mathrm{Im}\xi} = -4\pi\quad\text{and}\quad \lim_{\mathrm{Im}\xi\to -\infty}\frac{\partial \mathrm{Im}U_{\boldsymbol \alpha}(\xi)}{\partial \mathrm{Im}\xi} = 4\pi.$$
\end{proposition}

\begin{proof} By a direct computation, we have
$$ \frac{\partial \mathrm{Im}L(\alpha+\mathbf il)}{\partial l} =2\alpha{-\pi}-2\arg\Big(1-e^{-2l+2\mathbf i\alpha}\Big),$$
and
$$ \frac{\partial  \mathrm{Im}L(\alpha-\mathbf il)}{\partial l}=-2\alpha{+\pi}+2\arg\Big(1-e^{2l+2\mathbf i\alpha}\Big).$$
As a consequence, for a fixed $\alpha\in [0,\frac{\pi}{2}],$ we have
$$\lim_{l\to +\infty} \frac{\partial \mathrm{Im}L(\alpha+\mathbf il)}{\partial l} =2\alpha {-\pi} -2\arg(1)=2\alpha {-\pi},$$
$$\lim_{l\to -\infty} \frac{\partial \mathrm{Im}L(\alpha+\mathbf il)}{\partial l} =2\alpha +{\pi} -2\arg(-e^{2\mathbf i\alpha})= {\pi}-2\alpha;$$
and 
$$\lim_{l\to +\infty} \frac{\partial  \mathrm{Im}L(\alpha-\mathbf il)}{\partial l} =-2\alpha{-\pi}+2\arg(-e^{2\mathbf i\alpha})=2\alpha{-\pi},$$
$$\lim_{l\to -\infty} \frac{\partial \mathrm{Im}L(\alpha-\mathbf il)}{\partial l} =-2\alpha+{\pi}+2\arg(1)={\pi}-2\alpha.$$
Then 
$$\frac{\partial \mathrm{Im}U_{\boldsymbol \alpha}(\xi)}{\partial \mathrm{Im}\xi} =\sum_{i=1}^4\frac{\partial L(\xi-\tau_i)}{\partial \mathrm{Im}\xi}+\sum_{j=1}^4\frac{\partial L(\eta_j-\xi)}{\partial \mathrm{Im}\xi},$$
with
$$\lim_{\mathrm{Im}\xi\to +\infty}\frac{\partial \mathrm{Im}U_{\boldsymbol \alpha}(\xi)}{\partial \mathrm{Im}\xi}=\sum_{i=1}^4{\big(2\mathrm{Re}(\xi-\tau_i)-\pi\big)}+\sum_{j=1}^4{\big(2\mathrm{Re}(\eta_j-\xi)-\pi\big)}=-4\pi,$$
and
$$\lim_{\mathrm{Im}\xi\to -\infty}\frac{\partial \mathrm{Im}U_{\boldsymbol \alpha}(\xi)}{\partial \mathrm{Im}\xi}=\sum_{i=1}^4{\big(\pi-2\mathrm{Re}(\xi-\tau_i)\big)}+\sum_{j=1}^4{\big(\pi-2\mathrm{Re}(\eta_j-\xi)\big)}=4\pi.\qedhere$$
\end{proof}

An immediate consequence of Propositions \ref{concave} and \ref{limder} is the following 

\begin{corollary}\label{cor3} For each $c\in(\frac{3\pi}{2},2\pi),$  $\mathrm{Im}U_{\boldsymbol \alpha}(\xi)$ has a unique critical point $\xi^*_c$ on $\Gamma_c,$ which achieves the maximum.
\end{corollary}

In the rest of this section, we will re-index $\tau_i$'s so  that $\mathrm{Im}\tau_{i_1}\leqslant \mathrm{Im}\tau_{i_2}\leqslant \mathrm{Im}\tau_{i_3}\leqslant \mathrm{Im}\tau_{i_4},$ and re-index $\eta_j$'s so that $\mathrm{Im}\eta_{j_1}\leqslant \mathrm{Im}\eta_{j_2}\leqslant \mathrm{Im}\eta_{j_3}\leqslant \mathrm{Im}\eta_{j_4},$ and as a notion let $\mathrm{Im}\tau_{i_0}=\mathrm{Im}\eta_{j_0}=-\infty$ and $\mathrm{Im}\tau_{i_5}=\mathrm{Im}\eta_{j_5}=+\infty.$

\begin{proposition}\label{PL} For $k\in\{0,\dots, 4\}$ and  $\xi\in \Gamma_{\frac{3\pi}{2}}$ with $\mathrm{Im}\xi\in (\mathrm{Im} {\tau_{i_k}} ,\mathrm{Im}\tau_{i_{k+1}})$ or $\xi\in \Gamma_{2\pi}$ with $\mathrm{Im}\xi\in (\mathrm{Im} {\eta_{j_k}}, \mathrm{Im}\eta_{j_{k+1}}),$  
$$\frac{\partial \mathrm{Im}U_{\boldsymbol \alpha}(\xi)}{\partial \mathrm{Im}\xi}=4\pi - 2k\pi. $$
In particular, for $k=2,$ 
$$\frac{\partial \mathrm{Im}U_{\boldsymbol \alpha}(\xi)}{\partial \mathrm{Im}\xi}=0.$$
\end{proposition}

\begin{proof}  We have 
$$\frac{\partial L(\xi-\tau_i)}{\partial \mathrm{Im}\xi}=2\mathrm{Re}(\xi-\tau_i){-\pi}-2\arg\Big(1-e^{-2\mathrm{Im}(\xi-\tau_i)+2\mathbf i\mathrm{Re}(\xi-\tau_i)}\Big)$$
for $i\in\{1,2,3,4\},$ and 
$$\frac{\partial L(\eta_j-\xi)}{\partial \mathrm{Im}\xi}=-2\mathrm{Re}(\eta_j-\xi){+\pi}+2\arg\Big(1-e^{-2\mathrm{Im}(\eta_j-\xi)+2\mathbf i\mathrm{Re}(\eta_j-\xi)}\Big)$$
for $j\in\{1,2,3,4\}.$

If $\xi\in \Gamma_{\frac{3\pi}{2}},$ then $\mathrm{Re}(\xi-\tau_i)=0.$ In addition,  if $\mathrm{Im}\xi>\mathrm{Im}\tau_i,$ then
$$1-e^{-2\mathrm{Im}(\xi-\tau_i)+2\mathbf i\mathrm{Re}(\xi-\tau_i)}= 1-e^{-2\mathrm{Im}(\xi-\tau_i)}>0;$$ and if $\mathrm{Im}\xi<\mathrm{Im}\tau_i,$ then
$$1-e^{-2\mathrm{Im}(\xi-\tau_i)+2\mathbf i\mathrm{Re}(\xi-\tau_i)}= 1-e^{-2\mathrm{Im}(\xi-\tau_i)}<0.$$
 As a consequence, for  $\xi\in \Gamma_{\frac{3\pi}{2}}$  with $\mathrm{Im} {\tau_{i_k}}< \mathrm{Im}\xi < \mathrm{Im} {\tau_{i_{k+1}}},$ we have
\begin{equation}\label{e1}
\frac{\partial L(\xi-\tau_i)}{\partial \mathrm{Im}\xi}={-\pi}
\end{equation}
if $\mathrm{Im}\tau_i\leqslant \mathrm{Im}\tau_{i_k},$ 
and 
\begin{equation}\label{e2}
\frac{\partial L(\xi-\tau_i)}{\partial \mathrm{Im}\xi}={\pi}
\end{equation}
if $\mathrm{Im}\tau_i\geqslant \mathrm{Im}\tau_{i_{k+1}}.$ In the latter case the argument  of $1-e^{-2\mathrm{Im}(\xi-\tau_i)+2\mathbf i\mathrm{Re}(\xi-\tau_i)}$ equals $-\pi$ because as $\xi\in D$ approaches $\Gamma_{\frac{3\pi}{2}},$ $\mathrm{Im}\xi$ approaches $\frac{3\pi}{2}$ from above. As a consequence, $1-e^{-2\mathrm{Im}(\xi-\tau_i)+2\mathbf i\mathrm{Re}(\xi-\tau_i)}$ approaches the negative real axis from below and the argument approaches $-\pi.$ On the other hand, we  have $\mathrm{Re}(\eta_j-\xi)=\frac{\pi}{2}$ and $1-e^{-2\mathrm{Im}(\eta_j-\xi)+2\mathbf i\mathrm{Re}(\eta_j-\xi)}>0,$ which implies that
\begin{equation}\label{e3}
\frac{\partial L(\eta_j-\xi)}{\partial \mathrm{Im}\xi}={0}
\end{equation}
for each $j\in\{1,2,3,4\}.$  Putting (\ref{e1}), (\ref{e2}) and (\ref{e3}) together, we have 
$$\frac{\partial \mathrm{Im}U_{\boldsymbol \alpha}(\xi)}{\partial \mathrm{Im}\xi}={k\cdot (-\pi)+(4-k)\cdot \pi-4\cdot 0}=4\pi - 2k\pi.$$

The case that $\xi\in \Gamma_{2\pi}$ and $\mathrm{Im}\xi\in (\mathrm{Im} {\eta_{j_k}} ,\mathrm{Im}\eta_{j_{k+1}})$ is very similar. Namely, we have  $\mathrm{Re}(\xi-\tau_i)=\frac{\pi}{2}$ and $1-e^{-2\mathrm{Im}(\xi-\tau_i)+2\mathbf i\mathrm{Re}(\xi-\tau_i)}>0,$ which implies
\begin{equation}\label{e4}
\frac{\partial L(\xi-\tau_i)}{\partial \mathrm{Im}\xi}={0}.
\end{equation}
One the other hand, we have $\mathrm{Re}(\eta_j-\xi)=0,$ $1-e^{-2\mathrm{Im}(\eta_j-\xi)+2\mathbf i\mathrm{Re}(\eta_j-\xi)}<0$
when $\mathrm{Im}\xi > \mathrm{Im}\eta_j,$ 
and
$1-e^{-2\mathrm{Im}(\eta_j-\xi)+2\mathbf i\mathrm{Re}(\eta_j-\xi)}>0$
when $\mathrm{Im}\xi > \mathrm{Im}\eta_j.$ As a consequence, we have
\begin{equation}\label{e5}
\frac{\partial L(\eta_j-\xi)}{\partial \mathrm{Im}\xi}={-\pi}
\end{equation}
if $\mathrm{Im}\eta_j\leqslant \mathrm{Im}\eta_{j_{k}},$ where the argument of $1-e^{-2\mathrm{Im}(\eta_j-\xi)+2\mathbf i\mathrm{Re}(\eta_j-\xi)}$ equals $-\pi$ by a similar reason as in the previous case;
and
\begin{equation}\label{e6}
\frac{\partial L(\eta_j-\xi)}{\partial \mathrm{Im}\xi}={\pi}
\end{equation}
if $\mathrm{Im}\eta_j\geqslant \mathrm{Im}\eta_{j_{k+1}}.$ 
Putting (\ref{e4}), (\ref{e5}) and (\ref{e6}) together, we have
$$\frac{\partial \mathrm{Im}U_{\boldsymbol \alpha}(\xi)}{\partial \mathrm{Im}\xi}={4\cdot 0+k\cdot (-\pi)+(4-k)\cdot \pi}=4\pi - 2k\pi. \qedhere$$
\end{proof}

\begin{proposition} \label{limder2} For all $i,j\in\{1,2,3,4\},$ we have
\begin{equation*}
\lim_{\xi\to\tau_i}\frac{\partial \mathrm{Im}U_{\boldsymbol \alpha}(\xi)}{\partial \mathrm{Re}\xi}= -\infty \quad\text{and}\quad \lim_{\xi\to\eta_j}\frac{\partial \mathrm{Im}U_{\boldsymbol \alpha}(\xi)}{\partial \mathrm{Re}\xi}= +\infty.
\end{equation*}
\end{proposition}

\begin{proof} By a direct computation, we have
$$ \frac{\partial \mathrm{Im}L(\alpha+\mathbf il)}{\partial \alpha} =2\log\Big|1-e^{2\mathbf i(\alpha+\mathbf i l)}\Big|+2l,$$
$$ \frac{\partial \mathrm{Im}L(-\alpha+\mathbf il)}{\partial \alpha} =-2\log\Big|1-e^{2\mathbf i(-\alpha+\mathbf i l)}\Big|-2l,$$
and $ \frac{\partial \mathrm{Im}L(\pm \alpha+\mathbf il)}{\partial \alpha} $ converges to a finite value as $\pm \alpha + \mathbf i l$ tend to any points other than $k\pi$ for any integer $k.$ 
As a consequence, 
$$\lim_{\alpha +\mathbf i  l\to 0} \frac{\partial \mathrm{Im}L(\alpha+\mathbf il)}{\partial \alpha} =-\infty\quad\text{and}\quad \lim_{-\alpha+\mathbf i l\to 0} \frac{\partial \mathrm{Im}L(-\alpha+\mathbf il)}{\partial \alpha} = + \infty.$$
Then 
$$\frac{\partial \mathrm{Im}U_{\boldsymbol \alpha}(\xi)}{\partial \mathrm{Re}\xi} =\sum_{i=1}^4\frac{\partial L(\xi-\tau_i)}{\partial \mathrm{Re}\xi}+\sum_{j=1}^4\frac{\partial L(\eta_j-\xi)}{\partial \mathrm{Re}\xi},$$
which implies 
$$\lim_{\xi\to\tau_i}\frac{\partial \mathrm{Im}U_{\boldsymbol \alpha}(\xi)}{\partial \mathrm{Re}\xi}= -\infty \quad\text{and}\quad \lim_{\xi\to\eta_j}\frac{\partial \mathrm{Im}U_{\boldsymbol \alpha}(\xi)}{\partial \mathrm{Re}\xi}= +\infty$$
for all $i$ and $j$ in $\{1,2,3,4\}.$
\end{proof}

By Proposition \ref{limder2} and the continuity of $\frac{\partial \mathrm{Im}U_{\boldsymbol \alpha}(\xi)}{\partial \mathrm{Re}\xi},$ we have the following

\begin{corollary} \label{cor4}
There exists an $\delta>0$ such that  for each $i\in\{1,2,3,4\},$ if $\xi\in \overline D$ lying  in the $\delta$-square neighborhood of $\tau_i,$ then 
$$\frac{\partial \mathrm{Im}U_{\boldsymbol \alpha}(\xi)}{\partial \mathrm{Re}\xi}<0;$$
and for each $j\in\{1,2,3,4\},$ if $\xi\in \overline D$ lying in the $\delta$-square neighborhood of $\eta_j,$ then
$$\frac{\partial \mathrm{Im}U_{\boldsymbol \alpha}(\xi)}{\partial \mathrm{Re}\xi}>0.$$ I.e., in the intersection of the $\overline D$ with the $\delta$-square neighborhood of $\tau_i$'s and $\eta_j$'s the vector field 
$$\mathbf v(\xi)= \bigg(-\frac{\partial \mathrm{Im}U_{\boldsymbol\alpha}(\xi)}{\partial \mathrm{Re}\xi},0\bigg)$$
points inside the region $D.$
\end{corollary}

For $z_1,z_2\in \mathbb C,$ we denote by $\mathrm I_{z_1z_2}$ the line segment connecting $z_1$ and $z_2.$ By Proposition \ref{PL}, if  $\boldsymbol l\in \partial \mathcal L$ and $\xi^*$ is the unique critical point of $U_{\boldsymbol\alpha},$ then either  $\xi^*\in \mathrm I_{\tau_{i_2}\tau_{i_3}}$ or  $\xi^*\in \mathrm I_{\eta_{j_2}\eta_{j_3}};$ and  if  $\boldsymbol l\in {\mathbb R_{>0}}^6\setminus (\mathcal L\cup\partial \mathcal L)$ and $\xi^*_1$ and $\xi^*_2$  are  the  critical points of $U_{\boldsymbol\alpha},$ then either  $\{\xi^*_1,\xi^*_2\}\subset  \mathrm I_{\tau_{i_2}\tau_{i_3}}$ or  $\{\xi^*_1,\xi^*_2\}\subset  \mathrm I_{\eta_{j_2}\eta_{j_3}}.$ Among $\mathrm I_{\tau_{i_2}\tau_{i_3}}$ and $\mathrm I_{\eta_{j_2}\eta_{j_3}},$ we denote by $\mathrm I$ the one  that contains the critical points; and by Proposition \ref{limder2}, the critical points cannot be the end points of $\mathrm I.$

\begin{proposition}\label{nonvanish} 
 Consider the vector field 
 $$\mathbf v(\xi)= \bigg(-\frac{\partial \mathrm{Im}U_{\boldsymbol\alpha}(\xi)}{\partial\mathrm{Re} \xi},0\bigg).$$
\begin{enumerate}[(1)]
\item If  $\boldsymbol l\in \partial \mathcal L,$  then $\mathbf v$ is nowhere vanishing and pointing  inside $D$ on $\mathrm I\setminus \{\xi^*\}.$

\item  If $\boldsymbol l\in {\mathbb R_{>0}}^6\setminus (\mathcal L\cup\partial \mathcal L),$ then  $\mathbf v$ is nowhere vanishing on $\mathrm I\setminus \{\xi^*_1,\xi^*_2\},$ pointing inside $D$ on $\mathrm I\setminus \mathrm I_{\xi^*_1\xi^*_2}$  and pointing outside $D$ on $\mathrm I_{\xi^*_1\xi^*_2}.$ 
\end{enumerate}
\end{proposition}

\begin{proof}  For (1), by Proposition \ref{PL}, $\mathrm{Im}U_{\boldsymbol\alpha}$ is a constant on $\mathrm I,$ hence $\frac{\partial \mathrm{Im}U_{\boldsymbol\alpha}}{\partial \mathrm{Im}\xi} \equiv 0$ on $\mathrm I;$ and since $\xi^*$ is the only critical point of $\mathrm U_{\boldsymbol\alpha},$ we have $\frac{\partial \mathrm{Im}U_{\boldsymbol\alpha}}{\partial \mathrm{Re}\xi}\neq 0$ on $\mathrm I\setminus \{\xi^*\}.$ By Proposition \ref{cor4}, $\frac{\partial \mathrm{Im}U_{\boldsymbol\alpha}}{\partial \mathrm{Re}\xi} < 0$ near the end points of $\mathrm I$ when  $\xi^*\in\Gamma_{\frac{3\pi}{2}},$ and $\frac{\partial \mathrm{Im}U_{\boldsymbol\alpha}}{\partial \mathrm{Re}\xi} > 0$ near the end points of $\mathrm I$ when  $\xi^*\in\Gamma_{2\pi}.$ Therefore, $\frac{\partial \mathrm{Im}U_{\boldsymbol\alpha}}{\partial \mathrm{Re}\xi} < 0$ on $\mathrm I\setminus \{\xi^*\}$ when  $\xi^*\in\Gamma_{\frac{3\pi}{2}},$ and $\frac{\partial \mathrm{Im}U_{\boldsymbol\alpha}}{\partial \mathrm{Re}\xi} > 0$ on $\mathrm I\setminus \{\xi^*\}$  when  $\xi^*\in\Gamma_{2\pi}.$

For (2), by  Proposition \ref{PL},  $\mathrm{Im}U_{\boldsymbol\alpha}$ is a constant on $\mathrm I,$ hence $\frac{\partial \mathrm{Im}U_{\boldsymbol\alpha}}{\partial \mathrm{Im}\xi} \equiv 0$ on $\mathrm I;$ and since $\xi_1^*$ and $\xi_2^*$ are  the only critical points of $\mathrm U_{\boldsymbol\alpha},$ we have $\frac{\partial \mathrm{Im}U_{\boldsymbol\alpha}}{\partial \mathrm{Re}\xi}\neq 0$ on $\mathrm I\setminus \{\xi_1^*,\xi_2^*\}.$ By Proposition \ref{cor4}, $\frac{\partial \mathrm{Im}U_{\boldsymbol\alpha}}{\partial \mathrm{Re}\xi} < 0$ near the end points of $\mathrm I$ when  $\xi^*\in\Gamma_{\frac{3\pi}{2}},$ and $\frac{\partial \mathrm{Im}U_{\boldsymbol\alpha}}{\partial \mathrm{Re}\xi} > 0$ near the end points of $\mathrm I$ when  $\xi^*\in\Gamma_{2\pi}.$ Therefore, $\frac{\partial \mathrm{Im}U_{\boldsymbol\alpha}}{\partial \mathrm{Re}\xi} < 0$ on $\mathrm I\setminus \mathrm I_{\xi^*_1\xi^*_2}$ when  $\xi^*\in\Gamma_{\frac{3\pi}{2}},$ and $\frac{\partial \mathrm{Im}U_{\boldsymbol\alpha}}{\partial \mathrm{Re}\xi} > 0$ on $\mathrm I\setminus \mathrm I_{\xi^*_1\xi^*_2}$  when  $\xi^*\in\Gamma_{2\pi}.$  On the other hand, since  $\mathrm I$ is a level set of $\mathrm{Im}U_{\boldsymbol\alpha}$ passing through non-degenerate critical points,  the monotonicity of the function along the direction orthogonal to the level set (which is the $\frac{\partial}{\partial \mathrm{Re}\xi}$ direction in our case) changes sign on the two sides of each of the critical points. As a consequence, we have $\frac{\partial \mathrm{Im}U_{\boldsymbol\alpha}}{\partial \mathrm{Re}\xi} > 0$ on  the interior of $\mathrm I_{\xi^*_1\xi^*_2}$ when  $\xi^*\in\Gamma_{\frac{3\pi}{2}},$ and $\frac{\partial \mathrm{Im}U_{\boldsymbol\alpha}}{\partial \mathrm{Re}\xi} < 0$ on the interior of $\mathrm I_{\xi^*_1\xi^*_2}$  when  $\xi^*\in\Gamma_{2\pi}.$ 
\end{proof}

For the proof of Theorem \ref{cov2}, we need to consider the following bigger  region $D_{\delta,c}=D^{\boldsymbol \alpha}_{\delta,c},$ where $\boldsymbol\alpha\in \big(\frac{\pi}{2}+\mathbf i\mathbb R\big)^6,$  $\delta >0$ is sufficiently small and $c \in (\delta, \frac{\pi}{2}-\delta),$ consisting of $\xi\in \mathbb C$ either with $\mathrm {Re}\xi\in[\frac{3\pi}{2}+\delta, 2\pi-\delta],$ or with $|\mathrm{Im}(\xi-\tau_i)|\geqslant \delta,$ $|\mathrm{Im}(\eta_j-\xi)|\geqslant \delta$ for all $i, j\in\{1,2,3,4\}$ and $\mathrm {Re}\xi\in[\frac{3\pi}{2}-c, 2\pi+c].$  See Figure \ref{Ddc}.  The significance of this region is that for all $\xi \in D_{\delta,c},$ all $\xi-\tau_i$  and $\eta_j - \xi,$ $i,j\in\{1,2,3,4\},$  are in the region $H_{\delta,K}$ in Proposition \ref{EST} for any  $K>c.$

\begin{figure}[htbp]
\centering
\includegraphics[scale=0.5]{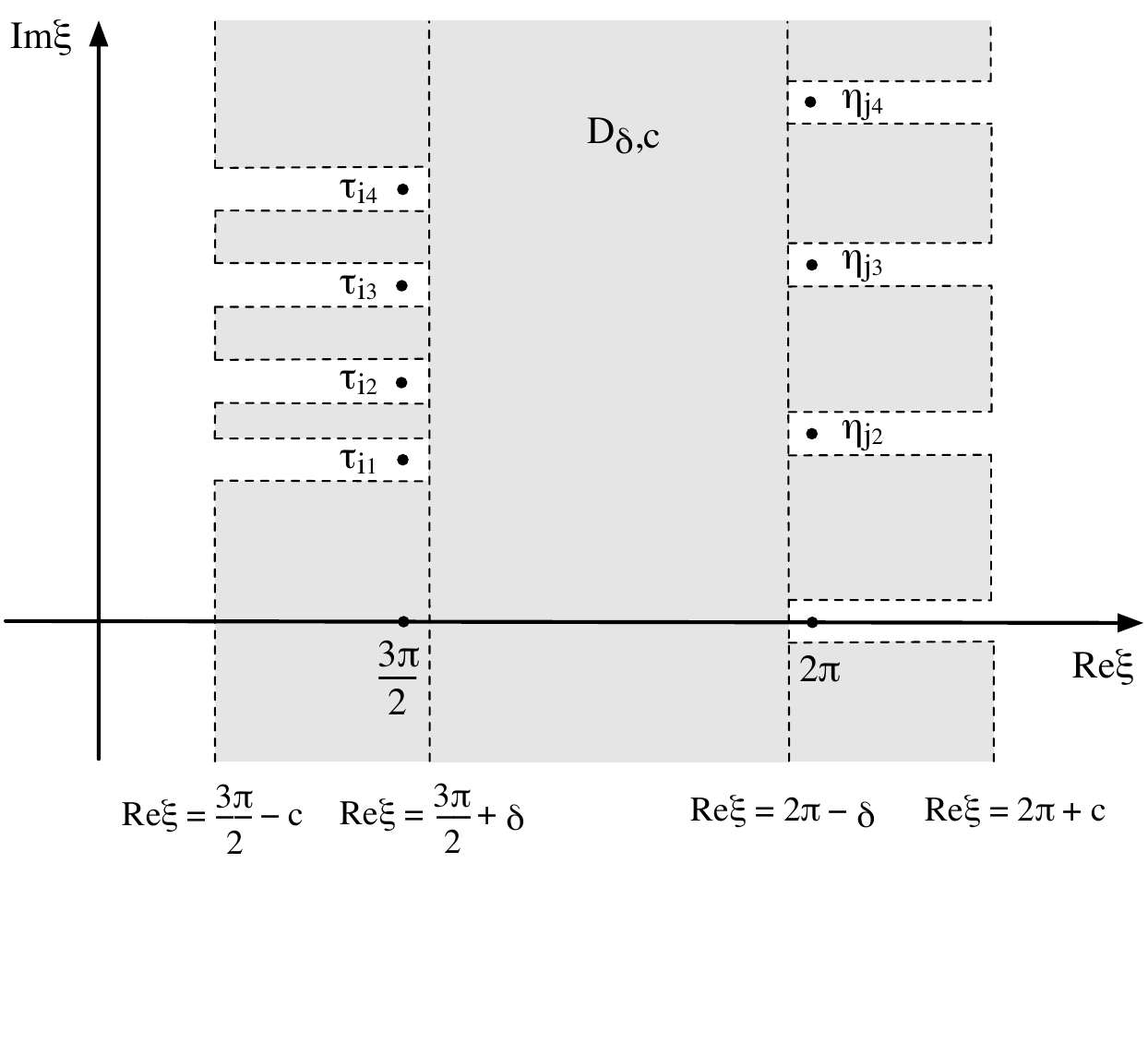}
\caption{Region $D_{\delta,c}$}
\label{Ddc}
\end{figure}

\begin{proposition}\label{bound3} 
 For $\delta> 0$  sufficiently small and $c\in (\delta, \frac{\pi}{2}-\delta),$ there exists a constant $K=K_{\delta,c}>0$ such that 
$$\Bigg|\frac{\partial \mathrm{Im}\kappa_{\boldsymbol\alpha}(\xi)}{\partial\mathrm{Im}\xi}\Bigg|<K$$
for all $\xi\in D_{\delta,c}.$
\end{proposition}

\begin{proof}  We first observe that for $x\in H_{\delta,K},$  there exists a constant $C_{\delta,K}>0$ depending on $\delta$ and $K$ such that
$$\bigg|\frac{\partial \log \big(1-e^{2\mathbf ix}\big)}{\partial \mathrm{Im} x}\bigg|<C_{\delta,K}.$$
 Indeed, by a direct computation, 
$$\bigg|\frac{\partial \log \big(1-e^{2\mathbf ix}\big)}{\partial \mathrm{Im} x}\bigg|=\frac{e^{-\mathrm{Im}x}}{\sqrt{\cosh^2\mathrm{Im}x-\cos^2\mathrm{Re}x}}$$
which converges to $0$ as $\mathrm{Im}x$ tends to $+\infty,$ and converges to $1$ as $\mathrm{Im}x$ tends to $-\infty.$ Also, when $\delta \leqslant \mathrm{Re}x\leqslant \pi-\delta$ the denominator $\sqrt{\cosh^2\mathrm{Im}x-\cos^2\mathrm{Re}x}\geqslant \sqrt{1-\cos^2\delta},$ which is bounded from below; and when $|\mathrm{Im}x| \geqslant \delta$  the denominator $\sqrt{\cosh^2\mathrm{Im}x-\cos^2\mathrm{Re}x}\geqslant \sqrt{\cosh^2\delta-1},$ which is also bounded from below.
As a consequence, the function is bounded from both sides by a constant $C_{\delta,K}$ on $H_{\delta,K}.$

Then by a direct computation, 
\begin{equation*}
\begin{split}
\Bigg|\frac{\partial \mathrm{Im}\kappa_{\boldsymbol\alpha}(\xi)}{\partial \mathrm{Im}\xi}\Bigg|=&\Bigg|-28\pi-4\pi\sum_{i=1}^4\mathrm{Re}\bigg(\frac{\partial \log \big(1-e^{2\mathbf i(\xi-\tau_i)}\big)}{\partial \mathrm{Im}\xi}\bigg)+3\pi\sum_{j=1}^4\mathrm{Re}\bigg(\frac{\partial \log \big(1-e^{2\mathbf i(\eta_j-\xi}\big)}{\partial \mathrm{Im}\xi}\bigg)\Bigg|\\
\leqslant & \,28\pi+4\pi\sum_{i=1}^4\bigg|\frac{\partial \log \big(1-e^{2\mathbf i(\xi-\tau_i)}\big)}{\partial \mathrm{Im}\xi}\bigg|+3\pi\sum_{j=1}^4\bigg|\frac{\partial \log \big(1-e^{2\mathbf i(\eta_j-\xi}\big)}{\partial \mathrm{Im}\xi}\bigg| < K
\end{split}
\end{equation*}
for all  $\xi\in D_{\delta,c},$ where $K=28\pi + 4\pi\sum_{i=1}^4C_{\delta,c}+3\pi\sum_{j=1}^4 C_{\delta,c}=28\pi (C_{\delta,c}+1).$
\end{proof}

More generally, for  $\boldsymbol \alpha\in \mathbb C^6,$ we let $D^{\boldsymbol \alpha}_{\delta,c}$ be the region consisting of $\xi\in \mathbb C$  with $\mathrm {Re}(\xi-\tau_i)\geqslant \delta,$  $|\mathrm{Im}(\xi-\tau_i)|\geqslant \delta$ for all $i\in\{1,2,3,4\},$  $\mathrm{Re}(\eta_j-\xi)\geqslant \delta,$ $|\mathrm{Im}(\eta_j-\xi)|\geqslant \delta$ for all $j\in\{1,2,3,4\}$ and $\mathrm {Re}\xi\in[\frac{3\pi}{2}-c, 2\pi+c].$ In the proof of Theorem \ref{VC},  we need to  consider the following  region 
$$\mathbf D_{\delta,c}=\Big\{ (\boldsymbol \alpha,\boldsymbol \xi)\in  \mathbb C^6\times \mathbb C\ \Big|\ \frac{\pi}{2}-\delta \leqslant \mathrm{Re}\alpha_k \leqslant \frac{\pi}{2}+\delta \text{ for all } k \text{ in } \{1,\dots, 6\}, \text{ and }\xi\in D^{\boldsymbol \alpha}_{\delta,c}\Big\}.$$
  The significance of this region is that for all $(\boldsymbol \alpha,\xi)\in \mathbf D_{\delta,c},$ all $\xi-\tau_i$  and $\eta_j - \xi,$ $i,j\in\{1,2,3,4\},$  are in the region $H_{\delta,K}$ in Proposition \ref{EST} for any  $K>c.$

\begin{proposition}\label{bound} 
For $b>0$ and $\delta> 0$ both sufficiently small and $c\in (\delta, \frac{\pi}{2}-\delta),$ there exists a constant $N=N_{\delta,c}>0$ independent of $b$ such that 
$$\mathrm{Im}\nu_{\boldsymbol\alpha,b}(\xi) \leqslant \big|\nu_{\boldsymbol\alpha,b}(\xi)\big|<N$$
for all $(\boldsymbol \alpha, \xi)$ in the region $\mathbf D_{\delta,c}.$
\end{proposition}

\begin{proof} Recall~\eqref{nuU} that \(\nu_{\boldsymbol\alpha,b}(\xi)=\frac{U_{\boldsymbol \alpha,b}(\xi)-U_{\boldsymbol \alpha}(\xi)-\kappa_{\boldsymbol \alpha}(\xi)b^2}{b^4}.
\)
By (\ref{b-6j}) and (\ref{6jint}), 
\begin{equation*}
\begin{split}
U_{\boldsymbol \alpha,b}(\xi)= &\,{\pi^2b^2} -\pi \mathbf i b^2 \sum_{i=1}^4\sum_{j=1}^4 \log S_b(q_j-t_i) + 2\pi \mathbf i b^2 \sum_{i=1}^4  \log S_b(u-t_i) + 2\pi \mathbf i b^2 \sum _{j=1}^4\log S_b(q_j-u).
\end{split}
\end{equation*}
First, for $(\boldsymbol \alpha, \xi)\in \mathbf D_{\delta,c}$ with $\delta$ small enough,  $\mathrm{Re}(\eta_j-\tau_i)\in  [\epsilon, \pi-\epsilon]$  for all  $i ,j\in\{1,2,3,4\}$ and for some $\epsilon>0.$ Hence all $\eta_j-\tau_i$'s are in $H_{\delta,K}$ for any $K>c,$ and Proposition \ref{EST} with the choice of $B$ therein tells us
$$\Big| 2\pi \mathbf i b^2 \log S_b(q_j-t_i) - L(\eta_j-\tau_i)\Big|<Bb^4.$$
 Next, since $(\boldsymbol \alpha,\xi)\in \mathbf D_{\delta,c},$ all $\xi-\tau_i$'s  and $\eta_j - \xi$'s  are in the region $H_{\delta,K}$ and $K>c.$ Then for $b$ small enough, all  $\xi-\tau_i-2\pi b^2$'s and  $\eta_j-\xi+\frac{3\pi b^2}{2}$'s are in $H_{\frac{\delta}{2},K+\frac{\delta}{2}},$ and Proposition \ref{EST} with the choice of $B$ therein gives 
 $$\Big|2\pi \mathbf i b^2  \log  S_b(u-t_i) -L\big(\xi-\tau_i-2\pi b^2\big)\Big|<Bb^4$$
and
$$\bigg|2\pi \mathbf i b^2  \log S_b(q_j-u) - L\Big(\eta_j-\xi+\frac{3\pi b^2}{2}\Big)\bigg|<Bb^4.$$
As a consequence,  we have
\begin{equation}\label{1}
\Big| U_{\boldsymbol \alpha,b}(\xi)- V_{\boldsymbol \alpha,b}(\xi)\Big|<16Bb^4,
\end{equation}
where
\begin{equation*}
V_{\boldsymbol \alpha,b}(\xi)=\,{\pi^2b^2}-\frac{1}{2}\sum_{i=1}^4\sum_{j=1}^4L(\eta_j-\tau_i)+\sum_{i=1}^4L\big(\xi-\tau_i-2\pi b^2\big)+\sum_{j=1}^4L\Big(\eta_j-\xi+\frac{3\pi b^2}{2}\Big).
\end{equation*}

We are now left to  estimate the difference between $V_{\boldsymbol \alpha,b}$ and $U_{\boldsymbol \alpha}+\kappa_{\boldsymbol \alpha}(\xi)b^2.$ On the one hand, we have
\begin{equation}\label{V-U-k}
\begin{split}
V_{\boldsymbol \alpha,b}-U_{\boldsymbol \alpha}-\kappa_{\boldsymbol \alpha}(\xi)b^2=& \sum_{i=1}^4\Big(L\big(\xi-\tau_i-2\pi b^2\big)-L\big(\xi-\tau_i\big)+ 2\pi b^2L'\big(\xi-\tau_i\big)\Big)\\
&+\sum_{j=1}^4\bigg(L\Big(\eta_j-\xi+\frac{3\pi b^2}{2}\Big)-L\big(\eta_j-\xi\big)-\frac{3\pi b^2}{2}L'\big(\eta_j-\xi\big)\bigg).
\end{split}
\end{equation}
On the other hand, by a direct computation, we have
 $$L''(x)=2+\frac{4e^{2\mathbf ix}}{1-e^{2\mathbf ix}},$$
 whose norm  is bounded from above  for any $\delta >0 $ and $K>0$ on the region $H_{\delta,K}.$ As a consequence, 
 $\big|L''\big(\xi-\tau_i\big)\big|$ and $\big|L''\big(\eta_j-\xi\big)\big|$ for all $i,j \in \{1,2,3,4\}$ are bounded from above on the region $\mathbf D_{\delta,c}$ for any sufficiently small $\delta >0.$   For $b$ sufficiently small so that $\xi-2\pi b^2$ and $\xi+\frac{3\pi b^2}{2}$ are in $D^{\boldsymbol \alpha}_{\frac{\delta}{2}, c},$  we let $C$ be an upper bound of $|L''|$ on $\mathbf D_{\frac{\delta}{2},c}.$ Then  we have 
 \begin{equation*}
 \begin{split}
 &\Big|L\big(\xi-\tau_i-2\pi b^2\big)-L\big(\xi-\tau_i\big)+ 2\pi b^2L'\big(\xi-\tau_i\big)\Big|\\
 =&\bigg|\int_{\xi-\tau_i}^{\xi-\tau_i-2\pi b^2}L''(t)\Big(\xi-\tau_i-2\pi b^2-t\Big)dt\bigg|<4\pi^2b^4C
 \end{split}
 \end{equation*}
 for each $i\in\{1,2,3,4\},$ and 
  \begin{equation*}
 \begin{split}
&\Big|L\Big(\eta_j-\xi+\frac{3\pi b^2}{2}\Big)-L\big(\eta_j-\xi\big)-\frac{3\pi b^2}{2} L'\big(\eta_j-\xi\big)\Big|\\
=&\bigg|\int_{\eta_j-\xi}^{\eta_j-\xi+\frac{3\pi b^2}{2}}L''(t)\bigg(\eta_j-\xi+\frac{3\pi b^2}{2}-t\bigg)dt\bigg|<\frac{9}{4}\pi^2b^4C
 \end{split}
 \end{equation*}
 for each $j\in\{1,2,3,4\}$.
 Together with (\ref{V-U-k}), we have
 \begin{equation}\label{3}
 \Big|V_{\boldsymbol\alpha,b}(\xi)-U_{\boldsymbol \alpha}(\xi)-\kappa_{\boldsymbol \alpha}(\xi)b^2\Big|<25\pi^2b^4C.
 \end{equation}

Putting (\ref{1}) and (\ref{3}) together and letting $N =16B+25\pi^2C,$ we have
\begin{equation}\label{2}
\Big| U_{\boldsymbol \alpha,b}(\xi)- U_{\boldsymbol \alpha}(\xi)- \kappa_{\boldsymbol \alpha}(\xi)b^2\Big|<N b^4
\end{equation}
for all $(\boldsymbol \alpha, \xi)$ in $\mathbf D_{\delta,c},$ and the result follows from (\ref{nuU}).
\end{proof}

\subsection{Proof of Theorem \ref{cov}}\label{subsec: thm2}

\begin{proof}[Proof of Theorem \ref{cov}] 
Recall the domain $D=\Big\{ \xi\in\mathbb C\ \Big|\ \frac{3\pi}{2}< \mathrm{Re}\xi < 2\pi \Big\}$. Let  $d>0$ be sufficiently small so that the region 
 $$B_{d}=\Big\{\xi\in \mathbb C \ \Big|\ |\mathrm{Re}\xi - \mathrm{Re}\xi^*| \leqslant d \text{ and } |\mathrm{Im}\xi-\mathrm{Im}\xi^*|\leqslant d \Big\}$$
 lies in $D.$ By Proposition \ref{concave} (1), $\frac{\partial \mathrm{Im}U_{\boldsymbol\alpha}(\xi)}{\partial \mathrm{Im}\xi}$ is strictly decreasing as $\xi$ moves along $\Gamma^*$ from $\xi^*$ with $\mathrm{Im}\xi$ approaching $+\infty,$ and is strictly increasing as $\xi$ moves along $\Gamma^*$ from $\xi^*$ with $\mathrm{Im}\xi$ approaching $-\infty.$ 
Therefore, there is an $\epsilon>0$ such that  
 \begin{equation}\label{ImU}
\mathrm{Im} U_{\boldsymbol\alpha}(\xi^*\pm \mathbf il )< -2\mathrm{Col}(\boldsymbol l)-4\epsilon
\end{equation}
and for $l>d;$ and by Propositions \ref{concave} (1) and \ref{limder},  there is an $L>d$ such that 
\begin{equation}\label{Dv1}
 \frac{\partial \mathrm{Im} U_{\boldsymbol\alpha}}{\partial\mathrm{Im}\xi} (\xi^*+  \mathbf i l)<-2\pi\quad\text{and}\quad  \frac{\partial \mathrm{Im} U_{\boldsymbol\alpha}}{\partial\mathrm{Im}\xi} (\xi^* -  \mathbf i l)>2\pi
\end{equation}
for $l>L.$  
Let
$$\Gamma^*_d=\Gamma^*\cap B_d=\Big\{ \xi \in \Gamma^*\ \Big|\ |\mathrm{Im}\xi-\mathrm{Im}\xi^*|<d\Big\}$$
and let 
$$\Gamma^*_L=\Big\{ \xi \in \Gamma^*\ \Big|\ |\mathrm{Im}\xi-\mathrm{Im}\xi^*|\leqslant L\Big\}.$$
We will show that, as $b\to 0,$
\begin{enumerate}[(I)]
\item 
$$\frac{1}{\pi b}\int_{\Gamma^*_d}\exp\bigg(\frac{U_{\boldsymbol \alpha}(\xi)+ \kappa_{\boldsymbol \alpha}(\xi)b^2+ \nu_{\boldsymbol \alpha,b}(\xi)b^4}{2\pi \mathbf i b^2} \bigg)d\xi = \frac{1}{2} \frac{e^{\frac{-\mathrm{Cov}(\boldsymbol l)}{\pi b^2}}}{\sqrt[4]{-\det\mathrm{Gram}(\boldsymbol l)}}\Big(1+O\big(b^2\big)\Big),
$$

\item $$\bigg|\frac{1}{\pi b}\int_{\Gamma^*_L\setminus \Gamma^*_d}\exp\bigg(\frac{U_{\boldsymbol \alpha}(\xi)+ \kappa_{\boldsymbol \alpha}(\xi)b^2+ \nu_{\boldsymbol \alpha,b}(\xi)b^4}{2\pi \mathbf i b^2} \bigg)d\xi\bigg|< O\Big(e^{\frac{-\mathrm{Cov}(\boldsymbol l)-\epsilon_1}{\pi b^2}}\Big),$$
and 

\item $$\bigg|\frac{1}{\pi b}\int_{\Gamma^*\setminus \Gamma^*_L}\exp\bigg(\frac{U_{\boldsymbol \alpha}(\xi)+ \kappa_{\boldsymbol \alpha}(\xi)b^2+ \nu_{\boldsymbol \alpha,b}(\xi)b^4}{2\pi \mathbf i b^2} \bigg)d\xi\bigg|<O\Big(e^{\frac{-\mathrm{Cov}(\boldsymbol l)-\epsilon_1}{\pi b^2}}\Big)$$
\end{enumerate}
for some $\epsilon_1>0,$ from which the result follows.
\\

For (I), we claim that  all the conditions of Proposition \ref{saddle} are satisfied by letting $\hbar=b^2,$ $D=B_d,$ $f=\frac{U_{\boldsymbol \alpha}}{2\pi \mathbf i},$ $g=\exp\big(\frac{\kappa_{\boldsymbol \alpha}}{2\pi \mathbf i }\big),$ $f_\hbar=\frac{U_{\boldsymbol \alpha}+\nu_{\boldsymbol \alpha,b}b^4}{2\pi \mathbf i},$ $\upsilon_h=\frac{\nu_{\boldsymbol \alpha,b}}{2\pi \mathbf i},$ $S=\Gamma^*_d$ and $c=\xi^*.$ 

Indeed, by Proposition \ref{critical2} (1), $\xi^*$ is a critical point of $f=\frac{U_{\boldsymbol \alpha}}{2\pi \mathbf i}$ in $B_{d},$ hence condition (i) is satisfied. 

By Propositions \ref{critical2} (1) and \ref{concave} (1), $\xi^*$ is the unique maximum point of $\mathrm{Re}f=\frac{\mathrm{Im}U_{\boldsymbol\alpha}}{2\pi}$ on $\Gamma^*_d,$ hence condition (ii) is satisfied.

For conditions (iii) and (iv), since $\xi^*\neq \tau_i$ and $\xi^*\neq\eta_j$ for any $i\in\{1,2,3,4\}$ and $j\in\{1,2,3\},$ $\kappa_{\boldsymbol \alpha}(\xi^*)$ is a finite value. As a consequence, $g(\xi^*)=\exp\big(\frac{\kappa_{\boldsymbol \alpha}(\xi^*)}{2\pi \mathbf i }\big)\neq 0,$ and condition (iv) is satisfied. Also by this,  Proposition \ref{tri} and Proposition \ref{Hess}, $U''_{\boldsymbol \alpha}(\xi^*)=-16{\exp\big(\frac{\kappa_{\boldsymbol \alpha}(\xi^*)}{\pi \mathbf i}\big)}\sqrt{\det\mathrm{Gram}(\boldsymbol l)}\neq 0,$ and condition (iii) is satisfied.

For condition (v), by Proposition \ref{bound},  $|\upsilon_{\hbar}(\xi)|=\big|\frac{\nu_{\boldsymbol \alpha,b}(\xi)}{2\pi \mathbf i}\big|<\frac{N}{2\pi}$ on $B_{d}.$ 

For condition (vi), since $\Gamma^*$ is a straight line, it is a smooth embedding near $\xi^*.$

Finally, by Proposition \ref{saddle}, Proposition \ref{critical2} (1)  and Proposition \ref{Hess}, we have as $b\to 0,$
\begin{equation*}
\begin{split}
\frac{1}{\pi b}\int_{\Gamma^*_d}\exp\bigg(\frac{U_{\boldsymbol \alpha,b}(\xi)}{2\pi \mathbf ib^2}\bigg) d\xi=& \frac{(2\pi b^2)^\frac{1}{2}}{\pi b} \frac{\exp\big(\frac{\kappa_{\boldsymbol \alpha}(\xi^*)}{2\pi \mathbf i}\big)}{\sqrt{-\frac{U_{\boldsymbol \alpha}''(\xi^*)}{2\pi \mathbf i}}}e^{\frac{U_{\boldsymbol\alpha}(\xi^*)}{2\pi \mathbf i b^2}}\Big(1+O\big(b^2\big)\Big)\\
=&\frac{1}{2}\frac{e^{\frac{-\mathrm{Cov}(\boldsymbol l)}{\pi b^2}}}{\sqrt[4]{-\det\mathrm{Gram}(\boldsymbol l)}}\Big(1+O\big(b^2\big)\Big).
\end{split}
\end{equation*}
This completes the proof of (I). 
\\

For (II) and (III), we have
\begin{equation*}
\begin{split}
& \bigg|\frac{1}{\pi b}\int_{\Gamma^*\setminus \Gamma^*_L}\exp\bigg(\frac{U_{\boldsymbol \alpha}(\xi)+ \kappa_{\boldsymbol \alpha}(\xi)b^2+ \nu_{\boldsymbol \alpha,b}(\xi)b^4}{2\pi \mathbf i b^2} \bigg)d\xi\bigg|\\
\leqslant & \frac{1}{\pi b}\int_{\Gamma^*\setminus \Gamma^*_L}\exp\bigg(\frac{\mathrm{Im}U_{\boldsymbol \alpha}(\xi)+\mathrm{Im}\kappa_{\boldsymbol \alpha}(\xi)b^2+\mathrm{Im}\nu_{\boldsymbol \alpha,b}(\xi)b^4}{2\pi b^2} \bigg)|d\xi|.
\end{split}
\end{equation*}

For  (II), by Proposition \ref{bound}, there is a  $b_1>0$ such that 
\begin{equation}\label{last2}
\mathrm{Im}\nu_{\boldsymbol\alpha,b}(\xi)b^4<Nb^4<\epsilon 
\end{equation}
for all $b<b_1$ and for all $\xi\in\Gamma^*;$  and together with  (\ref{ImU}), we have
\begin{equation}\label{last3}
\mathrm{Im}U_{\boldsymbol \alpha}(\xi)+\mathrm{Im}  \nu_{\boldsymbol \alpha,b}(\xi)b^4< -2\mathrm{Cov}(\boldsymbol l) -3\epsilon
\end{equation}
for all $b<b_1$ and  $\xi\in\Gamma^*\setminus \Gamma^*_d.$ By the compactness of ${\Gamma^*_L}\setminus \Gamma^*_d,$ there exists an $M>0$ such that 
\begin{equation}\label{Imk}
\mathrm{Im}\kappa_{\boldsymbol\alpha}(\xi)<M
\end{equation}
for all $\xi\in\Gamma^*_L\setminus \Gamma^*_d.$ As a consequence of (\ref{last3}) and (\ref{Imk}), we have
\begin{equation*}
\begin{split}
&\frac{1}{\pi b}\int_{\Gamma^*_L\setminus \Gamma^*_d}\exp\bigg(\frac{\mathrm{Im}U_{\boldsymbol \alpha}(\xi)+\mathrm{Im}\kappa_{\boldsymbol \alpha}(\xi)b^2+\mathrm{Im}\nu_{\boldsymbol \alpha,b}(\xi)b^4}{2\pi b^2} \bigg)|d\xi|\\
<  & \frac{2(L-d) e^{\frac{M}{2\pi}}}{\pi b} \exp\bigg(\frac{-\mathrm{Cov}(\boldsymbol l)-\epsilon }{\pi b^2}   \bigg )< O\Big(e^{\frac{-\mathrm{Cov}(\boldsymbol l)-\epsilon_1}{\pi b^2}}\Big)
\end{split}
\end{equation*}
for any $\epsilon_1<\epsilon.$ This completes the proof of (II).

For (III), there  is a $b_0\in (0, b_1)$ such that for all $b<b_0,$
\begin{equation}\label{ImK}
\mathrm{Im}\kappa_{\boldsymbol\alpha}(\xi^*\pm \mathbf i L) b^2<\epsilon,
\end{equation}
$Kb^2<\epsilon$ and $Nb^4<\epsilon,$ where $K$ and $N$ are respectively the constants in Propositions \ref{bound3} and \ref{bound}. We claim that, for $\xi\in\Gamma^*\setminus \Gamma^*_L$ and $b<b_0,$ 
\begin{equation}\label{cl}
\mathrm{Im}U_{\boldsymbol \alpha}(\xi)+\mathrm{Im}\kappa_{\boldsymbol \alpha}(\xi)b^2+\mathrm{Im}\nu_{\boldsymbol \alpha,b}(\xi)b^4<-2\big(\mathrm{Cov}(\boldsymbol l)+\epsilon\big)-(2\pi-\epsilon)\big(|\xi-\xi^*|-L\big),
\end{equation}
as a consequence of which, we have,
\begin{equation}\label{CI}
\begin{split}
& \frac{1}{\pi b}\int_{\Gamma^*\setminus \Gamma^*_L}\exp\bigg(\frac{\mathrm{Im}U_{\boldsymbol \alpha}(\xi)+\mathrm{Im}\kappa_{\boldsymbol \alpha}(\xi)b^2+\mathrm{Im}\nu_{\boldsymbol \alpha,b}(\xi)b^4}{2\pi b^2} \bigg)|d\xi| \\
 < & \frac{1}{\pi b} \exp\bigg(\frac{-\mathrm{Cov}(\boldsymbol l)-\epsilon}{\pi b^2}\bigg)\int_{\Gamma^*\setminus \Gamma^*_L}\exp\bigg(\frac{-(2\pi-\epsilon)\big(|\xi-\xi^*|-L\big)}{2\pi}\bigg) |d\xi|\\
< & O\Big(e^{\frac{-\mathrm{Cov}(\boldsymbol l)-\epsilon_1}{\pi b^2}}\Big)
\end{split}
\end{equation}
for any $\epsilon_1<\epsilon.$ 

For the proof of the claim, by (\ref{Dv1}) and the choice of $b_0,$ for $l>L,$ we have 
$$\frac{\partial}{\partial \mathrm{Im}\xi} \Big(\mathrm{Im}U_{\boldsymbol\alpha}(\xi^*+\mathbf il)+\mathrm{Im}\kappa_{\boldsymbol\alpha}(\xi^*+\mathbf il)b^2\Big)<-2\pi+\epsilon,$$
and 
$$\frac{\partial}{\partial \mathrm{Im}\xi} \Big(\mathrm{Im}U_{\boldsymbol\alpha}(\xi^*-\mathbf il)+\mathrm{Im}\kappa_{\boldsymbol\alpha}(\xi^*-\mathbf il)b^2\Big)>2\pi-\epsilon.$$
Together with the Mean Value Theorem, (\ref{ImU}) and (\ref{ImK}), we have 
\begin{equation}\label{Bou}
\begin{split}
\mathrm{Im}U_{\boldsymbol\alpha}(\xi)+ \mathrm{Im}\kappa_{\boldsymbol\alpha}(\xi)b^2 < & \mathrm{Im}U_{\boldsymbol\alpha}(\xi^*\pm \mathbf iL) + \mathrm{Im}
\kappa_{\boldsymbol\alpha}(\xi^*\pm \mathbf iL)b^2 - (2\pi-\epsilon) \big |  \xi - (\xi^*\pm \mathbf iL ) \big|\\
< & -2\mathrm{Cov}(\boldsymbol l)-3\epsilon  -(2\pi-\epsilon)\big(|\xi-\xi^*|-L \big)
\end{split}
\end{equation}
for all $\xi \in \Gamma^*\setminus \Gamma^*_L.$  Finally, putting (\ref{Bou}) and (\ref{last2}) together, we have  (\ref{cl}) and the first  inequality in (\ref{CI}); and since 
$$|\xi-\xi^*|-L\to+\infty$$
as $\xi \in \Gamma^*\setminus \Gamma^*_L$ approaches $\infty,$ we have the second inequality in (\ref{CI}). This completes the proof of (III).
\\

Putting (I), (II)  and (III) together, we have as $b\to 0,$ 
$$\bigg\{\begin{matrix} a_1 & a_2 & a_3 \\ a_4 & a_5 & a_6 \end{matrix} \bigg\}_b=\frac{1}{2}\frac{e^{\frac{-\mathrm{Cov}(\boldsymbol l)}{\pi b^2}}}{\sqrt[4]{-\det\mathrm{Gram}(\boldsymbol l)}}\Big(1+O\big(b^2\big)\Big). \qedhere$$ 
\end{proof}

\subsection{Proof of Theorem \ref{cov2}} \label{subsec:thm3}

\begin{proof}[Proof of Theorem \ref{cov2}] The proof is similar to the proof of Theorem \ref{cov}, namely, we will carefully choose an integral contour  and use the Saddle Point Approximation. A seemingly natural choice of the contour is $\Gamma_{2\pi}$ that passes through the unique critical point $\xi^*$ of $U_{\boldsymbol\alpha}$ in Case (a) and critical points $\xi^*_1$ and $\xi^*_2$ of $U_{\boldsymbol\alpha}$ in Case (b) respectively given by Proposition \ref{critical2} (2) and (3).  However, there are  issues with this choice as:
\begin{enumerate}[(a)]
\item  $\Gamma_{2\pi}$ pass through the points  $\eta_j$'s where $U_{\boldsymbol \alpha}$ is only continuous but not holomorphic, and 
\item  the critical points are not the unique maximum of $\mathrm{Im}U_{\boldsymbol\alpha}$ on the contour, hence condition (ii) of Proposition \ref{saddle} is not satisfied.
\end{enumerate} 
We will resolve these issues as follows.

By Proposition \ref{critical2} (2) and (3), $U_{\boldsymbol\alpha}$ either has a unique critical point $\xi^*$ or two critical points $\xi^*_1$ and $\xi^*_2$ lying on $\Gamma_{2\pi},$  which pass through the non-holomorphic points $\eta_j$'s of $U_{\boldsymbol \alpha}.$ To avoid  these non-holomorphic points, we let $\delta>0$ be sufficiently small so that Corollary \ref{cor4} holds for all $\xi$ in the  $\delta$-square neighborhood of each  $\eta_j,$ and consider the region $D_{\delta,c}$ as depicted in Figure \ref{Ddc} and defined there-around. We observe that $\Gamma_{2\pi}$ intersects the complement of $D_{\delta,c},$ and we deform $\Gamma_{2\pi}$ to the new contour $\Gamma'$ by pushing the parts out of $D_{\delta,c}$ into the boundary of $D_{\delta,c}.$ See Figure \ref{Ddc2}.

\begin{figure}[htbp]
\centering
\includegraphics[scale=0.25]{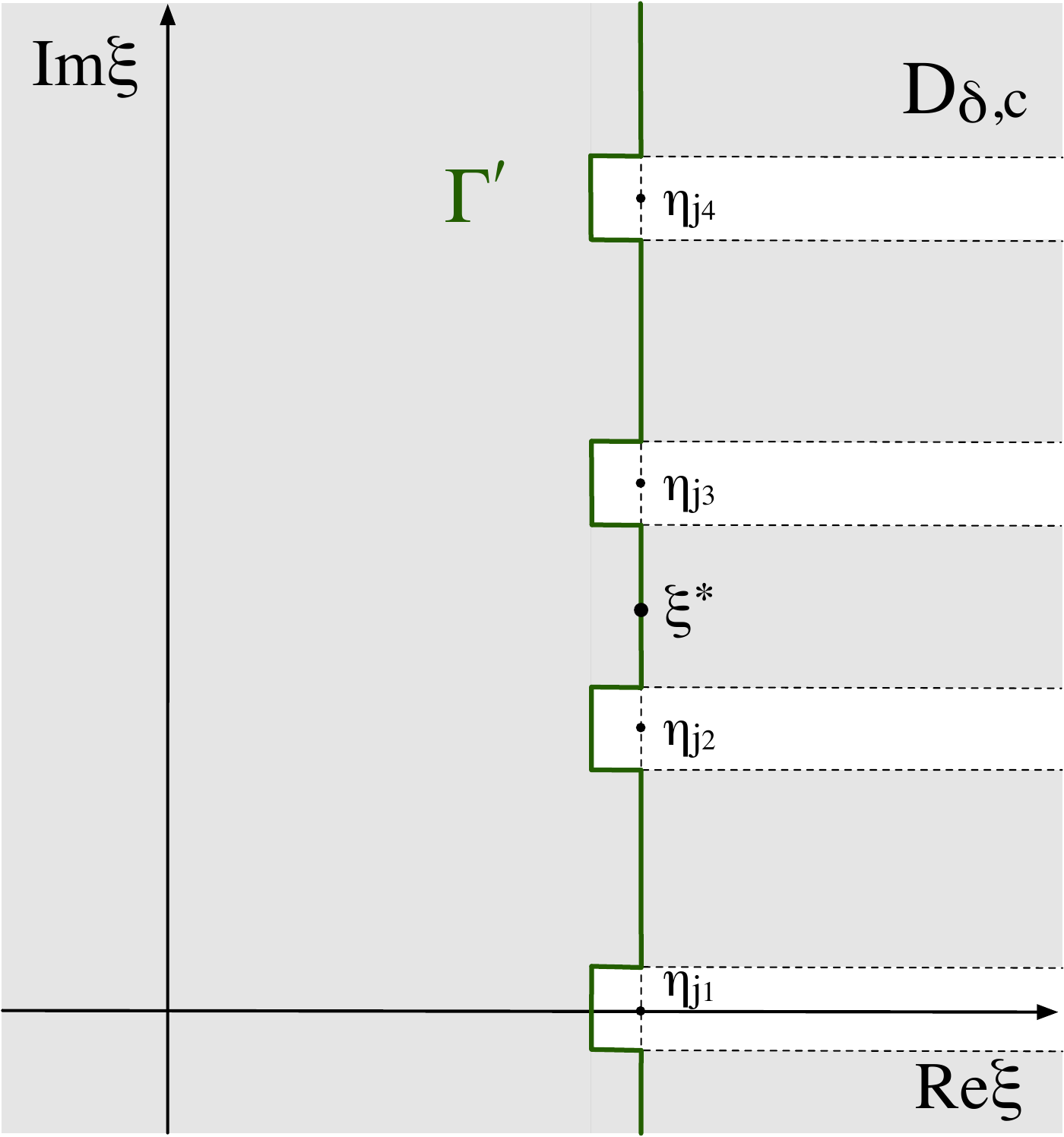}
\caption{Contour $\Gamma'$}
\label{Ddc2}
\end{figure}

By abuse of notation let $\xi^*$ denote the unique critical point in Case (a) or either of the two critical points in Case (b). 
By Propositions \ref{PL}, we see that $\mathrm{Im}\xi^* \in [\mathrm{Im}\eta_{j_2},\mathrm{Im}\eta_{j_3}];$ and by Proposition \ref{limder2}, $\xi^*$ cannot be the end points  of the intervals. Therefore,  $\mathrm{Im}\xi^* \in (\mathrm{Im}\eta_{j_2},\mathrm{Im}\eta_{j_3}).$ 
 By Proposition \ref{PL}, $\mathrm{Im}U_{\boldsymbol\alpha}$ is a constant function on the line segment connecting $\eta_{j_2}$ and $\eta_{j_3},$  hence $\xi^*$ is not the unique maximum point of $\mathrm{Im}U_{\boldsymbol\alpha}$ on $\Gamma'$ and  we need to further deform it. We will do it case by case as follows.
 \\

For  Case (a), we choose an $L>0$ large enough so that 
$$[\mathrm{Im}\eta_{i_1}-\delta,\mathrm{Im}\eta_{i_4}+\delta ]\subset (\mathrm{Im}\xi^*-L+\delta, \mathrm{Im}\xi^*+L-\delta),$$
where $\delta>0$ is as above, and consider the following smooth vector filed  $\psi\mathbf v$ on  
$ D_{\delta,c},$
 where 
 $\psi$ is a 
$C^\infty$-smooth bump function on $D_{\delta,c}$ satisfying 
  \begin{equation*}
\left\{
    \begin{array}{rcl}
 \psi(\xi ) = 1  & \text{if} & |\mathrm{Im}\xi  -\mathrm{Im}\xi^*| \leqslant  L-\delta  \\
    0 <   \psi(\xi ) <1 & \text{if}  & L-\delta <  |\mathrm{Im}\xi -\mathrm{Im}\xi^*| < L \\
 \psi(\xi )  = 0  & \text{if} & |\mathrm{Im}\xi  -\mathrm{Im}\xi^*| \geqslant  L
    \end{array}\right.
\end{equation*}
with $\delta>0$ as above, and $\mathbf v$ is the vector field on $D_{\delta,c}$ defined by  
$$\mathbf v =\bigg(-\frac{\partial\mathrm{Im} U_{\boldsymbol\alpha}}{\partial \mathrm{Re}\xi},0 \bigg).$$
Now let $\Gamma^*$ be the contour obtained from $\Gamma'$ by following the flow lines of  $\psi \mathbf v$ for a short time $t>0$.  
See Figure \ref{Ddc3}. 
Then by Corollary \ref{cor4} and the choice of $\delta,$ 
$\Gamma^*$ stays inside $D_{\delta,c},$ and by Proposition \ref{nonvanish} (1), 
\begin{equation}\label{um2}
\mathrm{Im}U_{\boldsymbol\alpha}(\xi) < \mathrm{Im}U_{\boldsymbol\alpha}(\xi^*)
\end{equation}
for all $\xi  \in \Gamma^*  \setminus \{\xi ^*\}.$

\begin{figure}[htbp]
\centering
\includegraphics[scale=0.25]{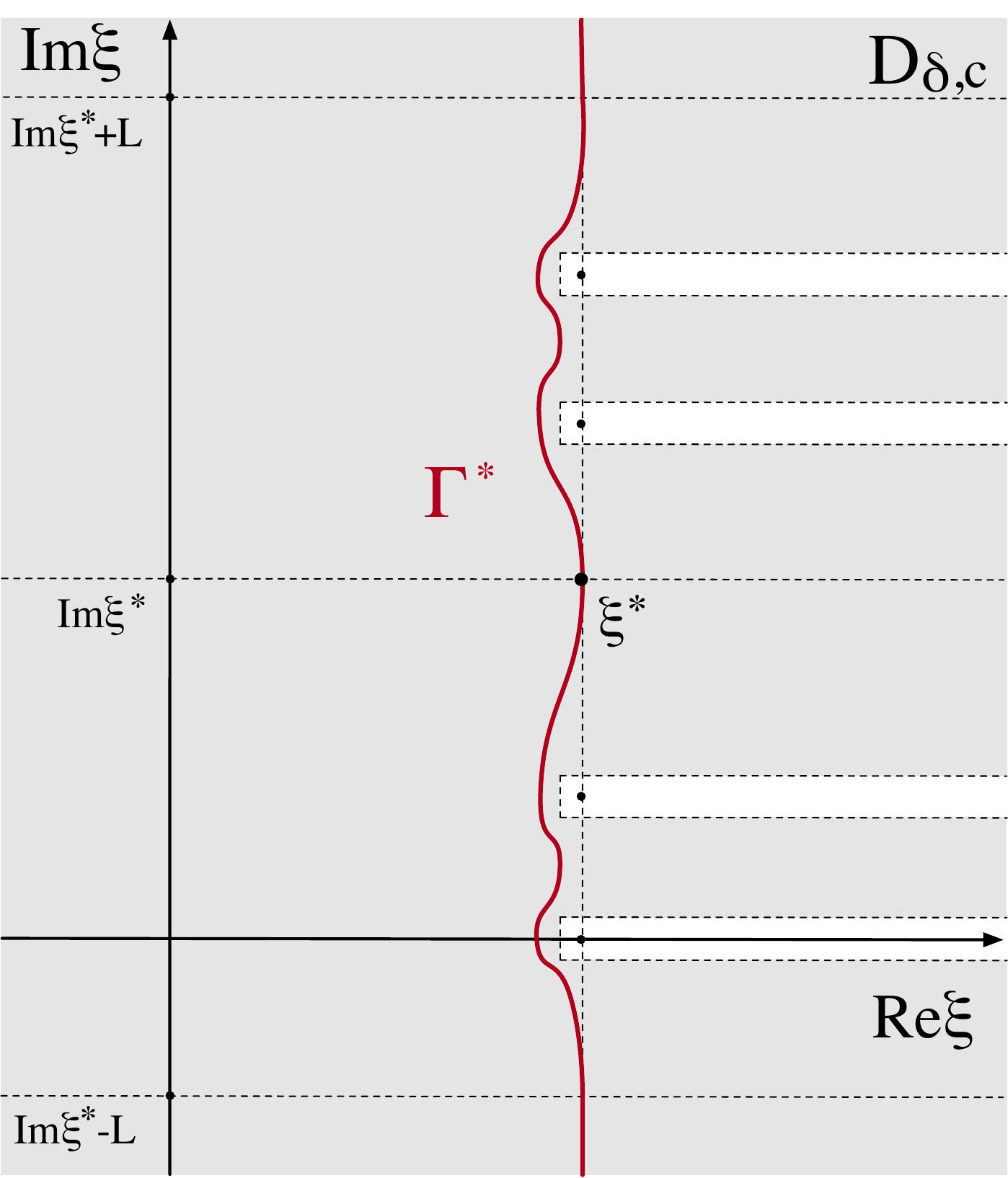}
\caption{Contour $\Gamma^*$ in Case (2)}
\label{Ddc3}
\end{figure}

Now in the computation of the $b$-$6j$ symbol (\ref{6jint}), we choose $\Gamma=\Gamma^*$ as the contour of integral.  Then we have
$$\bigg\{\begin{matrix} a_1 & a_2 & a_3 \\ a_4 & a_5 & a_6 \end{matrix} \bigg\}_b=\frac{1}{\pi b}\int_{\Gamma^*}\exp\bigg(\frac{U_{\boldsymbol \alpha}(\xi)+\kappa_{\boldsymbol \alpha}(\xi)b^2+\nu_{\boldsymbol \alpha,b}(\xi)b^4}{2\pi \mathbf i b^2} \bigg)d\xi.$$
Let $d>0$ be sufficiently small so that the region
$$B_d=\Big\{\xi\in \mathbb C\ \Big|\ |\mathrm{Re} \xi-\mathrm{Re} \xi^*|<d, |\mathrm{Im}\xi -\mathrm{Im} \xi^*|<d \Big\}$$
is  in $D_{\delta,c}.$ 
Let
$$\Gamma^*_d=\Gamma^*\cap B_d=\Big\{ \xi \in \Gamma^*\ \Big|\ |\mathrm{Im}\xi-\mathrm{Im}\xi^*|<d\Big\}$$
and let
$$\Gamma^*_L=\Big\{ \xi \in \Gamma^*\ \Big|\ |\mathrm{Im}\xi-\mathrm{Im}\xi^*|\leqslant L\Big\}.$$
We will show that, as $b\to 0,$ there exists an $\epsilon_2>0$ such that 
\begin{enumerate}[(I)]
\item 
$$\frac{1}{\pi b}\int_{\Gamma^*_d}\exp\bigg(\frac{U_{\boldsymbol \alpha}(\xi)+ \kappa_{\boldsymbol \alpha}(\xi)b^2+ \nu_{\boldsymbol \alpha,b}(\xi)b^4}{2\pi \mathbf i b^2} \bigg)d\xi =  C (2\pi b^2)^\frac{1}{3 }\frac{\exp\big(\frac{\kappa_{\boldsymbol \alpha}(\xi^*)}{2\pi \mathbf i}\big)}{\sqrt[3]{-\frac{U'''_{\boldsymbol \alpha}(\xi^*)}{2\pi \mathbf i}}} e^{\frac{-\widetilde{\mathrm{Cov}}(\boldsymbol l)}{\pi b^2}}\Big(1+O\big(b^2\big)\Big)
$$
for some  nonzero constant $C,$ 
\item $$\bigg|\frac{1}{\pi b}\int_{\Gamma^*_L\setminus \Gamma^*_d}\exp\bigg(\frac{U_{\boldsymbol \alpha}(\xi)+ \kappa_{\boldsymbol \alpha}(\xi)b^2+ \nu_{\boldsymbol \alpha,b}(\xi)b^4}{2\pi \mathbf i b^2} \bigg)d\xi\bigg|< O\Big(e^{\frac{-\widetilde{\mathrm{Cov}}(\boldsymbol l)-\epsilon_2}{\pi b^2}}\Big),$$
and 

\item $$\bigg|\frac{1}{\pi b}\int_{\Gamma^*\setminus \Gamma^*_L}\exp\bigg(\frac{U_{\boldsymbol \alpha}(\xi)+ \kappa_{\boldsymbol \alpha}(\xi)b^2+ \nu_{\boldsymbol \alpha,b}(\xi)b^4}{2\pi \mathbf i b^2} \bigg)d\xi\bigg|<O\Big(e^{\frac{-\widetilde{\mathrm{Cov}}(\boldsymbol l)-\epsilon_2}{\pi b^2}}\Big),$$
\end{enumerate}
 from which the result follows.
\\

For (I),  we claim that by  letting $\hbar=b^2,$ $D=B_d,$  $f=\frac{U_{\boldsymbol \alpha}}{2\pi \mathbf i},$ $g=\exp\big(\frac{\kappa_{\boldsymbol \alpha}}{2\pi \mathbf i }\big),$ $f_\hbar=\frac{U_{\boldsymbol \alpha}+\nu_{\boldsymbol \alpha,b}b^4}{2\pi \mathbf i},$ $\upsilon_{\hbar}=\frac{\nu_{\boldsymbol \alpha,b}}{2\pi \mathbf i},$ $S=\Gamma^*_d$ and $c=\xi^*,$  all the conditions  in Proposition \ref{saddle2}  are satisfied.

Indeed,  by Proposition \ref{critical2} (2), $\xi^*$ is a critical point of $f=\frac{U_{\boldsymbol \alpha}}{2\pi \mathbf i}$ in $B_d,$ hence condition (i) is satisfied. 

By (\ref{um2}), $\xi^*$ is the unique maximum point of $\mathrm{Re}f=\frac{\mathrm{Im}U_{\boldsymbol\alpha}}{2\pi}$ on $\Gamma^*,$ hence condition (ii) is satisfied.

For conditions (iii), (iv) and (v), since  $\xi^*_i\in D_{\epsilon, c},$ $\kappa_{\boldsymbol \alpha}(\xi^*)$ is a finite value. As a consequence, $g(\xi^*)=\exp\big(\frac{\kappa_{\boldsymbol \alpha}(\xi^*)}{2\pi \mathbf i }\big)\neq 0,$ and condition (v) is satisfied;  and by Proposition \ref{critical2} (3),  conditions (iii) and (iv) are satisfied.

For condition (vi), by Proposition \ref{bound},  $|\upsilon_{\hbar}(\xi)|=\big|\frac{\nu_{\boldsymbol \alpha,b}(\xi)}{2\pi \mathbf i}\big|<\frac{N}{2\pi}$ on $B_d.$

For condition (vii), since near $\xi^*,$ $\Gamma^*$ is obtained from a straight line segment $\Gamma'$ by moving along the flow lines of a smooth vector field $\mathbf v,$ $\Gamma_d$ is a smooth embedding around $\xi^*.$

Finally, by Proposition \ref{saddle2} and Proposition \ref{critical2} (2),  we have as $b\to 0,$
$$\frac{1}{\pi b}\int_{\Gamma^*_d}\exp\bigg(\frac{U_{\boldsymbol \alpha}(\xi)+ \kappa_{\boldsymbol \alpha}(\xi)b^2+ \nu_{\boldsymbol \alpha,b}(\xi)b^4}{2\pi \mathbf i b^2} \bigg)d\xi =  C (2\pi b^2)^\frac{1}{3 }\frac{\exp\big(\frac{\kappa_{\boldsymbol \alpha}(\xi^*)}{2\pi \mathbf i}\big)}{\sqrt[3]{-\frac{U'''_{\boldsymbol \alpha}(\xi^*)}{2\pi \mathbf i}}} e^{\frac{-\widetilde{\mathrm{Cov}}(\boldsymbol l)}{\pi b^2}}\Big(1+O\big(b^2\big)\Big)
$$for a nonzero constant $C.$ This completes the proof of (I).
\\

For (II) and (III), we have
\begin{equation*}
\begin{split}
& \bigg|\frac{1}{\pi b}\int_{\Gamma^*\setminus \Gamma^*_L}\exp\bigg(\frac{U_{\boldsymbol \alpha}(\xi)+ \kappa_{\boldsymbol \alpha}(\xi)b^2+ \nu_{\boldsymbol \alpha,b}(\xi)b^4}{2\pi \mathbf i b^2} \bigg)d\xi\bigg|\\
\leqslant & \frac{1}{\pi b}\int_{\Gamma^*\setminus \Gamma^*_L}\exp\bigg(\frac{\mathrm{Im}U_{\boldsymbol \alpha}(\xi)+\mathrm{Im}\kappa_{\boldsymbol \alpha}(\xi)b^2+\mathrm{Im}\nu_{\boldsymbol \alpha,b}(\xi)b^4}{2\pi b^2} \bigg)|d\xi|.
\end{split}
\end{equation*}

For  (II),  by (\ref{um2}) and Proposition \ref{critical2} (2), $\mathrm{Im}U_{\boldsymbol\alpha}(\xi)<-2\mathrm{Cov}(\boldsymbol l)$ for all $\xi\in\Gamma^*\setminus\{\xi^*\};$ and by the compactness of ${\Gamma^*_L}\setminus \Gamma^*_d,$ there exists an $\epsilon>0$ such that 
 \begin{equation}\label{ImU2}
\mathrm{Im} U_{\boldsymbol\alpha}(\xi)< -2\mathrm{Col}(\boldsymbol l)-4\epsilon
\end{equation}
for all $\xi\in {\Gamma^*_L}\setminus \Gamma^*_d.$ Also by the compactness, there is an $M>0$ such that 
\begin{equation}\label{Imk2}
\mathrm{Im}\kappa_{\boldsymbol\alpha}(\xi)<M
\end{equation}
for all $\xi\in\Gamma^*_L\setminus \Gamma^*_d.$ By Proposition \ref{bound}, there is a  $b_1>0$ such that 
\begin{equation}\label{last5}
\mathrm{Im}\nu_{\boldsymbol\alpha,b}(\xi)b^4<Nb^4<\epsilon 
\end{equation}
for all $b<b_1$ and $\xi\in\Gamma^*;$ and together with  (\ref{ImU2}), we have
\begin{equation}\label{last4}
\mathrm{Im}U_{\boldsymbol \alpha}(\xi)+\mathrm{Im}  \nu_{\boldsymbol \alpha,b}(\xi)b^4< -2\mathrm{Cov}(\boldsymbol l) - 3\epsilon
\end{equation}
for all $b<b_1$ and  $\xi\in\Gamma^*_L\setminus \Gamma^*_d.$  Putting (\ref{Imk2}) and  (\ref{last4}) together, we have
\begin{equation*}
\begin{split}
&\bigg|\frac{1}{\pi b}\int_{\Gamma^*_L\setminus \Gamma^*_d}\exp\bigg(\frac{U_{\boldsymbol \alpha}(\xi)+ \kappa_{\boldsymbol \alpha}(\xi)b^2+ \nu_{\boldsymbol \alpha,b}(\xi)b^4}{2\pi \mathbf i b^2} \bigg)d\xi\bigg|\\
 <  & \frac{2(L-d) e^{\frac{M}{2\pi}}}{\pi b} \exp\bigg(\frac{-\mathrm{Cov}(\boldsymbol l)-\epsilon }{\pi b^2}   \bigg )\\
<&  O\Big(e^{\frac{-\mathrm{Cov}(\boldsymbol l)-\epsilon_2}{\pi b^2}}\Big)
\end{split}
\end{equation*}
for any $\epsilon_2<\epsilon.$ This completes the proof of (II). 
\\

For (III), there  is a $b_0\in (0, b_1)$ such that for all $b<b_0,$
\begin{equation}\label{ImK2}
\mathrm{Im}\kappa_{\boldsymbol\alpha}(\xi^*\pm \mathbf i L) b^2<\epsilon,
\end{equation}
$Kb^2<\epsilon$ and $Nb^4<\epsilon,$ where $K$ and $N$ are respectively the constants in Propositions \ref{bound3} and \ref{bound}. We claim that, for $\xi\in\Gamma^*\setminus \Gamma^*_L$ and $b<b_0,$ 
\begin{equation}\label{cl2}
\mathrm{Im}U_{\boldsymbol \alpha}(\xi)+\mathrm{Im}\kappa_{\boldsymbol \alpha}(\xi)b^2+\mathrm{Im}\nu_{\boldsymbol \alpha,b}(\xi)b^4<-2\big(\mathrm{Cov}(\boldsymbol l)+\epsilon\big)-(4\pi-\epsilon)\big(|\xi-\xi^*|-L\big),
\end{equation}
as a consequence of which, we have,
\begin{equation}\label{CI2}
\begin{split}
& \frac{1}{\pi b}\int_{\Gamma^*\setminus \Gamma^*_L}\exp\bigg(\frac{\mathrm{Im}U_{\boldsymbol \alpha}(\xi)+\mathrm{Im}\kappa_{\boldsymbol \alpha}(\xi)b^2+\mathrm{Im}\nu_{\boldsymbol \alpha,b}(\xi)b^4}{2\pi b^2} \bigg)|d\xi| \\
 < & \frac{1}{\pi b} \exp\bigg(\frac{-\mathrm{Cov}(\boldsymbol l)-\epsilon}{\pi b^2}\bigg)\int_{\Gamma^*\setminus \Gamma^*_L}\exp\bigg(\frac{-(4\pi-\epsilon)\big(|\xi-\xi^*|-L\big)}{2\pi}\bigg) |d\xi|\\
 < & O\Big(e^{\frac{-\mathrm{Cov}(\boldsymbol l)-\epsilon_2}{\pi b^2}}\Big)
\end{split}
\end{equation}
for any $\epsilon_2<\epsilon.$ 

For the proof of the claim, by  Proposition \ref{PL} and (\ref{ImU2}),  for $l>L,$ i.e., $\xi^*\pm \mathbf il \in \Gamma^*\setminus\Gamma^*_L,$ we have
\begin{equation}\label{ImU2'}
\mathrm{Im}U_{\boldsymbol \alpha}(\xi^*\pm \mathbf il )\leqslant \mathrm{Im}U_{\boldsymbol\alpha}(\xi^*\pm \mathbf i L) < -2\mathrm{Cov}(\boldsymbol l)-4\epsilon; 
\end{equation} 
and by Proposition \ref{PL}  again and the choice of $L$ and $b_0,$ we have 
$$\frac{\partial}{\partial \mathrm{Im}\xi} \Big(\mathrm{Im}U_{\boldsymbol\alpha}(\xi^*+\mathbf il)+\mathrm{Im}\kappa_{\boldsymbol\alpha}(\xi^*+\mathbf il)b^2\Big)<-4\pi+\epsilon,$$
and 
$$\frac{\partial}{\partial \mathrm{Im}\xi} \Big(\mathrm{Im}U_{\boldsymbol\alpha}(\xi^*-\mathbf il)+\mathrm{Im}\kappa_{\boldsymbol\alpha}(\xi^*-\mathbf il)b^2\Big)>4\pi-\epsilon.$$
Together with the Mean Value Theorem, (\ref{ImU2'}) and (\ref{ImK2}), we have 
\begin{equation}\label{Bou2}
\begin{split}
\mathrm{Im}U_{\boldsymbol\alpha}(\xi)+ \mathrm{Im}\kappa_{\boldsymbol\alpha}(\xi)b^2 < & \mathrm{Im}U_{\boldsymbol\alpha}(\xi^*\pm \mathbf iL) + \mathrm{Im}
\kappa_{\boldsymbol\alpha}(\xi^*\pm \mathbf iL)b^2 - (4\pi-\epsilon) \big |  \xi - (\xi^*\pm \mathbf iL ) \big|\\
< & -2\mathrm{Cov}(\boldsymbol l)-3\epsilon  -(4\pi-\epsilon)\big(|\xi-\xi^*|-L \big)
\end{split}
\end{equation}
for all $\xi \in \Gamma^*\setminus \Gamma^*_L.$  Finally, putting (\ref{Bou2}) and (\ref{last5}) together, we have (\ref{cl2}) and  the first  inequality in  (\ref{CI2}); and since 
$$|\xi-\xi^*|-L\to+\infty$$
as $\xi\in \Gamma^*\setminus \Gamma^*_L$ approaches $\infty,$ we have the second inequality in (\ref{CI2}). This completes the proof of (III).
\\

Putting (I), (II)  and (III) together, we have 
\begin{equation*}
\lim_{b\to 0} \pi b^2\log \bigg\{\begin{matrix} a_1 & a_2 & a_3 \\ a_4 & a_5 & a_6 \end{matrix} \bigg\}_b =-\widetilde{\mathrm{Cov}}(\boldsymbol l).
\end{equation*}
in Case (a).
\\

For Case (b), let $\xi_{\text{mid}}=\frac{1}{2}(\xi^*_1+\xi^*_2)$ be the mid-point of $\xi^*_1$ and $\xi^*_2,$ and we choose an $L>0$ large enough so that 
$$[\mathrm{Im}\eta_{i_1}-\delta,\mathrm{Im}\eta_{i_4}+\delta]\subset (\mathrm{Im}\xi_{\text{mid}}-L+\delta, \mathrm{Im}\xi_{\text{mid}}+L-\delta),$$
where $\delta>0$ is as above. Similar to Case (2), we consider the following smooth vector filed  $\psi\mathbf v$ on  
$ D_{\delta,c},$
 where 
 $\psi$ is a 
$C^\infty$-smooth bump function on $D_{\delta,c}$ satisfying 
  \begin{equation*}
\left\{
    \begin{array}{rcl}
 \psi(\xi ) = 1  & \text{if} & |\mathrm{Im}\xi  -\mathrm{Im}\xi_{\text{mid}}| \leqslant  L-\delta  \\
    0 <   \psi(\xi ) <1 & \text{if}  & L-\delta <  |\mathrm{Im}\xi -\mathrm{Im}\xi_{\text{mid}}| < L \\
 \psi(\xi )  = 0  & \text{if} & |\mathrm{Im}\xi  -\mathrm{Im}\xi_{\text{mid}}| \geqslant  L
    \end{array}\right.
\end{equation*}
with $\delta>0$ as above, and $\mathbf v$ is the vector field on $D_{\delta,c}$ defined by  
$$\mathbf v =\bigg(-\frac{\partial\mathrm{Im} U_{\boldsymbol\alpha}}{\partial \mathrm{Re}\xi},0 \bigg).$$
Now let $\Gamma^*$ be the contour obtained from $\Gamma'$ by following the flow lines of  $\psi \mathbf v$ for a short time $t>0.$ See Figure \ref{Ddc4}. 
Then by Corollary \ref{cor4} and the choice of $\delta,$ 
$\Gamma^*$ stays inside $D_{\delta,c},$ and by Proposition \ref{nonvanish} (2),
\begin{equation}\label{um3}
\mathrm{Im}U_{\boldsymbol\alpha}(\xi) < \mathrm{Im}U_{\boldsymbol\alpha}(\xi^*_1)= \mathrm{Im}U_{\boldsymbol\alpha}(\xi^*_2)
\end{equation}
for all $\xi  \in \Gamma^*  \setminus \{\xi ^*_1,\xi^*_2\}.$

\begin{figure}[htbp]
\centering
\includegraphics[scale=0.25]{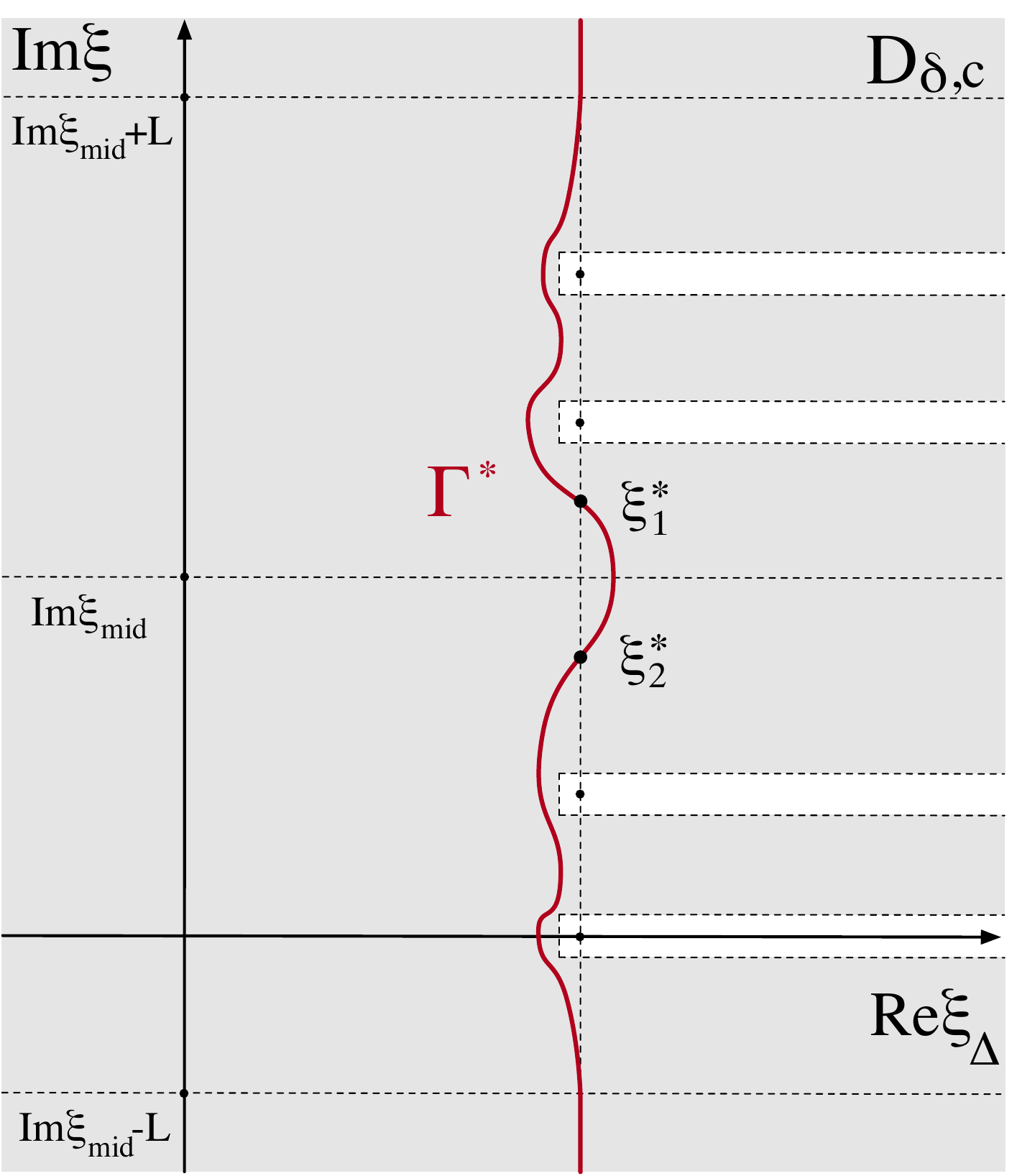}
\caption{Contour $\Gamma^*$ in Case (3)}
\label{Ddc4}
\end{figure}

Now in the computation of the $b$-$6j$ symbol (\ref{6jint}), we choose $\Gamma=\Gamma^*$ as the contour of integral. Then we have
$$\bigg\{\begin{matrix} a_1 & a_2 & a_3 \\ a_4 & a_5 & a_6 \end{matrix} \bigg\}_b=\frac{1}{\pi b}\int_{\Gamma^*}\exp\bigg(\frac{U_{\boldsymbol \alpha}(\xi)+\kappa_{\boldsymbol \alpha}(\xi)b^2+\nu_{\boldsymbol \alpha,b}(\xi)b^4}{2\pi \mathbf i b^2} \bigg)d\xi.$$
Let $d>0$ be sufficiently small so that both of  the regions 
$$B_{i,d}=\Big\{\xi\in \mathbb C\ \Big|\ |\mathrm{Re}\xi -\mathrm{Re} \xi^*_i|<d, |\mathrm{Im}\xi -\mathrm{Im} \xi^*_i|<d \Big\},$$
$i=1,2,$ are in $D_{\delta,c}.$  
For $i=1,2$ let 
$$\Gamma^*_{i,d}=\Gamma^*\cap B_{i,d}=\Big\{ \xi \in \Gamma^*\ \Big|\  |\mathrm{Im}\xi-\mathrm{Im}\xi_i ^*|<d\Big\},$$
let 
$$\Gamma^*_d=\Gamma^*_{1,d}\cup\Gamma^*_{2,d}$$
and let 
$$\Gamma^*_L=\Big\{ \xi \in \Gamma^*\ \Big|\  |\mathrm{Im}\xi-\mathrm{Im}\xi_{\text{mid}}|\leqslant L \Big\}.$$
We will show that, as $b\to 0,$ there exists an $\epsilon_3>0$ such that:
\begin{enumerate}[(I)]
\item For $i=1,2,$
\begin{equation*}
\frac{1}{\pi b}\int_{\Gamma^*_{1,d}}\exp\bigg(\frac{U_{\boldsymbol \alpha}(\xi)+ \kappa_{\boldsymbol \alpha}(\xi)b^2+ \nu_{\boldsymbol \alpha,b}(\xi)b^4}{2\pi \mathbf i b^2} \bigg)d\xi = \frac{1}{2}\frac{e^{-\big(\frac{\mathrm{Re}U_{\boldsymbol\alpha}(\xi_1^*)}{2\pi b^2}+\frac{\pi}{4}\big)\mathbf i }}{\sqrt[4]{\det\mathrm{Gram}(\boldsymbol l)}} e^{\frac{-\widetilde{\mathrm{Cov}}(\boldsymbol l)}{\pi b^2}} \Big(1+O\big(b^2\big)\Big)
\end{equation*}
and 
\begin{equation*}
 \frac{1}{\pi b}\int_{\Gamma^*_{2,d}}\exp\bigg(\frac{U_{\boldsymbol \alpha}(\xi)+ \kappa_{\boldsymbol \alpha}(\xi)b^2+ \nu_{\boldsymbol \alpha,b}(\xi)b^4}{2\pi \mathbf i b^2} \bigg)d\xi = \frac{1}{2}\frac{e^{\big(\frac{\mathrm{Re}U_{\boldsymbol\alpha}(\xi_1^*)}{2\pi\boldsymbol b^2}+\frac{\pi}{4}\big)\mathbf i } }{\sqrt[4]{\det\mathrm{Gram}(\boldsymbol l)}} e^{\frac{-\widetilde{\mathrm{Cov}}(\boldsymbol l)}{\pi b^2}} \Big(1+O\big(b^2\big)\Big),
\end{equation*}

\item $$\bigg|\frac{1}{\pi b}\int_{\Gamma^*_L\setminus \Gamma^*_d}\exp\bigg(\frac{U_{\boldsymbol \alpha}(\xi)+ \kappa_{\boldsymbol \alpha}(\xi)b^2+ \nu_{\boldsymbol \alpha,b}(\xi)b^4}{2\pi \mathbf i b^2} \bigg)d\xi\bigg|< O\Big(e^{\frac{-\widetilde{\mathrm{Cov}}(\boldsymbol l)-\epsilon_3}{\pi b^2}}\Big),$$
and 

\item $$\bigg|\frac{1}{\pi b}\int_{\Gamma^*\setminus \Gamma^*_L}\exp\bigg(\frac{U_{\boldsymbol \alpha}(\xi)+ \kappa_{\boldsymbol \alpha}(\xi)b^2+ \nu_{\boldsymbol \alpha,b}(\xi)b^4}{2\pi \mathbf i b^2} \bigg)d\xi\bigg|<O\Big(e^{\frac{-\widetilde{\mathrm{Cov}}(\boldsymbol l)-\epsilon_3}{\pi b^2}}\Big),$$
\end{enumerate}
 from which the result follows.
\\

For (I), we claim that by letting $\hbar=b^2,$ $D_i=B_{i,d}$ for $i=1,2,$ $f=\frac{U_{\boldsymbol \alpha}}{2\pi \mathbf i},$ $g=\exp\big(\frac{\kappa_{\boldsymbol \alpha}}{2\pi \mathbf i }\big),$ $f_\hbar=\frac{U_{\boldsymbol \alpha}+\nu_{\boldsymbol \alpha,b}b^4}{2\pi \mathbf i},$ $\upsilon_{\hbar}=\frac{\nu_{\boldsymbol \alpha,b}}{2\pi \mathbf i},$ $S_i=\Gamma^*_{i,d}$ and $c_i=\xi_i^*$  for $i=1$ and $2,$ all the conditions  in Proposition \ref{saddle}  are satisfied respectively.

Indeed,  by Proposition \ref{critical2} (3), $\xi^*_i$ is a critical point of $f=\frac{U_{\boldsymbol \alpha}}{2\pi \mathbf i}$ in $B_{i,d},$ hence condition (i) is satisfied. 

By (\ref{um3}), $\xi_i^*$ is the unique maximum point of $\mathrm{Re}f=\frac{\mathrm{Im}U_{\boldsymbol\alpha}}{2\pi}$ on $\Gamma^*_i,$ hence condition (ii) is satisfied.

For conditions (iii) and (iv), since  $\xi^*_i\in D_{\epsilon, c},$ $\kappa_{\boldsymbol \alpha}(\xi_i^*)$ is a finite value. As a consequence, $g(\xi_i^*)=\exp\big(\frac{\kappa_{\boldsymbol \alpha}(\xi_i^*)}{2\pi \mathbf i }\big)\neq 0,$ and condition (iv) is satisfied;  and by Proposition \ref{critical2} (3),  condition (iii) is satisfied.

For condition (v), by Proposition \ref{bound},  $|\upsilon_{\hbar}(\xi)|=\big|\frac{\nu_{\boldsymbol \alpha,b}(\xi)}{2\pi \mathbf i}\big|<\frac{N}{2\pi}$ on $B_{i,d}.$

For condition (vi), since near each $\xi_i,$ $\Gamma^*$ is obtained from a straight line segment $\Gamma'$ by moving along the flow lines of a smooth vector field $\mathbf v,$ $\Gamma^*_{i,d}$ is a smooth embedding around $\xi^*_i.$

Finally, by Proposition \ref{saddle}, Proposition \ref{critical2} (3)  and Proposition \ref{Hess}, we have as $b\to 0,$
\begin{equation*}
\begin{split}
\frac{1}{\pi b}\int_{\Gamma^*_{1,d}}\exp\bigg(\frac{U_{\boldsymbol \alpha}(\xi)+ \kappa_{\boldsymbol \alpha}(\xi)b^2+ \nu_{\boldsymbol \alpha,b}(\xi)b^4}{2\pi \mathbf i b^2} \bigg)d\xi = &  \frac{1}{2}\frac{e^{\frac{U_{\boldsymbol \alpha}(\xi_1^*)}{2\pi\mathbf i b^2}}}{\sqrt{\sqrt{-\det\mathrm{Gram}(\boldsymbol l)}}}\Big(1+O\big(b^2\big)\Big)\\
= &  \frac{1}{2} \frac{e^{-\big(\frac{\mathrm{Re}U_{\boldsymbol\alpha}(\xi_1^*)}{2\pi b^2}+\frac{\pi}{4}\big)\mathbf i }}{\sqrt[4]{\det\mathrm{Gram}(\boldsymbol l)}} e^{\frac{-\widetilde{\mathrm{Cov}}(\boldsymbol l)}{\pi b^2}} \Big(1+O\big(b^2\big)\Big),
\end{split}
\end{equation*}
and 
\begin{equation*}
\begin{split}
 \frac{1}{\pi b}\int_{\Gamma^*_{2,d}}\exp\bigg(\frac{U_{\boldsymbol \alpha}(\xi)+ \kappa_{\boldsymbol \alpha}(\xi)b^2+ \nu_{\boldsymbol \alpha,b}(\xi)b^4}{2\pi \mathbf i b^2} \bigg)d\xi = & \frac{1}{2}\frac{e^{\frac{U_{\boldsymbol \alpha}(\xi_2^*)}{2 \pi \mathbf i b^2}}}{\sqrt{-\sqrt{-\det\mathrm{Gram}(\boldsymbol l)}}}\Big(1+O\big(b^2\big)\Big)\\
=& \frac{1}{2} \frac{e^{\big(\frac{\mathrm{Re}U_{\boldsymbol\alpha}(\xi_1^*)}{2\pi b^2}+\frac{\pi}{4}\big)\mathbf i }}{\sqrt[4]{\det\mathrm{Gram}(\boldsymbol l)}} e^{\frac{-\widetilde{\mathrm{Cov}}(\boldsymbol l)}{\pi b^2}} \Big(1+O\big(b^2\big)\Big).
\end{split}
\end{equation*}
 This completes the proof of (I). 
\\

The proof of  (II) and (III) is quite similar to that of Case (a). Namely,  we have
\begin{equation*}
\begin{split}
& \bigg|\frac{1}{\pi b}\int_{\Gamma^*\setminus \Gamma^*_L}\exp\bigg(\frac{U_{\boldsymbol \alpha}(\xi)+ \kappa_{\boldsymbol \alpha}(\xi)b^2+ \nu_{\boldsymbol \alpha,b}(\xi)b^4}{2\pi \mathbf i b^2} \bigg)d\xi\bigg|\\
\leqslant & \frac{1}{\pi b}\int_{\Gamma^*\setminus \Gamma^*_L}\exp\bigg(\frac{\mathrm{Im}U_{\boldsymbol \alpha}(\xi)+\mathrm{Im}\kappa_{\boldsymbol \alpha}(\xi)b^2+\mathrm{Im}\nu_{\boldsymbol \alpha,b}(\xi)b^4}{2\pi b^2} \bigg)|d\xi|.
\end{split}
\end{equation*}

For  (II),  by (\ref{um3}) and Proposition \ref{critical2} (3), $\mathrm{Im}U_{\boldsymbol\alpha}(\xi)<-2\widetilde{\mathrm{Cov}}(\boldsymbol l)$ for all $\xi\in\Gamma^*\setminus\{\xi^*_1,\xi^*_2\};$ and by the compactness of ${\Gamma^*_L}\setminus \Gamma^*_d,$ there exists an $\epsilon>0$ such that 
 \begin{equation}\label{ImU3}
\mathrm{Im} U_{\boldsymbol\alpha}(\xi)< -2\widetilde{\mathrm{Col}}(\boldsymbol l)-4\epsilon
\end{equation}
for all $\xi\in {\Gamma^*_L}\setminus \Gamma^*_d.$ Also by the compactness, there is an $M>0$ such that 
\begin{equation}\label{Imk3}
\mathrm{Im}\kappa_{\boldsymbol\alpha}(\xi)<M
\end{equation}
for all $\xi\in\Gamma^*_L\setminus \Gamma^*_d.$ By Proposition \ref{bound},  there is a $b_1>0$ such that 
\begin{equation}\label{last7}
\mathrm{Im}\nu_{\boldsymbol\alpha,b}(\xi)b^4<Nb^4<\epsilon 
\end{equation}
for all $b<b_1$ and $\xi\in\Gamma^*;$ and together with  (\ref{ImU3}), we have
\begin{equation}\label{last6}
\mathrm{Im}U_{\boldsymbol \alpha}(\xi)+\mathrm{Im}  \nu_{\boldsymbol \alpha,b}(\xi)b^4< -2\widetilde{\mathrm{Cov}}(\boldsymbol l) - 3\epsilon
\end{equation}
for all $b<b_1$ and  $\xi\in\Gamma^*_L\setminus \Gamma^*_d.$  Putting (\ref{Imk3}) and  (\ref{last6}) together, we have
\begin{equation*}
\begin{split}
&\bigg|\frac{1}{\pi b}\int_{\Gamma^*_L\setminus \Gamma^*_d}\exp\bigg(\frac{U_{\boldsymbol \alpha}(\xi)+ \kappa_{\boldsymbol \alpha}(\xi)b^2+ \nu_{\boldsymbol \alpha,b}(\xi)b^4}{2\pi \mathbf i b^2} \bigg)d\xi\bigg| \\
<  & \frac{2(L-d) e^{\frac{M}{2\pi}}}{\pi b} \exp\bigg(\frac{-\widetilde{\mathrm{Cov}}(\boldsymbol l)-\epsilon }{\pi b^2}   \bigg )\\
< &  O\Big(e^{\frac{-\widetilde{\mathrm{Cov}}(\boldsymbol l)-\epsilon_3}{\pi b^2}}\Big)
\end{split}
\end{equation*}
for any $\epsilon_3<\epsilon.$ This completes the proof of (II). 
\\

For (III),  there  is a $b_0\in (0,b_1)$ such that for all $b<b_0,$
\begin{equation}\label{ImK3}
\mathrm{Im}\kappa_{\boldsymbol\alpha}(\xi_{\text{mid}}\pm \mathbf i L) b^2<\epsilon,
\end{equation}
$Kb^2<\epsilon$ and $Nb^4<\epsilon,$ where $K$ and $N$ are respectively the constants in Propositions \ref{bound3} and \ref{bound}. We claim that, for $\xi\in\Gamma^*\setminus \Gamma^*_L$ and $b<b_0,$ 
\begin{equation}\label{cl3}
\mathrm{Im}U_{\boldsymbol \alpha}(\xi)+\mathrm{Im}\kappa_{\boldsymbol \alpha}(\xi)b^2+\mathrm{Im}\nu_{\boldsymbol \alpha,b}(\xi)b^4<-2\big(\widetilde{\mathrm{Cov}}(\boldsymbol l)+\epsilon\big)-(4\pi-\epsilon)\big(|\xi-\xi_{\text{mid}}|-L\big),
\end{equation}
as a consequence of which, we have,
\begin{equation}\label{CI3}
\begin{split}
& \frac{1}{\pi b}\int_{\Gamma^*\setminus \Gamma^*_L}\exp\bigg(\frac{\mathrm{Im}U_{\boldsymbol \alpha}(\xi)+\mathrm{Im}\kappa_{\boldsymbol \alpha}(\xi)b^2+\mathrm{Im}\nu_{\boldsymbol \alpha,b}(\xi)b^4}{2\pi b^2} \bigg)|d\xi| \\
 < & \frac{1}{\pi b} \exp\bigg(\frac{-\widetilde{\mathrm{Cov}}(\boldsymbol l)-\epsilon}{\pi b^2}\bigg)\int_{\Gamma^*\setminus \Gamma^*_L}\exp\bigg(\frac{-(4\pi-\epsilon)\big(|\xi-\xi_{\text{mid}}|-L\big)}{2\pi}\bigg) |d\xi|\\
< & O\Big(e^{\frac{-\widetilde{\mathrm{Cov}}(\boldsymbol l)-\epsilon_3}{\pi b^2}}\Big)
\end{split}
\end{equation}
for any $\epsilon_3<\epsilon.$ For the proof of the claim, by  Proposition \ref{PL} and (\ref{ImU3}),  for $l>L,$ i.e., $\xi_{\text{mid}}\pm \mathbf il \in \Gamma^*\setminus\Gamma^*_L,$ we have
\begin{equation}\label{ImU3'}
\mathrm{Im}U_{\boldsymbol \alpha}(\xi_{\text{mid}}\pm \mathbf il )\leqslant \mathrm{Im}U_{\boldsymbol\alpha}(\xi_{\text{mid}}\pm \mathbf i L) < -2\widetilde{\mathrm{Cov}}(\boldsymbol l)-4\epsilon; 
\end{equation} 
and by   Proposition \ref{PL}  again and the choice of $L$ and $b_0,$ we have 
$$\frac{\partial}{\partial \mathrm{Im}\xi} \Big(\mathrm{Im}U_{\boldsymbol\alpha}(\xi^*+\mathbf il)+\mathrm{Im}\kappa_{\boldsymbol\alpha}(\xi^*+\mathbf il)b^2\Big)<-4\pi+\epsilon,$$
and 
$$\frac{\partial}{\partial \mathrm{Im}\xi} \Big(\mathrm{Im}U_{\boldsymbol\alpha}(\xi^*-\mathbf il)+\mathrm{Im}\kappa_{\boldsymbol\alpha}(\xi^*-\mathbf il)b^2\Big)>4\pi-\epsilon.$$
Together with the Mean Value Theorem, (\ref{ImU3'}) and (\ref{ImK3}), we have 
\begin{equation}\label{Bou3}
\begin{split}
\mathrm{Im}U_{\boldsymbol\alpha}(\xi)+ \mathrm{Im}\kappa_{\boldsymbol\alpha}(\xi)b^2 < & \mathrm{Im}U_{\boldsymbol\alpha}(\xi^*\pm \mathbf iL) + \mathrm{Im}
\kappa_{\boldsymbol\alpha}(\xi_{\text{mid}}\pm \mathbf iL)b^2 - (4\pi-\epsilon) \big |  \xi - (\xi_{\text{mid}}\pm \mathbf iL ) \big|\\
< & -2\widetilde{\mathrm{Cov}}(\boldsymbol l)-3\epsilon  -(4\pi-\epsilon)\big(|\xi-\xi_{\text{mid}}|-L \big)
\end{split}
\end{equation}
for all $\xi \in \Gamma^*\setminus \Gamma^*_L.$  Finally, putting (\ref{Bou3}) and (\ref{last7}) together, we have (\ref{cl3}) and  the first  inequality  in  (\ref{CI3}); and since 
$$|\xi-\xi_{\text{mid}}|-L\to+\infty$$
as $\xi\in \Gamma^*\setminus \Gamma^*_L$ approaches $\infty,$ we have the second  inequality in (\ref{CI3}). This completes the proof of (III).
\\

Putting (I), (II)  and (III) together, {we have
\begin{equation}\label{eq:ads}
\bigg\{\begin{matrix} a_1 & a_2 & a_3 \\ a_4 & a_5 & a_6 \end{matrix} \bigg\}_b=\frac{e^{\frac{-\widetilde{\mathrm{Cov}}(\boldsymbol l)}{\pi b^2}} }{\sqrt[4]{\det\mathrm{Gram}(\boldsymbol l)}} \Bigg(\cos\bigg(\frac{\mathrm{Re}U_{\boldsymbol\alpha}(\xi_1^*)}{2\pi b^2}+\frac{\pi}{4}\bigg) +O\big(b^2\big)\Bigg).
\end{equation}
If $\mathrm{Re}U_{\boldsymbol\alpha}(\xi_1^*)=0$, then we have 
\begin{equation*}
\limsup_{b\to 0} \pi b^2\log\Bigg|\bigg\{\begin{matrix} a_1 & a_2 & a_3 \\ a_4 & a_5 & a_6 \end{matrix} \bigg\}_b\Bigg| = \lim_{b\to 0} \pi b^2\log\Bigg|\bigg\{\begin{matrix} a_1 & a_2 & a_3 \\ a_4 & a_5 & a_6 \end{matrix} \bigg\}_b\Bigg| = -\widetilde{\mathrm{Cov}}(\boldsymbol l);
\end{equation*}
and if $\mathrm{Re}U_{\boldsymbol\alpha}(\xi_1^*)\neq 0$, then we take a sequence $b_n\to0$ such that $\frac{\mathrm{Re}U_{\boldsymbol\alpha}(\xi_1^*)}{2\pi b_n^2}+ \frac{\pi}{4}=0$ {(mod $\pi$)} for each $n$, and have 
\begin{equation*}
\lim_{n\to \infty} \pi b_n^2\log \bigg\{\begin{matrix} a_1 & a_2 & a_3 \\ a_4 & a_5 & a_6 \end{matrix} \bigg\}_{b_n} = -\widetilde{\mathrm{Cov}}(\boldsymbol l),
\end{equation*}
and
\begin{equation*}
\limsup_{b\to 0} \pi b^2\log\Bigg|\bigg\{\begin{matrix} a_1 & a_2 & a_3 \\ a_4 & a_5 & a_6 \end{matrix} \bigg\}_b\Bigg| = -\widetilde{\mathrm{Cov}}(\boldsymbol l)
\end{equation*}
in Case (b).}
\end{proof}

\begin{remark}\label{rmk:ads2}
In \cite{ADS25}, we will prove Eq. (\ref{AdS/CFT}) by showing that $\mathrm{Re}U_{\boldsymbol\alpha}(\xi_1^*)$ in~\eqref{eq:ads} equals $2\mathrm{Cov(\Delta)}$, where $\Delta$ is the truncated hyperideal tetrhadron in the anti-de Sitter space $\mathbb A\mathrm d\mathbb S^3$ with  $(l_1,\dots,l_6)$ the edge lengths, and  $\mathrm{Cov(\Delta)}$ is the co-volume of $\Delta$ as defined in Remark~\ref{rmk:ads}.
\end{remark}


\section{Convergence of the state-integral}\label{sec:convergence}

Theorem~\ref{Converge} is an immediate consequence of  the following Theorem \ref{angleconv} and Proposition \ref{KA}. 

\begin{theorem}\label{angleconv} Let $M$ be a hyperbolic $3$-manifold with totally geodesic boundary and let $\mathcal T$ be an ideal triangulation of $M$ that supports an angle structure (see section \ref{anglestr}). Then  for all $b\in (0,1),$ the integral 
$$\mathrm{TV}_b(M,\mathcal T)=\int_{\big(\frac{Q}{2}+\mathbf i{\mathbb R_{>0}}\big)^E} \prod_{e\in E}|e|_{\boldsymbol a}\prod_{\Delta\in T}|\Delta|_{\boldsymbol a}d\mathrm{Im}\boldsymbol a$$
converges absolutely.
\end{theorem}

The proof of Theorem \ref{angleconv} relies on the following

\begin{proposition}\label{cor:6jbound}
Let $\boldsymbol \theta =(\theta_1,\dots,\theta_6)$ be the dihedral angles of a hyperideal tetrahedron, and let $\boldsymbol a=(a_1,\dots,a_6)\in \big(\frac{Q}{2}+\mathbf i \mathbb R_{>0}\big)^6.$ Then there is a $C_{b,\boldsymbol \theta}>0$ depending only on $b$ and $\boldsymbol \theta$ and an $\epsilon_{\boldsymbol \theta}>0$ such that 
\begin{equation}
\Bigg|\bigg\{\begin{matrix} a_1 & a_2 & a_3 \\ a_4 & a_5 & a_6 \end{matrix} \bigg\}_b\Bigg|\leqslant C_{b,\boldsymbol \theta}\big(1+\boldsymbol 1\cdot \mathrm{Im}\boldsymbol a\big)e^{-Q (\boldsymbol\theta\cdot \mathrm{Im}\boldsymbol a)-\epsilon_{\boldsymbol \theta} Q(\boldsymbol\pi\cdot \mathrm{Im}\boldsymbol a)},
\end{equation}
where $\boldsymbol 1=(1,\dots, 1),$ $\mathrm{Im}\boldsymbol a=(\mathrm{Im}a_1,\dots,\mathrm{Im}a_6)$ and $\boldsymbol \pi=(\pi,\dots,\pi).$
\end{proposition}

To prove Proposition  \ref{cor:6jbound} , we need the following Lemma \ref{lem:sumbound}, Lemma \ref{lem:6jbound1} and Lemma \ref{lem:6jbound2}.

\begin{lemma}\label{lem:sumbound}
For real numbers $x_1\leqslant \cdots  \leqslant  x_{2n}$ and $a>0,$ 
we have 
$$\int_{-\infty}^{\infty}e^{-a\sum_{i=1}^{2n}|x-x_i|}dx \leqslant \bigg(\frac{1}{an}+x_{2n}-x_1\bigg) e^{-a\sum_{i=1}^n(x_{2n-i+1}-x_i)}.$$
\end{lemma}

\begin{proof}
Estimating the integral on the intervals  $(-\infty, x_{1}],$  $[x_{2n}, \infty)$ and $[x_{1}, x_{2n}]$ respectively, we have 
$$\int_{-\infty}^{x_{1}}e^{-a\sum_{i=1}^{2n}|x-x_{i}|}dx=\int_{-\infty}^{x_1} e^{2anx-a\sum_{i=1}^{2n}x_{i}}dx=\frac{1}{2an}e^{2anx_{1}-a\sum_{i=1}^{2n}x_{i}}\leqslant \frac{1}{2an}e^{-a\sum_{i=1}^{n} (x_{2n-i+1}-x_{i})},$$
$$\int_{x_{2n}}^{\infty}e^{-a\sum_{i=1}^{2n}|x-x_{i}|}dx=\int_{x_{2n}}^{\infty}e^{-2anx+a\sum_{i=1}^{2n}x_{i}}dx=\frac{1}{2an}e^{-2a{nx_{2n}}+a\sum_{i=1}^{2n}x_{i}}\leqslant \frac{1}{2an}e^{-a\sum_{i=1}^{n} (x_{2n-i+1}-x_{i})},$$
and 
$$\int_{x_{1}}^{x_{2n}}e^{-a\sum_{i=1}^{2n}|x-x_{i}|}dx\leqslant(x_{2n}-x_{1})\max_{x\in [x_1,x_{2n}]} e^{-a\sum_{i=1}^{2n}|x-x_{i}|}=(x_{2n}-x_{1})e^{-a\sum_{i=1}^{n} (x_{2n-i+1}-x_{i})}.$$
Putting these together, we have the result. 
\end{proof}

\begin{lemma}\label{lem:6jbound1}
There is the constant  $C_b$  depending only on $b$ such that for all $(a_1,\dots, a_6)\in \big(\frac{Q}{2}+\mathbf i\mathbb R_{>0}\big)^6,$ 
\begin{equation}
\Bigg|\bigg\{\begin{matrix} a_1 & a_2 & a_3 \\ a_4 & a_5 & a_6 \end{matrix} \bigg\}_b\Bigg|\leqslant C_b\bigg(1+\sum_{k=1}^6\mathrm{Im}a_k\bigg)e^{-\frac{\pi Q}{2}\max\{\mathrm{Im}t_1,\mathrm{Im}t_2,\mathrm{Im}t_3,\mathrm{Im}t_4\}},
\end{equation}
\end{lemma}

\begin{proof}  Let $\Gamma=\frac{7Q}{4}+\mathbf i\mathbb{R}$ in the integral (\ref{b-6j}). By (\ref{=1}) and (\ref{Sbestimate}) with the constant $C$ therein, we have 
$$\Bigg|\bigg\{\begin{matrix} a_1 & a_2 & a_3 \\ a_4 & a_5 & a_6 \end{matrix} \bigg\}_b\Bigg|\leqslant \int_{\Gamma}C^8 e^{-\frac{\pi Q}{4}\big(\sum_{i=1}^4|\mathrm{Im}(u-t_i)|+\sum_{j=1}^4|\mathrm{Im}(u-q_j)|\big)}d\mathrm{Im} u.
$$
Now let $D$ be the sum of the four largest numbers minus the sum of the four smallest numbers in $\{\mathrm{Im}t_i\}_{i\in\{1,2,3,4\}}\cup\{\mathrm{Im}q_j\}_{j\in\{1,2,3,4\}}.$ 
Then 
$$D\geqslant \Big(\sum_{j=1}^3\mathrm{Im}q_j+\max_{i}\mathrm{Im}t_{i}\Big)-\Big(\sum_{i=1}^{4}\mathrm{Im}t_{i}-\max_{i}\mathrm{Im}t_{i}\Big)=2\max_{i\in\{1,2,3,4\}}\{\mathrm{Im}t_i\};$$
and by Lemma \ref{lem:sumbound}, we have 
\begin{align*}
\Bigg|\bigg\{\begin{matrix} a_1 & a_2 & a_3 \\ a_4 & a_5 & a_6 \end{matrix} \bigg\}_b\Bigg|
&\leqslant C^8\bigg(\frac{1}{\pi Q}+\max_{i,j \in\{1,2,3,4\}}\{\mathrm{Im}t_i,\mathrm{Im}q_{j}\}\bigg)e^{-\frac{\pi Q}{4}D}\\
&\leqslant  C^8\bigg(1+\sum_{k=1}^6\mathrm{Im}a_k\bigg)e^{-\frac{\pi Q}{2}\max_{i\in\{1,2,3,4\}}\{\mathrm{Im}t_i\}},
\end{align*}
where the last inequality comes from that $Q>2$ and $\max_{j\in\{1,2,3\}}\{\mathrm{Im}q_j\}\leqslant  \sum_{k=1}^6\mathrm{Im}a_k;$ and the result holds with $C_b=C^8.$
\end{proof}

\begin{lemma}\label{lem:6jbound2}
For all $\delta\in \big(0,\frac{Q}{2}\big),$  there is a constant $C_{b,\delta}>0$ depending only on $\delta$ and $b$ such that 
for all $(a_1,\dots, a_6)\in \big(\frac{Q}{2}+\mathbf i\mathbb R_{>0}\big)^6,$ 
\begin{equation}
\Bigg|\bigg\{\begin{matrix} a_1 & a_2 & a_3 \\ a_4 & a_5 & a_6 \end{matrix} \bigg\}_b\Bigg|\leqslant C_{b,\delta}\bigg(1+\sum_{k=1}^6\mathrm{Im}a_k\bigg)e^{-\pi(Q-\delta)\max\big\{\mathrm{Im}a_1+\mathrm{Im}a_4,  \mathrm{Im}a_2+\mathrm{Im}a_5, \mathrm{Im}a_3+\mathrm{Im}a_6\big\}}.
\end{equation}
\end{lemma}

\begin{proof}
Let $\Gamma=2Q-\frac{\delta}{2}+\mathbf i\mathbb{R}$ in the integral (\ref{b-6j}). By (\ref{=1}) and (\ref{Sbestimate}) with the constant $C$ therein, we have 
\begin{equation*}
\begin{split}
    \Bigg|\bigg\{\begin{matrix} a_1 & a_2 & a_3 \\ a_4 & a_5 & a_6 \end{matrix} \bigg\}_b\Bigg|  
        &\leqslant \int_{\Gamma}C^8 
e^{-\pi\big(\frac{\delta}{2}\sum_{i=1}^4|\mathrm{Im}(u-t_i)|+\frac{Q-\delta}{2}\sum_{j=1}^4|\mathrm{Im}(q_j-u)|\big)}
d\mathrm{Im} u\\
&\leqslant \int_{\Gamma}C^8 
e^{-\pi\big(\frac{Q-\delta}{2}\big)\sum_{j=1}^4|\mathrm{Im}(u-q_j)|}
d\mathrm{Im} u.
\end{split}
\end{equation*}
Then by Lemma \ref{lem:sumbound}, we have
\begin{align*}
\Bigg|\bigg\{\begin{matrix} a_1 & a_2 & a_3 \\ a_4 & a_5 & a_6 \end{matrix} \bigg\}_b\Bigg|&\leqslant  C^8\bigg(\frac{1}{\pi(Q-\delta)}+\max_{j\in\{1,2,3\}}\{\mathrm{Im}q_{j}\}\bigg)e^{-\pi\big(\frac{Q-\delta}{2}\big)\big(\sum_{j=1}^3 \mathrm{Im}q_{j}-2\min_{j\in\{1, 2, 3\}}\{\mathrm{Im}q_{j}\}\big)}\\
&\leqslant C^8\bigg(1+\sum_{k=1}^6\mathrm{Im}a_k\bigg)e^{-\pi(Q-\delta)\max\big\{\mathrm{Im}a_1+\mathrm{Im}a_4,  \mathrm{Im}a_2+\mathrm{Im}a_5, \mathrm{Im}a_3+\mathrm{Im}a_6\big\}},
\end{align*}
where the last inequality comes from that $Q-\delta>\frac{Q}{2}>1;$ and the result holds with $C_{b,\delta}=C^8.$
\end{proof}

\begin{proof}[Proof of Proposition \ref{cor:6jbound}]
Recall that the set $\mathcal A$ of all possible dihedral angles of a hyperideal tetrahedron is a convex polytope in $(0,\pi)^6$ whose closure has the following vertices: $\boldsymbol u_1=\big(\frac{\pi}{2},\frac{\pi}{2},\frac{\pi}{2},0,0,0\big),$ $\boldsymbol u_2=\big(\frac{\pi}{2},0,0,0,\frac{\pi}{2},\frac{\pi}{2}\big),$  $\boldsymbol u_3=\big(0,\frac{\pi}{2},0,\frac{\pi}{2},0,\frac{\pi}{2}\big),$ $\boldsymbol u_4=\big(0,0,\frac{\pi}{2},\frac{\pi}{2},\frac{\pi}{2},0\big),$ $\boldsymbol v_1=(\pi, 0,0,\pi,0,0),$ $\boldsymbol v_2=(0,\pi,0,0,\pi,0),$ $\boldsymbol v_3=(0,0,\pi,0,0,\pi),$ and $\boldsymbol w_1,\dots, \boldsymbol w_6,$ where for each $i,k\in\{1,\dots,6\},$ the $i$-th entry of $\boldsymbol w_k$ equals $\pi \delta_{ik},$ and the origin   $\boldsymbol 0=(0,0,0,0,0,0).$ As $\boldsymbol \theta\in \mathcal A,$ there exist positive real numbers $\{a_i\}_{i\in\{1,2,3,4\}},$ $\{b_j\}_{j\in\{1,2,3\}},$ $\{c_k\}_{k\in\{1,\dots,6\}}$ and $d$ such that 
\begin{equation}\label{theta=}
    \boldsymbol \theta = \sum_{i=1}^4 a_i\boldsymbol u_i + \sum_{j=1}^3 b_j\boldsymbol v_j+ \sum_{k=1}^6 c_k\boldsymbol w_k+d\boldsymbol 0
\end{equation}
and 
\begin{equation}\label{=12}
\sum_{i=1}^4 a_i  + \sum_{j=1}^3 b_i + \sum_{k=1}^6 c_i +d=1.
\end{equation}
Then by (\ref{=12}), Lemma \ref{lem:6jbound1} and Lemma \ref{lem:6jbound2},
we have 
\begin{equation}\label{b-6jbound}
    \Bigg|\bigg\{\begin{matrix} a_1 & a_2 & a_3 \\ a_4 & a_5 & a_6 \end{matrix} \bigg\}_b\Bigg|\leqslant C_{b,\boldsymbol \theta}\bigg(1+\sum_{k=1}^6\mathrm{Im}a_k\bigg)e^{-\frac{\pi Q}{2}\sum_{i=1}^4a_i\mathrm{Im}t_i-\pi(Q-\delta)\sum_{j=1}^3\big(b_j+c_j+c_{j+3}+\frac{d}{3}\big)(\mathrm{Im}a_j+\mathrm{Im}a_{j+3})}
\end{equation}
for some $\delta
\in\big(0,\frac{Q}{2}\big)$ to be determined later, and 
\begin{equation}\label{Cbtheta}
    C_{b,\boldsymbol \theta}=C_b^{\sum_{i=1}^4a_i}C_{b,\delta}^{\sum_{j=1}^3b_j+\sum_{k=1}^6c_k+d}.
\end{equation}
Computing  the exponent of (\ref{b-6jbound}), we have 
\begin{equation}\label{exponent}
\begin{split}
  &-\frac{\pi Q}{2}\sum_{i=1}^4a_i\mathrm{Im}t_i-\pi(Q-\delta)\sum_{j=1}^3\bigg(b_j+c_j+c_{j+3}+\frac{d}{3}\bigg)\big(\mathrm{Im}a_j+\mathrm{Im}a_{j+3}\big) \\
  = &  -\frac{\pi Q}{2}\sum_{i=1}^4a_i\mathrm{Im}t_i-\pi Q\sum_{j=1}^3b_j\big(\mathrm{Im}a_j+\mathrm{Im}a_{j+3}\big) -\pi Q\sum_{j=1}^3(c_j+c_{j+3})(\mathrm{Im}a_j+\mathrm{Im}a_{j+3}) \\
  & -\frac{\pi Q d}{3}\sum_{k=1}^6\mathrm{Im}a_k + \pi\delta \sum_{j=1}^3 \bigg(b_j+c_j+c_{j+3}+\frac{d}{3}\bigg)\big(\mathrm{Im}a_j+\mathrm{Im}a_{j+3}\big).
\end{split}
\end{equation}
Now by (\ref{theta=}) and that $\sum_{j=1}^3(c_j+c_{j+3})(\mathrm{Im}a_j+\mathrm{Im}a_{j+3})> \sum_{k=1}^6c_k\mathrm{Im}a_k,$ the first row of the right hand side of (\ref{exponent}) 
\begin{equation}\label{first}
\begin{split}
 &-\frac{\pi Q}{2}\sum_{i=1}^4a_i\mathrm{Im}t_i-\pi Q\sum_{j=1}^3b_j\big(\mathrm{Im}a_j+\mathrm{Im}a_{j+3}\big) -\pi Q\sum_{j=1}^3(c_j+c_{j+3})(\mathrm{Im}a_j+\mathrm{Im}a_{j+3})\\
< & -\frac{\pi Q}{2}\sum_{i=1}^4a_i\mathrm{Im}t_i-\pi Q\sum_{j=1}^3b_j\big(\mathrm{Im}a_j+\mathrm{Im}a_{j+3}\big) -\pi Q\sum_{k=1}^6 c_k\mathrm{Im} a_k \\
= & -Q\bigg( \sum_{i=1}^4 a_i\boldsymbol u_i + \sum_{j=1}^3 b_j\boldsymbol v_j+ \sum_{k=1}^6 c_k\boldsymbol w_k\bigg)\cdot \mathrm{Im}\boldsymbol a
= -Q(\boldsymbol \theta\cdot \mathrm{Im}\boldsymbol a).
\end{split}
\end{equation}
Next, letting  
$\delta>0$ be small enough so that 
$$\delta< \frac{Qd}{6\max_{j\in\{1,2,3\}} \big\{b_j+c_j+c_{j+3}+\frac{d}{3}\big\}},$$
the second row of the right hand side of (\ref{exponent}) 
\begin{equation}\label{second}
\begin{split}
    &-\frac{\pi Q d}{3}\sum_{k=1}^6\mathrm{Im}a_k + \pi \delta \sum_{j=1}^3 \bigg(b_j+c_j+c_{j+3}+\frac{d}{3}\bigg)\big(\mathrm{Im}a_j+\mathrm{Im}a_{j+3}\big)\\
    < &-\frac{\pi Q d}{3}\sum_{k=1}^6\mathrm{Im}a_k + \pi \delta \max_{j\in\{1,2,3\}} \bigg\{b_j+c_j+c_{j+3}+\frac{d}{3}\bigg\} \sum_{k=1}^6\mathrm{Im}a_k\\
    <&   -\frac{d}{6} Q (\boldsymbol \pi\cdot \mathrm{Im}\boldsymbol a).
\end{split}
\end{equation}
Putting (\ref{b-6jbound}, (\ref{exponent}), (\ref{first}) and (\ref{second}) together, we have  the result  with   $C_{b,\boldsymbol \theta}$ as in (\ref{Cbtheta}) and $\epsilon_{\boldsymbol \theta}=\frac{d}{6}.$
\end{proof}

\begin{proof}[Proof of Theorem \ref{angleconv}]
Let $\boldsymbol \theta$ be any angle structure of $(M,\mathcal T).$ For each $\Delta\in T$ with edges $(e_1,\dots,e_6),$ we let $\boldsymbol \theta_\Delta=(\theta_{(\Delta,e_1)},\dots,\theta_{(\Delta,e_6)})$ be the six dihedral angles of $\Delta$ assigned by $\boldsymbol \theta;$ and for each $\boldsymbol a \in \big(\frac{Q}{2}+\mathbf i\mathbb R_{>0}\big)^E,$ we let $\mathrm{Im}\boldsymbol a_\Delta =(\mathrm{Im}a_{e_1},\dots, \mathrm{Im}a_{e_6}).$ Then we have 
\begin{equation}\label{anglesum}
2\pi \sum_{e\in E} \mathrm{Im}a_e=\sum_{e\in E}\bigg(\sum_{\Delta\sim e}\theta_{(\Delta,e)}\bigg)\mathrm{Im}a_e  = \sum_{\Delta\in T} \boldsymbol \theta_\Delta\cdot  \mathrm{Im} \boldsymbol a_\Delta.
\end{equation}
We also observe that  for each $\boldsymbol a = \big(\frac{Q}{2}+\mathbf i\frac{l_e}{2\pi b}\big)_{e\in E}\in \big(\frac{Q}{2}+\mathbf i\mathbb R_{>0}\big)^E$ and each $e\in E,$
\begin{equation}\label{ebound}
    |e|_{\boldsymbol a}=4\sinh(l_e)\sinh \Big(\frac{l_e}{b^2}\Big)< e^{2\pi Q \mathrm{Im}a_e}.
\end{equation}
Then we have 
\begin{equation*}
\begin{split}
    \big|\mathrm{TV}_b(\mathcal M,\mathcal T)\big| &< \int_{\big(\frac{Q}{2}+\mathbf i\mathbb R_{>0}\big)^E}\bigg| e^{2\pi Q\sum_{e\in E}\mathrm{Im}a_e} \prod_{\Delta\in T}|\Delta|_{\boldsymbol a}\bigg|d\mathrm{Im}\boldsymbol a\\
    &=\int_{\big(\frac{Q}{2}+\mathbf i\mathbb R_{>0}\big)^E}\prod_{\Delta\in T}\bigg| e^{ Q(\boldsymbol \theta_\Delta\cdot \mathrm{Im}\boldsymbol a_\Delta)} |\Delta|_{\boldsymbol a}\bigg|d\mathrm{Im}\boldsymbol a\\
    &< \int_{\big(\frac{Q}{2}+\mathbf i\mathbb R_{>0}\big)^E} \prod_{\Delta\in T}\bigg|C_{b,\boldsymbol \theta_\Delta}(1+\boldsymbol 1\cdot \mathrm{Im}\boldsymbol a_\Delta)e^{-\epsilon_{\boldsymbol \theta_\Delta}Q(\boldsymbol \pi\cdot \mathrm{Im}\boldsymbol a_\Delta)}\bigg|d\mathrm{Im}\boldsymbol a< +\infty,
\end{split}
\end{equation*}
where the first equality comes from (\ref{ebound}), the equality comes from (\ref{anglesum}), the second inequality comes from Proposition \ref{cor:6jbound} with $C_{b,\boldsymbol \theta_\Delta}$ and $\epsilon_{\boldsymbol \theta_\Delta}$ the constants therein, and the last inequality comes from that for each $\Delta \in T,$ $C_{b,\boldsymbol \theta_\Delta}$ is a constant independent of $\mathrm{Im}\boldsymbol a,$ $1+\boldsymbol 1\cdot \mathrm{Im}\boldsymbol a_\Delta$ is a linear function in $\mathrm{Im}\boldsymbol a,$ and $e^{-\epsilon_{\boldsymbol \theta_\Delta}Q(\boldsymbol \pi\cdot \mathrm{Im}\boldsymbol a_\Delta)}$ is a function that decays exponentially as $|\mathrm{Im}\boldsymbol a|$ tend to $+\infty.$
\end{proof}

\begin{proof}[Proof of Theorem \ref{Converge}]
 Combining   Proposition \ref{KA} and Theorem \ref{angleconv} gives  Theorem \ref{Converge}. 
\end{proof}


\section{Asymptotics of $\mathrm U_{q\tilde q}\mathfrak{sl}(2;\mathbb R)$ Turaev-Viro invariants}\label{sec:manifold}

Let $M$ be a hyperbolic $3$-manifold with a totally geodesic boundary with an ideal triangulation $\mathcal T.$ Let $E$ and $T$ respectively be the sets of edges and tetrahedra of $\mathcal T.$ For $e\in E,$ let $\alpha_e=\pi b a_e-\frac{\pi b^2}{2}.$ For $\Delta\in T$ with edges $e_1,\dots,e_6,$ let $\boldsymbol \alpha_\Delta=(\alpha_{e_1},\dots,\alpha_{e_6})$ and let $\xi_\Delta=\pi b u_\Delta.$ Let $U_b\big(\boldsymbol \alpha_\Delta,\xi_\Delta\big)=U_{\boldsymbol \alpha_\Delta,b}(\xi_\Delta)$
be as defined in (\ref{6jint}) and let $U\big(\boldsymbol \alpha_\Delta,\xi_\Delta\big)=U_{\boldsymbol \alpha_\Delta}(\xi_\Delta)$
be as defined in (\ref{U}).

Then by (\ref{6jint}), writing $\sinh(\frac{l_e}{b^2})=\frac{1}{2}\big(\exp(\frac{l_e}{b^2})-\exp(-\frac{l_e}{b^2})\big)$, we have
\begin{equation}\label{TVint}
\begin{split}
\mathrm{TV}_b(M,\mathcal T)= \frac{(-2\mathbf i)^{|E|}}{(\pi b )^{|E|+|T|}}\sum_{\boldsymbol \lambda\in\{-1,1\}^E}
\int_{\Gamma}\bigg(\prod_{e\in E}\lambda_e\sinh l_e\bigg)\exp\Bigg(\frac{\mathcal W_b^{\boldsymbol \lambda}(\boldsymbol \alpha, \boldsymbol \xi)}{2\pi \mathbf i b^2}\Bigg)d \boldsymbol \alpha  d \boldsymbol \xi,
\end{split}
\end{equation}
where the contour 
$$\Gamma=  \bigg\{ (\boldsymbol \alpha,\boldsymbol \xi)\in  \Big(\frac{\pi}{2}+\mathbf i{\mathbb R_{>0}}\Big)^E \times \mathbb C^T   \ \bigg|\ \xi_\Delta\in\Gamma_{\boldsymbol\alpha_\Delta} \text{ for all }\Delta\text{ in }T\bigg\}$$
and $\Gamma_{\boldsymbol\alpha_\Delta}$ is the contour of integral in (\ref{6jint}) that 
 will be specified later,  $d\boldsymbol \alpha=\prod _{e\in E}d\alpha_e,$ $d\boldsymbol \xi=\prod _{\Delta\in T}d\xi_\Delta,$ and
\begin{equation}\label{Wnu}
\begin{split}
\mathcal W_b^{\boldsymbol \lambda}(\boldsymbol \alpha, \boldsymbol \xi)=&2\pi \sum_{e\in E}\lambda_e(2\alpha_e-\pi)+\sum_{\Delta\in T}U_b\big(\boldsymbol \alpha_\Delta,\xi_\Delta\big)\\
=&\mathcal W^{\boldsymbol \lambda}(\boldsymbol \alpha, \boldsymbol \xi)+\kappa(\boldsymbol \alpha,\boldsymbol \xi)b^2+\nu_b(\boldsymbol \alpha,\boldsymbol \xi)b^4
\end{split}
\end{equation}
with
$$\mathcal W^{\boldsymbol \lambda}(\boldsymbol \alpha, \boldsymbol \xi)=2\pi \sum_{e\in E}\lambda_e(2\alpha_e-\pi)+\sum_{\Delta\in T}U\big(\boldsymbol \alpha_\Delta,\xi_\Delta\big),$$

 $$\kappa(\boldsymbol\alpha,\boldsymbol\xi)=\sum_{\Delta\in T}\kappa_{\boldsymbol \alpha_\Delta}(\xi_\Delta)$$
 with $\kappa_{\boldsymbol\alpha}$ computed in (\ref{kappa}),  and $$\nu_b(\boldsymbol\alpha,\boldsymbol\xi)=\sum_{\Delta\in T}\nu_{\boldsymbol \alpha_\Delta,b}(\xi_\Delta)$$  with  $\nu_{\boldsymbol\alpha,b}$  as in (\ref{6jint}). 
To simplify the notations, we write $\mathcal W_b\doteq\mathcal W_b^{\boldsymbol \lambda}$ and  $\mathcal W\doteq \mathcal W^{\boldsymbol \lambda}$ when $\boldsymbol \lambda=(1,\dots,1).$

By (\ref{6jint}) again, we can also write 
\begin{equation}\label{TVint2}
\begin{split}
\mathrm{TV}_b(M,\mathcal T)= \frac{(-\mathbf i)^{|E|}}{(\pi b )^{|E|+|T|}}\sum_{\boldsymbol \mu\in\{-1,1\}^E}\sum_{\boldsymbol \lambda\in\{-1,1\}^E}\bigg(\prod_{e\in E}\mu_e\lambda_e\bigg)
\int_{\Gamma}\exp\Bigg(\frac{\mathcal V_b^{\boldsymbol\mu\boldsymbol \lambda}(\boldsymbol \alpha, \boldsymbol \xi)}{2\pi \mathbf i b^2}\Bigg)d \boldsymbol \alpha  d \boldsymbol \xi,
\end{split}
\end{equation}
where 
\begin{equation}\label{Vnu}
\begin{split}
\mathcal V_b^{\boldsymbol \mu\boldsymbol \lambda}(\boldsymbol \alpha, \boldsymbol \xi)=&2\pi \sum_{e\in E}(\mu_eb^2+\lambda_e)(2\alpha_e-\pi)+\sum_{\Delta\in T}U_b\big(\boldsymbol \alpha_\Delta,\xi_\Delta\big)\\
=&\mathcal V^{\boldsymbol\mu\boldsymbol \lambda}(\boldsymbol \alpha, \boldsymbol \xi)+\kappa(\boldsymbol \alpha,\boldsymbol \xi)b^2+\nu_b(\boldsymbol \alpha,\boldsymbol \xi)b^4
\end{split}
\end{equation}
with
$$\mathcal V^{\boldsymbol\mu\boldsymbol \lambda}(\boldsymbol \alpha, \boldsymbol \xi)=2\pi \sum_{e\in E}(\mu_eb^2+\lambda_e)(2\alpha_e-\pi)+\sum_{\Delta\in T}U\big(\boldsymbol \alpha_\Delta,\xi_\Delta\big).$$ 
 To simplify the notations, we write $\mathcal V_b\doteq\mathcal V_b^{\boldsymbol\mu\boldsymbol \lambda}$ and  $\mathcal V\doteq\mathcal V^{\boldsymbol\mu\boldsymbol \lambda}$ when $\boldsymbol\mu=\boldsymbol \lambda=(1,\dots,1).$

To establish the asymptotics in Theorem~\ref{VC}, we apply the  saddle point approximations to the integral in~\eqref{TVint2}; and the style of the argument is similar to those in Section~\ref{sec:tetrahedron}. A major difference is that the integral under consideration here is with multi-variables $\xi_\Delta$'s and the $l_e$'s. To establish the desired properties of the relevant functions, we further employ geometric tools including Luo-Yang's result~\cite{LY} on the relationship between angle structures and hyperbolic polyhedral metrics, the Murakami-Yano volume formula in the form of~\cite{BY} and the Wong-Yang formula~\cite{WY3} of the adjoint twisted Reidemeister torsions. 
\\

\subsection{Properties of relevant functions}\label{subsec:property4}

Let $\mathcal T$ be a Kojima ideal triangulation of $M$  with edge lengths $\boldsymbol l^*=(l_e^*),$ which is a generalized hyperbolic polyhedral metric with all the cone angles $2\pi.$ (See subsection \ref{HPM}.) 

For $\Delta\in T$ with edges $e_1,\dots,e_6,$ let 
 $\boldsymbol \alpha^*_{\Delta}=\Big(\frac{\pi}{2}+ \mathbf i\frac{l^*_{e_1}}{2},\dots,\frac{\pi}{2}+ \mathbf i\frac{l^*_{e_6}}{2}\Big).$ 
Guaranteed by Proposition \ref{critical2}, 
\begin{enumerate}[(1)]
\item if $(l^*_{e_1},\dots, l^*_{e_6})\in\mathcal L,$ then we let $\xi^*_{\Delta}$ be the unique critical point of $U_{\boldsymbol \alpha^*_{\Delta}}$ in $D;$ and 
\item
 if  $(l^*_{e_1},\dots, l^*_{e_6})\in \partial \mathcal L,$ then we let $\xi^*_{\Delta}$ be the unique critical point of $U_{\boldsymbol \alpha^*_{\Delta}}$ on $\partial D.$ 
 \end{enumerate}  
By Proposition \ref{critical2} (2), $\mathrm{Re}\xi^*_{\Delta}=2\pi$ in case (2). 
\\

 Let  $\overline D=\{ \xi\in\mathbb C\ |\ \frac{3\pi}{2}\leqslant \mathrm{Re}\xi\leqslant 2\pi\}.$ Then we have the following

\begin{proposition}\label{cpcv5} 
The function 
$$\mathcal W(\boldsymbol \alpha, \boldsymbol \xi)=2\pi\sum_{e\in E} (2\alpha_e-\pi)+\sum_{\Delta\in T}U\big(\boldsymbol \alpha_\Delta,\xi_\Delta\big)$$  has a critical point 
$$\boldsymbol z^*=\bigg(\Big(\frac{\pi}{2}+\mathbf i\frac{l^*_{e}}{2}\Big)_{e\in E},\big(\xi^*_{\Delta}\big)_{\Delta\in T}\bigg)$$
in $\big(\frac{\pi}{2}+\mathbf i{\mathbb R_{>0}}\big)^E\times \overline D^T$
with the critical value
$$\mathcal W(\boldsymbol z^*)=-2\mathbf i\mathrm{Vol}(M).$$ 
\end{proposition}

\begin{proof}  
Let $\boldsymbol z^*=(\boldsymbol \alpha^* ,\boldsymbol \xi^* ),$ where $\boldsymbol \alpha^* =\Big(\frac{\pi}{2}+\mathbf i\frac{l^*_{ e}}{2}\Big)_{e\in E}$ and $\boldsymbol \xi^* =\big(\xi^*_{ \Delta}\big)_{\Delta\in T}.$ For $\Delta\in T$ with edges $e_1,\dots,e_6,$ let $l^*_{ \Delta}=(l^*_{ e_1},\dots,l^*_{ e_6}).$ 

We first have
\begin{equation}\label{xi=0b}
\frac{\partial  \mathcal W}{\partial \xi_\Delta}\Big|_{(\boldsymbol \alpha^* ,\boldsymbol \xi^* )}=\frac{\partial U_{\boldsymbol \alpha^*_{ \Delta}}(\xi_\Delta)}{\partial \xi_\Delta}\Big|_{\xi^*_{ \Delta}}=0.
\end{equation}
Next,  let $W(\boldsymbol \alpha_\Delta)=U(\boldsymbol \alpha_\Delta,\xi^*_\Delta)=U_{\boldsymbol \alpha_\Delta}(\xi^*_\Delta)$ be the function defined in (\ref{W}). Then for any edge $e$ of $\Delta,$ 
\begin{equation*}
\begin{split}
\frac{\partial W(\boldsymbol \alpha_\Delta)}{\partial \alpha_e}\Big|_{\boldsymbol \alpha^*_{ \Delta}}=&\frac{\partial U(\boldsymbol \alpha_\Delta,\xi_\Delta)}{\partial \alpha_e}\Big|_{(\boldsymbol \alpha^*_{ \Delta},\xi^*_{ \Delta})}+\frac{\partial U(\boldsymbol \alpha_\Delta,\xi_\Delta)}{\partial \xi_\Delta}\Big|_{(\boldsymbol \alpha^*_{ \Delta},\xi^*_{ \Delta})}\cdot\frac{\partial \xi^*_\Delta}{\partial \alpha_e}\Big|_{\boldsymbol \alpha^*_{ \Delta}}\\
=&\frac{\partial U(\boldsymbol \alpha_\Delta,\xi_\Delta)}{\partial \alpha_e}\Big|_{(\boldsymbol \alpha^*_{ \Delta},\xi^*_{ \Delta})}.
\end{split}
\end{equation*}
As a consequence, by \eqref{dCovdlin} and \eqref{dCovdlout} we have
\begin{equation*}
\frac{\partial U(\boldsymbol \alpha_\Delta,\xi_\Delta)}{\partial \alpha_e}\Big|_{(\boldsymbol \alpha^*_{ \Delta},\xi^*_{ \Delta})}=\frac{\partial W(\boldsymbol \alpha_\Delta)}{\partial \alpha_e}\Big|_{\boldsymbol \alpha^*_{ \Delta}}=-4\frac{\partial \mathrm{Cov}(\boldsymbol l_\Delta)}{\partial l_e}\Big|_{\boldsymbol l^*_{ \Delta}}=-2\widetilde \theta_{(\Delta,e)},
\end{equation*}
where
$\widetilde \theta_{(\Delta,e)}$ is the extended dihedral angle of the $\boldsymbol l^*_{ \Delta}$ at the edge $e.$ Then for each $e\in E,$ 
\begin{equation}\label{alpha=0b}
\frac{\partial  \mathcal W}{\partial \alpha_e}\Big|_{(\boldsymbol \alpha^* ,\boldsymbol \xi^* )}=4\pi+ \sum_{\Delta\in T}\frac{\partial U(\boldsymbol \alpha_\Delta,\xi_\Delta)}{\partial \alpha_e}\Big|_{(\boldsymbol \alpha^*_{ \Delta},\xi^*_{ \Delta})}=2\Big(2\pi-\sum_{\Delta\sim e}\widetilde\theta_{(\Delta,e)}\Big)=0,
\end{equation}
where the last summation is over all the corners around $e,$ and the last equality comes from the fact that the sum of the dihedral angles around an edge equals is the cone angle at this edge, which equals $2\pi.$ By (\ref{xi=0b}) and (\ref{alpha=0b}), $(\boldsymbol \alpha^* ,\boldsymbol \xi^* )$ is a critical point of $ \mathcal W.$

 Finally, for the critical value, by Proposition \ref{critical2}, we have
\begin{equation*}
\begin{split}
\mathcal W(\boldsymbol \alpha^*,\boldsymbol \xi^*)=&2\pi \mathbf i \sum_{e\in E} l^*_e-2\mathbf i\sum_{\Delta\in T}\Big(\mathrm{Vol}(\Delta)+\sum_{e\sim \Delta}\frac{\widetilde\theta_{(\Delta, e)}l^*_e}{2}\Big)\\
=& -2\mathbf i\sum_{\Delta\in T}\mathrm{Vol}(\Delta)+\mathbf i\sum_{e\in E} \Big(2\pi -\sum_{\Delta\sim e}\widetilde\theta_{(\Delta,e)}\Big)l^*_e\\
=& -2\mathbf i\mathrm{Vol}(M),
\end{split}
\end{equation*}
where the last summation in the first row is over all the edges $e$ in $\Delta,$ and the last summation in the second row is over all the corners around $e.$ 
\end{proof}

By Proposition \ref{KC}, there is a one-parameter family of hyperbolic polyhedral metrics $\boldsymbol l^*_s$ converging to $\boldsymbol l^*.$  For $e\in E,$ let $\theta_{s,e}$ be the cone angle of $\boldsymbol l^*_s$ at $e.$ For $\Delta\in T$ with edges $e_1,\dots,e_6,$ let 
 $\boldsymbol \alpha^*_{s,\Delta}=\Big(\frac{\pi}{2}+ \mathbf i\frac{l^*_{s, e_1}}{2},\dots,\frac{\pi}{2}+ \mathbf i\frac{l^*_{s, e_6}}{2}\Big).$ 
Guaranteed by Proposition \ref{critical2} (1), we let $\xi^*_{s,\Delta}=\xi^*(\boldsymbol \alpha^*_{s,\Delta})$ be the unique critical point of $U_{\boldsymbol \alpha^*_{s,\Delta}}$  in $D.$

\begin{proposition}\label{cpcv6} 
The function 
$$\mathcal W_s(\boldsymbol \alpha, \boldsymbol \xi)=\sum_{e\in E} \theta_{s,e}(2\alpha_e-\pi)+\sum_{\Delta\in T}U\big(\boldsymbol \alpha_\Delta,\xi_\Delta\big)$$  has a critical point 
$$\boldsymbol z_s^*=\bigg(\Big(\frac{\pi}{2}+\mathbf i\frac{l^*_{s,e}}{2}\Big)_{e\in E},\big(\xi^*_{s,\Delta}\big)_{\Delta\in T}\bigg)$$
in $\big(\frac{\pi}{2}+\mathbf i{\mathbb R_{>0}}\big)^E\times D^T.$
\end{proposition}

\begin{proof}  The proof follows verbatim the proof of Proposition \ref{cpcv5}. Namely, let $\boldsymbol z_s^*=(\boldsymbol \alpha_s^* ,\boldsymbol \xi_s^* ),$ where $\boldsymbol \alpha_s^* =\Big(\frac{\pi}{2}+\mathbf i\frac{l^*_{s, e}}{2}\Big)_{e\in E}$ and $\boldsymbol \xi_s^* =\big(\xi^*_{s, \Delta}\big)_{\Delta\in T}.$ For $\Delta\in T$ with edges $e_1,\dots,e_6,$ let $l^*_{s, \Delta}=(l^*_{s, e_1},\dots,l^*_{s, e_6}).$ 

We first have
\begin{equation}\label{xi=0b3}
\frac{\partial  \mathcal W_s}{\partial \xi_\Delta}\Big|_{(\boldsymbol \alpha_s^* ,\boldsymbol \xi_s^* )}=\frac{\partial U_{\boldsymbol \alpha^*_{s,\Delta}}(\xi_\Delta)}{\partial \xi_\Delta}\Big|_{\xi^*_{s, \Delta}}=0.
\end{equation}

Next, let $W(\boldsymbol \alpha_\Delta)=U(\boldsymbol \alpha_\Delta,\xi^*_\Delta)=U_{\boldsymbol \alpha_\Delta}(\xi^*_\Delta)$ be the function defined in (\ref{W}). Then for any edge $e$ of $\Delta,$ 
\begin{equation*}
\begin{split}
\frac{\partial W(\boldsymbol \alpha_\Delta)}{\partial \alpha_e}\Big|_{\boldsymbol \alpha^*_{s,\Delta}}=&\frac{\partial U(\boldsymbol \alpha_\Delta,\xi_\Delta)}{\partial \alpha_e}\Big|_{(\boldsymbol \alpha^*_{s,\Delta},\xi^*_{s,\Delta})}+\frac{\partial U(\boldsymbol \alpha_\Delta,\xi_\Delta)}{\partial \xi_\Delta}\Big|_{(\boldsymbol \alpha^*_{s,\Delta},\xi^*_{s,\Delta})}\cdot\frac{\partial \xi^*_\Delta}{\partial \alpha_e}\Big|_{\boldsymbol \alpha^*_{s, \Delta}}\\
=&\frac{\partial U(\boldsymbol \alpha_\Delta,\xi_\Delta)}{\partial \alpha_e}\Big|_{(\boldsymbol \alpha^*_{s,\Delta},\xi^*_{s, \Delta})}.
\end{split}
\end{equation*}
As a consequence, and by Proposition \ref{critical2} (1),
\begin{equation}\label{pUpa2}
\frac{\partial U(\boldsymbol \alpha_\Delta,\xi_\Delta)}{\partial \alpha_e}\Big|_{(\boldsymbol \alpha^*_{s, \Delta},\xi^*_{s, \Delta})}=\frac{\partial W(\boldsymbol \alpha_\Delta)}{\partial \alpha_e}\Big|_{\boldsymbol \alpha^*_{s,\Delta}}=-4\frac{\partial \mathrm{Cov}(\boldsymbol l_\Delta)}{\partial l_e}\Big|_{\boldsymbol l^*_{s, \Delta}}=-2 \theta_{s, (\Delta,e)},
\end{equation}
where
$\theta_{s, (\Delta,e)}$ is the dihedral angle of the $\boldsymbol l^*_{s, \Delta}$ at the edge $e.$ Then for each $e\in E,$ 
\begin{equation}\label{alpha=0b3}
\frac{\partial  \mathcal W_s}{\partial \alpha_e}\Big|_{(\boldsymbol \alpha_s^* ,\boldsymbol \xi_s^* )}=2\theta_{s,e} + \sum_{\Delta\in T}\frac{\partial U(\boldsymbol \alpha_\Delta,\xi_\Delta)}{\partial \alpha_e}\Big|_{(\boldsymbol \alpha^*_{s, \Delta},\xi^*_{s, \Delta})}=2\Big(\theta_{s,e}-\sum_{\Delta\sim e}\theta_{s, (\Delta,e)}\Big)=0,
\end{equation}
where the last summation is over all the corners around $e,$ and the last equality comes from the fact that the sum of the dihedral angles around an edge equals is the cone angle at this edge, which equals $\theta_{s,e}.$ By (\ref{xi=0b3}) and (\ref{alpha=0b3}), $(\boldsymbol \alpha_s^* ,\boldsymbol \xi_s^* )$ is a critical point of $ \mathcal W_s.$  
\end{proof}

\begin{proposition}\label{HessTor2} 
\begin{equation}\label{tor}
\frac{\det\big(-\mathrm{Hess}\mathcal  W(\boldsymbol z^*)\big)}{\exp\big(\frac{\kappa(\boldsymbol z^*)}{\pi \mathbf i }\big) \prod_{e\in E} \sinh^2 l^*_e}=\pm 2^{5|E|+|T|} \mathrm{Tor}(DM, \mathrm{Ad}_{\rho_{DM}}),
\end{equation}
where $\mathrm{Tor}(DM, \mathrm{Ad}_{\rho_{DM}})$ is the Reidemeister torsion of the double $DM$ of $M$ twisted by the adjoint action of  the double $\rho_{DM}$ of the holonomy representation of the hyperbolic structure on $M.$ In particular, 
$$\det\big(-\mathrm{Hess}\mathcal  W(\boldsymbol z^*)\big)\neq 0.$$
\end{proposition}

\begin{proof} Guaranteed by Proposition \ref{KC}, let $\{\boldsymbol l^*_s\}_{s\in (0,s_0)}$ be a one-parameter family of  hyperbolic polyhedral metrics on $M$ that converges to $\boldsymbol l^*.$ Let $N$ be the $3$-manifold obtained from  the double $DM$  of $M$ by removing the double of the edges of $\mathcal T,$ and for $s\in (0,s_0)$ let  $\rho_s$ be the holonomy representation of the hyperbolic cone metric on $N$ obtained by doubling $\boldsymbol l^*_s.$ Let $\boldsymbol m$ be the system of the meridians of a tubular neighborhood of the double of the edges, and let   $ \mathbb T_{(N,\boldsymbol m)}$ be the twisted adjoint Reidemeister torsion function of  the $\mathrm{PSL}(2,\mathbb C)$-characters of $\pi_1(N).$ 
Let $\boldsymbol z_s^*=\Big(\big(\frac{\pi}{2}+\mathbf i\frac{l^*_{s,e}}{2}\big)_{e\in E},\big(\xi^*_{s,\Delta}\big)_{\Delta\in T}\Big)$ be the critical point of the function $\mathcal W_s$ in Proposition \ref{cpcv6}. We will prove that 
\begin{equation}\label{tor2}
\frac{\det\big(-\mathrm{Hess}\mathcal  W(\boldsymbol z^*_s)\big)}{\exp\big(\frac{\kappa(\boldsymbol z^*_s)}{\pi \mathbf i }\big)}=\pm 2^{3|E|+|T|} \mathbb{T}_{(N,\boldsymbol m)}([\rho_s])
\end{equation}
for all $s\in (0,s_0).$ 
Assuming this, then as $\det(-\mathrm{Hess}\mathcal W)$ is continuous in $(\boldsymbol\alpha,\boldsymbol\xi),$ $\mathbb T_{(N,\boldsymbol m)}$ by Theorem \ref{funT} (i) and Remark \ref{rm} is analytic and hence continuous in a neighborhood of  the character of the holonomy representation $\rho$ of the hyperbolic cone metric on $N$ (with all the cone angles $2\pi$) obtained by doubling  the generalizaed hyperbolic polyhedral metric $\boldsymbol l^*$ of $M,$ and $\boldsymbol z^*_s$ approaches $\boldsymbol z^*$ as $s$ tends to $0,$ we have 
\begin{equation*}
\frac{\det\big(-\mathrm{Hess}\mathcal  W(\boldsymbol z^*)\big)}{\exp\big(\frac{\kappa(\boldsymbol z^*)}{\pi \mathbf i }\big)}=\pm 2^{3|E|+|T|} \mathbb{T}_{(N,\boldsymbol m)}([\rho]).
\end{equation*}
Then by Theorem \ref{funT} (iii), as $DM$ can be considered as obtained from $N$ by doing a hyperbolic Dehn filling along the meridians $\boldsymbol m,$ and the double of the edges form the system of curves $\boldsymbol \gamma$ with the logarithmic holonomy $\mathrm H(\gamma_e)= 2l_e$ for each $e\in E,$  we have
\begin{equation*}
\begin{split}
\frac{\det\big(-\mathrm{Hess}\mathcal  W(\boldsymbol z^*)\big)}{\exp\big(\frac{\kappa(\boldsymbol z^*)}{\pi \mathbf i }\big) \prod_{e\in E} \sinh^2 l^*_e}=&\pm 2^{3|E|+|T|} \mathbb{T}_{(N,\boldsymbol m)}([\rho]) \prod_{e\in E}\frac{1}{\sinh^2l^*_e}\\
=&\pm 2^{5|E|+|T|} \mathrm{Tor}(DM, \mathrm{Ad}_{\rho_{DM}}).
\end{split}
\end{equation*}
Since $\pm \mathrm{Tor}(DM, \mathrm{Ad}_{\rho_{DM}})\neq 0,$ we have $\det\big(-\mathrm{Hess}\mathcal  W(\boldsymbol z^*)\big)\neq 0,$ which completes the proof. 
\\

We are left to prove (\ref{tor2}). We use the argument adapted from that of \cite[Lemma 4.3]{WY}. Namely, for $\Delta\in T$ and $e\in E$ we denote by $\Delta\sim e$ and $e\sim \Delta$ if the tetrahedron $\Delta$ intersects the edge $e;$  and for $\{e_i,e_j\}\subset E$ we denote by $\Delta \sim e_i,e_j$ if $\Delta$ intersects both $e_i$ and $e_j.$ Let $(\theta_e)_{e\in E}$ be the cone angle functions in terms of the edge lengths of $M.$  Then we claim:
\begin{enumerate}[(1)]
\item For $\Delta\in T,$
$$\frac{\partial^2 \mathcal W(\boldsymbol\alpha,\boldsymbol\xi)}{\partial \xi_\Delta^2}\bigg|_{\boldsymbol z^*_s}=\frac{\partial^2 U(\boldsymbol\alpha_\Delta,\xi_\Delta)}{\partial \xi_\Delta^2}\bigg|_{(\boldsymbol\alpha^*_{s,\Delta},\xi^*_{s,\Delta})}.$$

\item For $\{\Delta_1,\Delta_2\}\subset T,$
$$\frac{\partial^2 \mathcal W(\boldsymbol\alpha,\boldsymbol\xi)}{\partial \xi_{\Delta_1}\partial \xi_{\Delta_2}}\bigg|_{\boldsymbol z^*_s}=0.$$

\item For $e\in E$ and $\Delta\in T,$
$$\frac{\partial^2 \mathcal W(\boldsymbol\alpha,\boldsymbol\xi)}{\partial \alpha_e\partial \xi_\Delta}\bigg|_{\boldsymbol z^*_s}=-\frac{\partial^2 U(\boldsymbol\alpha_\Delta,\xi_\Delta)}{\partial \xi_\Delta^2}\bigg|_{(\boldsymbol\alpha^*_{s,\Delta},\xi^*_{s,\Delta})}\frac{\partial \xi^*(\boldsymbol\alpha_\Delta)}{\partial \alpha_e}\bigg|_{\boldsymbol\alpha^*_{s,\Delta}}.$$

\item For $e\in E,$ 
$$\frac{\partial^2 \mathcal W(\boldsymbol\alpha,\boldsymbol\xi)}{\partial \alpha_e^2}\bigg|_{\boldsymbol z^*_s}=  4\mathbf i \frac{\partial \theta_e}{\partial l_e}\bigg|_{\boldsymbol l^*_s}+\sum_{\Delta\sim e}\frac{\partial^2 U(\boldsymbol\alpha_\Delta,\xi_\Delta)}{\partial \xi_\Delta^2}\bigg|_{(\boldsymbol\alpha^*_{s,\Delta},\xi^*_{s,\Delta})} \bigg(\frac{\partial \xi^*(\boldsymbol\alpha_\Delta)}{\partial \alpha_e}\bigg|_{\boldsymbol\alpha^*_{s,\Delta}}\bigg)^2.$$

\item For $\{e_i,e_j\}\subset E,$ 
$$\frac{\partial^2 \mathcal W(\boldsymbol\alpha,\boldsymbol\xi)}{\partial \alpha_{e_i}\partial \alpha_{e_j}}\bigg|_{\boldsymbol z^*_s}= 4\mathbf i \frac{\partial \theta_{e_i}}{\partial l_{e_j}}\bigg|_{\boldsymbol l^*_s}+\sum_{\Delta\sim e_i,e_j}\frac{\partial^2 U(\boldsymbol\alpha_\Delta,\xi_\Delta)}{\partial \xi_\Delta^2}\bigg|_{(\boldsymbol\alpha^*_{s,\Delta},\xi^*_{s,\Delta})} \frac{\partial \xi^*(\boldsymbol\alpha_\Delta)}{\partial \alpha_{e_i}}\bigg|_{\boldsymbol\alpha^*_{s,\Delta}}\frac{\partial \xi^*(\boldsymbol\alpha_\Delta)}{\partial \alpha_{e_j}}\bigg|_{\boldsymbol\alpha^*_{s,\Delta}}.$$
\end{enumerate}
We will prove these claims (1) -- (5) at the end.  Assuming them, we have
\begin{equation}\label{congurent}
\mathrm{Hess}\mathcal W(\boldsymbol z_s^*)=S \cdot D\cdot S ^T,
\end{equation}
where $D$ is a block matrix with the left-top block the $|E|\times|E|$ matrix 
$$\bigg(4\mathbf i  \frac{\partial \theta_{e_i}}{\partial l_{e_j}}\bigg|_{\boldsymbol l^*_s}\bigg)_{e_i,e_j\in E},$$
the right-top and the left-bottom blocks consisting of $0$'s, and the right-bottom block the $|T|\times|T|$ diagonal matrix with the diagonal entries
$\frac{\partial^2 U}{\partial \xi_\Delta^2}\Big|_{(\boldsymbol\alpha^*_{s,\Delta},\xi^*_{s,\Delta})};$ and  $S $ is a block matrix with the left-top and the right-bottom blocks respectively the $|E|\times|E|$ and $|T|\times|T|$ identity matrices, the left-bottom block consisting of $0$'s and the right-top block the $|E|\times |T|$ matrix with entries $s_{e,\Delta},$ $e\in E$ and $\Delta\in T,$ given by
$$s_{e,\Delta}=-\frac{\partial \xi^*(\boldsymbol\alpha_\Delta)}{\partial \alpha_{e}}\bigg|_{\boldsymbol\alpha^*_{s,\Delta}}$$
if $\Delta\sim e,$ and $s_{e,\Delta}=0$
if otherwise.
Then 
\begin{equation}\label{detD}
\begin{split}
\det D=&\det\bigg( 4\mathbf i  \frac{\partial \theta_{e_i}}{\partial l_{e_j}}\bigg|_{\boldsymbol l^*_s}\bigg)_{e_i,e_j\in E} \prod _{\Delta\in T} \frac{\partial ^2U}{\partial \xi_\Delta^2}\bigg|_{(\boldsymbol\alpha^*_{s,\Delta},\xi^*_{s,\Delta})}\\
=&(-1)^{\frac{|E|}{2}}2^{2|E|}\det\bigg(\frac{\partial \theta_i}{\partial l_j}\bigg|_{\boldsymbol l^*_s}\bigg)_{e_i,e_j\in E} \prod _{\Delta\in T} \frac{\partial ^2U}{\partial \xi_\Delta^2}\bigg|_{(\boldsymbol\alpha^*_{s,\Delta},\xi^*_{s,\Delta})};
\end{split}
\end{equation}
and since $S$ is upper-triangular with all diagonal entries equal to $1,$ 
\begin{equation}\label{detA}
\det S=1.
\end{equation}
Putting (\ref{congurent}), (\ref{detD}) and (\ref{detA}) together,  we have
$$\det\big(-\mathrm{Hess}\mathcal  W(\boldsymbol z^*_s)\big)= (-1)^{\frac{|E|}{2}}2^{2|E|}\mathrm{det}\bigg(\frac{\partial \theta_{e_i}}{\partial l_{e_j}}\bigg|_{\boldsymbol l_s^*}\bigg)_{e_i,e_j\in E}\prod_{\Delta\in T}\bigg(-\frac{\partial^2 U_{\boldsymbol \alpha_\Delta}}{\partial \xi_\Delta^2}\bigg|_{\xi^*_{s,\Delta}}\bigg).$$
Then by Proposition \ref{Hess} and  Theorem \ref{Torsion2},
\begin{equation*}
\begin{split}
\frac{\det\big(-\mathrm{Hess}\mathcal  W(\boldsymbol z_s^*)\big)}{\exp\big(\frac{\kappa(\boldsymbol z_s^*)}{\pi \mathbf i }\big)}=&(-1)^{\frac{|E|}{2}}2^{2|E|+4|T|}\mathrm{det}\bigg(\frac{\partial \theta_{e_i}}{\partial l_{e_j}}\bigg|_{\boldsymbol l_s^*}\bigg)\prod_{\Delta\in T}\sqrt{\det\mathrm{Gram}(\boldsymbol l^*_{s, \Delta})}\\
=&\pm 2^{3|E|+|T|}\mathbb{T}_{(N,\boldsymbol m)}([\rho_s]).
\end{split}
\end{equation*}
This proves (\ref{tor2}) under the assumption that claims (1) -- (5) hold. 
\\

Now we prove the claims (1) -- (5). Claims (1) and (2) follow  straightforwardly   from the definition of $\mathcal W.$ 
For (3), we have
\begin{equation}\label{WU}
\frac{\partial \mathcal W}{\partial \xi_\Delta}\bigg|_{(\boldsymbol\alpha,\boldsymbol\xi)}=\frac{\partial U}{\partial \xi_\Delta}\bigg|_{(\boldsymbol\alpha_\Delta,\xi_\Delta)}.
\end{equation}
Let $f(\boldsymbol\alpha_\Delta,\xi_\Delta)\doteq \frac{\partial U}{\partial \xi_\Delta}\Big|_{(\boldsymbol\alpha_\Delta,\xi_\Delta)}$
and $g(\boldsymbol\alpha_\Delta)\doteq f(\boldsymbol\alpha_\Delta,\xi^*(\boldsymbol\alpha_\Delta)).$
Then 
$$g(\boldsymbol\alpha_\Delta)=\frac{\partial U}{\partial \xi_\Delta}\bigg|_{(\boldsymbol\alpha_\Delta,\xi^*(\boldsymbol\alpha_\Delta))}=\frac{dU_{\boldsymbol\alpha_\Delta}}{d \xi_\Delta}\bigg|_{\xi^*(\boldsymbol\alpha_\Delta)}\equiv 0,$$
and hence
\begin{equation}\label{g=0}
\frac{\partial g}{\partial \alpha_{e}}\bigg|_{\boldsymbol\alpha_\Delta}=0
\end{equation}
for all $e\sim \Delta.$
On the other hand, for  $e \sim \Delta$ we have
\begin{equation}\label{+}
\begin{split}
\frac{\partial g}{\partial \alpha_{e}}\bigg|_{\boldsymbol\alpha_\Delta}=&\frac{\partial f}{\partial \alpha_{e}}\bigg|_{(\boldsymbol\alpha_\Delta,\xi^*(\boldsymbol\alpha_\Delta))}+\frac{\partial f}{\partial \xi_\Delta}\bigg|_{(\boldsymbol\alpha_\Delta,\xi^*(\boldsymbol\alpha_\Delta))}  \frac{\partial \xi^*(\boldsymbol\alpha_\Delta)}{\partial \alpha_e}\bigg|_{\boldsymbol\alpha_\Delta}\\
=&\frac{\partial^2 U}{\partial \alpha_e\partial \xi_\Delta}\bigg|_{(\boldsymbol\alpha_\Delta,\xi^*(\boldsymbol\alpha_\Delta))}+\frac{\partial^2 U}{\partial \xi_\Delta^2}\bigg|_{(\boldsymbol\alpha_\Delta,\xi^*(\boldsymbol\alpha_\Delta))} \frac{\partial \xi^*(\boldsymbol\alpha_\Delta)}{\partial \alpha_e}\bigg|_{\boldsymbol\alpha_\Delta}.
\end{split}
\end{equation}
Putting (\ref{g=0}) and (\ref{+}) together, we have
\begin{equation}\label{alphaxi}
\frac{\partial^2 U}{\partial \alpha_e\partial \xi_\Delta}\bigg|_{(\boldsymbol\alpha_\Delta,\xi^*(\boldsymbol\alpha_\Delta))}=-\frac{\partial^2 U}{\partial \xi_\Delta^2}\bigg|_{(\boldsymbol\alpha_\Delta,\xi^*(\boldsymbol\alpha_\Delta))} \frac{\partial \xi^*(\boldsymbol\alpha_\Delta)}{\partial \alpha_e}\bigg|_{\boldsymbol\alpha_\Delta},
\end{equation}
and (3) follows from (\ref{WU}) and (\ref{alphaxi}).

For (4) and (5), we have for $e\in E$ 
\begin{equation}\label{4.10}
\frac{\partial ^2\mathcal W}{\partial \alpha_{e}^2}\bigg|_{\boldsymbol z^*_s}=\sum_{\Delta\sim e}\frac{\partial ^2U}{\partial \alpha_{e}^2}\bigg|_{(\boldsymbol\alpha^*_{s,\Delta},\xi^*_{s,\Delta})},
\end{equation}
and for $\{e_i,e_j\}\subset E$ 
\begin{equation}\label{4.11}
\frac{\partial ^2\mathcal W}{\partial \alpha_{e_i}\partial \alpha_{e_j}}\bigg|_{\boldsymbol z^*_s}=\sum_{\Delta\sim e_i, e_j}\frac{\partial ^2U}{\partial \alpha_{e_i}\partial \alpha_{e_j}}\bigg|_{(\boldsymbol\alpha^*_{s,\Delta},\xi^*_{s,\Delta)}}.
\end{equation}
Let $W$ be the function defined in (\ref{W}). By the Chain Rule and (\ref{=0}), we have
$$\frac{\partial U}{\partial \xi_\Delta}\bigg|_{(\boldsymbol\alpha_\Delta,\xi^*(\boldsymbol\alpha_\Delta))}=\frac{d U_{\boldsymbol\alpha_\Delta}}{d \xi_\Delta}\bigg|_{\xi^*(\boldsymbol\alpha_\Delta)} =0,$$
and hence for $e\sim \Delta$ we have 
$$\frac{\partial W}{\partial \alpha_{e}}\bigg|_{\boldsymbol\alpha_\Delta}=\frac{\partial U}{\partial \alpha_{e}}\bigg|_{\boldsymbol\alpha_\Delta}+\frac{\partial U}{\partial \xi_\Delta}\bigg|_{(\boldsymbol\alpha_\Delta,\xi^*(\boldsymbol\alpha_\Delta))}\frac{\partial \xi^*(\boldsymbol\alpha_\Delta)}{\partial \alpha_{e}}\bigg|_{\boldsymbol\alpha_\Delta}=\frac{\partial U}{\partial \alpha_{e}}\bigg|_{\boldsymbol\alpha_\Delta}.$$
Then using the Chain Rule again, for  $e_i,e_j\sim \Delta$ we have
\begin{equation*}\label{3.12}
\begin{split}
\frac{\partial^2 W}{\partial \alpha_{e_i}\partial \alpha_{e_j}}\bigg|_{\boldsymbol\alpha_\Delta}
=\frac{\partial ^2U}{\partial \alpha_{e_i}\partial \alpha_{e_j}}\bigg|_{(\boldsymbol\alpha_\Delta,\xi^*(\boldsymbol\alpha_\Delta))}+\frac{\partial^2 U}{\partial \alpha_{e_j}\partial \xi_\Delta}\bigg|_{(\boldsymbol\alpha_\Delta,\xi^*(\boldsymbol\alpha_\Delta))}  \frac{\partial \xi^*(\boldsymbol\alpha_\Delta)}{\partial \alpha_{e_i}}\bigg|_{\boldsymbol\alpha_\Delta}.
\end{split}
\end{equation*}
Together with (\ref{alphaxi}), for $e_i,e_j\sim \Delta$ we have
\begin{equation}\label{4.13}
\begin{split}
\frac{\partial ^2U}{\partial \alpha_{e_i}\partial \alpha_{e_j}}\bigg|_{(\boldsymbol\alpha_\Delta,\xi^*(\boldsymbol\alpha_\Delta))}=&\frac{\partial^2 W}{\partial \alpha_{e_i}\partial \alpha_{e_j}}\bigg|_{\boldsymbol\alpha_\Delta}-\frac{\partial^2 U}{\partial \alpha_{e_j}\partial \xi_\Delta}\bigg|_{(\boldsymbol\alpha_\Delta,\xi^*(\boldsymbol\alpha_\Delta))}  \frac{\partial \xi^*(\boldsymbol\alpha_\Delta)}{\partial \alpha_{e_i}}\bigg|_{\boldsymbol\alpha_\Delta}\\
=&\frac{\partial^2 W}{\partial \alpha_{e_i}\partial \alpha_{e_j}}\bigg|_{\boldsymbol\alpha_\Delta}+\frac{\partial^2 U}{\partial \xi_\Delta^2}\bigg|_{(\boldsymbol\alpha_\Delta,\xi^*(\boldsymbol\alpha_\Delta))} \frac{\partial \xi^*(\boldsymbol\alpha_\Delta)}{\partial \alpha_{e_i}}\bigg|_{\boldsymbol\alpha_\Delta}  \frac{\partial \xi^*(\boldsymbol\alpha_\Delta)}{\partial \alpha_{e_j}}\bigg|_{\boldsymbol\alpha_\Delta}.
\end{split}
\end{equation}
By Proposition \ref{critical2}, we have
\begin{equation*}
\frac{\partial^2 W}{\partial \alpha_{e_i}\partial \alpha_{e_j}}\bigg|_{\boldsymbol\alpha^*_{s,\Delta}}=4\mathbf i \frac{\partial \theta_{(\Delta,e_i)}}{\partial l_{e_j}}\bigg|_{\boldsymbol l^*_{s,\Delta}},
\end{equation*}
where $\theta_{(\Delta,e_i)}$ is the dihedral angle of  $\Delta$ at $e_i$ as a function of $\boldsymbol l_\Delta,$ 
Then (4) and (5) follow from (\ref{4.10}),  (\ref{4.11}), (\ref{4.13}) and the fact that
$$\theta_e=\sum_{\Delta\sim e}\theta_{(\Delta,e)},$$
where the sum is over all the corners around $e.$
This completes the proof of the claims (1) -- (5). 
\end{proof}

\begin{proposition}\label{cd}
The function $\mathrm{Im}\mathcal W$ is  strictly concave down  in $(\mathrm{Im}\boldsymbol \alpha, \mathrm{Im}\boldsymbol \xi)$ on $\big(\frac{\pi}{2}+\mathbf i{\mathbb R_{>0}}\big)^E\times D^T.$ 
  \end{proposition}

\begin{proof}

 Since $2\pi\sum_{e\in E}\mathrm{Im}\alpha_e$ is a  linear function,  it suffices to prove that for each $\Delta\in T,$ $\mathrm{Im}U$ is strictly concave down in $(\mathrm{Im}\boldsymbol\alpha_\Delta,\mathrm{Im}\xi_\Delta)$ on $\big(\frac{\pi}{2}+\mathbf i{\mathbb R_{>0}}\big)^6\times D,$ so that their sum is a strictly concave down function.

To this end, we will show that:
\begin{enumerate}[(a)]
\item $-\frac{1}{2}\sum_{i=1}^4\sum_{j=1}^4\mathrm{Im}L(\eta_j-\tau_i)$ is linear in  $\mathrm{Im}\alpha_1,\dots, \mathrm{Im}\alpha_6,$ and
\item $\sum_{i=1}^4\mathrm{Im}L(\xi-\tau_i)+\sum_{j=1}^4\mathrm{Im}L(\eta_j-\xi)$ is strictly concave down in $\mathrm{Im}\alpha_1,\dots, \mathrm{Im}\alpha_6$ and $\mathrm{Im}\xi.$
\end{enumerate}
Then by (\ref{U}), $\mathrm{Im}U$ is strictly concave down in $\mathrm{Im}\alpha_1,\dots, \mathrm{Im}\alpha_6$ and $\mathrm{Im}\xi.$
\\

To see (a), for $i\in\{1,2,3,4\}$ and $j\in\{1,2,3\},$ suppose $\eta_j-\tau_i=\alpha_l+\alpha_m-\alpha_n$ for some $\{l,m,n\}\subset \{1,\dots, 6\},$ then by a direct computation, for $s,t\in\{l,m,n\},$
\begin{equation*}
-\frac{\partial^2 \mathrm{Im}L(\eta_j-\tau_i)}{\partial \mathrm{Im}\alpha_s\partial \mathrm{Im}\alpha_t}=\pm\frac{4\sin\big(2\mathrm{Re}(\eta_j-\tau_i)\big)}{\big|1-e^{2\mathbf i(\eta_j-\tau_i)}\big|^2}=0,
\end{equation*}
where the last equality come from $\mathrm{Re}(\eta_j-\tau_i)=\frac{\pi}{2}.$  For $i\in\{1,2,3,4\},$ suppose $\tau_i=\alpha_l+\alpha_m+\alpha_n$ for some $\{l,m,n\}\subset \{1,\dots, 6\},$ then by a direct computation, for $s,t\in\{l,m,n\},$
\begin{equation*}
-\frac{\partial^2 \mathrm{Im}L(2\pi-\tau_i)}{\partial \mathrm{Im}\alpha_s\partial \mathrm{Im}\alpha_t}=-\frac{4\sin\big(2\mathrm{Re}(2\pi-\tau_i)\big)}{\big|1-e^{2\mathbf i(2\pi-\tau_i)}\big|^2}=0,
\end{equation*}
where the last equality comes from $\mathrm{Re}(2\pi -\tau_i)=\frac{\pi}{2}.$ 

As a consequence, $-\mathrm{Im}L(\eta_j-\tau_i)$'s are linear in $\mathrm{Im}\alpha_1,\dots, \mathrm{Im}\alpha_6,$ from which (a) follows. 
\\

To see (b), we make the following change of variables: For $i\in\{1,2,3,4\}$ let $x_i=\xi-\tau_i,$ and for $j\in\{1,2,3\}$ let $x_{j+4}=\eta_j-\xi.$ Then 
we have $2\pi-\xi=2\pi-\sum_{k=1}^7x_k,$ and 
$$\sum_{i=1}^4\mathrm{Im}L(\xi-\tau_i)+\sum_{j=1}^4\mathrm{Im}L(\eta_j-\xi)=\sum_{k=1}^7\mathrm{Im}L(x_k)+L\Big(2\pi-\sum_{k=1}^7x_k\Big).$$
We will show that for each  $k\in\{1,\dots,7\},$ $\mathrm{Im}L(x_k)$ is strictly concave down in $\mathrm{Im} x_k;$ and $\mathrm{Im}L\big(2\pi-\sum_{k=1}^7x_k\big)$ is concave down in $\mathrm{Im}x_1,\dots,\mathrm{Im}x_7.$ Then from this, (2) follows.

Indeed, for each $k\in\{1,\dots,7\},$ by a direct computation we have
$$\frac{\partial^2\mathrm{Im}L(x_k)}{\partial \mathrm{Im} x_k ^2}=-\frac{4\sin\big(2\mathrm{Re}x_k\big)}{\big|1-e^{2\mathbf ix_k}\big|^2}<0,$$
where the inequality comes from that $\mathrm{Re}x_k$ equals either $\mathrm{Re}(\xi-\tau_i)$ or $\mathrm{Re}(\eta_j-\xi),$ in either case is in $(0,\pi).$ As a consequence, $\mathrm{Im}L(x_k)$ is strictly concave down in $\mathrm{Im} x_k.$ For all $k,l\in\{1,\dots,7\},$ by a direct computation we have
$$\frac{\partial^2\mathrm{Im}L\big(2\pi-\sum_{k=1}^7x_k\big)}{\partial \mathrm{Im} x_k \partial \mathrm{Im} x_l}=-\frac{4\sin\big(2\mathrm{Re}\big(2\pi-\sum_{k=1}^7x_k\big)\big)}{\big|1-e^{2\mathbf i(2\pi-\sum_{k=1}^7x_k)}\big|^2}<0,$$
where the inequality comes from that $\mathrm{Re}\big(2\pi-\sum_{k=1}^7x_k\big)=\mathrm{Re}(2\pi-\xi)\in(0,\frac{\pi}{2}).$ Therefore, the Hessian matrix of $\mathrm{Im}L\big(2\pi-\sum_{k=1}^7x_k\big)$ with respect to $\mathrm{Im}x_1,\dots,\mathrm{Im}x_7$ is a $7\times 7$ matrix with all of its entries the same negative constant, hence is negative semi-definite. As a consequence, $\mathrm{Im}L\big(2\pi-\sum_{k=1}^7x_k\big)$ is concave down in $\mathrm{Im}x_1,\dots,\mathrm{Im}x_7.$
 \end{proof}

For any ideal triangulation $\mathcal T$ of $M$ with the space of angle structures $\mathcal A(M,\mathcal T)\neq \emptyset,$ by Corollary \ref{Cor3} with the constant $t_0>0$ therein,  there is a generalized hyperbolic polyhedral metric $\boldsymbol l_{t_0} ^*=\big(l^*_{t_0,e}\big)$ whose cone angles are $2\pi+t_0$ at all the edges.

For $\Delta\in T$ with edges $e_1,\dots,e_6,$ let 
 $\boldsymbol \alpha^*_{t_0 ,\Delta}=\Big(\frac{\pi}{2}+ \mathbf i\frac{l^*_{t_0,e_1}}{2},\dots,\frac{\pi}{2}+ \mathbf i\frac{l^*_{t_0,e_6}}{2}\Big).$ 
Guaranteed by Proposition \ref{critical2}, we have
\begin{enumerate}[(1)]
\item if $(l^*_{t_0,e_1},\dots, l^*_{t_0,e_6})\in\mathcal L,$ then we let $\xi^*_{t_0,\Delta}$ be the unique critical point of $U_{\boldsymbol \alpha^*_{t_0,\Delta}}$ in $D;$ 
\item
 if  $(l^*_{t_0,e_1},\dots, l^*_{t_0,e_6})\in \partial \mathcal L,$ then we let $\xi^*_{t_0,\Delta}$ be the unique critical point of $U_{\boldsymbol \alpha^*_{t_0,\Delta}}$ on $\partial D;$ and 
  \item
 if $(l^*_{t_0,e_1},\dots, l^*_{t_0,e_6})\in {\mathbb R_{>0}}^6\setminus (\mathcal L\cup\partial \mathcal L),$ then we can choose from the two a critical point  $\xi^*_{t_0,\Delta}$
of $U_{\boldsymbol \alpha^*_{t_0,\Delta}}$ on $\partial D.$ For consistency, we choose the one corresponding to $\xi_1^*$ in Proposition \ref{critical2} (3).
\end{enumerate}  
By Proposition \ref{critical2} (2) and (3), $\mathrm{Re}\xi^*_{t_0,\Delta}=2\pi$ in the last two cases. 
\\

Let  $\overline D=\{ \xi\in\mathbb C\ |\ \frac{3\pi}{2}\leqslant \mathrm{Re}\xi\leqslant 2\pi\}.$   Then we have the following

\begin{proposition}\label{cpcv33}  Let $$\mathcal U(\boldsymbol \alpha, \boldsymbol \xi)=(2\pi+{t_0} ) \sum_{e\in E} (2\alpha_e-\pi)+\sum_{\Delta\in T}U\big(\boldsymbol \alpha_\Delta,\xi_\Delta\big).$$ Then $\mathrm{Im}\mathcal U$ has a critical point 
$$\boldsymbol z^*_{t_0}=\bigg(\Big(\frac{\pi}{2}+\mathbf i\frac{l^*_{t_0,e}}{2}\Big)_{e\in E},\big(\xi^*_{t_0,\Delta}\big)_{\Delta\in T}\bigg)$$
in $ \big(\frac{\pi}{2}+\mathbf i{\mathbb R_{>0}}\big)^E\times \overline D^T.$
\end{proposition}

\begin{proof} Let $\boldsymbol z^*_{t_0}=(\boldsymbol \alpha^*_{t_0},\boldsymbol \xi^*_{t_0}),$ where $\boldsymbol \alpha^*_{t_0}=\Big(\frac{\pi}{2}+\mathbf i\frac{l^*_{t_0,e}}{2}\Big)_{e\in E}$ and $\boldsymbol \xi^*_{t_0}=\big(\xi^*_{t_0,\Delta}\big)_{\Delta\in T}.$ For $\Delta\in T$ with edges $e_1,\dots,e_6,$ let $l^*_{t_0,\Delta}=(l^*_{t_0,e_1},\dots,l^*_{t_0,e_6}).$ 
We first have
\begin{equation}\label{xi=0b2}
\frac{\partial \mathrm{Im}\mathcal U}{\partial \xi_\Delta}\Big|_{(\boldsymbol \alpha^*_{t_0},\boldsymbol \xi^*_{t_0})}=\frac{\partial \mathrm{Im}U_{\boldsymbol \alpha_{t_0,\Delta}}(\xi_\Delta)}{\partial  \xi_\Delta}\Big|_{\xi^*_{t_0,\Delta}}=0.
\end{equation}
Now let $W(\boldsymbol \alpha_\Delta)=U(\boldsymbol \alpha_\Delta,\xi^*_\Delta)=U_{\boldsymbol \alpha_\Delta}(\xi^*_\Delta)$ be the function defined in (\ref{W}). Then for any edge $e$ of $\Delta,$ 
\begin{equation*}
\begin{split}
\frac{\partial \mathrm{Im} W(\boldsymbol \alpha_\Delta)}{\partial \mathrm{Im} \alpha_e}\Big|_{\boldsymbol \alpha^*_{t_0,\Delta}}=&\frac{\partial \mathrm{Im} U(\boldsymbol \alpha_\Delta,\xi_\Delta)}{\partial \mathrm{Im} \alpha_e}\Big|_{(\boldsymbol \alpha^*_{t_0,\Delta},\xi^*_{t_0,\Delta})}+\frac{\partial \mathrm{Im} U(\boldsymbol \alpha_\Delta,\xi_\Delta)}{\partial \mathrm{Im} \xi_\Delta}\Big|_{(\boldsymbol \alpha^*_{t_0,\Delta},\xi^*_{t_0,\Delta})}\cdot\frac{\partial \mathrm{Im} \xi^*}{\partial \mathrm{Im} \alpha_e}\Big|_{\boldsymbol \alpha^*_{t_0,\Delta}}\\
=&\frac{\partial \mathrm{Im} U(\boldsymbol \alpha_\Delta,\xi_\Delta)}{\partial \mathrm{Im} \alpha_e}\Big|_{(\boldsymbol \alpha^*_{t_0,\Delta},\xi^*_{t_0,\Delta})}.
\end{split}
\end{equation*}
As a consequence, and by Proposition \ref{critical2} (1) (2) and (3) 
\begin{equation}\label{pUpa}
\frac{\partial \mathrm{Im} U(\boldsymbol \alpha_\Delta,\xi_\Delta)}{\partial \mathrm{Im} \alpha_e}\Big|_{(\boldsymbol \alpha^*_{t_0,\Delta},\xi^*_{t_0,\Delta})}=\frac{\partial \mathrm{Im} W(\boldsymbol \alpha_\Delta)}{\partial \mathrm{Im} \alpha_e}\Big|_{\boldsymbol \alpha^*_{t_0,\Delta}}=-4\frac{\partial  \mathrm{Cov}(\boldsymbol l_\Delta)}{\partial  l_e}\Big|_{\boldsymbol l^*_{t_0,\Delta}}=-2\widetilde \theta_{t_0, (\Delta,e)},
\end{equation}
where
$\widetilde \theta_{t_0,(\Delta,e)}$ is the extended dihedral angle of the $\boldsymbol l^*_{t_0,\Delta}$ at the edge $e.$ Then for each $e\in E,$ 
\begin{equation}\label{alpha=0b2}
\frac{\partial \mathrm{Im} \mathcal U}{\partial \mathrm{Im} \alpha_e}\Big|_{(\boldsymbol \alpha^*_{t_0},\boldsymbol \xi^*_{t_0})}=2(2\pi+{t_0} )+ \sum_{\Delta\in T}\frac{\partial \mathrm{Im} U(\boldsymbol \alpha_\Delta,\xi_\Delta)}{\partial \mathrm{Im} \alpha_e}\Big|_{(\boldsymbol \alpha^*_{t_0,\Delta},\xi^*_{t_0,\Delta})}=2\Big((2\pi+{t_0} )-\sum_{((\Delta,e))\in C}\widetilde\theta_{t_0,(\Delta,e)}\Big)=0,
\end{equation}
where the last summation is over all the corners around $e,$ and the last equality comes from the fact that the sum of the dihedral angles around an edge equals is the cone angle at this edge, which equals $2\pi+{t_0} .$ By (\ref{xi=0b2}) and (\ref{alpha=0b2}), $(\boldsymbol \alpha^*_{t_0},\boldsymbol \xi^*_{t_0})$ is a critical point of $\mathrm{Im}\mathcal U.$   
\end{proof}

In the generalized hyperbolic cone metric $\boldsymbol l^*_{t_0},$ we call a tetrahedron $\Delta$ \emph{degenerate} if  $\boldsymbol l^*_{t_0,\Delta}\in {\mathbb R_{>0}}\setminus \mathcal L.$ For $c\in (\frac{3\pi}{2},2\pi),$ let $\xi^*_{c,t_0,\Delta}$  the unique  maximum point  of $\mathrm{Im}U_{\boldsymbol \alpha^*_{t_0,\Delta}}$ on $\Gamma_{c}$ guaranteed by Corollary \ref{cor3}. For a corner $(\Delta, e),$ let $\widetilde{\theta}_{c,t_0,(\Delta, e)}=\widetilde \theta_{t_0, (\Delta, e)}$ be the original dihedral angle of $\boldsymbol l^*_{t_0}$ if $\Delta$ is non-degenerate,  and let   
 $$\widetilde{\theta}_{c,t_0,(\Delta, e)}=-\frac{1}{2}\frac{\partial \mathrm{Im}U(\boldsymbol \alpha_\Delta,\xi_\Delta)}{\partial \mathrm{Im}\alpha_e}\Big|_{(\boldsymbol \alpha^*_{t_0,\Delta},\xi^*_{c,t_0,\Delta})}$$
  if $\Delta$ is degenerate. For $e\in E,$ let 
  \begin{equation}\label{cac}
  \widetilde\theta_{c,t_0, e}=\sum_{\Delta\sim e}\widetilde\theta_{c,t_0, (\Delta,e)},
  \end{equation}
 where the sum is over all the corners around $e.$  Let $\boldsymbol \xi^*_{c,t_0}\in D^T$ be obtained from $\boldsymbol \xi^*_{t_0}\in\overline D^T$ by replacing the component $\xi^*_{t_0,\Delta}$ by $\xi^*_{c,t_0,\Delta}$ for each degenerate $\Delta$ in $T.$ We note that, as $c$ tends to $2\pi,$  $\widetilde{\theta}_{c,t_0,(\Delta, e)}$ converges to  $\widetilde{\theta}_{t_0,(\Delta, e)}$ for each corner $(\Delta,e),$ 
 and therefore we have \begin{equation}
     \widetilde\theta_{c,t_0, e}\to \widetilde\theta_{t_0,e} =2\pi+t_0
 \end{equation}  for each edge $e.$

 \begin{proposition}\label{critical55} Let
$$\mathcal U_c(\boldsymbol \alpha, \boldsymbol \xi)= \sum_{e\in E} \widetilde{\theta}_{c,t_0,  e}(2\alpha_e-\pi)+\sum_{\Delta\in T}U\big(\boldsymbol \alpha_\Delta,\xi_\Delta\big).$$
Then $\mathrm{Im}\mathcal U_c$ has a unique maximum point $\boldsymbol z^*_{c,t_0}=(\boldsymbol \alpha^*_{t_0}, \boldsymbol \xi^*_{c,t_0})$ on the contour 
$$\Gamma^*_{c,t_0}= \bigg\{ (\boldsymbol \alpha,\boldsymbol \xi)\in \Big(\frac{\pi}{2}+\mathbf i{\mathbb R_{>0}}\Big)^E\times D^T   \ \bigg|\ \mathrm{Re}\boldsymbol \xi=\mathrm{Re}\boldsymbol \xi^*_{c,t_0}\bigg\}.$$
 \end{proposition}

\begin{proof} First, by Corollary \ref{cor3}, if $\Delta$ is non-degenerate,  we have
\begin{equation*}
\frac{\partial\mathrm{Im} \mathcal U_c}{\partial \mathrm {Im}\xi_\Delta}\Big|_{(\boldsymbol \alpha^*_{t_0},\boldsymbol \xi^*_{c,t_0})}=\frac{\partial \mathrm{Im}U_{\boldsymbol \alpha_{t_0,\Delta}}(\xi_\Delta)}{\partial \mathrm{Im}\xi_\Delta}\Big|_{\xi^*_{t_0,\Delta}}=0;
\end{equation*}
and  if $\Delta$ is degenerate, by the choice of $\xi^*_{c,t_0,\Delta},$  we have
\begin{equation*}
\frac{\partial\mathrm{Im} \mathcal U_c}{\partial \mathrm {Im}\xi_\Delta}\Big|_{(\boldsymbol \alpha^*_{t_0},\boldsymbol \xi^*_{c,t_0})}=\frac{\partial \mathrm{Im}U_{\boldsymbol \alpha_{t_0,\Delta}}(\xi_\Delta)}{\partial \mathrm{Im}\xi_\Delta}\Big|_{\xi^*_{c,t_0,\Delta}}=0.
\end{equation*}
Next, if $\Delta$ is non-degenerate, then by (\ref{pUpa}), we have
\begin{equation*}
\frac{\partial\mathrm{Im} U(\boldsymbol \alpha_\Delta,\xi_\Delta)}{\partial\mathrm{Im} \alpha_e}\Big|_{(\boldsymbol \alpha^*_{t_0,\Delta},\xi^*_{t_0,\Delta})}=-2\widetilde \theta_{t_0,(\Delta,e)}=-2\widetilde{\theta}_{c,t_0,(\Delta, e)};
\end{equation*}
and if $\Delta$ is degenerate, then by definition 
$$\frac{\partial\mathrm{Im} U(\boldsymbol \alpha_\Delta,\xi_\Delta)}{\partial\mathrm{Im} \alpha_e}\Big|_{(\boldsymbol \alpha^*_{t_0,\Delta},\xi^*_{c,t_0,\Delta})}=-2\widetilde{\theta}_{c,t_0,(\Delta, e)}.$$
Now for each $e\in E,$ 
\begin{equation*}
\begin{split}
&\frac{\partial \mathcal U_c}{\partial \alpha_e}\Big|_{(\boldsymbol \alpha^*_{t_0},\boldsymbol \xi^*_{c,t_0})}\\
=&2 \widetilde{\theta}_{c,t_0, e}+ \sum_{\Delta \text{ non-degenerate}}\frac{\partial \mathrm{Im}U(\boldsymbol \alpha_\Delta,\xi_\Delta)}{\partial \mathrm{Im}\alpha_e}\Big|_{(\boldsymbol \alpha^*_{t_0,\Delta},\xi^*_{t_0,\Delta})}+ \sum_{\Delta \text{ degenerate}}\frac{\partial \mathrm{Im}U(\boldsymbol \alpha_\Delta,\xi_\Delta)}{\partial \mathrm{Im}\alpha_e}\Big|_{(\boldsymbol \alpha^*_{t_0,\Delta},\xi^*_{c,t_0, \Delta})}\\
=&2\widetilde {\theta}_{c,t_0, e}-2\sum_{\Delta\sim e }\widetilde{\theta}_{c,t_0, (\Delta, e)}=0.
\end{split}
\end{equation*}
Therefore, $\boldsymbol z^*_{c,t_0}$ is a critical point of $\mathrm{Im}\mathcal U_c$ on $\Gamma^*_c.$ 
Finally, since $\mathrm{Im}\mathcal U_c$ and $\mathrm{Im}\mathcal W$ differ by a linear function $\sum_{e\in E}(\widetilde{\theta}_{c,t_0,e}-2\pi)l_e,$ they share the same concavity; and by Proposition \ref{cd},  $\mathrm{Im}\mathcal U_c$  is strictly concave down  in $(\mathrm{Im}\boldsymbol \alpha, \mathrm{Im}\boldsymbol \xi)$ on $\big(\frac{\pi}{2}+\mathbf i\mathbb R\big)^E\times D^T.$ As a consequence, $\boldsymbol z^*_{c,t_0}$ is the unique critical point of $\mathrm{Im}\mathcal U_c$ on $\Gamma^*_{c,t_0},$ and achieves the maximum value.  
\end{proof}

For $\boldsymbol \alpha\in  \big(\frac{\pi}{2}+\mathbf i{\mathbb R_{>0}}\big)^6,$ $\delta >0$ sufficiently small and $c \in (\delta, \frac{\pi}{2}-\delta),$ recall that $D^{\boldsymbol \alpha}_{\delta,c}$ is the region depicted in Figure \ref{Ddc} and defined there-around consisting of $\xi\in \mathbb C$ either with $\mathrm {Re}\xi\in[\frac{3\pi}{2}+\delta, 2\pi-\delta]$ or with $|\mathrm{Im}(\xi-\tau_i)|\geqslant \delta,$ $|\mathrm{Im}(\eta_j-\xi)|\geqslant \delta$ for all $i, j\in\{1,2,3,4\}$ and $\mathrm {Re}\xi\in[\frac{3\pi}{2}-c, 2\pi+c].$ We consider the following region 
$$\mathcal D_{\delta,c} =\bigg\{ (\boldsymbol\alpha, \boldsymbol \xi)\in \Big(\frac{\pi}{2}+\mathbf i{\mathbb R_{>0}}\Big)^E\times \mathbb C^T\ \bigg|\ \xi_\Delta\in D^{\boldsymbol \alpha_\Delta}_{\delta,c}\text{ for all }\Delta\in T \bigg\}.$$
 We notice that if $(\boldsymbol \alpha,\boldsymbol \xi)\in \mathcal D_{\delta,c},$ then for all $\Delta\in T,$ all $\xi_\Delta-\tau_i(\boldsymbol\alpha_\Delta)$  and $\eta_j(\boldsymbol\alpha_\Delta) - \xi_\Delta,$ $i,j\in\{1,2,3,4\},$  are in the region $H_{\delta,K}=\big\{ x\in \mathbb C\ |\ \delta \leqslant \mathrm{Re}x\leqslant \pi-\delta, \text{ or } |\mathrm{Re}x|\leqslant K \text{ and } |\mathrm{Im}x| \geqslant \delta \big\}$ depicted in Figure \ref{DdK}  for any  $K>c.$

\begin{proposition}\label{bound4}
There exists a constant $K>0$ such that, for all unit vector $\mathbf u$ of the form
$$\mathbf u=\bigg(\bigg(u_e\cdot  \frac{\partial}{\partial \mathrm{Im}\alpha_e}\bigg)_{e\in E}, \bigg(\Big(0,u_\Delta\cdot \frac{\partial }{\partial \mathrm{Im}\xi_\Delta}\Big)\bigg)_{\Delta\in T}\bigg)$$ 
in $\mathbb R^{|E|+2|T|}$ identified with the tangent space of $\mathcal D_{\delta,c}$ at each of its points, 
 the directional derivative 
$$D_{\mathbf u}\mathrm{Im}\kappa(\boldsymbol \alpha,\boldsymbol \xi)<K$$
for all for $(\boldsymbol \alpha,\boldsymbol \xi)$ in $\mathcal D_{\delta,c} .$
\end{proposition}

\begin{proof} As argued in the Proof of Proposition \ref{bound3},  for $x\in H_{\delta,K},$  there exists a constant $C_{\delta,K}>0$ depending on $\delta$ and $K$ such that
$$\bigg|\frac{\partial \log \big(1-e^{2\mathbf ix}\big)}{\partial \mathrm{Im} x}\bigg|<C_{\delta,K}.$$

Then by a direction computation, we have  for $k\in\{1,\dots,6\},$
\begin{equation*}\label{ka}
\begin{split}
\frac{\partial \mathrm{Im}\kappa(\boldsymbol\alpha,\xi)}{\partial \mathrm{Im}\alpha_k}=&14\pi-4\pi\sum_{\tau_i\sim \alpha_k}\mathrm{Re}\bigg(\frac{\partial \log \big(1-e^{2\mathbf i(\xi-\tau_i)}\big)}{\partial \mathrm{Im}\alpha_k}\bigg)+3\pi\sum_{\eta_j\sim \alpha_k}\mathrm{Re}\bigg(\frac{\partial \log \big(1-e^{2\mathbf i(\eta_j-\xi}\big)}{\partial \mathrm{Im}\alpha_k}\bigg)\\
\leqslant &14\pi+4\pi\sum_{\tau_i\sim \alpha_k}\bigg|\frac{\partial \log \big(1-e^{2\mathbf i(\xi-\tau_i)}\big)}{\partial \mathrm{Im}\alpha_k}\bigg|+3\pi\sum_{\eta_j\sim \alpha_k}\bigg|\frac{\partial \log \big(1-e^{2\mathbf i(\eta_j-\xi}\big)}{\partial \mathrm{Im}\alpha_k}\bigg| <K_1
\end{split}
\end{equation*}
for all $\boldsymbol \alpha \in \big(\frac{\pi}{2}+\mathbf i{\mathbb R_{>0}}\big)^6$ and $\xi\in D^{\boldsymbol \alpha}_{\delta,c},$ 
where the sums are respectively over the two $\tau_i$'s and the two $\eta_j$'s containing $\alpha_k,$ and
 $$K_1=14\pi + 4\pi\sum_{\tau_i\sim\alpha_k}C_{\delta,c}+3\pi\sum_{\eta_j\sim\alpha_k}C_{\delta,c}=14\pi+14\pi C_{\delta,c}.$$
Also by a direct computation, 
\begin{equation*}
\begin{split}
\frac{\partial \mathrm{Im}\kappa(\boldsymbol\alpha,\xi)}{\partial \mathrm{Im}\xi}=&-28\pi-4\pi\sum_{i=1}^4\mathrm{Re}\bigg(\frac{\partial \log \big(1-e^{2\mathbf i(\xi-\tau_i)}\big)}{\partial \mathrm{Im}\xi}\bigg)+3\pi\sum_{j=1}^4\mathrm{Re}\bigg(\frac{\partial \log \big(1-e^{2\mathbf i(\eta_j-\xi}\big)}{\partial \mathrm{Im}\xi}\bigg)\\
\leqslant &- 28\pi+4\pi\sum_{i=1}^4\bigg|\frac{\partial \log \big(1-e^{2\mathbf i(\xi-\tau_i)}\big)}{\partial \mathrm{Im}\xi}\bigg|+3\pi\sum_{j=1}^4\bigg|\frac{\partial \log \big(1-e^{2\mathbf i(\eta_j-\xi}\big)}{\partial \mathrm{Im}\xi}\bigg| < K_2
\end{split}
\end{equation*}
for all $\boldsymbol \alpha \in \big(\frac{\pi}{2}+\mathbf i{\mathbb R_{>0}}\big)^6$ and $\xi\in D^{\boldsymbol \alpha}_{\delta,c},$ where  
$$K_2=-28\pi + 4\pi\sum_{i=1}^4C_{\delta,c}+3\pi\sum_{j=1}^4 C_{\delta,c}=-28\pi+28\pi C_{\delta,c}.$$
As a consequence, for any vector $\mathbf u\in \mathbb R^7$ with norm $|\mathbf u|\leqslant 1,$ the directional derivative 
 \begin{equation}\label{H}
 D_{\mathbf u}\mathrm{Im}\kappa (\boldsymbol \alpha, \xi)=\Big\langle \mathbf u, \Big(\frac{\partial \mathrm{Im}\kappa (\boldsymbol \alpha, \xi)}{\partial \mathrm{Im}\alpha_1},\dots,  \frac{\partial \mathrm{Im}\kappa  (\boldsymbol \alpha, \xi)}{\partial \mathrm{Im}\alpha_6}, \frac{\partial \mathrm{Im}\kappa (\boldsymbol \alpha, \xi)}{\partial \mathrm{Im}\xi}\Big)\Big\rangle < \max\{K_1, K_2\}
 \end{equation}
for all $\boldsymbol \alpha \in \big(\frac{\pi}{2}+\mathbf i{\mathbb R_{>0}}\big)^6$ and $\xi\in D^{\boldsymbol \alpha}_{\delta,c}.$

 Now for a unit vector  $\mathbf u=\big((u_e)_{e\in E}, (u_\Delta)_{\Delta\in T}\big)\in \mathbb R^{|E|+2|T|}$ of the form in the statement of the  proposition,  let $\mathbf u_\Delta=(u_{e_1},\dots,u_{e_6},u_\Delta)$ for a $\Delta\in T$ with edges $e_1,\dots,e_6.$  
 Then by (\ref{H}), we have 
\begin{equation*}
D_{\mathbf u}\mathrm{Im}\kappa(\boldsymbol \alpha,\boldsymbol \xi)=\sum_{\Delta\in T}D_{\mathbf u_\Delta}\mathrm{Im}\kappa(\boldsymbol \alpha_\Delta,\xi_\Delta)
<K
\end{equation*}
for all $(\boldsymbol\alpha,\boldsymbol \xi)\in  \mathcal D_{\delta,c},$ where  $K=|T|\max\{K_1,K_2\}.$   
  \end{proof}

More generally, for  $\boldsymbol \alpha\in \mathbb C^6,$ $\delta >0$ sufficiently small and $c \in (\delta, \frac{\pi}{2}-\delta),$ recall that  $D^{\boldsymbol \alpha}_{\delta,c}$ is the region  defined before Proposition \ref{bound} consisting of $\xi\in \mathbb C$  with $\mathrm {Re}(\xi-\tau_i)\geqslant \delta,$  $|\mathrm{Im}(\xi-\tau_i)|\geqslant \delta$ for all $i\in\{1,2,3,4\},$  $\mathrm{Re}(\eta_j-\xi)\geqslant \delta,$ $|\mathrm{Im}(\eta_j-\xi)|\geqslant \delta$ for all $j\in\{1,2,3,4\}$ and $\mathrm {Re}\xi\in[\frac{3\pi}{2}-c, 2\pi+c].$ We consider the region 
 $$\mathbf D_{\delta,c}=\bigg\{ (\boldsymbol \alpha,\boldsymbol \xi)\in  \Big(\Big[\frac{\pi}{2}-\delta,\frac{\pi}{2}+\delta \Big]+\mathbf i{\mathbb R_{>0}}\Big)^{E}\times\mathbb C^{T}\ \bigg|\ \xi_\Delta\in D^{\boldsymbol \alpha_\Delta}_{\delta,c}\text{ for all }\Delta\in T \bigg\}.$$
 We notice that if $(\boldsymbol \alpha,\boldsymbol \xi)\in \mathbf D_{\delta,c},$ then for all $\Delta\in T,$ all $\xi_\Delta-\tau_i(\boldsymbol\alpha_\Delta)$  and $\eta_j(\boldsymbol\alpha_\Delta)- \xi_\Delta,$ $i,j\in\{1,2,3,4\},$  are in the region $H_{\delta,K}$ depicted in Figure \ref{DdK} and defined there around for any  $K>c.$

\begin{proposition}\label{bound2}  For $b>0$ and $\delta >0$ sufficiently small and $c \in (\delta, \frac{\pi}{2}-\delta),$ there exists a constant $N>0$ independent of $b$ such that 
$$\mathrm{Im}\nu_b(\boldsymbol \alpha, \boldsymbol \xi)\leqslant \big|\nu_b(\boldsymbol \alpha, \boldsymbol \xi)\big|<N$$
for all $(\boldsymbol\alpha,\boldsymbol \xi)\in  \mathbf D_{\delta,c}.$  
\end{proposition}

\begin{proof}  Let $N_{\delta,c}>0$ be the constant as in Proposition \ref{bound} for the chosen $\delta$  and $c.$ 
Then by (\ref{2}) and (\ref{Wnu}),
\begin{equation*}
\begin{split}
\big|\nu_b(\boldsymbol \alpha, \boldsymbol \xi)\big|=&\Big|\frac{\mathcal W_b(\boldsymbol \alpha, \boldsymbol \xi)-\kappa(\boldsymbol \alpha,\boldsymbol \xi)b^2-\mathcal W(\boldsymbol \alpha, \boldsymbol \xi)}{b^4}\Big|\\
\leqslant &\sum_{\Delta\in T}\Big|\frac{U_b(\boldsymbol \alpha_\Delta,\xi_\Delta)-\kappa_{\boldsymbol {\alpha}_\Delta}(\xi_\Delta)b^2-U(\boldsymbol \alpha_\Delta,\xi_\Delta)}{b^4}\Big|<N
\end{split}
\end{equation*}
for all $(\boldsymbol\alpha,\boldsymbol \xi)\in  \mathbf D_{\delta,c},$ where $N=|T|N_{\delta,c}.$  
\end{proof}

\subsection{Proof of Theorem \ref{VC}}\label{sec:4.3}
In this subsection we will show that for any Kojima ideal triangulation $\mathcal T$, $\mathrm{TV}_b(M,\mathcal T)$ has the desired asymptotics in (\ref{eq:1.7}); and in the next section we will show that   $\mathrm{TV}_b(M,\mathcal T)$ does not depend on $\mathcal T$ with a separate argument. We present this proof before that of Theorem \ref{WD3} because the style of the argument here is close to the ones in the previous section. 
\\

\begin{proof}[Proof of Theorem \ref{VC}] 
 Let $\mathcal T$ be a Kojima ideal triangulation of $M$ with the set of edges $E$ and the set of tetrahedra $T.$ Then by (\ref{int}) and  (\ref{6jint}),  we have 
\begin{equation}
\begin{split}
&\mathrm{TV}_b(M,\mathcal T)=\frac{(-\mathbf i)^{|E|}}{(\pi b )^{|E|+|T|}} \int_{\big(\frac{\pi}{2}+\mathbf i{\mathbb R_{>0}} \big)^E}\bigg(\prod_{e\in E}4\sinh l_e\sinh \frac{l_e}{b^2}\bigg)\Bigg(\int_{\Gamma^*_{\boldsymbol\alpha}}\exp\bigg(\frac{\sum_{\Delta\in T}U_b(\boldsymbol \alpha_\Delta, \xi_\Delta)}{2\pi \mathbf i b^2}\bigg) d\boldsymbol \xi \Bigg) d\boldsymbol \alpha,
\end{split}
\end{equation}
where for each $\boldsymbol\alpha\in \big(\frac{Q}{2}+\mathbf i{\mathbb R_{>0}}\big)^E,$ $\Gamma^*_{\boldsymbol\alpha}=\prod_{\Delta\in T}\Gamma^*_{\boldsymbol\alpha_\Delta},$  and each $\Gamma^*_{\boldsymbol\alpha_\Delta}$ is a contour of integral for $\xi_\Delta$ that is to be specified.

By Proposition \ref{KA}, Corollary \ref{Cor3}  with the constant $t_0>0$ therein, there exists $\boldsymbol l^*_{t_0}$ a  generalized hyperbolic polyhedral metric on $(M,\mathcal T)$ with cone angles $2\pi+t_0$ at all the edges. Let 
$$\mathcal U(\boldsymbol \alpha, \boldsymbol \xi)=(2\pi+t_0)\sum_{e\in E}(2\alpha_e-\pi)+\sum_{\Delta\in T}U\big(\boldsymbol \alpha_\Delta,\xi_\Delta\big)$$  be the function in Proposition \ref{cpcv33} whose imaginary part has a critical point  $\boldsymbol z_{t_0} ^*=\big(\boldsymbol \alpha_{t_0}^*,\boldsymbol\xi_{t_0}^*\big)$
in $\big(\frac{\pi}{2}+\mathbf i{\mathbb R_{>0}}\big)^E\times \overline D^T.$

Let $L_0$ large enough that is to be specified later,
and let 
$$B(\boldsymbol \alpha^*_{t_0}, L_0) = \Big\{ \boldsymbol\alpha \in \Big(\frac{\pi}{2}+\mathbf i{\mathbb R_{>0}} \Big)^E \ \Big|\  |\boldsymbol\alpha - \boldsymbol\alpha_{t_0}^*|\leqslant L_0\Big\}.$$
Then  
$$  \mathrm{TV}_b(M,\mathcal T) = \mathrm I_{\text{out}}+  \mathrm I_{\text{in}},$$
where
\begin{equation}\label{Iout}
 \mathrm I_{\text{out}} = \frac{(-\mathbf i)^{|E|}}{(\pi b )^{|E|+|T|}} \int_{\big(\frac{\pi}{2}+\mathbf i{\mathbb R_{>0}} \big)^E\setminus B(\boldsymbol \alpha^*_{t_0}, L_0) }\bigg(\prod_{e\in E}4\sinh l_e\sinh \frac{l_e}{b^2}\bigg)\Bigg(\int_{\Gamma^*_{\boldsymbol\alpha}}\exp\bigg(\frac{\sum_{\Delta\in T}U_b(\boldsymbol \alpha_\Delta, \xi_\Delta)}{2\pi \mathbf i b^2}\bigg) d\boldsymbol \xi \Bigg) d\boldsymbol \alpha,
 \end{equation}
and 
\begin{equation}\label{Iin}
\mathrm I_{\text{in}}= \frac{(-\mathbf i)^{|E|}}{(\pi b )^{|E|+|T|}} \int_{B(\boldsymbol \alpha^*_{t_0}, L_0) }\bigg(\prod_{e\in E}4\sinh l_e\sinh \frac{l_e}{b^2}\bigg)\Bigg(\int_{\Gamma^*_{\boldsymbol\alpha}}\exp\bigg(\frac{\sum_{\Delta\in T}U_b(\boldsymbol \alpha_\Delta, \xi_\Delta)}{2\pi \mathbf i b^2}\bigg) d\boldsymbol \xi \Bigg) d\boldsymbol \alpha.
\end{equation}
We will show that 
\begin{enumerate}[(I)]
\item 
\begin{equation*}\label{eout}
\mathrm I_{\text{out}} <  O\bigg(e^{\frac{-\mathrm{Vol}(M)-\epsilon_1}{\pi b^2}}\bigg)
\end{equation*}
for some $\epsilon_1>0,$ and 
\item  \begin{equation*}\label{ein}
\mathrm I_{\text{in}} = (2\mathbf i)^{\frac{\chi(M)}{2}}\frac{e^{\frac{-\mathrm{Vol}(M)}{\pi b^2}}}{\sqrt{\pm \mathrm{Tor}(DM, \mathrm{Ad}_{\rho_{DM}})}}\Big(1+O\big(b^2\big)\Big).
\end{equation*}
\end{enumerate}
Then putting (I) and (II) together, we have the result. 
\\

To prove (I), for all $\boldsymbol\alpha \in \big(\frac{\pi}{2}+\mathbf i{\mathbb R_{>0}} \big)^E\setminus B(\boldsymbol \alpha^*_{t_0}, L_0)$ and $\Delta\in T,$   we identically choose $\Gamma^*_{\boldsymbol\alpha_\Delta}$ to be 
$$\Gamma^*_{t_0,\Delta} = \Big\{ \xi_\Delta \in D\ \Big|\ \mathrm{Re}\xi_\Delta = \mathrm{Re}\xi^*_{t_0,\Delta}\Big\}$$
as in the proof of Theorem \ref{Converge}, and let 
$$\Gamma^*_{\text{out}} = \bigg(\Big(\frac{\pi}{2}+\mathbf i{\mathbb R_{>0}} \Big)^E\setminus B(\boldsymbol \alpha^*_{t_0}, L_0) \bigg)\times \prod_{\Delta\in T}\Gamma^*_{t_0,\Delta}.$$
Then by (\ref{Iout}), 
\begin{equation}\label{sout}
\begin{split}
\mathrm I_{\text{out}} = &\frac{(-\mathbf i)^{|E|}}{(\pi b )^{|E|+|T|}} \int_{\Gamma^*_{\text{out}}}\bigg(\prod_{e\in E}4\sinh l_e\sinh \frac{l_e}{b^2} \bigg)\exp\Bigg(\frac{\sum_{\Delta\in T}U_b(\boldsymbol \alpha_\Delta, \xi_\Delta)}{2\pi \mathbf i b^2}\Bigg)d \boldsymbol \alpha  d \boldsymbol \xi\\
=&\frac{(-\mathbf i)^{|E|}}{(\pi b )^{|E|+|T|}}\sum_{\boldsymbol \mu\in\{-1,1\}^E}\sum_{\boldsymbol \lambda\in\{-1,1\}^E}
 \int_{\Gamma^*_{\text{out}}} \bigg(\prod_{e\in E}\mu_e\lambda_e\bigg) \exp\Bigg(\frac{ \mathcal V_b^{\boldsymbol\mu\boldsymbol \lambda}(\boldsymbol \alpha, \boldsymbol \xi)}{2\pi \mathbf i b^2}\Bigg)d \boldsymbol \alpha  d \boldsymbol \xi.
\end{split}
\end{equation} 

Now we let $c\in (\frac{3\pi}{2},2\pi),$ and let 
$$\mathcal U_c(\boldsymbol \alpha, \boldsymbol \xi)= \sum_{e\in E} \widetilde{\theta}_{c,t_0, e}(2\alpha_e-\pi)+\sum_{\Delta\in T}U\big(\boldsymbol \alpha_\Delta,\xi_\Delta\big)$$
be the function defined in Proposition \ref{critical55} whose imaginary part has the unique maximum point $\boldsymbol z^*_{c,t_0}=(\boldsymbol \alpha^*_{t_0}, \boldsymbol \xi^*_{c,t_0})$ on the contour 
$$\Gamma^*_{c,t_0}= \bigg\{ (\boldsymbol \alpha,\boldsymbol \xi)\in \Big(\frac{\pi}{2}+\mathbf i{\mathbb R_{>0}}\Big)^E\times D^T   \ \bigg|\ \mathrm{Re}\boldsymbol \xi=\mathrm{Re}\boldsymbol \xi^*_{c,t_0}\bigg\}.$$
Since $\widetilde\theta_{c_0,t_0,e}$ converges to $\widetilde\theta_{t_0,e}=2\pi + t_0$ as $c$ tends to $2\pi,$ we can choose a $c_0\in  (\frac{3\pi}{2},2\pi)$ sufficiently close to $2\pi$ such that 
\begin{equation}\label{429}
\widetilde\theta_{c_0,t_0,e}>   2\pi + \frac{t_0}{2}
\end{equation}
for each $e\in E.$ For each $\Delta\in T,$ we consider the contour 
$$\Gamma^*_{c_0, t_0,\Delta}=\Big\{\xi_\Delta\in D \ \Big|\ \mathrm{Re}\xi_\Delta = \mathrm{Re}\xi^*_{c_0, t_0,\Delta} \Big\}$$
and 
consider
\begin{equation*}
\begin{split}
\Gamma^*_{c_0,t_0}= &   \bigg\{ (\boldsymbol \alpha,\boldsymbol \xi)\in  \Big(\frac{\pi}{2}+\mathbf i{\mathbb R_{>0}}\Big)^E \times D^T   \ \bigg|\ \mathrm{Re}\boldsymbol \xi=\mathrm{Re}\boldsymbol \xi^*_{c_0,t_0}\bigg\}\\
= & \Big(\frac{\pi}{2}+\mathbf i{\mathbb R_{>0}}\Big)^E\times \prod_{\Delta\in T} \Gamma^*_{c_0, t_0, \Delta}.
\end{split}
\end{equation*}
As 
$$\mathcal V(\boldsymbol\alpha,\boldsymbol \xi)= \mathcal U_{c_0} (\boldsymbol\alpha,\boldsymbol \xi)+\sum_{e\in E}\big(2\pi(1+b^2)-\widetilde\theta_{c_0,t_0,e}\big)(2\alpha_e-\pi),$$
and as a consequence of (\ref{429}), 
 we have
\begin{equation}\label{VUc5}
\begin{split}
\mathrm{Im} \mathcal V(\boldsymbol\alpha,\boldsymbol \xi) = &  \mathrm{Im}  \mathcal U_{c_0} (\boldsymbol\alpha,\boldsymbol \xi) + \sum_{e\in E}\big(2\pi(1+b^2)-\widetilde\theta_{c_0,t_0,e}\big)l_e \\
< &  \mathrm{Im}  \mathcal U_{c_0} (\boldsymbol\alpha,\boldsymbol \xi) + \Big(2\pi b^2-\frac{t_0}{2}\Big)\sum_{e\in E}l_e\\
<  &\mathrm{Im}  \mathcal U_{c_0} (\boldsymbol\alpha,\boldsymbol \xi)
\end{split}
\end{equation}
for all $(\boldsymbol\alpha,\boldsymbol \xi)\in \Gamma^*_{c_0,t_0}$ and $b<\sqrt{\frac{t_0}{4\pi}}.$

Next, for $L>0,$ let 
$$B(\boldsymbol z^*_{c_0,t_0},L)= \Big\{ (\boldsymbol \alpha,\boldsymbol \xi)\in \Gamma^*_{c_0,t_0}   \ \Big|\ \big| (\boldsymbol \alpha,\boldsymbol \xi) -   (\boldsymbol \alpha^*_{t_0}, \boldsymbol \xi^*_{c_0,t_0})  \big| \leqslant  L\Big\}$$
and let 
  $$\partial B(\boldsymbol z^*_{c_0,t_0},L)= \big \{ (\boldsymbol \alpha,\boldsymbol \xi)\in \Gamma^*_{c_0,t_0}   \ \big|\ | (\boldsymbol \alpha,\boldsymbol \xi) -   (\boldsymbol \alpha^*_{t_0}, \boldsymbol \xi^*_{c_0,t_0}) | =L \big\}.$$
Then as a consequence of Proposition \ref{cd}, the directional derivative of $\mathrm{Im}\mathcal U_{c_0}$  is strictly decreasing along each ray in $ \Gamma^*_{c_0,t_0}\cong {\mathbb R_{>0}}^E\times \mathbb R^T$ starting from $\boldsymbol z^*_{c_0,t_0};$ and by the compactness of the spheres in $\mathbb R^{|E|+|T|},$ there is an $L_0>0$ such that
\begin{equation}\label{Mpartial}
\max \big \{  \mathrm{Im}\mathcal U_{c_0} (\boldsymbol\alpha,\boldsymbol\xi)\ \big |\ \partial B(\boldsymbol z^*_{c_0,t_0},L_0)\big \} < -2\mathrm{Vol}(M)-4\epsilon_0
\end{equation}
for a given $\epsilon_0>0,$ and  for any unit vector $\mathbf u\in \mathbb R^{|E|+|T|}$ 
 the directional derivative
\begin{equation}\label{Dv5}
D_{\mathbf u}\mathrm{Im}\mathcal U_{c_0}(\boldsymbol z^*_{c_0,t_0}+l \mathbf u)<-2\epsilon
\end{equation}
for all $l >L_0$ and for some sufficiently small $\epsilon>0.$ 
By (\ref{Vnu}), we have 
\begin{equation*}
\begin{split}
&\Bigg|\int_{\Gamma^*_{c_0,t_0}\setminus B(\boldsymbol z^*_{c_0,t_0},L_0)}\exp\bigg(\frac{\mathcal V_b(\boldsymbol \alpha, \boldsymbol \xi)}{2\pi \mathbf i b^2}\bigg)d \boldsymbol \alpha  d \boldsymbol \xi\Bigg|\\
\leqslant &\int_{\Gamma^*_{c_0,t_0}\setminus B(\boldsymbol z^*_{c_0,t_0},L_0)}\exp\bigg(\frac{\mathrm{Im}\mathcal V(\boldsymbol \alpha, \boldsymbol \xi)+\mathrm{Im}\kappa(\boldsymbol\alpha,\boldsymbol\xi)b^2+\mathrm{Im}\nu_b(\boldsymbol\alpha,\boldsymbol\xi)b^4}{2\pi   b^2}\bigg)|d \boldsymbol \alpha  d \boldsymbol \xi|.
\end{split}
\end{equation*}
By the compactness of $\partial B(\boldsymbol z^*_{c_0,t_0},L_0),$ there is a $b_0<\sqrt{\frac{t_0}{4\pi}}$ such that for all $b<b_0,$ 
\begin{equation}\label{mpartial}
\max \big \{  \mathrm{Im}\kappa (\boldsymbol\alpha,\boldsymbol\xi) \cdot b^2 \ \big |\ (\boldsymbol\alpha,\boldsymbol\xi) \in \partial B(\boldsymbol z^*_{c_0,t_0},L_0)\big \} < \epsilon_0, 
\end{equation}
$Kb^2<\epsilon$ and $Nb^4<\epsilon_0,$ where $K$ and $N$ are respectively the constants in Proposition \ref{bound4} and Proposition \ref{bound2}. 

We claim that, on $\Gamma^*_{c_0,t_0}\setminus B(\boldsymbol z^*_{c_0,t_0},L_0)$ and for $b<b_0,$ 
$$\mathrm{Im}\mathcal V(\boldsymbol \alpha, \boldsymbol \xi)+\mathrm{Im}\kappa(\boldsymbol\alpha,\boldsymbol\xi)b^2+\mathrm{Im}\nu_b(\boldsymbol\alpha,\boldsymbol\xi)b^4< -2\Big(\mathrm{Vol}(M)+\epsilon_0\Big) -\epsilon \Big(\big|(\boldsymbol\alpha,\boldsymbol\xi)-  (\boldsymbol \alpha^*_{t_0}, \boldsymbol \xi^*_{c_0,t_0})  \big|-L_0\Big),$$
and as a consequence we have
\begin{equation}\label{I2}
\begin{split}
&\int_{\Gamma^*_{c_0,t_0}\setminus B(\boldsymbol z^*_{c_0,t_0},L_0)}\exp\bigg(\frac{\mathrm{Im}\mathcal V(\boldsymbol \alpha, \boldsymbol \xi)+\mathrm{Im}\kappa(\boldsymbol\alpha,\boldsymbol\xi)b^2+\mathrm{Im}\nu_b(\boldsymbol\alpha,\boldsymbol\xi)b^4}{2\pi   b^2}\bigg)|d \boldsymbol \alpha  d \boldsymbol \xi|\\
< & \exp\bigg(\frac{-\mathrm{Vol}(M)-\epsilon_0}{\pi b^2}\bigg) \int_{\Gamma^*_{c_0,t_0}\setminus B(\boldsymbol z^*_{c_0,t_0},L_0)}\exp\Bigg(\frac{-\epsilon\big(|(\boldsymbol\alpha,\boldsymbol\xi)-  (\boldsymbol \alpha^*_{t_0}, \boldsymbol \xi^*_{c_0,t_0}) |-L_0\big)}{2\pi   b^2}\Bigg)|d \boldsymbol \alpha  d \boldsymbol \xi|\\
= &\ O\bigg(e^{\frac{-\mathrm{Vol}(M)-\epsilon_0}{\pi b^2}}\bigg).
\end{split}
\end{equation}

Suppose the claim (\ref{I2}) is true, then for $\boldsymbol\mu=\boldsymbol \lambda=(1,\dots, 1),$ since $\Gamma^*_{\text{out}}\subset \Gamma^*_{c_0,t_0}\setminus B(\boldsymbol z^*_{c_0,t_0},L_0),$ we have 
\begin{equation}\label{Int1}
 \Bigg|  \frac{(-2\mathbf i)^{|E|}}{(\pi b )^{|E|+|T|}}  \int_{\Gamma^*_{\text{out}}}\exp\bigg(\frac{\mathcal V_b(\boldsymbol \alpha, \boldsymbol \xi)}{2\pi \mathbf i b^2}\bigg)d \boldsymbol \alpha  d \boldsymbol \xi\Bigg| < \ O\bigg(e^{\frac{-\mathrm{Vol}(M)-\epsilon_1}{\pi b^2}}\bigg)
 \end{equation}
 for any $\epsilon_1<\epsilon_0;$ and for $\boldsymbol \mu\neq (1,\dots,1)$ or  $\boldsymbol \lambda\neq (1,\dots,1),$ we have
$$\mathrm{Im}\mathcal V_b^{\boldsymbol\mu\boldsymbol \lambda}(\boldsymbol\alpha,\boldsymbol\xi)-\mathrm{Im}\mathcal V_b (\boldsymbol\alpha,\boldsymbol\xi) = -2\pi\sum_{e\in E}\big((1-\mu_e)b^2+(1-\lambda_e)\big)l_e\leqslant 0$$
for all for $(\boldsymbol \alpha,\boldsymbol \xi)$ in $\Gamma^*_{\text{out}}.$ Then by (\ref{Int1}), for any $\boldsymbol\mu$  and $\boldsymbol \lambda$ in $\{-1,1\}^E,$
 we have
\begin{equation}\label{Int4}
\begin{split}
&\Bigg| \frac{(-\mathbf i)^{|E|}}{(\pi b )^{|E|+|T|}} \int_{\Gamma^*_{\text{out}}} \bigg(\prod_{e\in E}\mu_e\lambda_e\bigg)\exp\bigg(\frac{\mathcal V_b^{\boldsymbol\mu\boldsymbol \lambda}(\boldsymbol \alpha, \boldsymbol \xi)}{2\pi \mathbf i b^2}\bigg)d \boldsymbol \alpha  d \boldsymbol \xi\Bigg|\\
\leqslant &  \frac{1}{(\pi b )^{|E|+|T|}} \int_{\Gamma^*_{\text{out}}}\exp\bigg(\frac{\mathrm{Im}\mathcal V^{\boldsymbol\mu\boldsymbol \lambda}_b(\boldsymbol \alpha, \boldsymbol \xi)}{2\pi b^2}\bigg)|d \boldsymbol \alpha  d \boldsymbol \xi|\\
\leqslant &  \frac{1}{(\pi b )^{|E|+|T|}} \int_{\Gamma^*_{\text{out}}}\exp\bigg(\frac{\mathrm{Im}\mathcal V_b(\boldsymbol \alpha, \boldsymbol \xi)}{2\pi b^2}\bigg)|d \boldsymbol \alpha  d \boldsymbol \xi| \\
< & \ O\bigg(e^{\frac{-\mathrm{Vol}(M)-\epsilon_1}{\pi b^2}}\bigg).
\end{split}
\end{equation}
By (\ref{sout}), we proved (I).

We are now left to prove the  claim. For $(\boldsymbol\alpha,\boldsymbol\xi)\in \Gamma^*_{c_0,t_0}\setminus B(\boldsymbol z^*_{c_0,t_0},L_0),$ we  let $(\boldsymbol \alpha',\boldsymbol\xi')$ be the intersection of 
  $\partial B(\boldsymbol z^*_{c_0,t_0},L_0)$ 
   with the line segment connecting $(\boldsymbol\alpha, \boldsymbol\xi)$ and $ (\boldsymbol \alpha_{t_0}^*, \boldsymbol \xi^*_{c_0,t_0}).$  Then 
  \begin{equation}\label{g}
\big|(\boldsymbol\alpha,\boldsymbol\xi)-  (\boldsymbol \alpha_{t_0}^*, \boldsymbol \xi^*_{c_0,t_0})  \big|-L_0=  \big|(\boldsymbol\alpha,\boldsymbol\xi)-  (\boldsymbol \alpha', \boldsymbol \xi')  \big|\to +\infty
  \end{equation}
 as $(\boldsymbol\alpha,\boldsymbol\xi)\to \infty.$ Let $\mathbf u=\frac{(\boldsymbol\alpha,\boldsymbol\xi)-  (\boldsymbol \alpha_{t_0}^*, \boldsymbol \xi^*_{c_0,t_0}) }{|(\boldsymbol\alpha,\boldsymbol\xi)-  (\boldsymbol \alpha_{t_0}^*, \boldsymbol \xi^*_{c_0,t_0}) |}.$   Then by (\ref{Dv5}) and  the choice of $b_0,$ for $l >L_0,$ i.e., $\boldsymbol z^*_{c_0,t_0}+l \mathbf u$ in  $\Gamma^*_{c_0,t_0}\setminus B(\boldsymbol z^*_{c_0,t_0},L_0),$ we have 
$$D_{\mathbf u}\Big(\mathrm{Im}\mathcal U_{c_0}(\boldsymbol z^*_{c_0,t_0}+l \mathbf u)+\mathrm{Im}\kappa(\boldsymbol z^*_{c_0,t_0}+l \mathbf u)b^2\Big)<-2\epsilon+\epsilon=-\epsilon$$
 for all $b<b_0.$ Together with the Mean Value Theorem,  (\ref{Mpartial}) and (\ref{mpartial}), we have
\begin{equation}\label{lM}
\begin{split}
\mathrm{Im}\mathcal U_{c_0}(\boldsymbol \alpha, \boldsymbol \xi)+\mathrm{Im}\kappa(\boldsymbol\alpha,\boldsymbol\xi)b^2 
<  & \mathrm{Im}\mathcal U_{c_0}(\boldsymbol \alpha', \boldsymbol \xi')+\mathrm{Im}\kappa(\boldsymbol\alpha',\boldsymbol\xi')b^2-\epsilon\big|(\boldsymbol\alpha,\boldsymbol\xi)-  (\boldsymbol \alpha', \boldsymbol \xi')  \big|\\
<  &  -2\mathrm{Vol}(M)-3\epsilon_0  - \epsilon\Big( \big|(\boldsymbol\alpha,\boldsymbol\xi)-  (\boldsymbol \alpha^*_{t_0}, \boldsymbol \xi^*_{c_0,t_0}) \big|- L_0 \Big) 
\end{split}
\end{equation}
for all for all $(\boldsymbol\alpha,\boldsymbol \xi)\in \Gamma^*_{c_0,t_0}\setminus B(\boldsymbol z^*_{c_0,t_0},L_0).$  Finally, by Proposition \ref{bound2} and the choice of $b_0$, we have
\begin{equation}\label{last}
\mathrm{Im}\nu_b(\boldsymbol\alpha,\boldsymbol\xi)b^4<Nb^4<\epsilon_0 
\end{equation}
for all $b<b_0$ and for all $(\boldsymbol\alpha,\boldsymbol \xi)\in \Gamma^*_{c_0,t_0}.$ 
 Putting (\ref{VUc5}), (\ref{lM}) and (\ref{last}) together, we have the  inequality in (\ref{I2}); and by (\ref{g}) we have the  equality in (\ref{I2}). This completes the proof of the claim (\ref{I2}), and hence the proof of (I).
\\

To prove (II), for each  $\boldsymbol\alpha\in B(\boldsymbol \alpha^*_{t_0}, L_0)$ and $\Delta\in T,$ we will choose the  contour $\Gamma^*_{\boldsymbol\alpha_\Delta}$ 
 as follows. Let $\boldsymbol l _{\boldsymbol\alpha_\Delta}=2\mathrm{Im}(\boldsymbol \alpha_\Delta).$ If $\boldsymbol l_{\boldsymbol\alpha_\Delta}\in \mathcal L \cup \partial \mathcal L,$ then we let
 $$\xi^*_\Delta(\boldsymbol\alpha_\Delta)=\xi^* (\boldsymbol\alpha_\Delta)$$
 be the unique critical point of $U_{\boldsymbol\alpha_\Delta}$ in $\overline D$ guaranteed by Proposition \ref{critical2} (1) (2);  and if $\boldsymbol l_{\boldsymbol\alpha_\Delta} \in {\mathbb R_{>0}}^6\setminus ( \mathcal L \cup \partial \mathcal L),$ 
 then we let 
 $$\xi^*_\Delta(\boldsymbol\alpha_\Delta)=\xi^*_1(\boldsymbol\alpha_\Delta)\cup \xi^*_2(\boldsymbol\alpha_\Delta)$$
be the union of the two critical points 
 $\xi^*_1(\boldsymbol\alpha_\Delta)$ and $\xi^*_2(\boldsymbol\alpha_\Delta)$ of $U_{\boldsymbol\alpha_\Delta}$ on $\partial\Delta$ guaranteed by Proposition \ref{critical2} (3).  
 In the latter case,  by Proposition \ref{critical2} (3) again, we have 
 $$\mathrm {Re} \xi^*_1(\boldsymbol\alpha_\Delta) = \mathrm{Re} \xi^*_2(\boldsymbol\alpha_\Delta) = 2\pi$$
 and 
 $$\mathrm{Im}U_{\boldsymbol\alpha_\Delta}\big(\xi^*_1(\boldsymbol\alpha_\Delta)\big)= \mathrm{Im}U_{\boldsymbol\alpha_\Delta}\big(\xi^*_2(\boldsymbol\alpha_\Delta) \big)= -2\widetilde{\mathrm{Cov}}(\boldsymbol l_{\boldsymbol\alpha_\Delta}).$$ 
 As such, we define 
 $$\mathrm{Re} \xi^*_\Delta(\boldsymbol\alpha_\Delta) \doteq 2\pi$$
 and 
 $$\mathrm{Im}U_{\boldsymbol\alpha_\Delta}\big(\xi^*_\Delta(\boldsymbol\alpha_\Delta)\big) \doteq -2\widetilde{\mathrm{Cov}}(\boldsymbol l_{\boldsymbol\alpha_\Delta})$$
 for each $\Delta\in T$ with $\boldsymbol l_{\boldsymbol\alpha_\Delta} \in {\mathbb R_{>0}}^6\setminus (\mathcal L\cup\partial \mathcal L).$

Then we let
$$\Gamma_{\boldsymbol\alpha_\Delta}=\Big\{\xi_\Delta \in \overline D\ \Big|\ \mathrm{Re}\xi_\Delta=\mathrm{Re}\xi^*_\Delta(\boldsymbol\alpha_\Delta)\Big \},$$
and let 
$$ \Gamma =  \bigg\{ (\boldsymbol \alpha,\boldsymbol \xi)\in B(\boldsymbol \alpha^*_{t_0}, L_0) \times \overline D^T   \ \bigg|\ \xi_\Delta\in\Gamma_{\boldsymbol\alpha_\Delta}\text{ for all }\Delta\text{ in }T\bigg\}.
$$
We notice that, for a $\Delta\in T,$  if $\boldsymbol l_{\boldsymbol\alpha_\Delta}\notin \mathcal L,$ then $\Gamma_{\boldsymbol\alpha_\Delta}$ passes through $\eta_j(\boldsymbol\alpha_\Delta)$'s which are the singular points of $U_{\boldsymbol\alpha_\Delta}.$  As a consequence,  if  there are flat tetrahedra in $T,$ then $\Gamma$ will pass through the singular points of $\mathcal W.$  Hence we need to deform $\Gamma$ to a new contour  $\Gamma'$ to avoid these singular points.
By  the compactness of $B(\boldsymbol \alpha^*_{t_0}, L_0),$ there is a $\delta>0$ such that Corollary \ref{limder2} holds for any $\boldsymbol\alpha\in B(\boldsymbol \alpha^*_{t_0}, L_0)$ and for any $\xi_\Delta$ in the $\delta$-square neighborhood of $\eta_{j}(\boldsymbol\alpha_\Delta)$'s.  Now for $\boldsymbol\alpha\in B(\boldsymbol \alpha^*_{t_0}, L_0)$ and $\Delta\in T,$  we consider a deformation $\Gamma'_{\boldsymbol\alpha_\Delta}$ of $\Gamma_{\boldsymbol\alpha_\Delta}$ defined as follows. If $\Gamma_{\boldsymbol\alpha_\Delta}$ does not intersect the complement of $D^{\boldsymbol\alpha_\Delta}_{\delta,c},$ then we let $\Gamma'_{\boldsymbol\alpha_\Delta}=\Gamma_{\boldsymbol\alpha_\Delta}$ (see (1) of Figure \ref{GA}); and if $\Gamma_{\boldsymbol\alpha_\Delta}$  intersects the complement of $D^{\boldsymbol\alpha_\Delta}_{\delta,c},$ then we let  $\Gamma'_{\boldsymbol\alpha_\Delta}$ be the contour obtained from $\Gamma_\Delta$ by pushing the parts out of $D^{\boldsymbol\alpha_\Delta}_{\delta,c}$ into the boundary of $D^{\boldsymbol\alpha_\Delta}_{\delta,c}$ (see (2), (3) of Figure \ref{GA}). 
 
\begin{figure}[htbp]
\centering
\includegraphics[scale=0.175]{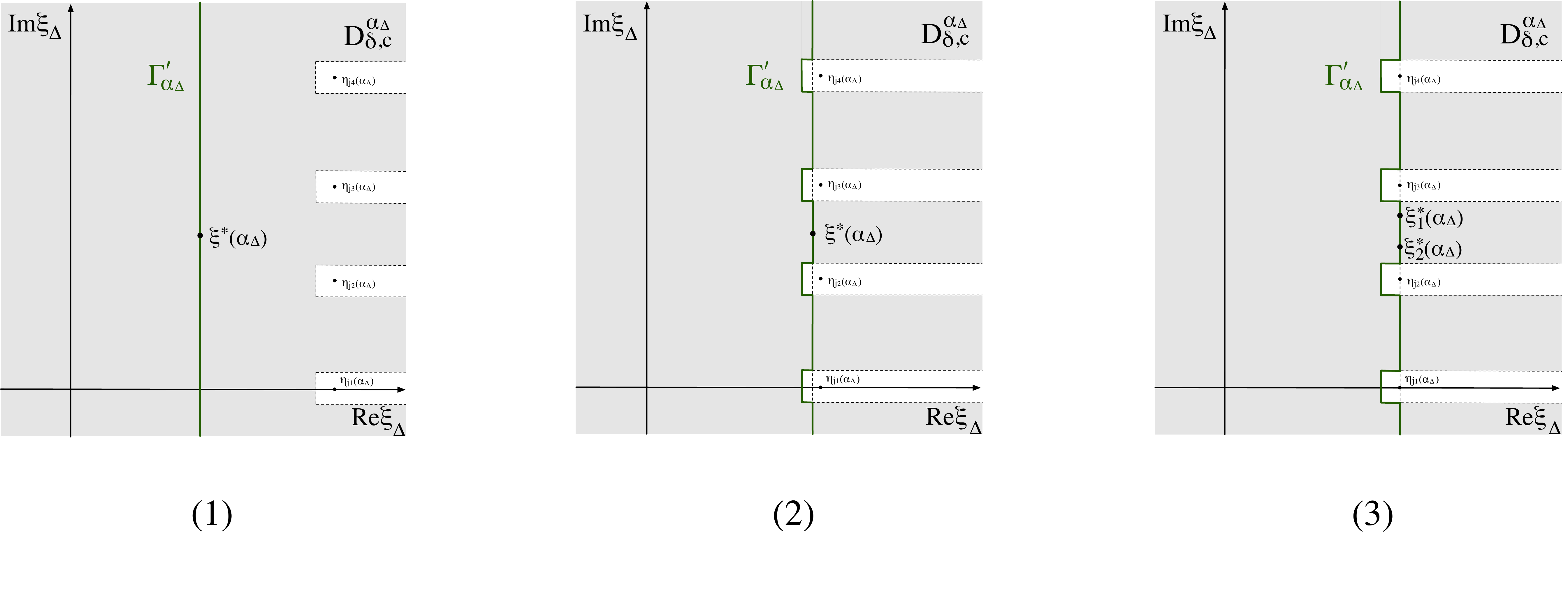}
\caption{Contour $\Gamma'_{\boldsymbol\alpha_\Delta}$ when: (1) $\boldsymbol l_{\boldsymbol\alpha_\Delta}\in \mathcal L$ and $\Gamma_{\boldsymbol\alpha_\Delta}$ does not intersect the complement of  $D^{\boldsymbol\alpha_\Delta}_{\delta,c},$ (2)  $\boldsymbol l_{\boldsymbol\alpha_\Delta} \in \mathcal L$ and $\Gamma_{\boldsymbol\alpha_\Delta}$ intersects the complement of $D^{\boldsymbol\alpha_\Delta}_{\delta,c},$ and (3)  $\boldsymbol l_{\boldsymbol\alpha_\Delta} \notin \mathcal L.$ }
\label{GA}
\end{figure}

Let 
$$\Gamma'= \bigg\{ (\boldsymbol \alpha,\boldsymbol \xi)\in B(\boldsymbol \alpha^*_{t_0}, L_0) \times \overline D^T   \ \bigg|\ \xi_\Delta\in\Gamma'_{\boldsymbol\alpha_\Delta}\text{ for all }\Delta\text{ in }T\bigg\}.$$ 
Then $\Gamma' $ passes through the critical point $(\boldsymbol\alpha^*,\boldsymbol\xi^*)$ of $\mathcal W,$ and by Proposition \ref{critical2} (1) (2) (3), Proposition \ref{Covconcave} and Proposition \ref{cpcv5}, we have 
\begin{equation}\label{nuiq}
\mathrm{Im}\mathcal W (\boldsymbol\alpha,\boldsymbol\xi) \leqslant   \sum_{e\in E} 2\pi l_e + \sum_{\Delta\in T}U\big(\boldsymbol\alpha_\Delta,\xi^*_\Delta(\boldsymbol\alpha_\Delta)\big)= -2  \widetilde{\mathrm{Cov}}(\boldsymbol l) \leqslant -2\mathrm{Vol}(M) = \mathrm{Im}\mathcal W(\boldsymbol\alpha^*,\boldsymbol\xi^*)
\end{equation}
for all $ (\boldsymbol \alpha,\boldsymbol \xi)\in \Gamma'\setminus \{(\boldsymbol\alpha^*,\boldsymbol\xi^*)\},$ with the equality held when $\boldsymbol\alpha=\boldsymbol\alpha^*,$ $\xi_\Delta=\xi^*_\Delta$ for all  non-flat $\Delta$ in $T,$ and $\xi_\Delta$ lies on the line segment connecting $\eta^*_{j_2,\Delta}$ and $\eta^*_{j_3,\Delta}$ for all flat $\Delta$ in $T.$  From this we see that, when there are flat tetrahedra in $T,$  the critical point $(\boldsymbol\alpha^*,\boldsymbol\xi^*)$ is not  the unique maximum point of $\mathrm{Im}\mathcal W$ on  $\Gamma'.$ Therefore, to be able to  apply Proposition \ref{saddle}, the Saddle Point Approximation,  we need to further deform $\Gamma'$  to decrease the value of  $\mathrm{Im}\mathcal W.$ To do this, by Proposition \ref{critical2} (1) (2), Proposition \ref{concave} (1), Proposition \ref{limder} and  the compactness of $B(\boldsymbol \alpha^*_{t_0}, L_0)$ we can choose an $L_1>L_0$ large  enough so that: For a given $\epsilon>0,$ and for  any $\boldsymbol\alpha\in B(\boldsymbol \alpha^*_{t_0}, L_0),$ $\Delta\in T$  and $\xi_\Delta\in \Gamma'_{\alpha_\Delta}$ with $|\mathrm {Im}\xi_\Delta - \mathrm {Im} \xi^*_{\Delta}| > L_1,$
\begin{enumerate}[(a)]
\item  \begin{equation}\label{447}
\mathrm{Im}U_{\boldsymbol\alpha_\Delta}(\xi_\Delta) < \mathrm{Im}U_{\boldsymbol\alpha_\Delta}\big(\xi^*_\Delta(\boldsymbol\alpha_\Delta) \big)-4\epsilon;
\end{equation}
\item if $\mathrm {Im}\xi_\Delta >  \mathrm {Im} \xi^*_{\Delta} + L_1,$ then 
\begin{equation}\label{48}
\frac{\partial \mathrm{Im}U_{\boldsymbol\alpha_\Delta}(\xi_\Delta)}{\partial \mathrm{Im}\xi_\Delta} < -2\pi,
\end{equation}
and if $\mathrm {Im}\xi_\Delta <   \mathrm {Im} \xi^*_{\Delta}  - L_1,$ then 
\begin{equation}\label{49}
\frac{\partial \mathrm{Im}U_{\boldsymbol\alpha_\Delta}(\xi_\Delta)}{\partial \mathrm{Im}\xi_\Delta} > 2\pi;
\end{equation}
\item and for  $\delta>0$ as above, 
\begin{equation}\label{ray}
[\mathrm{Im}\eta_{j_1}(\boldsymbol\alpha_\Delta)-\delta, \mathrm{Im}\eta_{j_4}(\boldsymbol\alpha_\Delta) + \delta ] \subset  \big(\mathrm{Im}\xi^*_{t_0, \Delta} - L_1 + \delta,  \mathrm{Im}\xi^*_{t_0, \Delta}+ L_1-\delta\big).
\end{equation}
\end{enumerate}

 Now for each $\boldsymbol\alpha \in  B(\boldsymbol \alpha^*_{t_0}, L_0)$ and $\Delta\in T,$ we consider the following smooth vector filed  $\psi_\Delta\mathbf v_\Delta $ on  
$ D_{\delta,c}^{\boldsymbol\alpha_\Delta}, $
 where 
 $\psi_\Delta$ is a 
$C^\infty$-smooth bump function on $D_{\delta,c}^{\boldsymbol\alpha_\Delta}$ satisfying 
  \begin{equation*}
\left\{
    \begin{array}{rcl}
 \psi_\Delta(\xi_\Delta) = 1  & \text{if} & |\mathrm{Im}\xi_\Delta -\mathrm{Im}\xi^*_{t_0,\Delta}| \leqslant  L_1-\delta  \\
    0 <   \psi_\Delta(\xi_\Delta) <1 & \text{if}  & L_1-\delta <  |\mathrm{Im}\xi_\Delta -\mathrm{Im}\xi^*_{t_0,\Delta}| < L_1 \\
 \psi_\Delta(\xi_\Delta)  = 0  & \text{if} & |\mathrm{Im}\xi_\Delta -\mathrm{Im}\xi^*_{t_0,\Delta}| \geqslant  L_1
    \end{array}\right.
\end{equation*}
with $\delta>0$ as above, and $\mathbf v_\Delta$ is the vector field on $ D_{\delta,c}^{\boldsymbol\alpha_\Delta}$ defined by  
$$\mathbf v_\Delta=\bigg(-\frac{\partial\mathrm{Im} U_{\boldsymbol\alpha_\Delta}}{\partial \mathrm{Re}\xi_\Delta}\bigg|_{\xi_\Delta},0 \bigg)=\bigg(-\frac{\partial\mathrm{Im}\mathcal W}{\partial \mathrm{Re}\xi_\Delta}\bigg|_{(\boldsymbol\alpha,\boldsymbol\xi)},0 \bigg).$$
Now let $\Gamma^*_{\boldsymbol\alpha_\Delta}$ be the contour obtained from $\Gamma'_{\boldsymbol\alpha_\Delta}$ by following the flow lines of  $\psi_\Delta \mathbf v_\Delta$ for a short time $t>0.$ See Figure \ref{GA*}. 
Then by Corollary \ref{cor4} and the choice of $\delta,$ 
$\Gamma^*_{\boldsymbol\alpha_\Delta}$ stays inside $D_{\delta,c}^{\boldsymbol\alpha_\Delta},$ and by Proposition \ref{nonvanish},
\begin{equation}\label{449}
\mathrm{Im}U_{\boldsymbol\alpha_\Delta}(\xi_\Delta) < \mathrm{Im}U_{\boldsymbol\alpha_\Delta}\big(\xi^*_\Delta(\boldsymbol\alpha_\Delta)\big)
\end{equation}
for all $\xi_\Delta \in \Gamma^*_{\boldsymbol\alpha_\Delta} \setminus \{\xi_\Delta^*(\boldsymbol\alpha_\Delta)\}.$

\begin{figure}[htbp]
\centering
\includegraphics[scale=0.175]{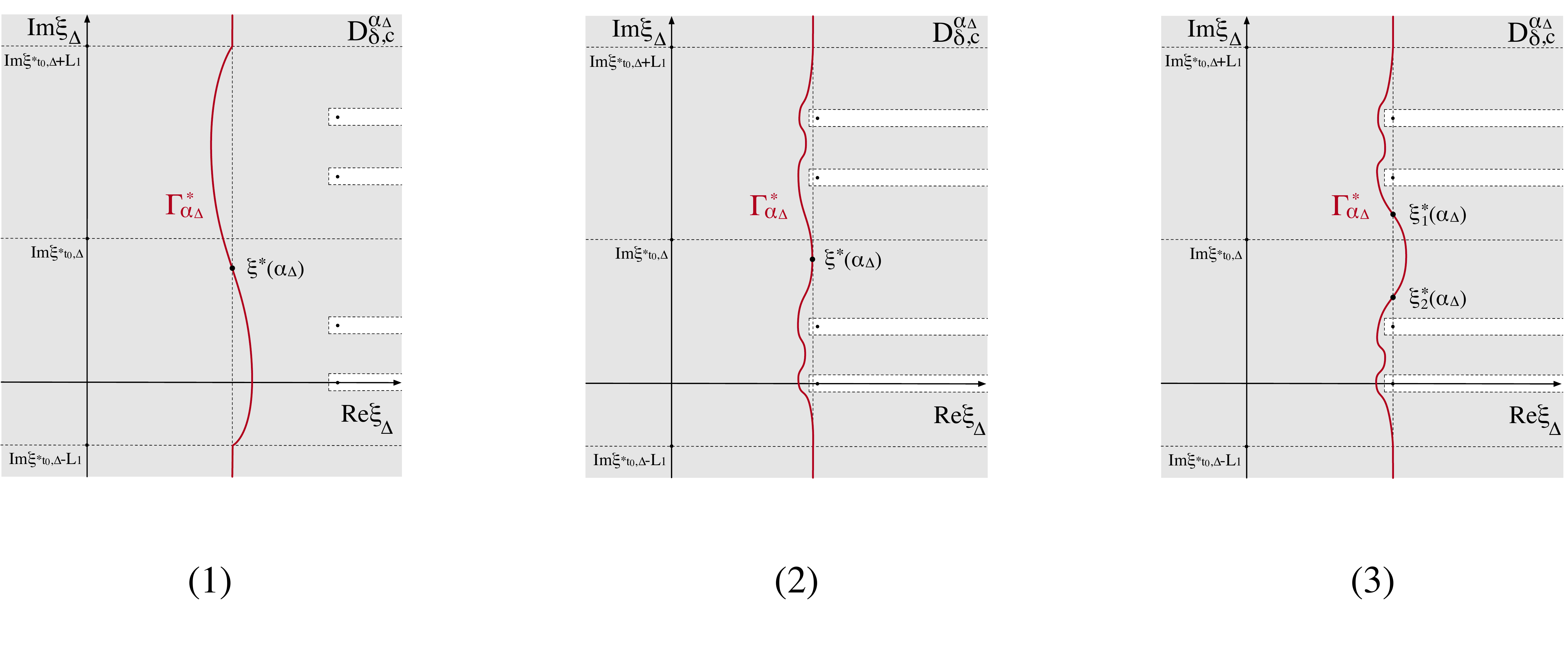}
\caption{Contour $\Gamma^*_{\boldsymbol\alpha_\Delta}$ when: (1) $\boldsymbol l_{\boldsymbol\alpha_\Delta}\in \mathcal L$ and $\Gamma_{\boldsymbol\alpha_\Delta}$ does not intersect the complement of  $D^{\boldsymbol\alpha_\Delta}_{\delta,c},$ (2)  $\boldsymbol l_{\boldsymbol\alpha_\Delta} \in \mathcal L$ and $\Gamma_{\boldsymbol\alpha_\Delta}$ intersects the complement of $D^{\boldsymbol\alpha_\Delta}_{\delta,c}$ or  $\boldsymbol l_{\boldsymbol\alpha_\Delta} \in \partial \mathcal L,$ and (3)  $\boldsymbol l_{\boldsymbol\alpha_\Delta} \in {\mathbb R_{>0}}^6\setminus (\mathcal L\cup \partial \mathcal L).$ }
\label{GA*}
\end{figure}

Let 
$$\Gamma^*_{\text{in}}=\bigg\{  (\boldsymbol \alpha,\boldsymbol \xi)\in B(\boldsymbol \alpha^*_{t_0}, L_0) \times \mathbb C ^T\ \bigg|\  \xi_\Delta \in \Gamma^*_{\boldsymbol\alpha_\Delta} \text { for all } \Delta \text{ in } T \bigg\}.$$
Then  $\Gamma^*_{\text{in}}$ is obtained from   $\Gamma'$ by following the flow lines of the smooth vector field 
$$\mathbf v = \bigg((0)_{e\in E}, \big(\psi_\Delta\mathbf v_\Delta\big)_{\Delta\in T} \bigg)$$
 on the region $\mathcal D_{\delta,c}=\big\{ (\boldsymbol\alpha, \boldsymbol \xi)\in \big(\frac{\pi}{2}+\mathbf i{\mathbb R_{>0}}\big)^E\times \mathbb C^T\ \big|\ \xi_\Delta\in D^{\boldsymbol \alpha_\Delta}_{\delta,c}\text{ for all }\Delta\in T \big\}$ defined before Proposition \ref{bound4} for a short time $t>0,$ hence is a connected domain, and is  smoothly embedded in  $\mathcal D_{\delta,c}$ near the critical point $(\boldsymbol\alpha^*,\boldsymbol\xi^*)$ of $\mathcal W;$  and by (\ref{449}), Proposition \ref{critical2} (1) (2) (3) and Proposition \ref{Covconcave}, we have 
\begin{equation}\label{450}
\mathrm{Im}\mathcal W (\boldsymbol\alpha,\boldsymbol\xi) \leqslant   \sum_{e\in E} 2\pi l_e + \sum_{\Delta\in T}U\big(\boldsymbol\alpha_\Delta,\xi^*_\Delta(\boldsymbol\alpha_\Delta)\big)= -2  \widetilde{\mathrm{Cov}}(\boldsymbol l)< -2\mathrm{Vol}(M)
\end{equation}
for all $(\boldsymbol\alpha,\boldsymbol\xi)\in \Gamma^*_{\text{in}}\setminus \{(\boldsymbol\alpha^*,\boldsymbol\xi^*)\}.$
\\

By (\ref{Iin}), we have 
\begin{equation}\label{Gin}
\mathrm I_{\text{in}} = \frac{(-\mathbf i)^{|E|}}{(\pi b )^{|E|+|T|}} \int_{\Gamma^*_{\text{in}}}\bigg(\prod_{e\in E}4\sinh l_e\sinh \frac{l_e}{b^2} \bigg)\exp\Bigg(\frac{\sum_{\Delta\in T}U_b(\boldsymbol \alpha_\Delta, \xi_\Delta)}{2\pi \mathbf i b^2}\Bigg)d \boldsymbol \alpha  d \boldsymbol \xi.
\end{equation}
To estimate this integral, we let $\delta>0$ be as above, and let $d_0>0$ be sufficiently small so that for any $\boldsymbol\alpha\in \Big(\Big[\frac{\pi}{2}-\delta, \frac{\pi}{2}+\delta\Big]+\mathbf i{\mathbb R_{>0}}\Big)^E$ with $|\boldsymbol \alpha - \boldsymbol \alpha^*|<d_0,$ 
$$\big[\mathrm{Im}\xi^*_\Delta - d_0, \mathrm{Im}\xi^*_\Delta+ d_0\big]\subset \big(\mathrm{Im}\eta_{j_2}(\boldsymbol\alpha_\Delta)+\delta, \mathrm{Im}\eta_{j_3}(\boldsymbol\alpha_\Delta)-\delta\big)$$
for all $\Delta$ in $T.$ Let $\boldsymbol D_{\delta,c}=\big\{ (\boldsymbol \alpha,\boldsymbol \xi)\in  \big([\frac{\pi}{2}-\delta,\frac{\pi}{2}+\delta ]+\mathbf i{\mathbb R_{>0}}\big)^E\times D^T\ \big|\ \xi_\Delta\in D^{\boldsymbol \alpha_\Delta}_{\delta,c}\text{ for all }\Delta\in T \big\}$ be the region defined before Proposition \ref{bound2}, let 
$$\boldsymbol B_{d_0}=\bigg\{ (\boldsymbol\alpha,\boldsymbol\xi)\in  \boldsymbol D_{\delta,c} \ \bigg|\  |\boldsymbol \alpha - \boldsymbol \alpha^*|<d_0, |\mathrm{Re}\xi_\Delta- \mathrm{Re}\xi^*_\Delta|< d_0 \text{ and }   |\mathrm{Im}\xi_\Delta- \mathrm{Im}\xi^*_\Delta| < d_0  \text{ for all }\Delta \text{ in } T  \bigg\},$$
and let 
$$\Gamma^*_{d_0}=\Gamma^*_{\text{in}}\cap \boldsymbol B_{d_0}.$$
We claim that 
\begin{enumerate}[(i)]
\item \begin{equation} \label{Int2}
\begin{split}
& \frac{(-\mathbf i)^{|E|}}{(\pi b )^{|E|+|T|}} \int_{\Gamma^*_{d_0}}\bigg(\prod_{e\in E}4\sinh l_e\sinh \frac{l_e}{b^2} \bigg)\exp\Bigg(\frac{\sum_{\Delta\in T}U_b(\boldsymbol \alpha_\Delta, \xi_\Delta)}{2\pi \mathbf i b^2}\Bigg)d \boldsymbol \alpha  d \boldsymbol \xi\\
  = & (2\mathbf i)^{\frac{\chi(M)}{2}}\frac{e^{\frac{-\mathrm{Vol}(M)}{\pi b^2}}}{\sqrt{\pm \mathrm{Tor}(DM, \mathrm{Ad}_{\rho_{DM}})}}\Big(1+O\big(b^2\big)\Big),
 \end{split}
\end{equation}
and 
\item \begin{equation} \label{Int3}
\begin{split}
 \frac{(-\mathbf i)^{|E|}}{(\pi b )^{|E|+|T|}} & \int_{\Gamma^*_{\text{in}}\setminus \Gamma^*_{d_0}}\bigg(\prod_{e\in E}4\sinh l_e\sinh \frac{l_e}{b^2} \bigg)\exp\Bigg(\frac{\sum_{\Delta\in T}U_b(\boldsymbol \alpha_\Delta, \xi_\Delta)}{2\pi \mathbf i b^2}\Bigg)d \boldsymbol \alpha  d \boldsymbol \xi \\
 < & O\bigg(e^{\frac{-\mathrm{Vol}(M)-\epsilon_2}{\pi b^2}}\bigg)
\end{split}
\end{equation}
for some $\epsilon_2>0.$
\end{enumerate}
Then putting (\ref{Gin}), (i) and (ii) together, we have (II).
\\

To see (i), we have
\begin{equation}\label{Isum}
\begin{split}
& \frac{(-\mathbf i)^{|E|}}{(\pi b )^{|E|+|T|}} \int_{\Gamma^*_{d_0}}\bigg(\prod_{e\in E}4\sinh l_e\sinh \frac{l_e}{b^2} \bigg)\exp\Bigg(\frac{\sum_{\Delta\in T}U_b(\boldsymbol \alpha_\Delta, \xi_\Delta)}{2\pi \mathbf i b^2}\Bigg)d \boldsymbol \alpha  d \boldsymbol \xi\\
 =&\sum_{\boldsymbol \lambda\in\{-1,1\}^E}\frac{(-2\mathbf i)^{|E|}}{(\pi b )^{|E|+|T|}} \int_{\Gamma^*_{d_0}}\bigg(\prod_{e\in E}\lambda_e\sinh l_e\bigg)\exp\Bigg(\frac{\mathcal W^{\boldsymbol \lambda}_b(\boldsymbol \alpha, \boldsymbol \xi)}{2\pi \mathbf i b^2}\Bigg)d \boldsymbol \alpha  d \boldsymbol \xi.
\end{split}
\end{equation}

For $\boldsymbol \lambda = (1,\dots, 1),$ we claim that all the conditions of Proposition \ref{saddle} are satisfied by letting $\hbar=b^2,$ $D= \boldsymbol  B_{d_0},$ $f=\frac{\mathcal W}{2\pi \mathbf i},$ $g$ be defined by $g(\boldsymbol z)=\prod_{e\in E}\sinh l_e\exp\big(\frac{\kappa(\boldsymbol z)}{2\pi \mathbf i }\big),$ $f_\hbar=\frac{\mathcal W+\nu_b b^4}{2\pi \mathbf i},$ $\upsilon_{\hbar}=\frac{\nu_b}{2\pi \mathbf i},$ $S=\Gamma^*_{d_0}$ and $c=\boldsymbol z^*= (\boldsymbol \alpha^*,\boldsymbol \xi^*).$ 

Indeed, by Proposition \ref{cpcv5}, $\boldsymbol z^*$ is a critical point of $f=\frac{\mathcal W}{2\pi \mathbf i}$ in $ \boldsymbol  B_{d_0},$ hence condition (i) is satisfied. 

By Proposition \ref{cpcv5} and (\ref{450}),  $\boldsymbol z^*$ is the unique maximum point of $\mathrm{Re}f=\frac{\mathrm{Im}\mathcal W}{2\pi}$ on $\Gamma^*_{d_0},$ hence condition (ii) is satisfied. 

By Proposition \ref{HessTor2}, $\det(-\mathrm{Hess}\mathcal W(\boldsymbol z^*))\neq 0,$ and condition (iii) is satisfied. 

Since $\xi_\Delta^*\in D_{\delta,c}^{\boldsymbol\alpha^*_\Delta}$ for any $\Delta\in T,$  $\kappa(\boldsymbol z^*)$ is a finite value; and since $l^*_e>0$ for each $e\in E,$  $g(\boldsymbol z^*)=\prod_{e\in E}\sinh l^*_e\exp\big(\frac{\kappa(\boldsymbol z^*)}{2\pi \mathbf i }\big)\neq 0,$ and condition (iv) is satisfied.

For condition (v), by Proposition \ref{bound2},  $|\upsilon_{\hbar}(\boldsymbol \alpha,\boldsymbol \xi)|=\big|\frac{\nu_{b}(\boldsymbol \alpha,\boldsymbol \xi)}{2\pi \mathbf i}\big|<\frac{N}{2\pi}$ on $ \boldsymbol  B_{d_0}.$

For condition (vi),  since $\Gamma^*_{\text{in}}$  is smoothly embedded around $\boldsymbol z^*,$ $\Gamma^*_{d_0}=\Gamma^*_{\text{in}}\cap \boldsymbol B_{d_0}$ is smoothly embedded around $\boldsymbol z^*.$

Finally, by Proposition \ref{saddle}, Proposition \ref{cpcv5}, Proposition \ref{HessTor2} and the fact that $\chi(M)=-|E|+|T|,$ we have as $b\to 0,$
\begin{equation}\label{eps1}
\begin{split}
& \frac{(-2\mathbf i)^{|E|}}{(\pi b )^{|E|+|T|}} \int_{\Gamma^*_{d_0}}\bigg(\prod_{e\in E}\sinh l_e\bigg)\exp\Bigg(\frac{\mathcal W_b(\boldsymbol \alpha, \boldsymbol \xi)}{2\pi \mathbf i b^2}\Bigg)d \boldsymbol \alpha  d \boldsymbol \xi\\
= &\frac{(-2\mathbf i)^{|E|}}{(\pi b )^{|E|+|T|}}
\int_{\Gamma^*_{d_0}}\bigg(\prod_{e\in E}\sinh l_e \bigg)\exp\bigg(\frac{\kappa(\boldsymbol\alpha,\boldsymbol\xi)}{2\pi \mathbf i}\bigg)\exp\bigg(\frac{\mathcal W(\boldsymbol \alpha, \boldsymbol \xi)+\nu_b(\boldsymbol\alpha,\boldsymbol\xi)b^4}{2\pi \mathbf i b^2}\bigg)d \boldsymbol \alpha  d \boldsymbol \xi\\
=& \frac{(-2\mathbf i)^{|E|}   (2\pi b^2)^\frac{|E|+|T|}{2}}{(\pi b )^{|E|+|T|}}
\frac{\exp\big(\frac{\kappa(\boldsymbol z^*)}{2\pi \mathbf i}\big)\prod_{e\in E} \sinh l_e}{\sqrt{\det\Big(-\mathrm{Hess}\frac{\mathcal W(\boldsymbol z^*)}{2\pi \mathbf i}\Big)}}e^{\frac{\mathcal W(\boldsymbol z^*)}{2\pi \mathbf i b^2}}\Big(1+O\big(b^2\big)\Big)\\
=& (2\mathbf i)^{\frac{\chi(M)}{2}}\frac{e^{\frac{-\mathrm{Vol}(M)}{\pi b^2}}}{\sqrt{\pm \mathrm{Tor}(DM, \mathrm{Ad}_{\rho_{DM}})}}\Big(1+O\big(b^2\big)\Big).
\end{split}
\end{equation}

For any $\boldsymbol \lambda\neq (1,\dots,1),$ we have
$$\mathrm{Im}\mathcal W_b^{\boldsymbol \lambda}(\boldsymbol\alpha,\boldsymbol\xi)-\mathrm{Im}\mathcal W_b (\boldsymbol\alpha,\boldsymbol\xi) = -2\pi\sum_{e\in E}\big(1-\lambda_e\big)l_e< 0$$
for all for $(\boldsymbol \alpha,\boldsymbol \xi)$ in $\Gamma^*_{d_0};$ and by the compactness of $\Gamma^*_{d_0},$ there is an $\epsilon>0$ such that 
$$\mathrm{Im}\mathcal W_b^{\boldsymbol \lambda}(\boldsymbol\alpha,\boldsymbol\xi) < \mathrm{Im}\mathcal W_b (\boldsymbol\alpha,\boldsymbol\xi)  -2\epsilon$$
for all $(\boldsymbol \alpha,\boldsymbol \xi)$ in $\Gamma^*_{d_0}.$ Then  by (\ref{eps1}), for any $\boldsymbol \lambda\neq (1,\dots,1),$
we have
\begin{equation}\label{Int8}
\begin{split}
&\Bigg| \frac{(-2\mathbf i)^{|E|}}{(\pi b )^{|E|+|T|}} \int_{\Gamma^*_{d_0}}\bigg(\prod_{e\in E}\lambda_e\sinh l_e\bigg)\exp\Bigg(\frac{\mathcal W^{\boldsymbol \lambda}_b(\boldsymbol \alpha, \boldsymbol \xi)}{2\pi \mathbf i b^2}\Bigg)d \boldsymbol \alpha  d \boldsymbol \xi\Bigg|\\
\leqslant  & \frac{2^{|E|}}{(\pi b)^{|E|+|T|}}
\int_{\Gamma^*_{d_0}}\bigg(\prod_{e\in E}\sinh l_e \bigg)\exp\bigg(\frac{ \mathrm{Im}\mathcal W^{\boldsymbol \lambda}_b(\boldsymbol \alpha, \boldsymbol \xi)}{2\pi  b^2}\bigg)\big|d \boldsymbol \alpha  d \boldsymbol \xi\big|\\
<  & \frac{2^{|E|}}{(\pi b)^{|E|+|T|}}
\int_{\Gamma^*_{d_0}}\bigg(\prod_{e\in E}\sinh l_e \bigg)\exp\bigg(\frac{ \mathrm{Im}\mathcal W_b(\boldsymbol \alpha, \boldsymbol \xi)-2\epsilon}{2\pi  b^2}\bigg)\big|d \boldsymbol \alpha  d \boldsymbol \xi\big|\\
<& O\bigg(e^{\frac{-\mathrm{Vol}(M)-\epsilon}{\pi b^2}}\bigg).
\end{split}
\end{equation}

Putting (\ref{Isum}), (\ref{eps1}) and (\ref{Int8}) together, we have (\ref{Int2}), completing the proof of (i). 
\\

To see (ii), for each $\boldsymbol \alpha\in B(\boldsymbol \alpha^*_{t_0}, L_0)$ and   $\Delta\in T,$ we let 
$$\Gamma^*_{\boldsymbol\alpha_\Delta,L_1}=\Big\{ \xi_\Delta \in \Gamma^*_{\boldsymbol\alpha_\Delta} \ \Big|\ |\mathrm{Im}\xi_\Delta- \mathrm{Im}\xi^*_{t_0,\Delta}|\leqslant L_1 \Big\}.$$
Then $\Gamma^*_{\boldsymbol\alpha_\Delta}\setminus \Gamma^*_{\boldsymbol\alpha_\Delta,L_1}$ is the disjoint union of two  rays. 
Let 
\begin{equation*}
\begin{split}
\Gamma^*_{L_1,T} = & \bigg\{ (\boldsymbol\alpha,\boldsymbol\xi)\in \Gamma^*_{\text{in}}\ \bigg|\ \boldsymbol\xi \in \prod_{\Delta\in T}\Gamma^*_{\boldsymbol\alpha_\Delta,L_1}\bigg\}\\
=&\bigg\{ (\boldsymbol\alpha,\boldsymbol\xi)\in \Gamma^*_{\text{in}}\ \bigg|\ \xi_\Delta\in \Gamma^*_{\boldsymbol\alpha_\Delta,L_1} \text{ for all } \Delta \text{ in } T\bigg\},
\end{split}
\end{equation*}
and for any proper  subset $T_0$ of $T,$ we let 
\begin{equation*}
\begin{split}
\Gamma^*_{L_1,T_0} = & \bigg\{ (\boldsymbol\alpha,\boldsymbol\xi)\in \Gamma^*_{\text{in}}\ \bigg|\ \boldsymbol\xi \in \prod_{\Delta\in T_0}\Gamma^*_{\boldsymbol\alpha_\Delta,L_1} \times \prod_{\Delta\notin T_0}\big(\Gamma^*_{\boldsymbol\alpha_\Delta}\setminus \Gamma^*_{\boldsymbol\alpha_\Delta,L_1}\big)\bigg\}\\
=&\bigg\{ (\boldsymbol\alpha,\boldsymbol\xi)\in \Gamma^*_{\text{in}}\ \bigg|\ \xi_\Delta\in \Gamma^*_{\boldsymbol\alpha_\Delta,L_1} \text{ for all } \Delta \text{ in } T_0, \text{ and } \xi_\Delta\notin \Gamma^*_{\boldsymbol\alpha_\Delta,L_1} \text{ for all } \Delta \text{ not in } T_0\bigg\}.
\end{split}
\end{equation*}
Then we have the following partition 
$$\Gamma^*_{\text{in}} = \Gamma^*_{L_1,T}\cup \bigg( \bigcup _{T_0\subset T} \Gamma^*_{L_1,T_0}\bigg)$$
of $\Gamma^*_{\text{in}}.$ 
We will show that:
\begin{enumerate}[(\text{ii}.1)]
\item \begin{equation} \label{Int9}
\begin{split}
 \frac{(-\mathbf i)^{|E|}}{(\pi b )^{|E|+|T|}} & \int_{\Gamma^*_{L_1,T}\setminus \Gamma^*_{d_0}}\bigg(\prod_{e\in E}4\sinh l_e\sinh \frac{l_e}{b^2} \bigg)\exp\Bigg(\frac{\sum_{\Delta\in T}U_b(\boldsymbol \alpha_\Delta, \xi_\Delta)}{2\pi \mathbf i b^2}\Bigg)d \boldsymbol \alpha  d \boldsymbol \xi\\
 < & O\bigg(e^{\frac{-\mathrm{Vol}(M)-\epsilon_3}{\pi b^2}}\bigg)
\end{split}
\end{equation}
for some $\epsilon_3>0.$
\item  For any proper subset $T_0$ of $T,$ 
\begin{equation}\label{Int10}
\begin{split}
 \frac{(-\mathbf i)^{|E|}}{(\pi b )^{|E|+|T|}} & \int_{\Gamma^*_{L_1, T_0}}\bigg(\prod_{e\in E}4\sinh l_e\sinh \frac{l_e}{b^2} \bigg)\exp\Bigg(\frac{\sum_{\Delta\in T}U_b(\boldsymbol \alpha_\Delta, \xi_\Delta)}{2\pi \mathbf i b^2}\Bigg)d \boldsymbol \alpha  d \boldsymbol \xi\\
 < & O\bigg(e^{\frac{-\mathrm{Vol}(M)-\epsilon_4}{\pi b^2}}\bigg)
\end{split}
\end{equation}
for some $\epsilon_4>0.$
\end{enumerate}
Then (ii) follows from (ii.1) and (ii.2) with $\epsilon_2= \min\{ \epsilon_3, \epsilon_4\}.$
\\

To see (ii,1), we have
\begin{equation}\label{II1sum}
\begin{split}
& \frac{(-\mathbf i)^{|E|}}{(\pi b )^{|E|+|T|}} \int_{\Gamma^*_{L_1,T}\setminus \Gamma^*_{d_0}}\bigg(\prod_{e\in E}4\sinh l_e\sinh \frac{l_e}{b^2} \bigg)\exp\Bigg(\frac{\sum_{\Delta\in T}U_b(\boldsymbol \alpha_\Delta, \xi_\Delta)}{2\pi \mathbf i b^2}\Bigg)d \boldsymbol \alpha  d \boldsymbol \xi\\
 =&\sum_{\boldsymbol \lambda\in\{-1,1\}^E}\frac{(-2\mathbf i)^{|E|}}{(\pi b )^{|E|+|T|}} \int_{\Gamma^*_{L_1,T}\setminus\Gamma^*_{d_0}}\bigg(\prod_{e\in E}\lambda_e\sinh l_e\bigg)\exp\Bigg(\frac{\mathcal W^{\boldsymbol \lambda}_b(\boldsymbol \alpha, \boldsymbol \xi)}{2\pi \mathbf i b^2}\Bigg)d \boldsymbol \alpha  d \boldsymbol \xi.
\end{split}
\end{equation}
For $\boldsymbol \lambda = (1,\dots, 1),$ by (\ref{450}) and the compactness of  $\Gamma^*_{L_1,T}\setminus \Gamma^*_{d_0},$ there is an $\epsilon_3>0,$ an $M>0$ and an $S>0$ such that 
\begin{equation}\label{3e}
\mathrm{Im}\mathcal W (\boldsymbol\alpha,\boldsymbol\xi) < -2\mathrm{Vol}(M) - 5\epsilon_3
\end{equation}
and 
\begin{equation}\label{M}
\mathrm{Im}\mathcal \kappa (\boldsymbol\alpha,\boldsymbol\xi) < M\quad\text{and}\quad \prod_{e\in E}\sinh l_e<S
\end{equation}
for all  $(\boldsymbol\alpha,\boldsymbol\xi)\in\Gamma^*_{L_1,T}\setminus \Gamma^*_{d_0};$ and by Proposition \ref{bound2},  there is a $b_0>0$ such that 
\begin{equation}\label{e}
\mathrm{Im}\mathcal \nu_b (\boldsymbol\alpha,\boldsymbol\xi) b^4 < Nb^4<\epsilon_3
\end{equation}
for all $b<b_0$ and for all  $(\boldsymbol\alpha,\boldsymbol\xi)\in \Gamma^*_{L_1,T}\setminus \Gamma^*_{d_0}.$ As a consequence of (\ref{3e}), (\ref{M}) and (\ref{e}),
 we have 
 $$\mathrm{Im}\mathcal W_b(\boldsymbol \alpha, \boldsymbol \xi)=\mathrm{Im}\mathcal W(\boldsymbol \alpha, \boldsymbol \xi)+\mathrm{Im}\kappa(\boldsymbol\alpha,\boldsymbol\xi) b^2+ \mathrm{Im}\nu_b(\boldsymbol\alpha,\boldsymbol\xi)b^4 <  -2\mathrm{Vol}(M) - 4\epsilon_3 + Mb^2;$$
 and letting $\mathrm V$ be the volume of $\Gamma^*_{L_1,T}\setminus \Gamma^*_{d_0},$ we have
\begin{equation}\label{eps3}
\begin{split}
&\Bigg| \frac{(-2\mathbf i)^{|E|}}{(\pi b )^{|E|+|T|}} \int_{\Gamma^*_{L_1,T}\setminus \Gamma^*_{d_0}}\bigg(\prod_{e\in E}\sinh l_e\bigg)\exp\Bigg(\frac{\mathcal W_b(\boldsymbol \alpha, \boldsymbol \xi)}{2\pi \mathbf i b^2}\Bigg)d \boldsymbol \alpha  d \boldsymbol \xi\Bigg|\\
\leqslant  & \frac{2^{|E|}}{(\pi b)^{|E|+|T|}}
\int_{\Gamma^*_{L_1,T}\setminus \Gamma^*_{d_0}}\bigg(\prod_{e\in E}\sinh l_e \bigg)\exp\bigg(\frac{ \mathrm{Im}\mathcal W_b(\boldsymbol \alpha, \boldsymbol \xi)}{2\pi  b^2}\bigg)\big|d \boldsymbol \alpha  d \boldsymbol \xi\big|\\
< &\frac{2^{|E|} S\mathrm Ve^{\frac{M}{2\pi}}}{(\pi b)^{|E|+|T|}}O\bigg(e^{\frac{-\mathrm{Vol}(M)-2\epsilon_3}{\pi b^2}}\bigg)\\
<& O\bigg(e^{\frac{-\mathrm{Vol}(M)-\epsilon_3}{\pi b^2}}\bigg).
\end{split}
\end{equation}
For any $\boldsymbol \lambda\neq (1,\dots,1),$ we have
$$\mathrm{Im}\mathcal W_b^{\boldsymbol \lambda}(\boldsymbol\alpha,\boldsymbol\xi)-\mathrm{Im}\mathcal W_b (\boldsymbol\alpha,\boldsymbol\xi) = -2\pi\sum_{e\in E}\big(1-\lambda_e\big)l_e\leqslant  0$$
for all for $(\boldsymbol \alpha,\boldsymbol \xi)$ in $\Gamma^*_{L_1,T}\setminus \Gamma^*_{d_0}.$ Then  by (\ref{eps3}), 
we have
\begin{equation}\label{Int11}
\begin{split}
&\Bigg| \frac{(-2\mathbf i)^{|E|}}{(\pi b )^{|E|+|T|}} \int_{\Gamma^*_{L_1,T}\setminus \Gamma^*_{d_0}}\bigg(\prod_{e\in E}\lambda_e\sinh l_e\bigg)\exp\Bigg(\frac{\mathcal W^{\boldsymbol \lambda}_b(\boldsymbol \alpha, \boldsymbol \xi)}{2\pi \mathbf i b^2}\Bigg)d \boldsymbol \alpha  d \boldsymbol \xi\Bigg|\\
\leqslant  & \frac{2^{|E|}}{(\pi b)^{|E|+|T|}}
\int_{\Gamma^*_{L_1,T}\setminus \Gamma^*_{d_0}}\bigg(\prod_{e\in E}\sinh l_e \bigg)\exp\bigg(\frac{ \mathrm{Im}\mathcal W^{\boldsymbol \lambda}_b(\boldsymbol \alpha, \boldsymbol \xi)}{2\pi  b^2}\bigg)\big|d \boldsymbol \alpha  d \boldsymbol \xi\big|\\
\leqslant  & \frac{2^{|E|}}{(\pi b)^{|E|+|T|}}
\int_{\Gamma^*_{L_1,T}\setminus \Gamma^*_{d_0}}\bigg(\prod_{e\in E}\sinh l_e \bigg)\exp\bigg(\frac{ \mathrm{Im}\mathcal W_b(\boldsymbol \alpha, \boldsymbol \xi)}{2\pi  b^2}\bigg)\big|d \boldsymbol \alpha  d \boldsymbol \xi\big|\\
<& O\bigg(e^{\frac{-\mathrm{Vol}(M)-\epsilon_3}{\pi b^2}}\bigg).
\end{split}
\end{equation}
Putting (\ref{II1sum}), (\ref{eps3}) and (\ref{Int11}) together, we have (\ref{Int9}), completing the proof of (II.1). 
\\

To see (ii.2),  for each  proper subset $T_0$ of $T,$  we have
\begin{equation}\label{II2sum}
\begin{split}
& \frac{(-\mathbf i)^{|E|}}{(\pi b )^{|E|+|T|}} \int_{\Gamma^*_{L_1,T_0}}\bigg(\prod_{e\in E}4\sinh l_e\sinh \frac{l_e}{b^2} \bigg)\exp\Bigg(\frac{\sum_{\Delta\in T}U_b(\boldsymbol \alpha_\Delta, \xi_\Delta)}{2\pi \mathbf i b^2}\Bigg)d \boldsymbol \alpha  d \boldsymbol \xi\\
 =& \frac{(-\mathbf i)^{|E|}}{(\pi b )^{|E|+|T|}}\sum_{\boldsymbol \mu\in\{-1,1\}^E}\sum_{\boldsymbol \lambda\in\{-1,1\}^E}\bigg(\prod_{e\in E}\mu_e\lambda_e\bigg)
\int_{\Gamma^*_{L_1,T_0}}\exp\Bigg(\frac{\mathcal V_b^{\boldsymbol\mu\boldsymbol \lambda}(\boldsymbol \alpha, \boldsymbol \xi)}{2\pi \mathbf i b^2}\Bigg)d \boldsymbol \alpha  d \boldsymbol \xi.
\end{split}
\end{equation}
For $\boldsymbol\mu=\boldsymbol \lambda=(1,\dots, 1),$ we have
$$\mathcal V(\boldsymbol\alpha,\boldsymbol \xi)= \mathcal W (\boldsymbol\alpha,\boldsymbol \xi)+2\pi b^2\sum_{e\in E}(2\alpha_e-\pi),$$
and hence 
\begin{equation}\label{vuc7}
\mathrm{Im} \mathcal V(\boldsymbol\alpha,\boldsymbol \xi)=\mathrm{Im}  \mathcal W(\boldsymbol\alpha,\boldsymbol \xi)+2\pi b^2\sum_{e\in E}l_e
\end{equation}
for all $(\boldsymbol\alpha,\boldsymbol \xi)\in \Gamma^*_{L_1,T_0}.$

Next, by (\ref{447}), (\ref{449}) and Proposition \ref{Covconcave} and letting $\epsilon$ be as in (\ref{447}), we have 
\begin{equation}\label{469}
\begin{split}
\mathrm{Im}\mathcal W (\boldsymbol\alpha,\boldsymbol\xi) \leqslant  & \sum_{e\in E} 2\pi l_e + \sum_{\Delta\in T}U\big(\boldsymbol\alpha_\Delta,\xi^*_\Delta(\boldsymbol\alpha_\Delta)\big)\\
< & -2  \widetilde{\mathrm{Cov}}(\boldsymbol l) - 4\big(|T|-|T_0|\big) \epsilon<   -2\mathrm{Vol}(M) - 4\epsilon
\end{split}
\end{equation} 
for all $(\boldsymbol\alpha,\boldsymbol \xi)\in \Gamma^*_{L_1,T_0},$ where the last inequality comes from that $T_0$ is a proper subset of $T$ hence $|T|-|T_0| \geqslant 1.$ By (\ref{48}) and (\ref{49}), for any $(\boldsymbol\alpha,\boldsymbol \xi)\in \Gamma^*_{L_1,T_0}$ and any unit vector $\mathbf u=\big( (u_e)_{e\in E}, (u_\Delta)_{\Delta\in T}\big)\in \mathbb R^{|E|+|T|}$ with $u_e=0$ for all $e\in E$ and $u_\Delta=0$ for all $\Delta \in T_0$ and  so that the ray  $\{ (\boldsymbol\alpha,\boldsymbol \xi) + l\mathbf u \ |\ l\in (0, +\infty)\}$  stays in $\Gamma^*_{L_1,T_0},$ we have that the directional derivative 
\begin{equation}\label{dv5}
D_{\mathbf u}\mathrm{Im}\mathcal W((\boldsymbol\alpha,\boldsymbol \xi)+l \mathbf u)<-2\pi.
\end{equation}
As a notation, let 
$$\partial  \Gamma^*_{L_1,T_0} \doteq \bigg\{ (\boldsymbol\alpha,\boldsymbol\xi)\in \Gamma^*_{\text{in}}\ \bigg|\ \boldsymbol\xi \in \prod_{\Delta\in T_0}\Gamma^*_{\boldsymbol\alpha_\Delta,L_1} \times \prod_{\Delta\notin T_0}\partial  \Gamma^*_{\boldsymbol\alpha_\Delta,L_1}\bigg\},$$
which is a compact set. Then by the compactness of $\partial  \Gamma^*_{L_1,T_0}$ and $B(\boldsymbol \alpha^*_{t_0}, L_0) ,$  there is a $b_0>0$ such that for all $b<b_0,$ 
\begin{equation}\label{ke}
\mathrm{Im}\kappa (\boldsymbol\alpha,\boldsymbol\xi) b^2 < \epsilon
\end{equation}
for all $(\boldsymbol\alpha,\boldsymbol \xi)\in \partial \Gamma^*_{L_1,T_0},$
\begin{equation}\label{ke2}
2\pi b^2\sum_{e\in E}l_e<\epsilon
\end{equation}
for all $\boldsymbol\alpha \in  B(\boldsymbol \alpha^*_{t_0}, L_0),$ $K|T|b^2<\epsilon$ and $Nb^4<\epsilon,$ where $K$ and $N$ are respectively the constants in Proposition \ref{bound3} and Proposition \ref{bound2}. For $\boldsymbol\xi \in \prod_{\Delta\in T_0}\Gamma^*_{\boldsymbol\alpha_\Delta,L_1} \times \prod_{\Delta\notin T_0}\big(\Gamma^*_{\boldsymbol\alpha_\Delta}\setminus \Gamma^*_{\boldsymbol\alpha_\Delta,L_1}\big),$ we let $\boldsymbol\xi_{\text{proj}}=(\xi_{\text{proj},\Delta})_{\Delta\in T} \in  \prod_{\Delta\in T_0}\Gamma^*_{\boldsymbol\alpha_\Delta,L_1} \times  \prod_{\Delta\notin T_0}\partial \big(\Gamma^*_{\boldsymbol\alpha_\Delta}\setminus \Gamma^*_{\boldsymbol\alpha_\Delta,L_1}\big)$  be the projection of $\boldsymbol \xi,$ i.e.,  for $\Delta\in T_0,$ $\xi_{\text{proj},\Delta}=\xi_\Delta,$ and for $\Delta\notin T_0,$  $\xi_{\text{proj},\Delta}\in \partial \big(\Gamma^*_{\boldsymbol\alpha_\Delta}\setminus \Gamma^*_{\boldsymbol\alpha_\Delta,L_1}\big)$ is the starting point of the ray  containing $\xi_\Delta.$ 
 Then we  claim that, for all $(\boldsymbol\alpha,\boldsymbol \xi)\in \Gamma^*_{L_1,T_0}$ and $b<b_0,$ 
\begin{equation}\label{claim}
\begin{split}
\mathrm{Im}\mathcal V(\boldsymbol \alpha, \boldsymbol \xi)+& \mathrm{Im}\kappa(\boldsymbol\alpha,\boldsymbol\xi)b^2+\mathrm{Im}\nu_b(\boldsymbol\alpha,\boldsymbol\xi)b^4\\
<& -2\Big(\mathrm{Vol}(M)+\epsilon\Big) -(2\pi -\epsilon ) \big| \boldsymbol\xi -  \boldsymbol \xi^*_{\text{proj}}  \big|.
\end{split}
\end{equation}
As a consequence of the claim and by letting 
$$\mathrm I_{T_0} =2^{|T|-|T_0|}\int_{{\mathbb R_{>0}}^{|T|-|T_0|}}e^{-(2\pi -\epsilon)|\boldsymbol x|}d\boldsymbol x<+\infty$$
and $\mathrm V_{T_0}$ be the volume of the compact set $\Big\{ (\boldsymbol \alpha,\boldsymbol \xi)\in B(\boldsymbol \alpha^*_{t_0}, L_0) \times \mathbb C^{T_0}   \ \Big|\  \boldsymbol\xi \in \prod_{\Delta\in T_0} \Gamma^*_{\boldsymbol\alpha_\Delta, L_1} \Big\},$
we have
\begin{equation}\label{474}
\begin{split}
&\Bigg|\frac{(-\mathbf i)^{|E|}}{(\pi b )^{|E|+|T|}} \int_{\Gamma^*_{L_1,T_0}}\exp\bigg(\frac{\mathcal V_b(\boldsymbol \alpha, \boldsymbol \xi)}{2\pi \mathbf i b^2}\bigg)d \boldsymbol \alpha  d \boldsymbol \xi\Bigg|\\
< &\frac{1}{(\pi b )^{|E|+|T|}}  \exp\bigg(\frac{-\mathrm{Vol}(M)-\epsilon}{\pi b^2}\bigg) \int_{\Gamma^*_{L_1,T_0}}\exp\Bigg(\frac{-(2\pi-\epsilon) \big| \boldsymbol\xi -  \boldsymbol \xi^*_{\text{proj}}  \big|}{2\pi   b^2}\Bigg)|d \boldsymbol \alpha  d \boldsymbol \xi|\\
=&  \frac{\mathrm I_{T_0}\mathrm V_{T_0}}{(\pi b )^{|E|+|T|}}  \exp\bigg(\frac{-\mathrm{Vol}(M)-\epsilon}{\pi b^2}\bigg)\\
< &\ O\bigg(e^{\frac{-\mathrm{Vol}(M)-\epsilon_4}{\pi b^2}}\bigg)
\end{split}
\end{equation}
for any $\epsilon_4 < \epsilon.$

To prove the  claim (\ref{claim}),   by (\ref{vuc7}) and (\ref{ke2}), we first  have 
 \begin{equation}\label{vuc6}
\mathrm{Im}\mathcal V(\boldsymbol\alpha,\boldsymbol \xi)<  \mathrm{Im}\mathcal W(\boldsymbol\alpha,\boldsymbol \xi)+ \epsilon
\end{equation}
  for all $b<b_0$ and for all $(\boldsymbol\alpha,\boldsymbol \xi)\in \Gamma^*_{L_1,T_0}.$  
 Let $\mathbf u=\frac{(\boldsymbol\alpha,\boldsymbol\xi)- (\boldsymbol \alpha, \boldsymbol\xi_{\text{proj}} )}{|(\boldsymbol\alpha,\boldsymbol\xi)- (\boldsymbol \alpha, \boldsymbol\xi_{\text{proj}} ) |}.$   Then by (\ref{dv5}) and  the choice of $b_0,$ for $l >0,$ we have 
$$D_{\mathbf u}\Big(\mathrm{Im}\mathcal W\big((\boldsymbol \alpha, \boldsymbol\xi_{\text{proj}})+l \mathbf u\big)+\mathrm{Im}\kappa\big((\boldsymbol \alpha, \boldsymbol\xi_{\text{proj}} )+l \mathbf u\big)b^2\Big)< -2\pi + K(|T|-|T_0|)b^2 < -2\pi +\epsilon$$
 for all $b<b_0.$ Together with the Mean Value Theorem,  (\ref{469}),  (\ref{ke}), we have
\begin{equation}\label{im}
\begin{split}
 \mathrm{Im}\mathcal W(\boldsymbol \alpha, \boldsymbol \xi)+\mathrm{Im}\kappa(\boldsymbol\alpha,\boldsymbol\xi)b^2 <  & \mathrm{Im}\mathcal W(\boldsymbol \alpha, \boldsymbol\xi_{\text{proj}})+\mathrm{Im}\kappa(\boldsymbol \alpha, \boldsymbol\xi_{\text{proj}})b^2-(2\pi-\epsilon)  \big|\boldsymbol\xi-  \boldsymbol \xi^*_{\text{proj}} \big|\\
<  &  -2\mathrm{Vol}(M)-3\epsilon  - (2\pi-\epsilon) \big| \boldsymbol\xi -  \boldsymbol \xi^*_{\text{proj}}  \big| 
\end{split}
\end{equation}
for all  $(\boldsymbol\alpha,\boldsymbol \xi)\in \Gamma^*_{L_1,T_0}.$  Finally, by Proposition \ref{bound2} and the choice of $b_0$, we have
\begin{equation}\label{Last}
\mathrm{Im}\nu_b(\boldsymbol\alpha,\boldsymbol\xi)b^4<Nb^4<\epsilon 
\end{equation}
for all $b<b_0$ and for all $(\boldsymbol\alpha,\boldsymbol \xi)\in \Gamma^*_{L_1,T_0}.$ 
 Putting (\ref{vuc6}), (\ref{im}) and (\ref{Last}) together, we have the first inequality in (\ref{474}).

For  $\boldsymbol \mu\neq (1,\dots,1)$ or  $\boldsymbol \lambda\neq (1,\dots,1),$ we observe that 
$$\mathrm{Im}\mathcal V_b^{\boldsymbol\mu\boldsymbol \lambda}(\boldsymbol\alpha,\boldsymbol\xi)-\mathrm{Im}\mathcal V_b (\boldsymbol\alpha,\boldsymbol\xi) = -2\pi\sum_{e\in E}\big((1-\mu_e)b^2+(1-\lambda_e)\big)l_e\leqslant 0$$
for all for $(\boldsymbol \alpha,\boldsymbol \xi)$ in $\Gamma^*_{L_1,T_0}.$ Then by (\ref{474}), 
\begin{equation*}
\begin{split}
&\Bigg| \frac{(-\mathbf i)^{|E|}}{(\pi b )^{|E|+|T|}}  \int_{\Gamma^*_{L_1,T_0}}\bigg(\prod_{e\in E}\mu_e\lambda_e\bigg)\exp\bigg(\frac{\mathcal V_b^{\boldsymbol\mu\boldsymbol \lambda}(\boldsymbol \alpha, \boldsymbol \xi)}{2\pi \mathbf i b^2}\bigg)d \boldsymbol \alpha  d \boldsymbol \xi\Bigg|\\
\leqslant & \frac{1}{(\pi b )^{|E|+|T|}}  \int_{\Gamma^*_{L_1,T_0}}\exp\bigg(\frac{\mathrm{Im}\mathcal V^{\boldsymbol\mu\boldsymbol \lambda}_b(\boldsymbol \alpha, \boldsymbol \xi)}{2\pi b^2}\bigg)|d \boldsymbol \alpha  d \boldsymbol \xi|\\
\leqslant & \frac{1}{(\pi b )^{|E|+|T|}}   \int_{\Gamma^*_{L_1,T_0}}\exp\bigg(\frac{\mathrm{Im}\mathcal V_b(\boldsymbol \alpha, \boldsymbol \xi)}{2\pi b^2}\bigg)|d \boldsymbol \alpha  d \boldsymbol \xi|\\
< &\ O\bigg(e^{\frac{-\mathrm{Vol}(M)-\epsilon_4}{\pi b^2}}\bigg).
\end{split}
\end{equation*}
Together with (\ref{474}) and  (\ref{II2sum}), this completes the proof of (ii.2).
\end{proof}

\section{Topological Invariance}\label{sec:invariance}

The goal of this section is to prove Theorem \ref{WD3}. The idea is similar to that of the original Turaev-Viro invariants\,\cite{TV13}, namely, the Orthogonality and the Biedenharn-Elliot Identity (Pentagon Identity) satisfied by the quantum $6j$-symbols there correspond to the invariance under the $0$-$2$ and $2$-$3$ Pachner Moves that connect the different ideal triangulations. In our setting, the corresponding identities still hold as distributions. However, as how our invariants are defined, what we really need are the identities between the integrals of certain functions, which do not always hold as the integrand may not be absolutely integrable and Fubini's Theorem does not apply. For instance,  due to the appearance of the Dirac delta function, the integrand fails to be absolutely integrable when a $0$-$2$ Pachner Move is performed. To resolve the issue, we restrict ourselves to the class of the Kojima ideal triangulations and introduce a $4$-$4$ move illustrated in Figure \ref{fig:4-4move} to avoid the appearance of Dirac delta functions. Then we show that two Kojima ideal triangulations can be connected by a sequence of $2$-$3$, $3$-$2$ and $4$-$4$ moves such that each intermediate ideal triangulation supports angle structures so that the corresponding integral absolutely converges (see Theorem \ref{angleconv}) and Fubini's Theorem applies to prove the desired identity.

\begin{figure}
    \centering
    \includegraphics[width=0.8\linewidth]{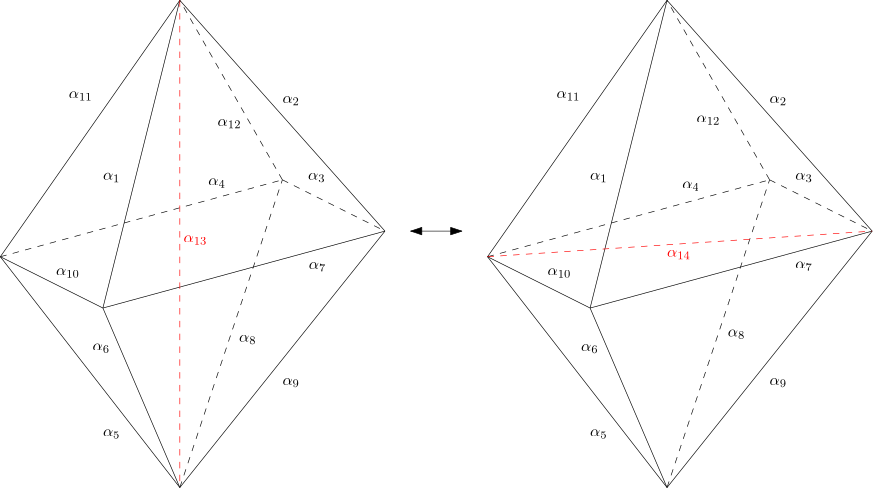}
    \caption{The 4-4 move}
    \label{fig:4-4move}
\end{figure}

 For the invariance under the $2$-$3$, $3$-$2$ and $4$-$4$ moves, we need the following identities for the  $b$-$6j$ symbols, where the Pentagon Identity was proved in~\cite{PT01}, and the $4$-$4$ Move Identity is new and whose proof is deferred to Section~\ref{sec:5.2}.

\begin{prop}\label{prop:32and44moves}
Let $\mathbb{S}\doteq \frac{Q}{2}+\mathbf{i}\mathbb R_{>0}$.
\begin{enumerate}[(1)]
    \item The following \emph{Pentagon Identity} 
\begin{equation}\label{pentagon}
    \resizebox{0.8\textwidth}{!}{$ \begin{split}
&\begin{Bmatrix} 
    a_5 & a_6 & a_{10}\\
      a_3 & a_4 & a_7 
   \end{Bmatrix}_b
   \begin{Bmatrix} 
   a_{10} & a_5 & a_6\\
      a_1 & a_2 & a_9 
   \end{Bmatrix}_b\\
=&\int_\mathbb{S}|S_b(2a_8)|^2\begin{Bmatrix} 
    a_3 & a_7 & a_6\\
      a_1 & a_2 & a_8 
   \end{Bmatrix}_b
   \begin{Bmatrix} 
    a_4 & a_5 & a_7\\
      a_1 & a_8 & a_9 
   \end{Bmatrix}_b
   \begin{Bmatrix} 
    a_4 & a_9 & a_8\\
      a_2 & a_3 & a_{10} 
   \end{Bmatrix}_b
   d\mathrm{Im}a_{8}
   \end{split}$}
    \end{equation}
   holds as  continuous functions in $(a_1,\dots,a_{10})\in {\mathbb{S}}^{10}$.

   \item  The following \emph{4-4 Move Identity} 
\begin{equation}\label{44moveidentity}
    \resizebox{0.9\textwidth}{!}{$ \begin{split}
     &\int_\mathbb{S} |S_b(2a_{13})|^2\begin{Bmatrix} 
    a_{13} & a_6 & a_1\\
      a_{10} & a_{11} & a_5 
   \end{Bmatrix}_b
\begin{Bmatrix} 
    a_{13} & a_{12} & a_8\\
      a_4 & a_5 & a_{11} 
   \end{Bmatrix}_b
\begin{Bmatrix} 
    a_{8} & a_{13} & a_{12}\\
      a_2 & a_3 & a_9 
   \end{Bmatrix}_b
\begin{Bmatrix} 
    a_9 & a_6 & a_{7}\\
      a_1 & a_2 & a_{13} 
   \end{Bmatrix}_bd\mathrm{Im}  a_{13}
\\
   =&\int_\mathbb{{\mathbb{S}}} |S_b(2a_{14})|^2
   \begin{Bmatrix} 
   a_{2} & a_{11} & a_{14}\\
      a_{4} & a_{3} & a_{12} 
   \end{Bmatrix}_b
\begin{Bmatrix} 
a_{7} & a_{10} & a_{14}\\
      a_{11} & a_2 & a_{1} 
   \end{Bmatrix}_b
\begin{Bmatrix} 
a_{7} & a_{9} & a_{6}\\
      a_5 & a_{10} & a_{14} 
   \end{Bmatrix}_b
\begin{Bmatrix} 
      a_5 & a_9 & a_{14}\\
      a_3 & a_4 & a_{8} 
   \end{Bmatrix}_b
   d\mathrm{Im}a_{14}
    \end{split}$}
    \end{equation}
\end{enumerate}
 holds as continuous functions in $(a_1,\dots,a_{14})\in \mathbb{S}^{14}$.
 \end{prop}

On the topological side, let $\{P_i\}_{i\in \mathrm I}$ be the set of hyperideal polyhedra in the Kojima polyhedral decomposition of $M$. Then as described in Section~\ref{sec:2.8}, a Kojima ideal triangulation depends on the choice of a tip $p_{i}$ for each convex polyhedron $P_{i}$ and the choice of a cone point $p_{ij}$ for each face $F_{ij}$ of the polyhedron $P_{i}$ not adjacent to $p_{i}$; and we will denote by $\mathcal{T}(\{p_{i}\}, \{p_{ij}\})$ the Kojima ideal triangulation determined by this choice. The following Propositions \ref{prop:movepolygoncone} and \ref{prop:movepolyhedracone}, which will be proved in Section~\ref{proofofpolyhedracone}, imply that any two different Kojima ideal triangulations can be connected by a sequence of $2$-$3$, $3$-$2$ and $4$-$4$ moves such that each of  the intermediate triangulations supports angle structures.

\begin{proposition}\label{prop:movepolygoncone}
Let $P_{u}$ be a hyperideal polyhedron in the Kojima polyhedral decomposition $M$ and let $F_{uv}$ be a face of $P_{u}$ that is not adjacent to $p_u$. Suppose $\mathcal{T}(\{p_{i}\}, \{p_{ij}\})$ and $\mathcal{T}(\{p_{i}'\}, \{p_{ij}'\})$ are two Kojima ideal triangulation of $M$ such that
\begin{enumerate}[(1)]
\item the tips $p_{i}=p_{i}'$ for all $i\in\mathrm I$, and
\item the cone points $p_{ij}=p_{ij}'$ for all $(i, j)\ne (u, v)$.
\end{enumerate}
Then $\mathcal{T}(\{p_{i}\}, \{p_{ij}\})$ and $\mathcal{T}(\{p_{i}'\}, \{p_{ij}'\})$ can be connected by a sequence of $2$-$3$ and $3$-$2$ moves such that each of the intermediate triangulations supports angle structures.
\end{proposition}

\begin{figure}
\centering
$
\vcenter{\hbox{\begin{tikzpicture}[scale=0.8]
    \foreach \x in {0,1,...,11} \draw (30*\x:3)--(30*\x+30:3);
    \foreach \x in {0,1,...,11} \coordinate (a\x) at (-30*\x:3);
    \draw (-1, 0) to (1, 0);
    \draw (-1, 0) to (a4);
    \draw (-1, 0) to (a8);
    \draw (1, 0) to (a2);
    \draw (1, 0) to (a10);
    \foreach \x in {2, 3, 5, 6, 7, 9, 10} \draw[blue, line width=1pt] (-1, 0) to (a\x);
    \foreach \x in {11, 0, 1} \draw[blue, line width=1pt] (1, 0) to (a\x);
    \draw[fill=white] (-1, 0) circle (0.3);
    \draw[fill=white] (1, 0) circle (0.3);
    \draw (-1, 0) node{$p_{u}$};
    \draw (1, 0) node{$p_{u}'$};
    \foreach \x in {0,1,...,11} \draw[fill=white] (a\x) circle (0.2);
\end{tikzpicture}}}
\longrightarrow 
\vcenter{\hbox{\begin{tikzpicture}[scale=0.8]
    \foreach \x in {0,1,...,11} \draw (30*\x:3)--(30*\x+30:3);
    \foreach \x in {0,1,...,11} \coordinate (a\x) at (-30*\x:3);
    \draw (-1, 0) to (1, 0);
    \draw (-1, 0) to (a4);
    \draw (-1, 0) to (a8);
    \draw (1, 0) to (a2);
    \draw (1, 0) to (a10);
    \foreach \x in {11, 0, 1, 3, 4, 9, 8} \draw[blue, line width=1pt] (1, 0) to (a\x);
    \foreach \x in {5, 6, 7} \draw[blue, line width=1pt] (-1, 0) to (a\x);
    \draw[fill=white] (-1, 0) circle (0.3);
    \draw[fill=white] (1, 0) circle (0.3);
    \draw (-1, 0) node{$p_{u}$};
    \draw (1, 0) node{$p_{u}'$};
    \foreach \x in {0,1,...,11} \draw[fill=white] (a\x) circle (0.2);
\end{tikzpicture}}}
$
    \caption{The figure above illustrates the difference between the two triangulations $\mathcal{T}(\{p_{i}\}, \{p_{ij}\})$ on the left and $\mathcal{T}(\{p_{i}'\}, \{p_{ij}'\})$ on the right described in Proposition \ref{prop:movepolyhedracone} over the polygonal faces in the neighbourhood of $p_u$ and $p_u'$; and all the other faces are triangulated identically. The edges of the polygonal faces  are colored in black and the edges coming from the subdivision are colored in blue.}
    \label{fig:neiborhoodofpu}
\end{figure}
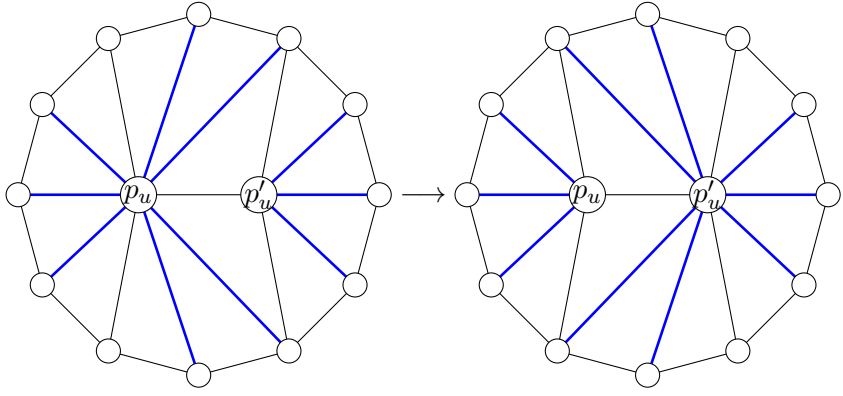

\begin{proposition}\label{prop:movepolyhedracone}
Let $P_{u}$ be a hyperideal polyhedron in the Kojima polyhedral decomposition. Suppose $\mathcal{T}(\{p_{i}\}, \{p_{ij}\})$ and $\mathcal{T}(\{p_{i}'\}, \{p_{ij}'\})$ are two Kojima ideal triangulation of $M$ such that
\begin{enumerate}[(1)]
    \item the tips $p_{i}=p_{i}'$ for all $i\ne u$, and $p_{u}$ and $p_{u}'$ are adjacent in $P_{u}$,
    
    \item the cone points $p_{ij}=p_{ij}'$ for all faces $F_{ij}$ not in $P_{u},$ and the cone points $p_{uj}=p_{uj}'$ for all faces $F_{uj}$ of $P_{u}$ that is adjacent to neither $p_{u}$ nor $p_{u}'$, and 
    
    \item the cone point $p_{u j}=p_{u}'$ for all faces $F_{uj}$ adjacent to $p_{u}'$ but not adjacent to $p_u$, and the cone point $p_{uj}'=p_{u}$ for all faces $F_{uj}$ adjacent to $p_{u}$. (See Figure \ref{fig:neiborhoodofpu})
\end{enumerate}
Then $\mathcal{T}(\{p_{i}\}, \{p_{ij}\})$ and $\mathcal{T}(\{p_{i}'\}, \{p_{ij}'\})$ can be  connected by a sequence of $2$-$3$, $3$-$2$ and $4$-$4$ moves such that each of the intermediate triangulations supports angle structures.
\end{proposition}

Combining these ingredients, we can now prove Theorem \ref{WD3}.

\begin{proof}[Proof of Theorem \ref{WD3}]
By Theorem \ref{angleconv}, if an ideal triangulation $\mathcal T$ admits angle structures, then the associated state integral $\mathrm{TV}_b(M,\mathcal T)$ converges absolutely. If two ideal triangulations $\mathcal T_1$ and $\mathcal T_2$ both admit angle structure and differ from each other by a single $2$-$3$, $3$-$2$ or $4$-$4$ move, then Fubini's Theorem together with Proposition \ref{prop:32and44moves} implies that the two  state integrals $\mathrm{TV}_b(M,\mathcal T_1)=\mathrm{TV}_b(M,\mathcal T_2)$. 
Therefore, to prove Theorem \ref{WD3}, it suffices to prove that any two Kojima ideal triangulations can be connected by a sequence of  $2$-$3$, $3$-$2$ and $4$-$4$ moves such that each of the intermediate triangulations supports angle structures.\par

By Proposition \ref{prop:movepolygoncone}, any two Kojima ideal triangulations determined by the same set of tips can be connected by a sequence of $2$-$3$ and $3$-$2$ moves such that each of the intermediate ideal triangulations supports angle structure. Now for two Kojima ideal triangulation $\mathcal{T}(\{p_{i}\}, \{p_{ij}\})$ and $\mathcal{T}(\{p_{i}'\}, \{p_{ij}'\})$, one can find a sequence of Kojima ideal triangulations $\mathcal{T}(\{p_{i}^{(k)}\}, \{p_{ij}^{(k)}\})$ such that for each $k$ there exist a $u_{k}\in \mathrm I$ such that
\begin{enumerate}[(1)] 
    \item for each pair of consecutive Kojima ideal triangulations $\mathcal{T}(\{p_{i}^{(k)}\}, \{p_{ij}^{(k)}\})$ and $\mathcal{T}(\{p_{i}^{(k+1)}\}, \{p_{ij}^{(k+1)}\})$ in the sequence, the tips $p_{u_{k}}^{(k)}$ and $p_{u_{k}}^{(k+1)}$ are adjacent in the polyhedron $P_{u_{k}}$, and
    \item the tips
    $p_{i}^{(k)}=p_{i}^{(k+1)}$ in all the other polyhedra $P_i$'s.
\end{enumerate}
Then by Propositions \ref{prop:movepolygoncone} and \ref{prop:movepolyhedracone}, consecutive Kojima ideal triangulations in the sequence, and hence $\mathcal{T}(\{p_{i}\}, \{p_{ij}\})$ and $\mathcal{T}(\{p_{i}'\}, \{p_{ij}'\})$, can be connected by a sequence of $2$-$3$, $3$-$2$ and $4$-$4$ moves such that each of the intermediate ideal triangulations supports angle structures.
\end{proof}

\subsection{Positive representations and proof of Proposition~\ref{prop:32and44moves}}\label{sec:5.2}

For $b\in (0,1)$ such that $q=e^{\pi\mathbf{i}b^2}$ that is not a root of unity, the identities in Proposition~\ref{prop:32and44moves}  are established by Ponsot-Teschner~\cite{PT01} via the positive representations of the modular double of $\mathrm{U}_q \mathfrak{sl}(2; \mathbb{R})$.
In this section, we will show that both sides of the identities in Proposition~\ref{prop:32and44moves} depend continuously on $b$, and it suffices to prove these identities for those $b$ with $q=e^{\pi\mathbf{i}b^2}$ not a root of unity. We now briefly recall the definition and basic properties of $\mathrm{U}_{q\tilde q} \mathfrak{sl}(2; \mathbb{R})$, 
the modular double of $\mathrm{U}_q \mathfrak{sl}(2; \mathbb{R})$, and its positive representations. Recall that for $q\notin \{0,\pm1\}$ the quantum group \( \mathrm{U}_q \mathfrak{sl}(2; \mathbb{R}) \) is the \(*\)-Hopf algebra generated by \( E, F, K, K^{-1} \) subject to the relations:
\[
K E = q E K, \quad K F = q^{-1} F K, \quad [E, F] = \frac{K^2 - K^{-2}}{q - q^{-1} },
\]
and equipped with the following co-product:
\[
\Delta(E) = E \otimes K + K^{-1} \otimes E, \quad \Delta(F) = F \otimes K + K^{-1} \otimes F, \quad \Delta(K) = K \otimes K.
\]
The star structure is $E^*=E$, $F^*=F$ and $K^*=K$. Let $\tilde q=e^{\pi\mathbf i b^{-2}}$. Then the  \emph{modular double} of $\mathrm{U}_{q} \mathfrak{sl}(2; \mathbb{R})$ is defined as 
$$\mathrm{U}_{q\tilde q} \mathfrak{sl}(2; \mathbb{R})=\mathrm{U}_{q} \mathfrak{sl}(2; \mathbb{R})\otimes\mathrm{U}_{\tilde q} \mathfrak{sl}(2; \mathbb{R}).$$
We denote the corresponding generators of $\mathrm{U}_{\tilde{q}} \mathfrak{sl}(2; \mathbb{R}) $ by  $ \tilde{E}, \tilde{F}, \tilde{K}$ and $\tilde{K}^{-1} $.

For each $a=\frac{Q}{2}+\mathbf{i}\lambda$ with $\lambda>0$, there is an irreducible representation $P_a$ of $\mathrm{U}_{q\tilde q} \mathfrak{sl}(2; \mathbb{R})$, called the \emph{positive representation}. It is defined using unbounded essentially self-adjoint operators over the dense domain 
$$\mathcal{W}\doteq\Big\{\sum_{k=1}^Ne^{-{r_k}x^2+{s_k}x+c_k}x^{i_k}\Big|r_k>0,s_k,c_k\in\mathbb C,i_k\in\mathbb N,N\in\mathbb N\Big\}$$ in $L^2(\mathbb R,dx)$ (c.f \cite[Section 2.5]{Ip2013RepresentationOT} ). The action of the generators $E,F,K^{\pm1},\tilde E,\tilde F,\tilde K^{\pm1}$ of $\mathrm{U}_{q\tilde q}\mathfrak{s l}(2; \mathbb{R})$ under $P_a$ is given by 
\begin{enumerate}[(1)]
    \item $P_a(E)(u(x))=\frac{e^{2\pi bx}}{2\sin\pi b^2}[e^{\mathbf{i}\pi b(\frac{b}{2}-\mathbf{i}\lambda)}u(x+\frac{\mathbf{i}b}{2})+e^{\mathbf{i}\pi b(\mathbf{i}\lambda-\frac{b}{2})}u(x-\frac{\mathbf{i}b}{2})]$, 
    \item $P_a(F)(u(x))=\frac{e^{-2\pi bx}}{2\sin\pi b^2}[e^{\mathbf{i}\pi b(\mathbf{i}\lambda-\frac{b}{2})}u(x+\frac{\mathbf{i}b}{2})+e^{\mathbf{i}\pi b(\frac{b}{2}-\mathbf{i}\lambda)}u(x-\frac{\mathbf{i}b}{2})]$,
    \item $P_a(K)(u(x))=u(x+\frac{\mathbf{i}b}{2})$, and
    
    \item $\tilde E,\tilde F$ and $\tilde K$ are sent respectively to the operators given by replacing the $b$ in $E,$ $F$ and $K$ by $b^{-1}.$
\end{enumerate}

We first follow~\cite{PT01} and recall the definition of the Clebsch-Gordan maps, which are unitary intertwining maps decomposing a tensor product $P_{a_2} \otimes P_{a_1}$ into a direct integral $\int_{\mathbb{S}}^{\oplus}  P_{a}d \rho(a)$.

\begin{theorem}[\cite{PT01}, Theorem 2]
\label{theorem:unitary representation}
Equip $\mathbb{S}=\frac{Q}{2}+\mathbf{i}\mathbb R_{>0}$ with the \emph{Plancherel measure}
$$
d\rho(a) \doteq \left|S_b(2a)\right|^2 d\operatorname{Im} a.
$$
Then the tensor product of positive representations of  \( \mathrm{U}_{q\tilde q} \mathfrak{sl}(2; \mathbb{R}) \) decomposes into a direct integral of irreducible positive representations:
\[
P_{a_2} \otimes P_{a_1} \simeq \int_{\mathbb{S}}^{\oplus} P_{a} \, d\rho(a).
\]
Moreover, the isomorphism $\simeq$ is realized by a unitary intertwiner $ (\mathcal{C}_{2,1}^3)^{\rm PT}: L^2(\mathbb{R} \times \mathbb{R}) \to L^2\left(\mathbb{S} \times \mathbb{R}, d\rho(a_3) dx_3\right)$
defined by
\begin{align} \label{equ:unitary 3j}
  (\mathcal{C}_{2,1}^3)^{\rm PT}(f)(a_3,x_3)
  =
  \int_{\mathbb{R}^2}
  \left[\begin{array}{c}
  a_3 \\
  x_3
  \end{array}\Bigg|
  \begin{array}{cc}
  a_2 & a_1 \\
  x_2 & x_1
  \end{array}\right]^{\rm PT}
  f(x_2, x_1) \, dx_2 dx_1,
\end{align}
where the kernel is the \emph{Clebsch–Gordan coefficient}
\begin{align*}
\left[\begin{array}{c}
a_3 \\
x_3
\end{array}\Bigg|
\begin{array}{cc}
a_2 & a_1 \\
x_2 & x_1
\end{array}\right]^{\rm PT}
&\doteq
e^{-\frac{\pi \mathbf{i}}{2} \left( a_3(Q-a_3) - a_1(Q-a_1) - a_2(Q-a_2) \right)} 
 \times
\frac{S_b\left( \frac{Q}{2} - \mathbf{i}(x_2 - x_1) - \frac{a_1 + a_2 + 2a_3 - Q}{2} \right)}
     {S_b\left( \frac{Q}{2} - \mathbf{i}(x_2 - x_1) + \frac{a_1 + a_2 - Q}{2} \right)} \\
& \times
\frac{S_b\left( \frac{Q}{2} - \mathbf{i}(x_2 - x_3) - \frac{Q + a_2 - a_3}{2} \right)}
     {S_b\left( \frac{Q}{2} - \mathbf{i}(x_2 - x_3) + \frac{Q + a_2 - a_3 - 2a_1}{2} \right)} 
 \times
\frac{S_b\left( \frac{Q}{2} - \mathbf{i}(x_3 - x_1) - \frac{Q + a_1 - a_3}{2} \right)}
     {S_b\left( \frac{Q}{2} - \mathbf{i}(x_3 - x_1) + \frac{Q + a_1 - a_3 - 2a_2}{2} \right)}.
\end{align*}
\end{theorem}

 \begin{remark}
The original proof in~\cite{PT01} was based on studying the eigenfunctions of the  Casimir operator acting on $P_{a_2}\otimes P_{a_1}$. The treatment of the unitarity of the Clebsch–Gordan map defined above was later simplified in  \cite{derkachov20133jsymbol,nidaiev2013}, using the connection between the Casimir operator and the Kashaev length operator introduced in \cite{Kashaev2001}. The orthogonality and completeness of the eigenfunctions of the Kashaev length operator were proven in \cite{Takhtajan2015OnTS}, and \cite{derkachov20133jsymbol,nidaiev2013} demonstrated how the Clebsch–Gordan coefficients - viewed as eigenfunctions of the Casimir operator - can be explicitly constructed from these eigenfunctions. Similar construction for the Fourier transform of the above representation was considered in~\cite{Ivan2021}.
\end{remark}

In~\cite{PT01}, a variant of the  $b$-$6j$ symbols in Definition~\ref{def：6j}  was obtained from the Clebsch–Gordan map  $(\mathcal{C}_{2,1}^3)^{\rm PT}$, which differs from our $b$-$6j$ symbols by a factor. To obtain our $b$-$6j$ symbols, as the convention taken in \cite{Teschner:2012em}, we need to renormalize the Clebsch–Gordan maps as 
\begin{align} \label{equ:C123}
  (\mathcal{C}_{2,1}^3)(f)(a_3,x_3)
\doteq 
  \int_{\mathbb{R}^2}
  \left[\begin{array}{c}
  a_3 \\
  x_3
  \end{array}\Bigg|
  \begin{array}{cc}
  a_2 & a_1 \\
  x_2 & x_1
  \end{array}\right]^{\rm TV}
  f(x_2, x_1) \, dx_2 dx_1,
\end{align}
where \[\left[\begin{array}{c}
  a_3 \\
  x_3
  \end{array}\Bigg|
  \begin{array}{cc}
  a_2 & a_1 \\
  x_2 & x_1
  \end{array}\right]^{\rm TV} \doteq  M(a_1, a_2, a_3)  \left[\begin{array}{c}
  a_3 \\
  x_3
  \end{array}\Bigg|
  \begin{array}{cc}
  a_2 & a_1 \\
  x_2 & x_1
  \end{array}\right]^{\mathrm {PT}},\] and 
\[
M(a_1, a_2, a_3) = \left[
S_b(2Q - a_1 - a_2 - a_3)
S_b(Q - a_1 - a_2 + a_3)
S_b(a_1 + a_3 - a_2)
S_b(a_2 + a_3 - a_1)
\right]^{-\frac{1}{2}}.
\]

Now we explain the relation between $\mathcal{C}_{2,1}^3$ and our $b$-$6j$ symbols. It is in analog to the case of finite dimensional representations of quantum groups with the summation replaced by an integral. See e.g.~\cite[Chapter VI]{Turaev:1994xb} for reference in that case.
We shall use subscripts with parentheses, such as \((t)\), to indicate that an integral operator acts on the component \(L^2(\mathbb{S}, d\rho(a_t))\), and subscripts without parentheses, such as \(t\), to indicate that it acts on \(L^2(\mathbb{R}, dx_t)\). The superscripts, such as $(s)$ and $s$, will indicate that the codomain is respectively \(L^2(\mathbb{S}, d\rho(a_s))\) and \(L^2(\mathbb{R}, dx_s)\). Consider the following compositions of the Clebsch–Gordan maps, both of which are unitary operators:
\begin{align*}
\mathcal{C}^{4}_{3,s} \circ (\mathrm{Id}_3 \otimes \mathcal{C}^{s}_{2,1}) &:
L^2(\mathbb{R}^3, dx_1 dx_2 dx_3)
\longrightarrow
L^2\left(\mathbb{S} \times \mathbb{S} \times \mathbb{R}, d\rho(a_4)\, d\rho(a_s)\, dx_4 \right), \\
\mathcal{C}^{4}_{t,1} \circ (\mathcal{C}^{t}_{3,2} \otimes \mathrm{Id}_1) &:
L^2(\mathbb{R}^3, dx_1 dx_2 dx_3)
\longrightarrow
L^2\left(\mathbb{S} \times \mathbb{S} \times \mathbb{R}, d\rho(a_4)\, d\rho(a_t)\, dx_4 \right).
\end{align*}
Given $a_1,a_2,a_3,a_4\in \mathbb S$, define  $\mathcal{F}_{(t)}^{(s)}: L^2(\mathbb{S},d\rho(a_t))\to L^2(\mathbb{S},d\rho(a_s))$ by  
$$(\mathcal{F}_{(t)}^{(s)}f)(a_s)= \int_{\mathbb{S}} 
\left\{ \begin{array}{ccc}
a_1 & a_2 & a_s \\
a_3 & a_4 & a_t
\end{array} \right\}_b f(a_t) d \rho(a_t).$$  
The map $\mathcal{F}_{(t)}^{(s)}$ is called the \emph{fusion transformation operator} and relates to the Clebsch–Gordan maps by 
\begin{equation}\label{eq:fusion}
    \mathcal{F}_{(t)}^{(s)} \otimes \mathrm{Id}_4=(\mathcal{C}^4_{3,s}\circ(\mathrm{Id}_3\otimes\mathcal{C}^s_{2,1} ))\circ (\mathcal{C}^4_{t,1} \circ(\mathcal{C}^t_{3,2}\otimes \mathrm{Id}_1))^{-1} \quad \textrm{on }L^2(\mathbb{S}^2\times \mathbb{R}, d\rho(a_4)d\rho(a_t)dx_4), 
\end{equation}  
as illustrated in the following commutative diagram:
\[\begin{tikzcd}
&L^2(\R^3,dx_1dx_2dx_3)\arrow{ld}[swap]{\mathcal{C}^4_{t,1} \circ(\mathcal{C}^t_{3,2}\otimes \mathrm{Id}_1)}\arrow{rd}{\mathcal{C}^4_{3,s}\circ(\mathrm{Id}_3\otimes\mathcal{C}^s_{2,1} )}&\\
L^2(\mathbb{S}^2\times \R, d\rho(a_4)d\rho(a_t)dx_4)\arrow{rr}{\mathcal{F}_{(t)}^{(s)}\otimes \mathrm{Id}_4}&&L^2(\mathbb{S}^2\times \R,d\rho(a_4)d\rho(a_s)dx_4).
\end{tikzcd}\]
The relation~\eqref{eq:fusion}  has a variant for $(\mathcal{C}_{2,1}^3)^{\rm PT}$ 
 which is established in~\cite{PT01}. Based on this, the relation~\eqref{eq:fusion} 
was obtained in \cite[Equations (2.23), (2.24), and (2.27)]{Teschner:2012em} after changing the normalization as in~\eqref{equ:C123}. The  finite dimensional counterpart of~\eqref{eq:fusion} between can be found in~\cite[Section 5]{Kirillov:191317}.

The following proposition gives the orthogonality of the $b$-$6j$ symbols.
\begin{prop}
    $\mathcal{F}_{(t)}^{(s)}:L^2(\mathbb{S},d\rho(a_t))\to L^2(\mathbb{S},d\rho(a_s))$ is a unitary operator.
\end{prop}
\begin{proof}
 This is because both $\mathcal{C}^4_{3,s}\circ(\mathrm{Id}_3\otimes\mathcal{C}^s_{2,1} )$ and $\mathcal{C}^4_{t,1} \circ(\mathcal{C}^t_{3,2}\otimes \mathrm{Id}_1)$ are unitary.
\end{proof}

We are ready to prove Proposition \ref{prop:32and44moves}. To simplify the presentation, we use the 
diagrammatic representation of the Clebsch–Gordan maps.  Each map can be represented using a binary tree with three univalent vertices.  The composition of two maps then corresponds to stacking such trees in different associative patterns. See Figure~\ref{fig:3j}. 
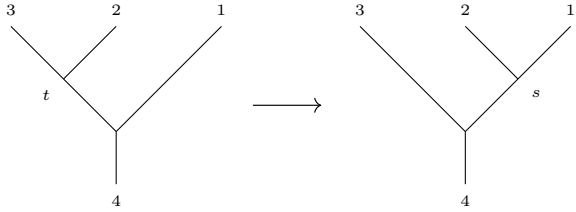
\begin{figure}[ht]
    \centering
    \begin{tikzcd}
        \vcenter{\hbox{
        \begin{tikzpicture}[scale=0.7]
            \draw (0, 0) -- (-2, 2) node[above]{\tiny{$3$}};
            \draw (0, 0) -- (2, 2) node[above]{\tiny{$1$}};
            \draw (0, 2) node[above]{\tiny{$2$}} -- (-1, 1) node[anchor=north east]{\tiny{$t$}};
            \draw (0, 0) -- (0, -1) node[below]{\tiny{$4$}};
        \end{tikzpicture}}}\ar[r]&
        \vcenter{\hbox{
        \begin{tikzpicture}[scale=0.7]
            \draw (0, 0) -- (-2, 2) node[above]{\tiny{$3$}};
            \draw (0, 0) -- (2, 2) node[above]{\tiny{$1$}};
            \draw (0, 2) node[above]{\tiny{$2$}} -- (1, 1) node[anchor=north west]{\tiny{$s$}};
            \draw (0, 0) -- (0, -1) node[below]{\tiny{$4$}};
        \end{tikzpicture}}}
    \end{tikzcd}
    \caption{
The tree on the left represents the map $\mathcal{C}^4_{t,1} \circ(\mathcal{C}^t_{3,2}\otimes \mathrm{Id}_1)$. The tree on the right represents $\mathcal{C}^4_{3,s}\circ(\mathrm{Id}_3\otimes\mathcal{C}^s_{2,1} )$. The arrow between them represents the fusion transformation operator $\mathcal{F}_{(t)}^{(s)}\otimes \mathrm{Id}_4$.}
    \label{fig:3j}
\end{figure}

\begin{proof}[Proof of Proposition~\ref{prop:32and44moves}]
Consider the following two unitary maps from $L^2(\mathbb{S}^3\times \R,d\rho(a_5)d\rho(a_9)d\rho(a_{10})dx_5) $ to $ L^2(\mathbb{S}^3\times \R,d\rho(a_5)d\rho(a_6)d\rho(a_{7})dx_5)$:

\begin{align*}
V_1:F(a_5,a_9,a_{10},x_5)\mapsto \int_{\mathbb{S}}\begin{Bmatrix} 
    a_5 & a_4 & a_{7}\\
    a_3 & a_6 & a_{10} 
   \end{Bmatrix}_b &
   \int_{\mathbb{S}}\begin{Bmatrix} 
    a_{10} & a_5 & a_6\\
    a_1 & a_2 & a_9 
\end{Bmatrix}_b\\
& F(a_5,a_9,a_{10},x_5)d\rho(a_9)d\rho(a_{10}),\\
V_2:F(a_5,a_9,a_{10},x_5) \mapsto \int_{\mathbb{S}}\begin{Bmatrix} 
      a_3 & a_7 & a_6\\
      a_1 & a_2 & a_8 
   \end{Bmatrix}_b &\int_{\mathbb{S}}
   \begin{Bmatrix} 
    a_4 & a_5 & a_7\\
    a_1 & a_8 & a_9 
   \end{Bmatrix}_b \int_{\mathbb{S}}
   \begin{Bmatrix} 
    a_4 & a_9 & a_8\\
    a_2 & a_3 & a_{10}
   \end{Bmatrix}_b\\
   &F(a_5,a_9,a_{10},x_5)d\rho(a_{10})d\rho(a_{9})d\rho(a_{8}).
   \end{align*}
As illustrated in Figure \ref{fig:pentagon}, by~\eqref{eq:fusion}, we have $V_1=V_2$.
Indeed, the first tree (labeled by (A))  represents the unitary map $U_1=\mathcal{C}_{9,1}^{5}\circ(\mathcal{C}_{10,2}^9\otimes\mathrm{Id}_1)\circ(\mathcal{C}_{4,3}^{10} \otimes\mathrm{Id}_2\otimes\mathrm{Id}_1)$ from $L^2( \mathbb R^4,dx_1dx_2dx_3dx_4)$ to $L^2(\mathbb{S}^3\times \mathbb R,d\rho(a_5)d\rho(a_9)d\rho(a_{10})dx_5)$. The tree labeled by (E)  is  the unitary map $U_2=\mathcal{C}_{4,7}^5\circ(\mathrm{Id}_4\otimes\mathcal{C}_{3,6}^7)\circ(\mathrm{Id}_4\otimes\mathrm{Id}_3\otimes\mathcal{C}_{2,1}^6)$ from  $ L^2( \mathbb R^4,dx_1dx_2dx_3dx_4)$ to $ L^2(\mathbb{S}^3\times \R,d\rho(a_5)d\rho(a_6)d\rho(a_{7})dx_5)$. 
From the path (A) $\to$ (D) $\to$ (E), we get $V_1=U_2 U_1^{-1}$; and from the path (A) $\to$ (B) $\to$ (C) $\to$ (E), we get $V_2=U_2 U_1^{-1}$. Therefore, $V_1=V_2$. This derivation is similar to the one for the finite-dimensional case given in~\cite[Chapter VI]{Turaev:1994xb}.

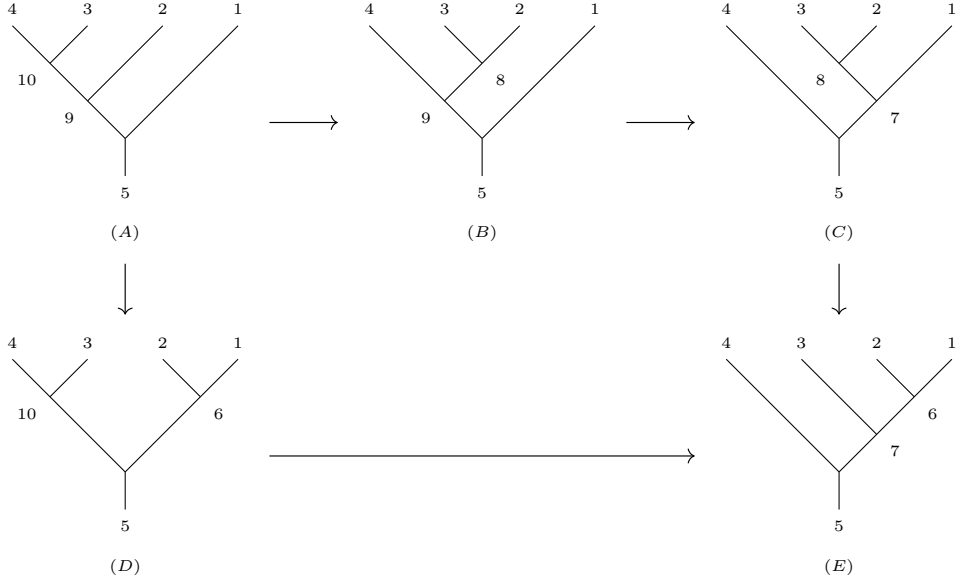
\begin{figure}
    \centering
    \begin{tikzcd}
\vcenter{\hbox{\begin{tikzpicture}[scale=0.5]
\draw (0, 0)--(-3, 3) node[above]{\tiny{$4$}};
\draw (0, 0)--(3, 3) node[above]{\tiny{$1$}};
\draw (-1, 3)node[above]{\tiny{$3$}}--(-2, 2)node[anchor=north east]{\tiny{$10$}};
\draw (1, 3)node[above]{\tiny{$2$}}--(-1, 1)node[anchor=north east]{\tiny{$9$}};
\draw (0, 0)--(0, -1) node[below]{\tiny{$5$}};
\draw (0, -2) node[below]{\tiny{$(A)$}};
\end{tikzpicture}}}\ar[r]\ar[d]
&
\vcenter{\hbox{\begin{tikzpicture}[scale=0.5]
\draw (0, 0)--(-3, 3) node[above]{\tiny{$4$}};
\draw (0, 0)--(3, 3) node[above]{\tiny{$1$}};
\draw (-1, 3)node[above]{\tiny{$3$}}--(0, 2)node[anchor=north west]{\tiny{$8$}};
\draw (1, 3)node[above]{\tiny{$2$}}--(-1, 1)node[anchor=north east]{\tiny{$9$}};
\draw (0, 0)--(0, -1) node[below]{\tiny{$5$}};
\draw (0, -2) node[below]{\tiny{$(B)$}};
\end{tikzpicture}}}\ar[r]
&
\vcenter{\hbox{\begin{tikzpicture}[scale=0.5]
\draw (0, 0)--(-3, 3) node[above]{\tiny{$4$}};
\draw (0, 0)--(3, 3) node[above]{\tiny{$1$}};
\draw (-1, 3)node[above]{\tiny{$3$}}--(1, 1)node[anchor=north west]{\tiny{$7$}};
\draw (1, 3)node[above]{\tiny{$2$}}--(0, 2)node[anchor=north east]{\tiny{$8$}};
\draw (0, 0)--(0, -1) node[below]{\tiny{$5$}};
\draw (0, -2) node[below]{\tiny{$(C)$}};
\end{tikzpicture}}}\ar[d]\\
\vcenter{\hbox{\begin{tikzpicture}[scale=0.5]
\draw (0, 0)--(-3, 3) node[above]{\tiny{$4$}};
\draw (0, 0)--(3, 3) node[above]{\tiny{$1$}};
\draw (-1, 3)node[above]{\tiny{$3$}}--(-2, 2)node[anchor=north east]{\tiny{$10$}};
\draw (1, 3)node[above]{\tiny{$2$}}--(2, 2)node[anchor=north west]{\tiny{$6$}};
\draw (0, 0)--(0, -1) node[below]{\tiny{$5$}};
\draw (0, -2) node[below]{\tiny{$(D)$}};
\end{tikzpicture}}}\ar[rr]
&&
\vcenter{\hbox{\begin{tikzpicture}[scale=0.5]
\draw (0, 0)--(-3, 3) node[above]{\tiny{$4$}};
\draw (0, 0)--(3, 3) node[above]{\tiny{$1$}};
\draw (-1, 3)node[above]{\tiny{$3$}}--(1, 1)node[anchor=north west]{\tiny{$7$}};
\draw (1, 3)node[above]{\tiny{$2$}}--(2, 2)node[anchor=north west]{\tiny{$6$}};
\draw (0, 0)--(0, -1) node[below]{\tiny{$5$}};
\draw (0, -2) node[below]{\tiny{$(E)$}};
\end{tikzpicture}}}
\end{tikzcd}
    \caption{Pentagon Identity}
    \label{fig:pentagon}
\end{figure}

For the $4$-$4$ move,  consider the unitary maps $W_1$ and $W_2$ from $L^2(\mathbb{S}^4\times \R,d\rho(a_7)d\rho(a_8)d\rho(a_{9})d\rho(a_{6})dx_6)$ to $L^2(\mathbb{S}^4\times \R,d\rho(a_{10})d\rho(a_{11})d\rho(a_{12})d\rho(a_{6})dx_6)$:
\begin{align*}
W_1:F(a_7,a_8,a_{9},a_{6},x_6)\mapsto &\int_{\mathbb{S}}
\begin{Bmatrix} 
    a_{5} & a_6 & a_{10}\\
    a_{1} & a_{11} & a_{13}  
   \end{Bmatrix}_b\int_{\mathbb{S}}
\begin{Bmatrix} 
    a_{13} & a_{5} & a_{11}\\
    a_4 & a_{12} & a_{8} 
   \end{Bmatrix}_b\int_{\mathbb{S}}
\begin{Bmatrix} 
    a_{8} & a_{13} & a_{12}\\
      a_2 & a_3 & a_9 
   \end{Bmatrix}_b\\
& \int_{\mathbb{S}}\begin{Bmatrix} 
    a_9 & a_2 & a_{13}\\
      a_1 & a_6 & a_{7} 
   \end{Bmatrix}_b
F(a_7,a_8,a_{9},a_{6},x_6)d\rho(a_7)d\rho(a_{9})d\rho(a_{8})d\rho(a_{13}),\\
W_2:F(a_7,a_8,a_{9},a_{6},x_6) \mapsto & \int_{\mathbb{S}}
\begin{Bmatrix} 
    a_{2} & a_{3} & a_{12}\\
      a_{4} & a_{11} & a_{14} 
   \end{Bmatrix}_b\int_{\mathbb{S}}
\begin{Bmatrix} 
    a_{1} & a_{10} & a_{11}\\
      a_{14} & a_2 & a_{7} 
   \end{Bmatrix}_b\int_{\mathbb{S}}
\begin{Bmatrix} 
    a_{7} & a_{14} & a_{10}\\
      a_5 & a_{6} & a_{9}
   \end{Bmatrix}_b\\
&\int_{\mathbb{S}}\begin{Bmatrix} 
    a_5 & a_9 & a_{14}\\
    a_3 & a_4 & a_{8} 
   \end{Bmatrix}_b
F(a_7,a_8,a_{9},a_{6},x_6)d\rho(a_8)d\rho(a_{9})d\rho(a_{7})d\rho(a_{14}).
   \end{align*}

As illustrated in Figure \ref{fig:4-4}, we have $W_1=W_2= \mathcal{C}_{5,10}^6\circ(\mathrm{Id}_5\otimes\mathcal{C}_{11,1}^{10})\circ(\mathrm{Id}_5\otimes\mathcal{C}_{4,12}^{11}\otimes\mathrm{Id}_1)\circ(\mathrm{Id}_5\otimes\mathrm{Id}_4\otimes\mathcal{C}_{3,2}^{12}\otimes\mathrm{Id}_1)\circ (\mathrm{Id}_5\otimes\mathrm{Id}_4\otimes\mathrm{Id}_3\otimes\mathcal{C}_{2,1}^{7} )^{-1}\circ(\mathcal{C}_{5,4}^{8} \otimes\mathrm{Id}_7)^{-1}\circ(\mathcal{C}_{8,3}^9\otimes\mathrm{Id}_7)^{-1}\circ(\mathcal{C}_{9,7}^{6})^{-1}$.

{
\begin{figure}
    \centering
\begin{tikzcd}[scale=0.1]
\vcenter{\hbox{\begin{tikzpicture}[scale=0.35]
\draw (0, 0)--(-4, 4) node[above]{\tiny{$5$}};
\draw (0, 0)--(4, 4) node[above]{\tiny{$1$}};
\draw (-2, 4)node[above]{\tiny{$4$}}--(-3, 3)node[anchor=north east]{\tiny{$8$}};
\draw (0, 4)node[above]{\tiny{$3$}}--(-2, 2)node[anchor=north east]{\tiny{$9$}};
\draw (2, 4)node[above]{\tiny{$2$}}--(3, 3)node[anchor=north west]{\tiny{$7$}};
\draw (0, 0)--(0, -1) node[below]{\tiny{$6$}};
\end{tikzpicture}}}\ar[r]\ar[d]
&
\vcenter{\hbox{\begin{tikzpicture}[scale=0.35]
\draw (0, 0)--(-4, 4) node[above]{\tiny{$5$}};
\draw (0, 0)--(4, 4) node[above]{\tiny{$1$}};
\draw (-2, 4)node[above]{\tiny{$4$}}--(-3, 3)node[anchor=north east]{\tiny{$8$}};
\draw (0, 4)node[above]{\tiny{$3$}}--(-2, 2)node[anchor=north east]{\tiny{$9$}};
\draw (2, 4)node[above]{\tiny{$2$}}--(-1, 1)node[anchor=north east]{\tiny{$13$}};
\draw (0, 0)--(0, -1) node[below]{\tiny{$6$}};
\end{tikzpicture}}}\ar[r]
&
\vcenter{\hbox{\begin{tikzpicture}[scale=0.35]
\draw (0, 0)--(-4, 4) node[above]{\tiny{$5$}};
\draw (0, 0)--(4, 4) node[above]{\tiny{$1$}};
\draw (-2, 4)node[above]{\tiny{$4$}}--(-3, 3)node[anchor=north east]{\tiny{$8$}};
\draw (0, 4)node[above]{\tiny{$3$}}--(1, 3)node[anchor=north west]{\tiny{$12$}};
\draw (2, 4)node[above]{\tiny{$2$}}--(-1, 1)node[anchor=north east]{\tiny{$13$}};
\draw (0, 0)--(0, -1) node[below]{\tiny{$6$}};
\end{tikzpicture}}}\ar[d]
\\
\vcenter{\hbox{\begin{tikzpicture}[scale=0.35]
\draw (0, 0)--(-4, 4) node[above]{\tiny{$5$}};
\draw (0, 0)--(4, 4) node[above]{\tiny{$1$}};
\draw (-2, 4)node[above]{\tiny{$4$}}--(-1, 3)node[anchor=north west]{\tiny{$14$}};
\draw (0, 4)node[above]{\tiny{$3$}}--(-2, 2)node[anchor=north east]{\tiny{$9$}};
\draw (2, 4)node[above]{\tiny{$2$}}--(3, 3)node[anchor=north west]{\tiny{$7$}};
\draw (0, 0)--(0, -1) node[below]{\tiny{$6$}};
\end{tikzpicture}}}\ar[d]
&&
\vcenter{\hbox{\begin{tikzpicture}[scale=0.35]
\draw (0, 0)--(-4, 4) node[above]{\tiny{$5$}};
\draw (0, 0)--(4, 4) node[above]{\tiny{$1$}};
\draw (-2, 4)node[above]{\tiny{$4$}}--(0, 2)node[anchor=north west]{\tiny{$11$}};
\draw (0, 4)node[above]{\tiny{$3$}}--(1, 3)node[anchor=north west]{\tiny{$12$}};
\draw (2, 4)node[above]{\tiny{$2$}}--(-1, 1)node[anchor=north east]{\tiny{$13$}};
\draw (0, 0)--(0, -1) node[below]{\tiny{$6$}};
\end{tikzpicture}}}\ar[d]
\\
\vcenter{\hbox{\begin{tikzpicture}[scale=0.35]
\draw (0, 0)--(-4, 4) node[above]{\tiny{$5$}};
\draw (0, 0)--(4, 4) node[above]{\tiny{$1$}};
\draw (-2, 4)node[above]{\tiny{$4$}}--(1, 1)node[anchor=north west]{\tiny{$10$}};
\draw (0, 4)node[above]{\tiny{$3$}}--(-1, 3)node[anchor=north east]{\tiny{$14$}};
\draw (2, 4)node[above]{\tiny{$2$}}--(3, 3)node[anchor=north west]{\tiny{$7$}};
\draw (0, 0)--(0, -1) node[below]{\tiny{$6$}};
\end{tikzpicture}}}\ar[r]
&
\vcenter{\hbox{\begin{tikzpicture}[scale=0.35]
\draw (0, 0)--(-4, 4) node[above]{\tiny{$5$}};
\draw (0, 0)--(4, 4) node[above]{\tiny{$1$}};
\draw (-2, 4)node[above]{\tiny{$4$}}--(1, 1)node[anchor=north west]{\tiny{$10$}};
\draw (0, 4)node[above]{\tiny{$3$}}--(-1, 3)node[anchor=north east]{\tiny{$14$}};
\draw (2, 4)node[above]{\tiny{$2$}}--(0, 2)node[anchor=north east]{\tiny{$11$}};
\draw (0, 0)--(0, -1) node[below]{\tiny{$6$}};
\end{tikzpicture}}}\ar[r]
&
\vcenter{\hbox{\begin{tikzpicture}[scale=0.35]
\draw (0, 0)--(-4, 4) node[above]{\tiny{$5$}};
\draw (0, 0)--(4, 4) node[above]{\tiny{$1$}};
\draw (-2, 4)node[above]{\tiny{$4$}}--(1, 1)node[anchor=north west]{\tiny{$10$}};
\draw (0, 4)node[above]{\tiny{$3$}}--(1, 3)node[anchor=north west]{\tiny{$12$}};
\draw (2, 4)node[above]{\tiny{$2$}}--(0, 2)node[anchor=north east]{\tiny{$11$}};
\draw (0, 0)--(0, -1) node[below]{\tiny{$6$}};
\end{tikzpicture}}}
\end{tikzcd}
    \caption{$4$-$4$ Move Identity}
    \label{fig:4-4}
\end{figure}
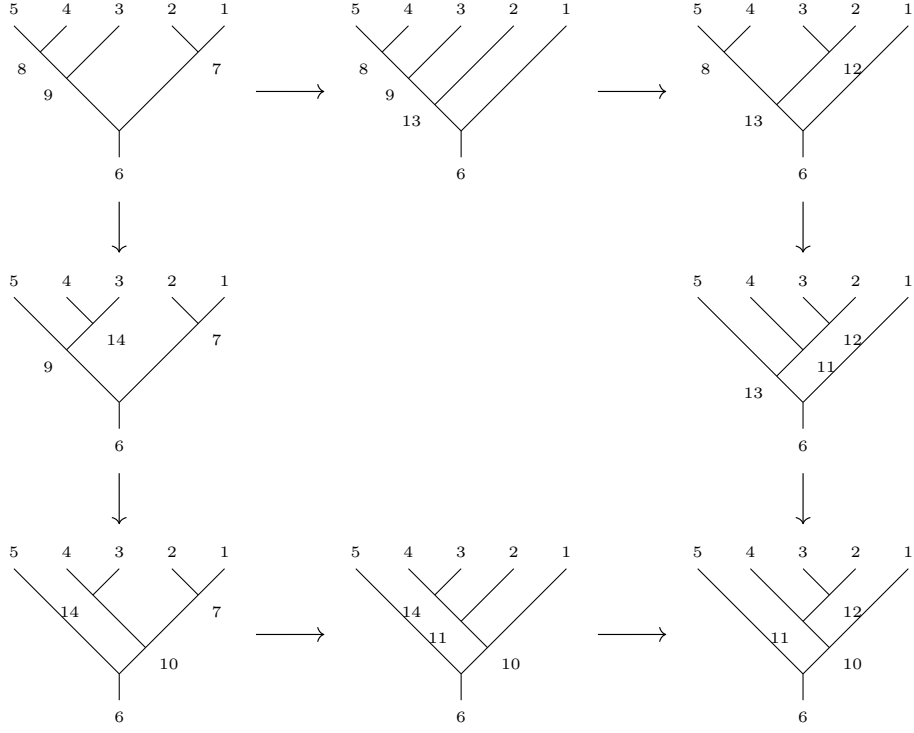
}

By Proposition~\ref{prop:6j-cont} for the continuity of the $b$-$6j$ symbols,
and Proposition~\ref{prop:conv15} (1) right below, the kernels of  $V_1$ and $V_2$ are both continuous functions in the variables $\alpha_i$'s. Therefore, the fact that $V_1=V_2$ as unitary maps from $L^2(\mathbb{S}^3\times \R,d\rho(a_5)d\rho(a_9)d\rho(a_{10})dx_5) $ to $ L^2(\mathbb{S}^3\times \R,d\rho(a_5)d\rho(a_6)d\rho(a_{7})dx_5)$ yields that their kernel agree as continuous functions. This gives the Pentagon Equation~\eqref{pentagon} for $b\in(0,1)$ with $q=e^{\pi\mathbf{i}b^2}$ not a root of unity.  By the continuity in $b$ of the relevant funcions given in Proposition~\ref{prop:conv15} (2), the  identity holds for all $b\in(0,1)$. The equation~\eqref{44moveidentity} for the 4-4 move follows from  $W_1=W_2$ in the same way.\qedhere
\end{proof}

\begin{prop}\label{prop:conv15}
\begin{enumerate}[(1)]
    \item The integrals 
\begin{align}\label{integral1}
&\int_\mathbb{S}|S_b(2a_8)|^2\begin{Bmatrix} 
    a_3 & a_7 & a_6\\
      a_1 & a_2 & a_8 
   \end{Bmatrix}_b
   \begin{Bmatrix} 
   a_4 & a_5 & a_7\\
      a_1 & a_8 & a_9 
   \end{Bmatrix}_b
   \begin{Bmatrix} 
   a_4 & a_9 & a_8\\
      a_2 & a_3 & a_{10}
   \end{Bmatrix}_b
   d\mathrm{Im}a_{8},
   \end{align}
   \begin{align}\label{integral2}
&\int_\mathbb{S} |S_b(2a_{13})|^2\begin{Bmatrix} 
    a_{13} & a_6 & a_1\\
      a_{10} & a_{11} & a_5 
   \end{Bmatrix}_b
\begin{Bmatrix} 
    a_{13} & a_{12} & a_8\\
      a_4 & a_5 & a_{11} 
   \end{Bmatrix}_b
\begin{Bmatrix} 
    a_{8} & a_{13} & a_{12}\\
      a_2 & a_3 & a_9 
   \end{Bmatrix}_b
\begin{Bmatrix} 
    a_9 & a_6 & a_{7}\\
      a_1 & a_2 & a_{13} 
   \end{Bmatrix}_bd\mathrm{Im}a_{13},
\end{align}
\begin{align}\label{integral3}
   &\int_\mathbb{S} |S_b(2a_{14})|^2
   \begin{Bmatrix} 
   a_{2} & a_{11} & a_{14}\\
      a_{4} & a_{3} & a_{12}
   \end{Bmatrix}_b
\begin{Bmatrix} 
    a_{7} & a_{10} & a_{14}\\
      a_{11} & a_2 & a_{1} 
   \end{Bmatrix}_b
\begin{Bmatrix} 
    a_{7} & a_{9} & a_{6}\\
      a_5 & a_{10} & a_{14} 
   \end{Bmatrix}_b
\begin{Bmatrix} 
    a_5 & a_9 & a_{14}\\
      a_3 & a_4 & a_{8} 
   \end{Bmatrix}_b
   d\mathrm{Im}a_{14}.
\end{align}
converge and provide the kernels of the compositions of the fusion transformation operators $V_2,W_1,W_2$, which are continuous in $\alpha_i$'s in $\mathbb{S}$.
\item The $b$-$6j$ symbols, as well as the integrals above, are continuous in $b\in(0,1)$.
\end{enumerate}   
\end{prop}

\begin{proof}
Recall the estimate \eqref{Integrandestimate}, which implies that for $\epsilon\in(0,\frac{Q}{2})$ and for $b$ in a compact interval $[b_{\min},b_{\max}]\subset (0,1)$, over the vertical line $\Gamma\doteq\{u\ |\ \mathrm{Re}u=2Q-\epsilon\}$ we have a uniform estimate:
\begin{align*}
    \Bigg|\prod_{i=1}^4S_b(u-t_i)\prod_{j=1}^4S_b(q_j-u)\Bigg|&\leqslant (C^8e^{-\pi \sum_{i}|\text{Im} u-\mathrm{Im}t_i|\epsilon}) e^{-\pi \sum_{j}|\text{Im} u-\mathrm{Im}q_j|(\frac{Q}{2}-\epsilon)}\\
    &\leqslant C(\epsilon,b_{\min},b_{\max})e^{-\pi \sum_{j}|\text{Im} u-\mathrm{Im}q_j|(\frac{b_{\min}}{2}-\epsilon)}.
\end{align*}
Here  $C=C(\epsilon,b)$ is as in \eqref{Integrandestimate}, which is a continuous function in $b$, and 
$$C(\epsilon,b_{\min},b_{\max})=\max_{b\in[b_{\min},b_{\max}]}C(\epsilon,b).$$
Thus for $\epsilon$ small enough, the Dominated Convergence Theorem gives that the $b$-$6j$ symbols are continuous in $b$.

Next, we show that for each $k\in\{1,\dots, 6\}$, the $b$-$6j$ symbol is dominated by a function that exponentially decays in $\mathrm{Im}a_k$. By the tetrahedral symmetry, it suffices to consider the case that $k=1.$  Let \begin{align*}
&\tilde q_1\doteq q_1-a_1=a_2+a_4+a_5\quad\text{and}\quad \tilde q_2\doteq q_2-a_2= a_3+a_4+a_6.
\end{align*}
Then over the vertical line $\Gamma$ the estimate \eqref{Sbestimate} results in the following:
\begin{align*}
&\Bigg|\int_\Gamma \prod_{i=1}^4S_b(u-t_i)\prod_{j=1}^4S_b(q_j-u)d\mathrm{Im}u\Bigg|
\leqslant C^8\int_{-\infty}^\infty  e^{-(\frac{Q}{2}-\epsilon)\pi(\sum_j|y-\mathrm{Im}q_j|)}dy\\
\leqslant & C^8\int_{-\infty}^0  e^{(\frac{Q}{2}-\epsilon)\pi(4y-\mathrm{Im}(q_3+\tilde q_2 +\tilde q_1)-2\mathrm{Im}a_1)}dy+ C^8\int^{\infty}_{2\mathrm{Im}a_1}  e^{-(\frac{Q}{2}-\epsilon)\pi(4y-\mathrm{Im}(q_3+\tilde q_2+\tilde q_1)-2\mathrm{Im}a_1)}dy\\
&+C^8(2\mathrm{Im}a_1+\mathrm{Im}(q_3+\tilde q_2+\tilde q_1))e^{-(\frac{Q}{2}-\epsilon)\pi(2\mathrm{Im}a_1-\mathrm{Im}(\tilde q_1+\tilde q_2+q_3))}\\
\leqslant  &C(Q,\epsilon,\mathrm{Im}\tilde q_1,\mathrm{Im}\tilde q_2,\mathrm{Im}q_3)(1+\mathrm{Im}a_1)e^{-(Q-2\epsilon)\pi \mathrm{Im}a_1}.
\end{align*}

Here $C(Q,\epsilon,\mathrm{Im}\tilde q_1,\mathrm{Im}\tilde q_2,\mathrm{Im}q_3)$ can be chosen as a continuous function in its inputs. By the estimate~\eqref{Sbestimate}, we have that over $\mathbb{S}$ the density of the Plancherel measure satisfies $|S_b(2a)|^2\leqslant e^{2 Q\pi \mathrm{Im}a}$. Therefore, for any compact interval $[b_{\min},b_{\max}]\subset (0,1)$, the integrand in \eqref{integral1}, which involves a product of three $b$-$6j$ symbols, is controlled by 
\begin{align*}
C_1(b_{\min},b_{\max},\epsilon,\{a_1,\dots,a_{10}\} \setminus\{a_8\})(1+\mathrm{Im}a_8)^3e^{-(b_{\min}-6\epsilon)\pi \mathrm{Im}a_8}
\end{align*}
for some continuous function $C_1(b_{\min},b_{\max},\epsilon,\{a_1,\dots,a_{10}\} \setminus\{a_8\})$ in its inputs; and the integrands \eqref{integral2}, \eqref{integral3}, which involves a product of four $b$-$6j$ symbols, are controlled by 
\begin{align*}
C_2(b_{\min},b_{\max},\epsilon,\{a_1,\dots,a_{12}\}) (1+\mathrm{Im}a_{13})^4e^{-(2b_{\min}-8\epsilon)\pi \mathrm{Im}a_{13}}.
\end{align*}
for some continuous function $C_2(b_{\min},b_{\max},\epsilon,\{a_1,\dots,a_{12}\}) $ in its inputs. Now for a sufficiently small $\epsilon$, these two bounds are integrable functions, respectively in $\mathrm{Im}a_8$ and $\mathrm{Im}a_{13}$. Then the Dominated Convergence Theorem ensures that these integrals continuously depend on the parameters $\{a_i\}$ and on $b$. 
\end{proof}

\subsection{Proof of Proposition \ref{prop:movepolygoncone} and Proposition \ref{prop:movepolyhedracone}} \label{proofofpolyhedracone}

\begin{proof}[Proof of Proposition \ref{prop:movepolygoncone}]
We orient the face $F_{uv}$ using the normal vector pointing inside the polyhedron $P_{u}$, and label its vertices in the cyclic order $B_{1}\ldots B_{n}$ such that $B_{1}=p_{uv}$ and $B_{t}$ be the cone point of the polygon glued to $P_u$ along $F_{uv}$. Notice that all the edges in the interior of $F_{uv}$ except $B_{1}B_{t}$ is adjacent to exactly $3$ tetrahedra, two being hyperideal and one being flat. 

If $B_{1}\ne B_{t}$, we subsequently perform $3$-$2$ moves with respect to the internal edges $B_{1}B_{t+1},$ $B_{1}B_{t+2},$ $\dots,$ $B_{1}B_{n-1}$ and $B_{1}B_{t-1},$ $B_{1}B_{t-2},\ldots B_{1}B_{3}$. Each of such $3$-$2$ moves involves two hyperideal tetrahedra and one flat tetrahedron. See Figure \ref{fig:3-2inprop5.2}. By doing so, all the flat tetrahedra are removed. We also observe that all the intermediate triangulations are deformable  since each edge is adjacent to at least one hyperideal tetrahedron, hence by Proposition \ref{DF} support angle structures. On the other hand, since the relative position of $B_{1}$ and $B_{t}$ is arbitrary, reversing this process by performing a sequence of $2$-$3$ moves as depicted in Figure \ref{fig:2-3inprop5.2} with $B_{1}=p_{uv}$   replaced by $p_{uv}'$ will result in $\mathcal{T}'$. By the same reason, all the intermediate triangulations support angle structures also. See Figure \ref{fig:basechange}.
\end{proof}

\begin{figure}
    \centering
    $
\vcenter{\hbox{
\begin{tikzpicture}
\begin{scope}[yscale=0.3, rotate=157.5]
\foreach \x in {1,...,6, 7, 8} \draw[green!50!black] (45*\x:3)--(45*\x+45:3);
\foreach \x in {1,...,6, 7, 8} \coordinate (a\x) at (-45*\x:3); 
\end{scope}
\coordinate (A) at (0, 4);
\draw (A) node[above]{$p_{u}$};
\draw (a1) node[above]{$p_{uv}$};
\draw (a5) node[below]{$B_{t}$};
\draw (a4) node[right]{$B_{4}$};
\draw (a3) node[right]{$p_{uv}'$};
\foreach \x in {3, 4, 5, 6, 7, 8} \draw[gray] (A) to (a\x);
\foreach \x in {1, 2} \draw[dashed, gray] (A) to (a\x);
\foreach \x in {2, 7, 8} \draw[dashed, blue!50!black] (a5) to (a\x);
\foreach \x in {3, 6, 7}
\draw[blue!50!black] (a1) to (a\x);
\draw[red] (a1) to (a5);
\draw[red] (a3) to (a5);
\path[fill=green, opacity=0.1](A)--(a1)--(a5)--(a3)--(A);
\end{tikzpicture}}}
\xrightarrow{3-2}
\vcenter{\hbox{
\begin{tikzpicture}
\begin{scope}[yscale=0.3, rotate=157.5]
\foreach \x in {1,...,6, 7, 8} \draw [green!50!black](45*\x:3)--(45*\x+45:3);
\foreach \x in {1,...,6, 7, 8} \coordinate (a\x) at (-45*\x:3); 
\end{scope}
\coordinate (A) at (0, 4);
\draw (A) node[above]{$p_{u}$};
\draw (a1) node[above]{$p_{uv}$};
\draw (a5) node[below]{$B_{t}$};
\draw (a4) node[right]{$B_{4}$};
\draw (a3) node[right]{$p_{uv}'$};
\foreach \x in {3, 4, 5, 6, 7, 8} \draw[gray] (A) to (a\x);
\foreach \x in {1, 2} \draw[dashed, gray] (A) to (a\x);
\foreach \x in {2, 3, 7, 8} \draw[dashed, blue!50!black] (a5) to (a\x);
\foreach \x in {6, 7}
\draw[blue!50!black] (a1) to (a\x);
\draw[red] (a1) to (a5);
\draw[red] (a3) to (a5);
\draw[red] (a2) to (a5);
\path[fill=green, opacity=0.1](A)--(a1)--(a5)--(a3)--(A);
\end{tikzpicture}}}
$
\caption{The $3$-$2$ move removing the edge $p_{uv}p_{uv}'$, turning the hyperideal tetrahedra $p_{u}p_{uv}p_{uv}'B_{2}$, $p_{u}p_{uv}p_{uv}'B_{t}$ and the  flat tetrahedron $p_{uv}p_{uv}'B_{2}B_{t}$ on the left to the hyperideal tetrahedra $p_{u}p_{uv}B_{2}B_{t}$ and $p_{u}p_{uv}'B_{2}B_{t}$ on the right.}
\label{fig:3-2inprop5.2}
\end{figure}

\begin{figure}
    \centering
$\vcenter{\hbox{\begin{tikzpicture}
\begin{scope}[yscale=0.3, rotate=157.5]
\foreach \x in {1,...,6, 7, 8} \draw[green!50!black] (45*\x:3)--(45*\x+45:3);
\foreach \x in {1,...,6, 7, 8} \coordinate (a\x) at (-45*\x:3); 
\end{scope}
\coordinate (A) at (0, 4);
\draw (A) node[above]{$p_{u}$};
\draw (a1) node[above]{$p_{uv}$};
\draw (a5) node[below]{$B_{t}$};
\draw (a2) node[above]{$B_{2}$};
\draw (a3) node[right]{$p_{uv}'$};
\foreach \x in {3, 4, 5, 6, 7, 8} \draw[gray] (A) to (a\x);
\foreach \x in {1, 2} \draw[dashed, gray] (A) to (a\x);
\foreach \x in {1, 2, 3, 7, 8} \draw[red] (a5) to (a\x);
\path[fill=green, opacity=0.1] (A)--(a1)--(a5)--(a3)--(A);
\end{tikzpicture}}}\xrightarrow{2-3}
\vcenter{\hbox{\begin{tikzpicture}
\begin{scope}[yscale=0.3, rotate=157.5]
\foreach \x in {1,...,6, 7, 8} \draw (45*\x:3)--(45*\x+45:3);
\foreach \x in {1,...,6, 7, 8} \coordinate (a\x) at (-45*\x:3); 
\end{scope}
\path[fill=green, opacity=0.1] (A)--(a1)--(a5)--(a3)--(A);
\coordinate (A) at (0, 4);
\draw (A) node[above]{$p_{u}$};
\draw (a1) node[above]{$p_{uv}$};
\draw (a5) node[below]{$B_{t}$};
\draw (a2) node[above]{$B_{2}$};
\draw (a3) node[right]{$p_{uv}'$};
\foreach \x in {3, 4, 5, 6, 7, 8} \draw[gray] (A) to (a\x);
\foreach \x in {1, 2} \draw[dashed, gray] (A) to (a\x);
\foreach \x in {1, 3, 7, 8} \draw[red] (a5) to (a\x);
\draw[blue!50!black] (a1) to (a3);
\draw[dashed, blue!50!black] (a2) to (a5);
\end{tikzpicture}}}$
    \caption{The $2$-$3$ move adding edge $p_{uv}p_{uv}'$, turning the hyperideal tetrahedra $p_{u}p_{uv}B_{t}B_{2}$ and $p_{u}p_{uv}'B_{t}B_{2}$ on the left to the  hyperideal tetrahedra $p_{u}p_{uv}p_{uv}'B_{2}$, $p_{u}p_{uv}p_{uv}'B_{t}$ and the flat tetrahedron $p_{uv}p_{uv}'B_{2}B_{t}$ on the right.}
    \label{fig:2-3inprop5.2}
\end{figure}

\begin{figure}
    \centering
$
\vcenter{\hbox{\begin{tikzpicture}[scale=0.9]
\coordinate (a1) at (0, 0);
\draw (a1)node[above]{$p_{uv}$};
\coordinate (a2) at (1.414, 0);
\draw (a2)node[above]{$B_{2}$};
\coordinate (a3) at (2.414, -1);
\draw (a3)node[right]{$p_{uv}'$};
\coordinate (a4) at (2.414, -2.414);
\draw (a4)node[right]{$B_{4}$};
\coordinate (a5) at (1.414, -3.414);
\draw (a5)node[below]{$B_{t}$};
\coordinate (a6) at (0, -3.414);
\draw (a6)node[below]{$B_{6}$};
\coordinate (a7) at (-1, -2.414);
\draw (a7)node[left]{$B_{7}$};
\coordinate (a8) at (-1, -1);
\draw (a8)node[left]{$B_{8}$};
\draw[green!50!black] (a1) to (a2) to (a3);
\draw[green!50!black] (a4) to (a5);
\draw[green!50!black] (a6) to (a7);
\draw[green!50!black] (a8) to (a1);
\draw[green!50!black] (a3) to (a4);
\draw[green!50!black] (a5) to (a6);
\draw[green!50!black] (a7) to (a8);
\foreach \x in {2, 3, 7, 8} \draw[dashed, blue!50!black] (a5) to (a\x);
\foreach \x in {3, 4, 6, 7}
\draw[blue!50!black] (a1) to (a\x);
\draw[red] (a1) to (a5);
\end{tikzpicture}}}
\xrightarrow{3-2}
\vcenter{\hbox{\begin{tikzpicture}[scale=0.9]
\coordinate (a1) at (0, 0);
\draw (a1)node[above]{$p_{uv}$};
\coordinate (a2) at (1.414, 0);
\draw (a2)node[above]{$B_{2}$};
\coordinate (a3) at (2.414, -1);
\draw (a3)node[right]{$p_{uv}'$};
\coordinate (a4) at (2.414, -2.414);
\draw (a4)node[right]{$B_{4}$};
\coordinate (a5) at (1.414, -3.414);
\draw (a5)node[below]{$B_{t}$};
\coordinate (a6) at (0, -3.414);
\draw (a6)node[below]{$B_{6}$};
\coordinate (a7) at (-1, -2.414);
\draw (a7)node[left]{$B_{7}$};
\coordinate (a8) at (-1, -1);
\draw (a8)node[left]{$B_{8}$};
\draw[green!50!black] (a1) to (a2) to (a3);
\draw[green!50!black] (a4) to (a5);
\draw[green!50!black] (a6) to (a7);
\draw[green!50!black] (a8) to (a1);
\draw[green!50!black] (a3) to (a4);
\draw[green!50!black] (a5) to (a6);
\draw[green!50!black] (a7) to (a8);
\foreach \x in {1, 2, 3, 7, 8} \draw[red] (a5) to (a\x);
\end{tikzpicture}}}
\xrightarrow{2-3}
\vcenter{\hbox{\begin{tikzpicture}[scale=0.9]
\coordinate (a1) at (0, 0);
\draw (a1)node[above]{$p_{uv}$};
\coordinate (a2) at (1.414, 0);
\draw (a2)node[above]{$B_{2}$};
\coordinate (a3) at (2.414, -1);
\draw (a3)node[right]{$p_{uv}'$};
\coordinate (a4) at (2.414, -2.414);
\draw (a4)node[right]{$B_{4}$};
\coordinate (a5) at (1.414, -3.414);
\draw (a5)node[below]{$B_{t}$};
\coordinate (a6) at (0, -3.414);
\draw (a6)node[below]{$B_{6}$};
\coordinate (a7) at (-1, -2.414);
\draw (a7)node[left]{$B_{7}$};
\coordinate (a8) at (-1, -1);
\draw (a8)node[left]{$B_{8}$};
\draw[green!50!black] (a1) to (a2) to (a3);
\draw[green!50!black] (a4) to (a5);
\draw[green!50!black] (a6) to (a7);
\draw[green!50!black] (a8) to (a1);
\draw[green!50!black] (a3) to (a4);
\draw[green!50!black] (a5) to (a6);
\draw[green!50!black] (a7) to (a8);
\foreach \x in {1, 2, 7, 8} \draw[dashed, blue!50!black] (a5) to (a\x);
\foreach \x in {1, 8, 6, 7}
\draw[blue!50!black] (a3) to (a\x);
\draw[red] (a3) to (a5);
\end{tikzpicture}}}$
\caption{Change of triangulations on $F_{uv}$: First apply $3$-$2$ moves to remove the edges $p_{uv}B_{4}, p_{uv}p_{uv}'$, $p_{uv}B_{6}$, and $p_{uv}B_{7}$ in order, removing all the layered flat tetrahedra in between $F_{uv}$ and the polygon been glued to it. Then apply $2$-$3$ moves to add the edges $p_{uv}'p_{uv}$ $p_{uv}'B_{8}, p_{uv}'B_{7}$, and $p_{uv}'B_{6}$ in order to obtain $\mathcal{T}'$, inserting the  layered flat tetrahedra.}
    \label{fig:basechange}
\end{figure}

In the rest of this subsection, we will prove Proposition \ref{prop:movepolyhedracone} by connecting $\mathcal{T}\doteq\mathcal{T}(\{p_{i}\}, \{p_{ij}\})$ and $\mathcal{T}'\doteq(\{p_{i}'\}, \{p_{ij}'\})$ described in Proposition \ref{prop:movepolyhedracone} by a sequence of $2$-$3$, $3$-$2$ and $4$-$4$ moves such that all the intermediate triangulations support angle structure. All these moves will take place among tetrahedra coming from subdividing $P_{u}$ and the flat tetrahedra inserted between faces of $P_{u}$ and the faces identified with them. We embed the convex hyperideal polyhedron $P_{u}$ into the hyperbolic space $\mathbb{H}^{3}$. To simplify the notation, we denote the adjacent vertices $p_{u}$ and $p_{u}'$ described in Proposition \ref{prop:movepolyhedracone} by $A$ and $B$ respectively in the rest of  this subsection, and further denote the two faces adjacent to $AB$ by $F_{1}$ and $F_{2}$. 
Let $\mathbf{B}$ be the subcomplex consisting of all faces of $P_{u}$ that are not adjacent to $A$ nor $B$ (see Figure~\ref{fig:polypu}), 
and equipped it with the triangulation inherited from $\mathcal{T}$.  By the convexity of $\mathbf{P}$, the base $\mathbf{B}$ is homeomorphic to a disk. Also observe that the restriction of the triangulations $\mathcal{T}$ and $\mathcal{T}'$ described in Proposition \ref{prop:movepolyhedracone}  are identical over $\mathbf{B}$. We remark that the triangulation on $\mathbf{B}$ will not be changed during the process.

Then we layout some standard terminology in topology, which will be used to describe the triangulations of $P_{u}$. Given two simplices $\Delta$ and $\Delta'$, the \textit{join} of $\Delta$ and $\Delta'$ is a simplex whose vertices are the union of vertices of $\Delta$ and $\Delta'$. The join of two simplicial complexes  $\mathbf{X}$ and $\mathbf{X}'$ is a simplicial complex consists of simplices equal to the join of a simplex in $\mathbf{X}$ with a simplex in $\mathbf{X}'$. For instance, the join of a vertex with a complex $\mathbf{X}$ is  the cone over $\mathbf X$.

Given a $2$-dimensional subcomplex $\mathbf{B}'$ of the base $\mathbf{B}$, we denote the join of $\{A\}$ (resp. $\{B\}$) and $\mathbf{B}'$ by $A\mathbf{B}'$ (resp. $B\mathbf{B}'$), which consists of tetrahedra spanned by the vertex $A$ and triangles in $\mathbf{B}'$. In the same manner, given a $1$-dimensional subcomplex $\mathbf{S}$ of $\mathbf{B}$, we denote the join of $\{A\}$ (resp. $\{B\}$) and $\mathbf{S}$ by $A\mathbf{S}$ (resp. $B\mathbf{S}$) and the join of $\{AB\}$ and $\mathbf{S}$ by $AB\mathbf{S}$. Given a complex $\mathbf{X}$ whose vertices belongs to the set of vertices of $P_{u}$, by the \textit{straightening} of $\textbf{X}$, we mean the complex isotopic to $\textbf{X}$ with all its edges geodesic segments, faces right-angled hexagons and tetrahedra either  hyperideal tetrahedra or flat tetrahedra. By the convexity of $P_u$, the straightening of $X$ lies in $P_u.$ In the rest of this subsection, for a complex, we will always automatically consider its straightening, unless stated otherwise.

We will be particularly interested in two types of $1$-dimensional subcomplexes $\mathbf{S}$ as defined in Definitions \ref{def:sep} and \ref{def:boundary}, whose joins with $\{AB\}$ will respectively contribute hyperideal tetrahedra and layered flat tetrahedra in the straightening.

A $1$-dimensional subcomplex $\mathbf{S}\subset \mathbf{B}$ is a \emph{separating path} if $\mathbf{S}\cap F_{1}$ is a singleton 
$\{S_{0}\}$, $\mathbf{S}\cap F_{2}$ is a singleton $\{S_{k}\}$, and  $\mathbf{S}$ is homeomorphic to an arc with $\partial S=\{S_{0}, S_{k}\}$. 
Observe that $\partial \mathbf{B}\setminus (\partial F_{1}\cup \partial F_{2})$ consists of two connected components, each of which satisfies the conditions of a separating path. We denote the one with its vertices connected to the vertex $A=p_{u}$ in $\partial P_{u}$ by $\mathbf{S}_{min}$, and deonote the other one whose vertices are connected to the vertex $B=p'_u$ in $\partial P_{u}$ by $\mathbf{S}_{max}$. See Figure \ref{fig:polypu}. A separating path $\mathbf{S}$ may separate $\mathbf{B}$ into disconnected components; and we denote the union of the connected components in between $\mathbf{S}$ and $\mathbf{S}_{min}$ by $\mathbf{B}_{L}^{\mathbf{S}}$, and denote the union of the connected components in between $\mathbf{S}$ and $\mathbf{S}_{max}$ by $\mathbf{B}_{R}^{\mathbf{S}}$.

\begin{figure}\label{fig:examplePu}
    \centering
$\vcenter{\hbox{
\begin{tikzpicture}
\coordinate (a1) at (-2, 1.5);
\coordinate (a2) at (1.9, 1.5);
\coordinate (a3) at (-2, -1.5);
\coordinate (a4) at (2, -1.5);
\coordinate (a15) at (0, 1.2);
\coordinate (a35) at (0, -1.8);
\draw[green!50!black, line width=1pt] (a1) to (a15) to (a2);
\draw[blue!70!black, line width=1pt] (a3) to (a35) to (a4);
\coordinate (b1) at (-2.6, -0.7);
\coordinate (b2) at (-2.8, 0.5);
\draw[purple, line width=1 pt] (a3)--(b1)--(b2) node[left]{$\mathbf{S}_{min}$}--(a1);
\coordinate (c1) at (3, -0.5);
\coordinate (c2) at (3, 0.5);
\draw[purple, line width=1 pt] (a4) --(c1)--(c2) node[right]{$\mathbf{S}_{max}$}--(a2);
\draw[red, line width=1 pt] (a35)node[below]{$S_{k}$}--(0.6, -1.1)--(-1, -0.5) node[below]{$\mathbf{S}$}--(0.8, 0.2)--(-0.8, 0.8)--(a15)node[below]{$S_{0}$};
\draw (-1, 0) node[left]{$\mathbf{B}_{L}^{S}$};
\draw (1, 0) node[right]{$\mathbf{B}_{R}^{S}$};
\coordinate (A) at (-1, 4);
\coordinate (B) at (1, 4);
\draw [line width=1 pt]  (A) node[above]{$A$}--(B) node[above]{$B$};
\draw [line width=0.6 pt] (A) to (b1);
\draw  [line width=0.6 pt] (A) to (b2);
\draw[dashed, green!50!black, line width=1pt] (A) to (a1);
\draw[blue!70!black, line width=1pt] (A) to (a3);
\draw[line width=0.6 pt, dashed] (A) to (a15);
\draw [line width=0.6 pt]  (A) to (a35);
\draw [line width=0.6 pt] (B) to (c1);
\draw [line width=0.6 pt] (B) to (c2);
\draw[dashed, green!50!black, line width=1pt] (B) to (a2);
\draw[blue!70!black, line width=1pt] (B) to (a4);
\draw[line width=0.6 pt, dashed] (B) to (a15);
\draw [line width=0.6 pt]  (B) to (a35);
\draw[green!50!black] (0,3) node{$F_1$};
\draw[blue!70!black] (0,2) node{$F_2$};
\end{tikzpicture}}}
$
\caption{The triangulation $\mathcal{T}_{\mathbf{S}}'$ of the polyhedron $P_{u}$: The red curve $\mathbf{S}$ separate base $\mathbf{B}$ into $\mathbf{B}_{L}^{\mathbf{S}}$ and $\mathbf{B}_{R}^{\mathbf{S}}$. The faces $F_{1}$ and $F_{2}$ are the two faces of $P_{u}$ adjacent to edge $AB\doteq p_{u}p_{u}'$, where $F_2$ is facing outside}.
\label{fig:polypu}
\end{figure}
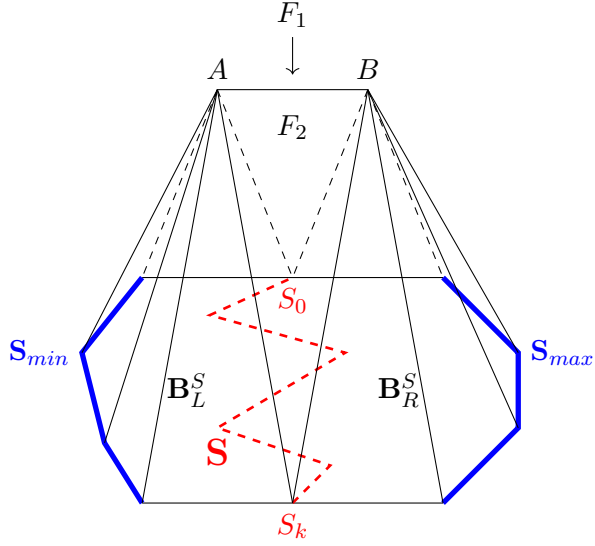

\begin{definition}\label{def:sep}
A separating path $\mathbf{S}\subset \mathbf{B}$ is \emph{admissible} if the straightening of 
$A\mathbf{B}_{L}^{\mathbf{S}}\cup AB\mathbf{S} \cup B\mathbf{B}_{R}^{\mathbf{S}}$ form a geometric triangulation of $P_{u}$ consisting of  only hyperideal tetrahedra.
\end{definition}

\begin{definition}\label{def:boundary}
    A $1$-dimensional subcomplex $\mathbf S\subset \mathbf B$ is a \emph{boundary path} if it consists of consecutive edges of either $F_{1}$ or $F_{2}$, in which case the join $AB\mathbf{S}$ consists of the layered flat tetrahedra.
\end{definition}

We will prove Proposition \ref{prop:movepolyhedracone} in the following  three steps. 
\begin{enumerate}[Step 1.]
    \item  To each admissible path $\mathbf{S}$, we associate it with one or two triangulations of $M$, respectively denoted by $\mathcal{T}_{\mathbf{S}}'$ and $\mathcal{T}_{\mathbf{S}}''$, which differ by a single $2$-$3$ or $3$-$2$ move. In addition, we have $\mathcal{T}_{\mathbf{S}_{max}}'=\mathcal T$ and $\mathcal{T}_{\mathbf{S}_{min}}'=\mathcal T'$ as described in Proposition \ref{prop:movepolyhedracone}.
    \item We show that  the triangulations $\mathcal{T}_{\mathbf{S}}'$ and $\mathcal{T}_{\mathbf{S}}''$ support  angle structures. See Proposition \ref{prop:anglestructure}.
    
    \item We define a partial order $\prec$ over the set of admissible paths such that $\mathbf{S}_{max}$ and $\mathbf{S}_{min}$ are respectively the maximum and the minimum, and show for each non-minimal admissible path $\mathbf{S}$, there exist another admissible path $\mathbf{S}^{*}\prec \mathbf{S}$ such that there is a pair of triangulations $\mathcal T_{\mathbf S}$ and $\mathcal T_{\mathbf S^*}$ each of which respecitvely belonging to one of the triangulations assigned to $\mathbf{S}$ and $\mathbf{S}^{*}$ in Step 1, and they differ by a single $2$-$3$, $3$-$2$ or $4$-$4$ move. See Proposition \ref{prop:order}. 
\end{enumerate}

\begin{proof}[Proof of Proposition \ref{prop:movepolyhedracone}]
The two triangulations $\mathcal{T}$ and $\mathcal{T}'$ in Proposition \ref{prop:movepolyhedracone} are respectively $\mathcal{T}_{\mathbf{S}_{max}}$ and $\mathcal{T}_{\mathbf{S}_{min}}$. Since the number of admissible paths is finite, by Proposition $\ref{prop:order}$, the triangulations $\mathcal{T}_{\mathbf{S}_{max}}$ and $\mathcal{T}_{\mathbf{S}_{min}}$ are connected by a sequence of triangulations of the form $\mathcal{T}_{\mathbf{S}}$  such that the adjacent ones differ by a single $2$-$3$, $3$-$2$ or $4$-$4$ move; and by Proposition \ref{prop:anglestructure}, each of the intermediate triangulations admits angle structures. This completes the proof.
\end{proof}

The three steps will respectively by carried out in Subsections~\ref{subsec:step1}, \ref{subsec:step2} and \ref{subsec:step3}.

\subsubsection{Step 1: Construction of the intermediate triangulations $\mathcal{T}_{\mathbf{S}}'$ and $\mathcal{T}_{\mathbf{S}}''$.}\label{subsec:step1}

Let $\mathbf{S}\subset \mathbf{B}$ be an admissible path with vertices $S_{0}\in F_{1}, S_{1}, \ldots, S_{k}\in F_{2}$ ordered following the path. We now define an ideal triangulation $\mathcal{T}_{\mathbf{S}}'$ as follows.

We first subdivide all the polyhedra $P_{i}\ne P_{u}$ in the Kojima polyhedral decomposition in the same way as $\mathcal{T}$ does, and subdivide $P_{u}$ into  $A\mathbf{B}_{L}^{\mathbf{S}}\cup AB\mathbf{S} \cup B\mathbf{B}_{R}^{\mathbf{S}}$. By doing so, the triangulations on the polygonal faces glued together might be different. If on each of the two polygonal faces that are glued together, the triangles of each polygon share a common vertex, then we insert layered flat tetrahedra in the same way as in the definition of Kojima triangulation  (see section \ref{sec:2.8}). Observe that all the polygonal faces except $F_{1}$ and $F_{2}$ in $P_{u}$ are triangulated in this way. It remains to describe how to insert flat tetrahedra between $F_{1}$ (resp. $F_{2}$) and the polygonal face glued to it. 

To this end, we first describe the triangulations of $F_{1}$ and $F_{2}$ inherited from the subdivision of $P_{u}$ givin by the admissible path $\mathbf{S}=\{S_{0},\ldots, S_{k}\}$. The triangulation over $F_{1}$ (resp. $F_{2}$) could be described as follow: First connect $A$ and $S_{0}$ (resp. $A$ and $S_{k}$) to separate $F_{1}$ (resp. $F_2$) into two components, connecting $B$ with all the other vertices in its component, and connecting all the other vertices with $A$. See Figure \ref{fig:triangulationf1}.

Now we are ready to describe $\mathcal{T}_{\mathbf{S}}'$.

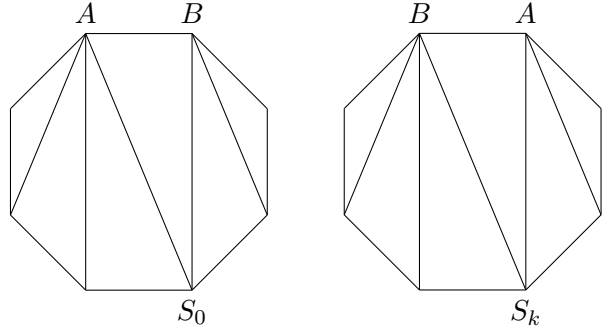
\begin{figure}
\centering
$\vcenter{\hbox{
\begin{tikzpicture}
\coordinate (a1) at (0, 0);
\draw (a1)node[above]{$A$};
\coordinate (a2) at (1.414, 0);
\draw (a2)node[above]{$B$};
\coordinate (a3) at (2.414, -1);
\draw (a3)node[right]{};
\coordinate (a4) at (2.414, -2.414);
\draw (a4)node[right]{};
\coordinate (a5) at (1.414, -3.414);
\draw (a5)node[below]{$S_{0}$};
\coordinate (a6) at (0, -3.414);
\draw (a6)node[below]{};
\coordinate (a7) at (-1, -2.414);
\draw (a7)node[left]{};
\coordinate (a8) at (-1, -1);
\draw (a8)node[left]{};
\draw [line width=0.8 pt] (a1) to (a2);
\draw [blue!70!black, line width=0.8 pt] (a2) to (a3);
\draw [blue!70!black, line width=0.8 pt](a4) to (a5);
\draw [blue!70!black, line width=0.8 pt](a6) to (a7);
\draw [blue!70!black, line width=0.8 pt](a8) to (a1);
\draw [blue!70!black, line width=0.8 pt](a3) to (a4);
\draw [blue!70!black, line width=0.8 pt](a5) to (a6);
\draw [blue!70!black, line width=0.8 pt](a7) to (a8);
\foreach \x in {5, 6, 7}
\draw (a1) to (a\x);
\draw (a2) to (a4);
\draw (a2) to (a5);
\end{tikzpicture}}} \hspace{4mm} \vcenter{\hbox{
\begin{tikzpicture}
\coordinate (a1) at (0, 0);
\draw (a1)node[above]{$B$};
\coordinate (a2) at (1.414, 0);
\draw (a2)node[above]{$A$};
\coordinate (a3) at (2.414, -1);
\draw (a3)node[right]{};
\coordinate (a4) at (2.414, -2.414);
\draw (a4)node[right]{};
\coordinate (a5) at (1.414, -3.414);
\draw (a5)node[below]{$S_{k}$};
\coordinate (a6) at (0, -3.414);
\draw (a6)node[below]{};
\coordinate (a7) at (-1, -2.414);
\draw (a7)node[left]{};
\coordinate (a8) at (-1, -1);
\draw (a8)node[left]{};
\draw [line width=0.8 pt](a1) to (a2);
\draw [green!50!black, line width=0.8 pt] (a2) to (a3);
\draw [green!50!black, line width=0.8 pt] (a4) to (a5);
\draw [green!50!black, line width=0.8 pt] (a6) to (a7);
\draw [green!50!black, line width=0.8 pt] (a8) to (a1);
\draw [green!50!black, line width=0.8 pt] (a3) to (a4);
\draw [green!50!black, line width=0.8 pt] (a5) to (a6);
\draw [green!50!black, line width=0.8 pt] (a7) to (a8);
\foreach \x in {5, 6, 7}
\draw (a1) to (a\x);
\draw (a2) to (a4);
\draw (a2) to (a5);
\end{tikzpicture}}}$
\caption{Triangulations over polygon $F_{1}$ and $F_{2}$, viewed from the inside of $P_{u}.$}
\label{fig:triangulationf1}
\end{figure}

\bigskip

\noindent Case 1: Suppose $F_{1}$ and $F_{2}$ are not identified with each other. In this case, we describe how to insert the flat tetrahedra between $F_{1}$ and the face $F$ glued to it. The construction between $F_2$ and the polygon glued to it follows verbatim  by replacing $S_{0}$ by $S_{k}$.

We label the vertices of $F_1$ in the cyclic order $B_{1}, B_{2}, \ldots, B_{n}$ such that $A=B_{1}, B=B_{2}, S_{0}=B_{s}$ and the cone point $\bar{A}$ of $F$ is identified with $B_{t}$. We call the face $F_{1}$ the \textit{front side} (because it is always facing us in the illustrative diagrams) and the other face the \textit{back side}. Now on the back side polygon, all triangles share the same vertex $\bar{A}=B_{t}$; and on the front side, $B=B_{2}$  is connected to $B_{i}$ for all $3\leqslant i\leqslant s$, and $A=B_{1}$  is connected to $B_{j}$ for all $s+1 \leqslant j \leqslant n$.  We have the following two subcases. 

\begin{enumerate}
\item[(1a)] $s\leqslant t$. Denote the boundary paths $\mathbf{S}_{1}=\{B_{3}, B_{4}\ldots B_{s}=S_{0}\}$, $\mathbf{S}_{2}=\{B_{s}, B_{s+2}\ldots B_{t-1}\}$ and $\mathbf{S}_{3}=\{B_{t+1}, B_{t+2}\ldots B_{n}\}$ which could be empty. We insert flat tetrahedra $B\bar{A}\mathbf{S}_{1}$, $ABS_{0}\bar{A}$, $A\bar{A}\mathbf{S}_{2}$, and $A\bar{A}\mathbf{S}_{3}$, layered in order from the bottom to the top.
\item[(1b)] $s>t$, Denote the boundary paths $\mathbf{S}_{1}=\{B_{3}, B_{4}\ldots B_{t-1}\}$, $\mathbf{S}_{2}=\{B_{t+1}, B_{s+2}\ldots B_{s}\}$ and $\mathbf{S}_{3}=\{B_{s}, B_{s+1}\ldots B_{n}\}$ which could be empty. We insert flat tetrahedra $A\bar{A}\mathbf{S}_{3}$, $ABS_{0}\bar{A}$, $B\bar{A}\mathbf{S}_{2}$ and $B\bar{A}\mathbf{S}_{1}$, layered in order from the bottom to the top.
\end{enumerate}
\begin{remark}
Notice that some tetrahedra in the list could be degenerate. For instance, if the vertex $S_{0}$ is identified with vertex $\bar{A}$, then the tetrahedra $ABS_{0}\bar{A}$ becomes degenerate, if this happen, we simply remove it from our list. We will follow this convention in the rest of the section, to reduce the number of situations without causing ambiguity.
\end{remark}
$$
\begin{array}{cc}
\begin{tikzpicture}[rotate=105]
    \foreach \x in {1,...,12} \draw (30*\x-30:3)--(30*\x:3);
    \foreach \x in {1,2,...,12} \coordinate (a\x) at (-30*\x+30:3);
    \foreach \x in {7, 8} \draw (a\x) node[below]{$B_{\x}$};
    \foreach \x in {9,12,11} \draw (a\x) node[left]{$B_{\x}$};
    \foreach \x in {5, 3, 4} \draw (a\x) node[right]{$B_{\x}$};
    \draw (a1) node[above]{$A$};
    \draw (a2) node[above]{$B$};
    \draw (a6) node[right]{$S_{0}$};
    \draw (a10) node[left]{$\bar{A}$};
    \foreach \x in {11, 6, 7, 8, 9, 10} \draw(a1) to (a\x);
    \foreach \x in {6, 4, 5} \draw(a2) to (a\x);
    \foreach \x in {7, 12, 8, 1, 2, 3, 4, 5, 6} \draw[dashed] (a10) to (a\x);
    \draw[red] (a1) to (a10);
\end{tikzpicture}&
\begin{tikzpicture}[rotate=105]
    \foreach \x in {0,1,...,11} \draw (30*\x:3)--(30*\x+30:3);
    \foreach \x in {1,2,...,12} \coordinate (a\x) at (-30*\x+30:3);
    \foreach \x in {8, 7} \draw (a\x) node[below]{$B_{\x}$};
    \foreach \x in {9,12,11} \draw (a\x) node[left]{$B_{\x}$};
    \foreach \x in {5, 3, 4} \draw (a\x) node[right]{$B_{\x}$};
    \draw (a1) node[above]{$A$};
    \draw (a2) node[above]{$B$};
    \draw (a6) node[right]{$\bar{A}$};
    \draw (a10) node[left]{$S_{0}$};
    \foreach \x in {11, 10} \draw(a1) to (a\x);
    \foreach \x in {10, 3, 4, 5, 6, 7, 8, 9} \draw(a2) to (a\x);
    \foreach \x in {4, 8, 9, 10, 11, 12, 1, 2, 3} \draw[dashed] (a6) to (a\x);
    \draw[red] (a2) to (a6);
\end{tikzpicture}\\
\text{Case (1a)}&\text{Case (1b)}\\
\end{array}
$$

Before moving to Case 2, we explain the convention of the colors in the illustrating figures for the layered tetrahedra inserted in between polygonal faces.
The solid edges come from the triangulations of the polygon on the front side and the dashed edges come from the triangulations of the polygon on the back side.
Here the front and back sides are relative to an observer inside the polygon $P_u$.
The red edges are shared by polygons on both side and the orange edges (only appear in Cases 2 and 3) do not belong to either of the polygons. 
\\

\noindent Case 2: Suppose $F_{1}$ and $F_{2}$ are glued together by an orientation reversing isometry. We will still call $F_{1}$ the front side, and call $F_{2}$ as the back side. To distinguish the vertices $A, B$ from $F_{1}$ and from $F_{2}$, we denote the latter  by $\bar{A}$ and $\bar{B}.$ We label the vertices of $F_{1}$ in the cyclic order $B_{1},\ldots B_{n}$, with $B_{1}=A$ and $B_{2}=B.$ The way that the layered flat tetrahedra are inserted depends on the cyclic order of the six distinguished vertices, namely, $S_{0}, S_{k}$, $A, B$ and $\bar A, \bar B.$ Suppose $\bar{A}=B_{t}, S_{0}=B_{s}$ and $S_{k}=B_{s'}$. Then the orientation-reversing assumption implies that  $\bar B=B_{t+1}$. We have the following six subcases. 

\begin{enumerate}
\item[(2a)] $2\leqslant s \leqslant t < s'$, i.e., $AB< S_{0}\leqslant \bar{A}\bar{B}< S_{k}$. Denote the boundary paths $\mathbf{S}_{1}=\{B_{3},\ldots, B_{s}\},  \mathbf{S}_{2}=\{B_{s},\ldots, B_{t-1}\}, \mathbf{S}_{3}=\{B_{t+2},\ldots, B_{s'}\}$, and $\mathbf{S}_{4}=\{B_{s'},\ldots, B_{n}\}$. We insert flat tetrahedra $\bar{A}B\mathbf{S}_{1},\bar{A}ABS_{0},  \bar{A}A\mathbf{S}_{2}, \bar{A}A\mathbf{S}_{4},$ $S_{k}A\bar{A}\bar{B}$, and
$\bar{B}A\mathbf{S}_{3}$, layered from the back to the front.

\item[(2b)] $2\leqslant s' \leqslant t < s$, i.e., $AB\leqslant S_{k}< \bar{A}\bar{B}\leqslant S_{0}$. Denote the boundary paths $\mathbf{S}_{1}=\{B_{3},\ldots, B_{s'}\},  \mathbf{S}_{2}=\{B_{s'},\ldots, B_{t-1}\}, \mathbf{S}_{3}=\{B_{t+2},\ldots, B_{s}\}$, and $\mathbf{S}_{4}=\{B_{s},\ldots, B_{n}\}$. We insert flat tetrahedra $\bar{B}B\mathbf{S}_{1}, S_{k}B\bar{A}\bar{B}, \bar{A}B\mathbf{S}_{2}, \bar{B}A\mathbf{S}_{4}, \bar{B}ABS_{0},$ and $\bar{B}B\mathbf{S}_{3}$, layered from the back to the front.
\end{enumerate}
$$\begin{array}{cc}\\
\begin{tikzpicture}[rotate=105]
    \foreach \x in {0,1,...,11} \draw (30*\x:3)--(30*\x+30:3);
    \foreach \x in {1,2,...,12} \coordinate (a\x) at (-30*\x+30:3);
    \foreach \x in {10,9,12} \draw (a\x) node[left]{$B_{\x}$};
    \foreach \x in {4, 3, 6} \draw (a\x) node[right]{$B_{\x}$};
    \draw (a1) node[above]{$A$};
    \draw (a2) node[above]{$B$};
    \draw (a7) node[below]{$\bar{A}$};
    \draw (a8) node[below]{$\bar{B}$};
    \draw (a5) node[right]{$S_{0}$};
    \draw (a11) node[left]{$S_{k}$};
    \foreach \x in {11, 5, 6, 7, 8, 9, 10} \draw (a1)--(a\x);
    \foreach \x in {5, 4} \draw (a2)--(a\x);
    \foreach \x in {11, 10} \draw[dashed] (a8)--(a\x);
    \foreach \x in {5, 1, 2, 3, 4, 12, 11} \draw[dashed] (a7)--(a\x);
    \draw[red] (a1)--(a7);
\end{tikzpicture}
&
\begin{tikzpicture}[xscale=-1, rotate=105]
    \foreach \x in {0,1,...,11} \draw (30*\x:3)--(30*\x+30:3);
    \foreach \x in {1,...,11,12} \coordinate (a\x) at (-30*\x:3);
    \foreach \x in {9,11,12} \draw (30*\x-60:3) node[left]{$B_{\x}$};
    \foreach \x in {3,5,6} \draw (30*\x-60:3) node[right]{$B_{\x}$};
    \draw (a12) node[above]{$B$};
    \draw (a1) node[above]{$A$};
    \draw (a6) node[below]{$\bar{B}$};
    \draw (a7) node[below]{$\bar{A}$};
    \draw (a4) node[left]{$S_{0}$};
    \draw (a10) node[right]{$S_{k}$};
    \foreach \x in {4, 5, 6, 7, 8, 9, 10} \draw (a12)--(a\x);
    \foreach \x in {3, 4} \draw (a1)--(a\x);
    \foreach \x in {9, 10} \draw[dashed] (a7)--(a\x);
    \foreach \x in {12, 1, 2, 3, 4, 10, 11} \draw[dashed] (a6)--(a\x);
    \draw[red] (a12)--(a6);
\end{tikzpicture}\\
\text{Case (2a)}&\text{Case (2b)}
\end{array}$$
\begin{enumerate}

\item[(2c)] $2\leqslant t < s' \leqslant s$, i.e., $AB\leqslant \bar{A}\bar{B}< S_{k}\leqslant S_{0}$. Denote the boundary paths $\mathbf{S}_{1}=\{B_{3},\ldots, B_{t-1}\},  \mathbf{S}_{2}=\{B_{t+2},\ldots, B_{s'}\}, \mathbf{S}_{3}=\{B_{s'},\ldots, B_{s}\}$, and $\mathbf{S}_{4}=\{B_{s},\ldots, B_{n}\}$. We insert flat tetrahedra $\bar{A}B\mathbf{S}_{1}, \bar{A}A\mathbf{S}_{4}, S_{0}AB\bar{A},$
$\bar{A}B\mathbf{S}_{3}, S_{k}B\bar{A}\bar{B},$ and 
$\bar{B}B\mathbf{S}_{2}$, layered from the back to the front.

\item[(2d)] $2\leqslant s \leqslant s' \leqslant t$, i.e., $AB\leqslant S_{0}\leqslant S_{k}< \bar{A}\bar{B}$. Denote the boundary  paths $\mathbf{S}_{1}=\{B_{3},\ldots, B_{s}\},   \mathbf{S}_{2}=\{B_{s},\ldots, B_{s'}\}, \mathbf{S}_{3}=\{B_{s'},\ldots, B_{t-1}\}$, and $\mathbf{S}_{4}=\{B_{t+2},\ldots, B_{n}\}$. We insert flat tetrahedra $\bar{A}B\mathbf{S}_{4}, \bar{A}A\mathbf{S}_{1}, S_{0}AB\bar{A}, \bar{A}B\mathbf{S}_{2}, S_{k}B\bar{A}\bar{B}$, and $\bar{B}B\mathbf{S}_{3}$, layered from the back to the front.
\end{enumerate}
$$
\begin{array}{cc}
   \begin{tikzpicture}[rotate=105]
    \foreach \x in {0,1,...,11} \draw (30*\x:3)--(30*\x+30:3);
    \foreach \x in {1,2,...,12} \coordinate (a\x) at (-30*\x+30:3);
    \foreach \x in {9,12,11} \draw (a\x) node[left]{$B_{\x}$};
    \foreach \x in {4, 3} \draw (a\x) node[right]{$B_{\x}$};
    \draw (a7) node[below]{$B_{7}$};
    \draw (a1) node[above]{$A$};
    \draw (a2) node[above]{$B$};
    \draw (a5) node[right]{$\bar{A}$};
    \draw (a6) node[right]{$\bar{B}$};
    \draw (a8) node[below]{$S_{k}$};
    \draw (a10) node[left]{$S_{0}$};
    \foreach \x in {10, 4, 5, 6, 7, 8, 9} \draw (a2)--(a\x);
    \foreach \x in {11, 10} \draw (a1)--(a\x);
    \foreach \x in {8} \draw[dashed] (a6)--(a\x);
    \foreach \x in {3, 1, 2, 12, 8, 9, 10, 11} \draw[dashed] (a5)--(a\x);
    \draw[red] (a2)--(a5);
\end{tikzpicture}  &  \begin{tikzpicture}[rotate=105]
    \foreach \x in {0,1,...,11} \draw (30*\x:3)--(30*\x+30:3);
    \foreach \x in {1,...,11,12} \coordinate (a\x) at (30-30*\x:3);
    \foreach \x in {11,12} \draw (a\x) node[left]{$B_{\x}$};
    \foreach \x in {3, 4, 6} \draw (a\x) node[right]{$B_{\x}$};
    \draw (a8) node[below]{$B_{8}$};
    \draw (a1) node[above]{$A$};
    \draw (a2) node[above]{$B$};
    \draw (a9) node[left]{$\bar{A}$};
    \draw (a10) node[left]{$\bar{B}$};
    \draw (a7) node[below]{$S_{k}$};
    \draw (a5) node[right]{$S_{0}$};
    \foreach \x in {11, 5, 6, 7, 8, 9, 10} \draw (a1)--(a\x);
    \foreach \x in {4, 5} \draw (a2)--(a\x);
    \foreach \x in {8, 7} \draw[dashed] (a9)--(a\x);
    \foreach \x in {7, 1, 2, 3, 4, 5, 6, 12} \draw[dashed] (a10)--(a\x);
    \draw[red] (a1)--(a10);
\end{tikzpicture}\\
   \text{Case (2c)}  & \text{Case (2d)} 
\end{array}
$$
\begin{enumerate}
    
\item[(2e)] $2\leqslant t \leqslant s< s'$, i.e., $AB\leqslant \bar{A}\bar{B}\leqslant S_{0}< S_{k}$. Denote the boundary paths $\mathbf{S}_{1}=\{B_{3},\ldots, B_{t-1}\}, \mathbf{S}_{2}=\{B_{t+2},\ldots, B_{s}\}, \mathbf{S}_{3}=\{B_{s},\ldots, B_{s'}\}$, and $\mathbf{S}_{4}=\{B_{s'},\ldots, B_{n}\}$. We add an extra edge connecting $A\bar{B}$ and insert flat tetrahedra 
$\bar{A}B\mathbf{S}_{1},$ $\bar{A}A\mathbf{S}_{4},$
$S_{k}A\bar{A}\bar{B},$
$\bar{B}A\mathbf{S}_{3},$
$ \bar{A}ABS_{0},$ 
$ AB\bar{A}\bar{B}$ and $\bar{B}B\mathbf{S}_{2}$, layered from the back to the front. We remark that in this case, if $S_{0}=\bar{B}$, then there is no flat tetrahedron on the top of the orange edge $A\bar{B}$. In this case, the edge $A\bar{B}$ belongs to $F_{1}$ and no extra edge is added. The remark also applies to the subcases (2f), (3e) and (3f) where the orange edge appears generically.

\item[(2f)] $2\leqslant s' < s \leqslant t$, i.e., $AB\leqslant S_{k}< S_{0}\leqslant \bar{A}\bar{B}$. Denote the boundary paths $\mathbf{S}_{1}=\{B_{3},\ldots, B_{s'}\},   \mathbf{S}_{2}=\{B_{s'},\ldots, B_{s}\}, \mathbf{S}_{3}=\{B_{s},\ldots, B_{t-1}\}$, and $\mathbf{S}_{4}=\{B_{t+2},\ldots, B_{n}\}$. We add an extra edge connecting $A\bar{B}$ and insert flat tetrahedra $\bar{B}A\mathbf{S}_{4},$
$ \bar{B}B\mathbf{S}_{1},$
$ \bar{A}ABS_{0},$
$ \bar{A}B\mathbf{S}_{2},$
$S_{k}B\bar{A}\bar{B},$ $AB\bar{A}\bar{B}$ and $\bar{A}A\mathbf{S}_{3}$, layered from the back to the front.
\end{enumerate}

$$
\begin{array}{cc}
\begin{tikzpicture}[rotate=105]
    \foreach \x in {0,1,...,11} \draw (30*\x:3)--(30*\x+30:3);
    \foreach \x in {1,...,11,12} \coordinate (a\x) at (30-30*\x:3);
    \foreach \x in {9,12,11} \draw (a\x) node[left]{$B_{\x}$};
    \foreach \x in {4, 3} \draw (a\x) node[right]{$B_{\x}$};
    \draw (a7) node[below]{$B_{7}$};
    \draw (a1) node[above]{$A$};
    \draw (a2) node[above]{$B$};
    \draw (a5) node[right]{$\bar{A}$};
    \draw (a6) node[right]{$\bar{B}$};
    \draw (a8) node[below]{$S_{0}$};
    \draw (a10) node[left]{$S_{k}$};
    \foreach \x in {8, 4, 5, 6, 7} \draw (a2)--(a\x);
    \foreach \x in {11, 8, 9, 10} \draw (a1)--(a\x);
    \foreach \x in {10, 8, 9} \draw[dashed] (a6)--(a\x);
    \foreach \x in {3, 1, 2, 12, 10, 11} \draw[dashed] (a5)--(a\x);
    \draw[red] (a2)--(a5);
    \draw[orange, densely dashed, line width=1pt] (a1)--(a6);
\end{tikzpicture}
&
\begin{tikzpicture}[rotate=105]
    \foreach \x in {0,1,...,11} \draw (30*\x:3)--(30*\x+30:3);
    \foreach \x in {1,...,11,12} \coordinate (a\x) at (30-30*\x:3);
    \foreach \x in {12,11} \draw (a\x) node[left]{$B_{\x}$};
    \foreach \x in {3, 4, 6} \draw (a\x) node[right]{$B_{\x}$};
    \draw (a8) node[below]{$B_{8}$};
    \draw (a1) node[above]{$A$};
    \draw (a2) node[above]{$B$};
    \draw (a9) node[left]{$\bar{A}$};
    \draw (a10) node[left]{$\bar{B}$};
    \draw (a7) node[below]{$S_{0}$};
    \draw (a5) node[right]{$S_{k}$};
    \foreach \x in {11, 7, 8, 9, 10} \draw (a1)--(a\x);
    \foreach \x in {7, 4, 5, 6} \draw (a2)--(a\x);
    \foreach \x in {8, 5, 6, 7} \draw[dashed] (a9)--(a\x);
    \foreach \x in {5, 1, 2, 3, 4, 12} \draw[dashed] (a10)--(a\x);
    \draw[red] (a1)--(a10);
    \draw[orange, densely dashed, line width=1pt] (a2)--(a9);
\end{tikzpicture}\\
\text{Case (2e)}&\text{Case (2f)}
\end{array}
$$
\bigskip

\noindent Case 3: Suppose the $F_{1}$ and $F_{2}$ are glued together by an orientation-preserving isometry. The way that the layered flat tetrahedra been inserted is almost identical to the previous case, simply with $\bar{A}$ and $\bar{B}$ interchanged in each of the sub-cases of Case (2). We still list the triangulations below for the readers' convenience. Adopting the same convention as in the Case (2), we label the vertices of $F_{1}$ in the cyclic order $B_{1},\ldots B_{n}$, and denote the vertices $A$ and $B$ from $F_{2}$ by $\bar{A}$ and $\bar{B}$ respectively. Set $A=B_{1}, \bar{A}=B_{t}, S_{0}=B_{s}$ and $S_{k}=B_{s'}$. 
Then the orientation-preserving assumption implies that  $\bar B=B_{t-1}$.

\begin{enumerate}
\item[(3a)] $2 < s < t < s'$, i.e., $AB< S_{0}\leqslant \bar{B}\bar{A}< S_{k}$. Denote the  boundary paths $\mathbf{S}_{1}=\{B_{3},\ldots, B_{s}\},  \mathbf{S}_{2}=\{B_{s},\ldots, B_{t-2}\}, \mathbf{S}_{3}=\{B_{t+1},\ldots, B_{s'}\}$, and $\mathbf{S}_{4}=\{B_{s'},\ldots, B_{n}\}$. We insert flat tetrahedra $\bar{B}B\mathbf{S}_{1},\bar{B}ABS_{0},  \bar{B}A\mathbf{S}_{2}, \bar{B}A\mathbf{S}_{4},$ $S_{k}A\bar{B}\bar{A}$, and
$\bar{A}A\mathbf{S}_{3}$, layered from the back to the front.

\item[(3b)] $2\leqslant s'< t\leqslant s$, i.e., $AB\leqslant S_{k}< \bar{B}\bar{A}\leqslant S_{0}$. Denote the boundary paths $\mathbf{S}_{1}=\{B_{3},\ldots, B_{s'}\},  \mathbf{S}_{2}=\{B_{s'},\ldots, B_{t-2}\}, \mathbf{S}_{3}=\{B_{t+1},\ldots, B_{s}\}$, and $\mathbf{S}_{4}=\{B_{s},\ldots, B_{n}\}$. We insert flat tetrahedra $\bar{A}B\mathbf{S}_{1}, S_{k}B\bar{B}\bar{A}, \bar{B}B\mathbf{S}_{2}, \bar{A}A\mathbf{S}_{4}, \bar{A}ABS_{0},$ and $\bar{A}B\mathbf{S}_{3}$, layered from the back to the front.
\end{enumerate}
$$\begin{array}{cc}\\
\begin{tikzpicture}[rotate=105]
    \foreach \x in {0,1,...,11} \draw (30*\x:3)--(30*\x+30:3);
    \foreach \x in {1,...,11,12} \coordinate (a\x) at (30-30*\x:3);
    \foreach \x in {10,9,12} \draw (a\x) node[left]{$B_{\x}$};
    \foreach \x in {3, 4, 6} \draw (a\x) node[right]{$B_{\x}$};
    \draw (a1) node[above]{$A$};
    \draw (a2) node[above]{$B$};
    \draw (a7) node[below]{$\bar{B}$};
    \draw (a8) node[below]{$\bar{A}$};
    \draw (a5) node[right]{$S_{0}$};
    \draw (a11) node[left]{$S_{k}$};
    \foreach \x in {11, 5, 6, 7, 8, 9, 10} \draw (a1)--(a\x);
    \foreach \x in {5, 4} \draw (a2)--(a\x);
    \foreach \x in {11, 10} \draw[dashed] (a8)--(a\x);
    \foreach \x in {5, 1, 2, 3, 4, 12, 11} \draw[dashed] (a7)--(a\x);
    \draw[red] (a1)--(a7);
\end{tikzpicture}
&
\begin{tikzpicture}[xscale=-1, rotate=105]
    \foreach \x in {0,1,...,11} \draw (30*\x:3)--(30*\x+30:3);
    \foreach \x in {1,...,11,12} \coordinate (a\x) at (30-30*\x:3);
    \foreach \x in {9,11,12} \draw (30*\x-60:3) node[left]{$B_{\x}$};
    \foreach \x in {3,6,5} \draw (30*\x-60:3) node[right]{$B_{\x}$};
    \draw (a1) node[above]{$B$};
    \draw (a2) node[above]{$A$};
    \draw (a7) node[below]{$\bar{A}$};
    \draw (a8) node[below]{$\bar{B}$};
    \draw (a5) node[left]{$S_{0}$};
    \draw (a11) node[right]{$S_{k}$};
    \foreach \x in {11, 5, 6, 7, 8, 9, 10} \draw (a1)--(a\x);
    \foreach \x in {5, 4} \draw (a2)--(a\x);
    \foreach \x in {11, 10} \draw[dashed] (a8)--(a\x);
    \foreach \x in {5, 1, 2, 3, 4, 12, 11} \draw[dashed] (a7)--(a\x);
    \draw[red] (a1)--(a7);
\end{tikzpicture}\\
\text{Case (3a)}&\text{Case (3b)}
\end{array}$$
\begin{enumerate}

\item[(3c)] $2<t<s'\leqslant s$, i.e., $AB\leqslant \bar{B}\bar{A}< S_{k}\leqslant S_{0}$. Denote the boundary paths $\mathbf{S}_{1}=\{B_{3},\ldots, B_{t-2}\},  \mathbf{S}_{2}=\{B_{t+1},\ldots, B_{s'}\}, \mathbf{S}_{3}=\{B_{s'},\ldots, B_{s}\}$, and $\mathbf{S}_{4}=\{B_{s},\ldots, B_{n}\}$. The inserted tetrahedra are $\bar{B}B\mathbf{S}_{1}, \bar{B}A\mathbf{S}_{4}, S_{0}AB\bar{B},$
$\bar{B}B\mathbf{S}_{3}, S_{k}B\bar{B}\bar{A},$ and 
$\bar{A}B\mathbf{S}_{2}$, layered from the back to the front.

\item[(3d)] $2<s\leqslant s'<t$, i.e., $AB\leqslant S_{0}\leqslant S_{k}< \bar{B}\bar{A}$. Denote the boundary  paths $\mathbf{S}_{1}=\{B_{3},\ldots, B_{s}\}, \mathbf{S}_{2}=\{B_{s},\ldots, B_{s'}\}, \mathbf{S}_{3}=\{B_{s'},\ldots, B_{t-2}\}$, and $\mathbf{S}_{4}=\{B_{t+1},\ldots, B_{n}\}$. The inserted tetrahedra are $\bar{B}B\mathbf{S}_{4}, \bar{B}A\mathbf{S}_{1}, S_{0}AB\bar{B}, \bar{B}B\mathbf{S}_{2}, S_{k}B\bar{B}\bar{A}$, and $\bar{A}B\mathbf{S}_{3}$, layered from the back to the front.
\end{enumerate}
$$
\begin{array}{cc}
   \begin{tikzpicture}[rotate=105]
    \foreach \x in {0,1,...,11} \draw (30*\x:3)--(30*\x+30:3);
    \foreach \x in {1,...,11,12} \coordinate (a\x) at (30-30*\x:3);
    \foreach \x in {9,12,11} \draw (a\x) node[left]{$B_{\x}$};
    \foreach \x in {4, 3} \draw (a\x) node[right]{$B_{\x}$};
    \draw (a7) node[below]{$B_{7}$};
    \draw (a1) node[above]{$A$};
    \draw (a2) node[above]{$B$};
    \draw (a5) node[right]{$\bar{B}$};
    \draw (a6) node[right]{$\bar{A}$};
    \draw (a8) node[below]{$S_{k}$};
    \draw (a10) node[left]{$S_{0}$};
    \foreach \x in {10, 4, 5, 6, 7, 8, 9} \draw (a2)--(a\x);
    \foreach \x in {11, 10} \draw (a1)--(a\x);
    \foreach \x in {8} \draw[dashed] (a6)--(a\x);
    \foreach \x in {3, 1, 2, 12, 8, 9, 10, 11} \draw[dashed] (a5)--(a\x);
    \draw[red] (a2)--(a5);
\end{tikzpicture}  &  \begin{tikzpicture}[rotate=105]
    \foreach \x in {0,1,...,11} \draw (30*\x:3)--(30*\x+30:3);
    \foreach \x in {1,...,11,12} \coordinate (a\x) at (30-30*\x:3);
    \foreach \x in {12,11} \draw (a\x) node[left]{$B_{\x}$};
    \foreach \x in {4, 3, 6} \draw (a\x) node[right]{$B_{\x}$};
    \draw (a8) node[below]{$B_{8}$};
    \draw (a1) node[above]{$A$};
    \draw (a2) node[above]{$B$};
    \draw (a9) node[left]{$\bar{B}$};
    \draw (a10) node[left]{$\bar{A}$};
    \draw (a7) node[below]{$S_{k}$};
    \draw (a5) node[right]{$S_{0}$};
    \foreach \x in {11, 5, 6, 7, 8, 9, 10} \draw (a1)--(a\x);
    \foreach \x in {5, 4} \draw (a2)--(a\x);
    \foreach \x in {8, 7} \draw[dashed] (a9)--(a\x);
    \foreach \x in {7, 1, 2, 3, 4, 5, 6, 12} \draw[dashed] (a10)--(a\x);
    \draw[red] (a1)--(a10);
\end{tikzpicture}\\
   \text{Case (3c)}  & \text{Case (3d)} 
\end{array}
$$
\begin{enumerate}
    
\item[(3e)] $2<t<s<s'$, i.e., $AB\leqslant \bar{B}\bar{A}\leqslant S_{0}< S_{k}$. Denote the boundary paths $\mathbf{S}_{1}=\{B_{3},\ldots,B_{t-2}\},$
$ \mathbf{S}_{2}=\{B_{t+1},\ldots, B_{s}\},$
$ \mathbf{S}_{3}=\{B_{s},\ldots, B_{s'}\}$, and $\mathbf{S}_{4}=\{B_{s'},\ldots, B_{n}\}$. We add an extra edge connecting $A\bar{A}$ and insert flat tetrahedra 
$\bar{B}B\mathbf{S}_{1},$
$ \bar{B}A\mathbf{S}_{4},$
$ S_{k}A\bar{B}\bar{A},$
$ \bar{A}A\mathbf{S}_{3},$
$ \bar{B}ABS_{0},$
$ AB\bar{B}\bar{A}$ and $\bar{A}B\mathbf{S}_{2}$, layered from the back to the front.

\item[(3f)] $2\leqslant s'<s<t$, i.e., $AB\leqslant S_{k}< S_{0}\leqslant \bar{B}\bar{A}$. Denote the boundary paths $\mathbf{S}_{1}=\{B_{3},\ldots, B_{s'}\},$   
$\mathbf{S}_{2}=\{B_{s'},\ldots, B_{s}\},$
$ \mathbf{S}_{3}=\{B_{s},\ldots, B_{t-2}\}$, and $\mathbf{S}_{4}=\{B_{t+1},\ldots, B_{n}\}$. we add an extra edge connecting $A\bar{A}$ and insert flat tetrahedra $\bar{A}A\mathbf{S}_{4},$
$ \bar{A}B\mathbf{S}_{1},$
$ \bar{B}ABS_{0},$
$ \bar{B}B\mathbf{S}_{2},$
$ S_{k}B\bar{B}\bar{A},$
$ AB\bar{B}\bar{A}$ and $\bar{B}A\mathbf{S}_{3}$, layered from the back to the front.
\end{enumerate}
$$
\begin{array}{cc}
\begin{tikzpicture}[rotate=105]
    \foreach \x in {0,1,...,11} \draw (30*\x:3)--(30*\x+30:3);
    \foreach \x in {1,...,11,12} \coordinate (a\x) at (30-30*\x:3);
    \foreach \x in {9,12,11} \draw (a\x) node[left]{$B_{\x}$};
    \foreach \x in {4, 3} \draw (a\x) node[right]{$B_{\x}$};
    \draw (a7) node[below]{$B_{7}$};
    \draw (a1) node[above]{$A$};
    \draw (a2) node[above]{$B$};
    \draw (a5) node[right]{$\bar{B}$};
    \draw (a6) node[right]{$\bar{A}$};
    \draw (a8) node[below]{$S_{0}$};
    \draw (a10) node[left]{$S_{k}$};
    \foreach \x in {8, 4, 5, 6, 7} \draw (a2)--(a\x);
    \foreach \x in {11, 8, 9, 10} \draw (a1)--(a\x);
    \foreach \x in {10, 8, 9} \draw[dashed] (a6)--(a\x);
    \foreach \x in {3, 1, 2, 12, 10, 11} \draw[dashed] (a5)--(a\x);
    \draw[red] (a2)--(a5);
    \draw[orange, densely dashed, line width=1pt] (a1)--(a6);
\end{tikzpicture}
&
\begin{tikzpicture}[rotate=105]
    \foreach \x in {0,1,...,11} \draw (30*\x:3)--(30*\x+30:3);
    \foreach \x in {1,...,11,12} \coordinate (a\x) at (30-30*\x:3);
    \foreach \x in {12,11} \draw (a\x) node[left]{$B_{\x}$};
    \foreach \x in {4, 3, 6} \draw (a\x) node[right]{$B_{\x}$};
    \draw (a8) node[below]{$B_{8}$};
    \draw (a1) node[above]{$A$};
    \draw (a2) node[above]{$B$};
    \draw (a9) node[left]{$\bar{B}$};
    \draw (a10) node[left]{$\bar{A}$};
    \draw (a7) node[below]{$S_{0}$};
    \draw (a5) node[right]{$S_{k}$};
    \foreach \x in {11, 7, 8, 9, 10} \draw (a1)--(a\x);
    \foreach \x in {7, 4, 5, 6} \draw (a1)--(a\x);
    \foreach \x in {8, 5, 6, 7} \draw[dashed] (a9)--(a\x);
    \foreach \x in {5, 1, 2, 3, 4, 12} \draw[dashed] (a10)--(a\x);
    \draw[red] (a1)--(a10);
    \draw[orange, densely dashed, line width=1pt] (a2)--(a9);
\end{tikzpicture}\\
\text{Case (3e)}&\text{Case (3f)}
\end{array}
$$

\begin{remark}\label{rem:edgecentral}
Observe that for each edge $B_{i}B_{i+1}$ of the polygon that is not adjacent to $AB$ and $\bar{A}\bar{B}$ (resp. $\bar{A}$ in Case 1) there is exactly one flat tetrahedra $C\bar{C}B_{i}B_{i+1}$ adjacent to it, where $C\in \{A, B\}$ and $\bar{C}\in \{\bar{A}, \bar{B}\}$ are chosen such that $CB_{i}B_{i+1}$ form a triangle in $F_{1}$ and $\bar{C}B_{i}B_{i+1}$ form a triangle in $F_{2}$. We call these flat tetrahedra \emph{sided} and the rest flat tetrahedra \emph{central}. Notice there are at most three central tetrahedra in all the subcases above.
\end{remark}

As in Case 2 and Case 3, suppose  $F_{1}$ and $F_{2}$ are glued together in the Kojima polyhedral decomposition. For the purpose of proving Proposition \ref{prop:movepolyhedracone}, we will need to construct another triangulation $\mathcal{T}_{\mathbf{S}}''$ of $M$ if the admissible path $\mathbf{S}$ is special in the following sense.

\begin{definition}\label{special}
    An admissible path $\mathbf S$ is \emph{special} if it satisfies one of the following conditions: 
\begin{enumerate}
    \item [(A)]  The cyclic order of the six distinguished vertices are as in the subcases (2c), (2d), (3c) or (3d). In addition, assume: (i) the edges $AB$ and $\bar{A}\bar{B}$ are disjoint, and (ii) $S_{0}=S_{k}$.
    
    \item [(B)] The cyclic order of the six distinguished vertices are as in the subcases (2e), (2f), (3e) or (3f). In addition, assume:  (i) $S_{0}$ and $S_{k}$ are adjacent, (ii) $AB$ and $\bar{A}\bar{B}$ are adjacent, and (iii) $S_{k}\notin AB$ and  $S_{0}\notin \bar{A}\bar{B}$.
\end{enumerate}
\end{definition}

When $\mathbf S$ is special, we define the ideal triangulation $\mathcal{T}_{\mathbf{S}}''$ of $M$ as follows. Under Condition (A),
the two central flat tetrahedra in $\mathcal{T}_{\mathbf{S}}'$ share a common triangle. We define $\mathcal{T}_{\mathbf{S}}''$ to be the triangulation obtained from $\mathcal{T}_{\mathbf{S}}'$ by performing the $2$-$3$ move with respect to the two central flat tetrahedra. By doing so, an extra edge is added, which is adjacent to exactly three flat tetrahedra, all central in the sense of Remark \ref{rem:edgecentral}. See Figure \ref{fig:alternativetrig1}. Under Condition (B), the orange edge (i.e., the edge belonging to neither $F_{1}$ nor $F_{2}$) is adjacent to exactly three flat tetrahedra, two of  which are central. We define $\mathcal{T}_{\mathbf{S}}''$ to be the triangulation obtained from $\mathcal{T}_{\mathbf{S}}'$ by performing the $3$-$2$ move with respect to the three tetrahedra. Notice that the two new tetrahedra are both sided in the sense of Remark \ref{rem:edgecentral}, and are both adjacent to $S_{0}S_{k}$. See Figure \ref{fig:alternativetrig2}.

\begin{figure}
    \centering
$
\begin{array}{ccc}
  \vcenter{\hbox{\begin{tikzpicture}[rotate=105]
    \foreach \x in {0,1,...,11} \draw (30*\x:3)--(30*\x+30:3);
    \foreach \x in {1,...,11,12} \coordinate (a\x) at (30-30*\x:3);
    \foreach \x in {12,10,11} \draw (a\x) node[left]{$B_{\x}$};
    \foreach \x in {4, 3} \draw (a\x) node[right]{$B_{\x}$};
    \draw (a7) node[below]{$B_{7}$};
    \draw (a8) node[below]{$B_{8}$};
    \draw (a1) node[above]{$A$};
    \draw (a2) node[above]{$B$};
    \draw (a5) node[right]{$\bar{A}$};
    \draw (a6) node[right]{$\bar{B}$};
    \draw (a9) node[left]{$S_{0}=S_{k}$};
    \foreach \x in {9, 4, 5, 6, 7, 8} \draw (a2)--(a\x);
    \foreach \x in {11, 9, 10} \draw (a1)--(a\x);
    \foreach \x in {9, 8} \draw[dashed] (a6)--(a\x);
    \foreach \x in {3, 1, 2, 12, 9, 10, 11} \draw[dashed] (a5)--(a\x);
    \draw[red] (a2)--(a5);
\end{tikzpicture}}}   &\rightarrow &\vcenter{\hbox{\begin{tikzpicture}[rotate=105]
    \foreach \x in {0,1,...,11} \draw (30*\x:3)--(30*\x+30:3);
    \foreach \x in {0,1,...,11} \coordinate (a\x) at (-30*\x:3);
    \foreach \x in {9,10,11} \draw (a\x) node[left]{$B_{\x}$};
    \foreach \x in {2, 3} \draw (a\x) node[right]{$B_{\x}$};
    \draw (a6) node[below]{$B_{6}$};
    \draw (a7) node[below]{$B_{7}$};
    \draw (a0) node[above]{$A$};
    \draw (a1) node[above]{$B$};
    \draw (a4) node[right]{$\bar{A}$};
    \draw (a5) node[right]{$\bar{B}$};
    \draw (a8) node[left]{$S_{0}=S_{k}$};
    \foreach \x in {3, 4, 5, 6, 7, 8} \draw (a1)--(a\x);
    \foreach \x in {8, 9, 10} \draw (a0)--(a\x);
    \foreach \x in {7, 8} \draw[dashed] (a5)--(a\x);
    \foreach \x in {0, 1, 2, 8, 9, 10, 11} \draw[dashed] (a4)--(a\x);
    \draw[red] (a1)--(a4);
    \draw[orange, densely dashed, line width=1pt] (a0)--(a5);
\end{tikzpicture}}} 
\end{array}
$
    \caption{The $2$-$3$ move from $\mathcal{T}_{\mathbf{S}}'$  to  $\mathcal{T}_{\mathbf{S}}''$ in subcase (2c): The two central flat tetrahedra in $\mathcal{T}_{\mathbf{S}}'$ are $AB\bar{A}S_{0}$ and $\bar{A}\bar{B}BS_{0}$ sharing  a common face $BS_{0}\bar{A}$. Performing a $2$-$3$ move to the two tetrahedra, one obtain three new tetrahedra $AB\bar{A}\bar{B}, A\bar{B}S_{0}A$ and $A\bar{B}\bar{A}S_{0}$, which are all central in the sense of Remark \ref{rem:edgecentral}.} 
    \label{fig:alternativetrig1}
\end{figure}

\begin{figure}
    \centering
$
\begin{array}{ccc}
    \vcenter{\hbox{\begin{tikzpicture}[rotate=105]
    \foreach \x in {0,1,...,11} \draw (30*\x:3)--(30*\x+30:3);
    \foreach \x in {1,...,11,12} \coordinate (a\x) at (30-30*\x:3);
    \foreach \x in {12,10,11} \draw (a\x) node[left]{$B_{\x}$};
    \foreach \x in {6, 4, 5} \draw (a\x) node[right]{$B_{\x}$};
    \draw (a7) node[below]{$B_{7}$};
    \draw (a1) node[above]{$A$};
    \draw (a2) node[above]{$B=\bar{A}$};
    \draw (a3) node[right]{$\bar{B}$};
    \draw (a8) node[below]{$S_{0}$};
    \draw (a9) node[left]{$S_{k}$};
    \foreach \x in {8, 4, 5, 6, 7} \draw (a2)--(a\x);
    \foreach \x in {11, 8, 9, 10} \draw (a1)--(a\x);
    \foreach \x in {9, 5, 6, 7, 8} \draw[dashed] (a3)--(a\x);
    \foreach \x in {8, 9, 10, 11} \draw[dashed] (a2)--(a\x);
    \draw[orange, densely dashed, line width=1pt] (a1)--(a3);
\end{tikzpicture}}} &\rightarrow & \vcenter{\hbox{\begin{tikzpicture}[rotate=105]
    \foreach \x in {0,1,...,11} \draw (30*\x:3)--(30*\x+30:3);
    \foreach \x in {1,...,11,12} \coordinate (a\x) at (30-30*\x:3);
    \foreach \x in {12,10,11} \draw (a\x) node[left]{$B_{\x}$};
    \foreach \x in {6, 4, 5} \draw (a\x) node[right]{$B_{\x}$};
    \draw (a7) node[below]{$B_{7}$};
    \draw (a1) node[above]{$A$};
    \draw (a2) node[above]{$B=\bar{A}$};
    \draw (a3) node[right]{$\bar{B}$};
    \draw (a8) node[below]{$S_{0}$};
    \draw (a9) node[left]{$S_{k}$};
    \foreach \x in {8, 4, 5, 6, 7} \draw (a2)--(a\x);
    \foreach \x in {11, 8, 9, 10} \draw (a1)--(a\x);
    \foreach \x in {9, 5, 6, 7, 8} \draw[dashed] (a3)--(a\x);
    \foreach \x in {12, 9, 10, 11} \draw[dashed] (a2)--(a\x);
\end{tikzpicture}}} 
\end{array}
$
    \caption{The $3$-$2$ move from  $\mathcal{T}_{\mathbf{S}}'$   to $\mathcal{T}_{\mathbf{S}}''$ in Case (2e): The three flat tetrahedra in $\mathcal{T}_{\mathbf{S}}'$  $A\bar{B}$ adjacent to $A\bar{B}$ are $A\bar{B}S_{0}B, A\bar{B}S_{k}B$ and $A\bar{B}S_{0}S_{k}$. Performing a $3$-$2$ move to the three tetrahedra, one obtains two new tetrahedra $ABS_{0}S_{k}$ and $B\bar{B}S_{0}S_{k}$, which are both sided in the sense of Remark \ref{rem:edgecentral}, and both are adjacent to $S_{0}S_{k}$.}
    \label{fig:alternativetrig2}
\end{figure}

\subsubsection{Step 2: $\mathcal{T}_{\mathbf{S}}'$ and $\mathcal{T}_{\mathbf{S}}''$ support angle structures}\label{subsec:step2}
\begin{proposition}\label{prop:anglestructure}For each admissible path $\mathbf{S} \subset \mathbf{B}$, the triangulations $\mathcal{T}_{\mathbf{S}}'$ and $\mathcal{T}_{\mathbf{S}}''$ support angle structures.
\end{proposition}
\begin{proof}
For the straightening of $\mathcal{T}_{\mathbf{S}}'$ in the subcases (1), (2a)-(2d), (3a)-(3d) and $\mathcal{T}_{\mathbf{S}}''$ in the subcases (2e), (2f), (3e) and (3f), all the edges are adjacent to at least one hyperideal tetrahedra, and hence are deformable. By Proposition \ref{DF}, these triangulations  support angle structures.

Next we show that the triangulation $\mathcal{T}_{\mathbf{S}}'$ in the subcase  (2e) supports angle structures; and the proof for the rest of the subcases goes identically, under the same crucial condition that the orange edge is the only edge that is adjacent to no hyperideal tetrahedra, and it is adjacent to at least one $0$-dihedral angle.

Let $T^{0}_{A\bar{B}}=\{A\bar{B}\mathbf{S}_{3}, A\bar{A}B\bar{B}\}$  be the set of flat tetrahedra in the straightening of  $\mathcal T_{\mathbf S}'$ adjacent to $A\bar{B}$ at an edge of dihedral angle $0,$ which is non-empty.
For an edge $e$ of $\mathcal T_{\mathbf S}'$, let $m(e)$, $n(e)$ and $k(e)$ respectively be the numbers of the $0$-dihedral angles around $e$, the $\pi$-dihedral angles around $e$, and the dihedral angles in $(0, \pi)$ around $e$, that do not belong to the tetrahedra in $T^{0}_{A\bar{B}}$; and let $m'(e)$ and $ n'(e)$ respectively be the numbers of the $0$-dihedral angles around $e$ and the $\pi$-dihedral angles around $e$ that belong to some tetrahedron in $T^{0}_{A\bar{B}}.$ 

In below the three steps,  we construct an angle structure on $(M,\mathcal T_{\mathbf S}')$ by deforming these dihedral angles into $(0, \pi)$ while keeping the cone angles unchanged. Let $\epsilon>0$ be sufficiently small.

\begin{enumerate}[Step (I)]

 \item  
    For each $\pi$-dihedral angle  belonging to the tetrahedra in $T^{0}_{A\bar{B}}$, we deform it to $\pi-\frac{13\epsilon}{|T^{0}_{A\bar{B}}|};$ and for each $0$-dihedral angle belonging to the tetrahedra in $T^{0}_{A\bar{B}}$, we deform it to $\frac{6\epsilon}{|T^{0}_{A\bar{B}}|}$.

    \item For the rest of the $\pi$-dihedral angles in $\mathcal T_{\mathbf S}'$, we deform them to $\pi-3\epsilon,$ and for the rest of the $0$-dihedral angles in $\mathcal T_{\mathbf S}'$, we deform them to $\epsilon$. 

    \item For each dihedral angle around $e$ with value $\alpha\in (0, \pi)$ not coming from the previous steps,  we deform it to 
    $$\alpha+\frac{\epsilon}{k(e)}\bigg(3 n(e)-m(e)+\frac{13n'(e)}{|\mathcal{T}_{A\bar{B}}^{0}|}-\frac{6m'(e)}{|\mathcal{T}_{A\bar{B}}^{0}|}\bigg).$$
\end{enumerate}
Then as $\epsilon$ is sufficiently small, after Step (I) all the tetrahedra in $T^0_{A\bar B}$ become hyperideal, after Step (II) all the other flat tetrahedra become hyperideal, and after Step (III) all the originally hyperideal tetrahedra remain hyperideal. Steps (I) and (II) are also designed to make sure that after the deformation the cone angle at the edge $A\bar B$ remains $2\pi,$ as there are exactly two $\pi$-dihedral angles adjacent to $A\bar B$ respectively coming from the tetrahedra $A\bar{B}BS_{0}$ and $A\bar{B}S_{k}\bar{A}$; and Step (III) is designed  to make sure that after the deformation the cone angles at all the other edges remain $2\pi.$ Therefore, $\mathcal{T}_{\mathbf{S}}'$ supports an angle structure.
\end{proof}

\subsubsection{Step 3: Connecting $\mathcal{T}$ and $\mathcal{T}'$ by intermediate triangulations }\label{subsec:step3}

We define a partial order $\prec$ on the set of all admissible paths.  For two admissible paths $\mathbf{S}^*$ and $\mathbf{S}$, we say that $\mathbf{S}^{*}\prec \mathbf{S}$ if
$\mathbf{B}_{L}^{\mathbf{S}^{*}}\subset
\mathbf{B}_{L}^{\mathbf{S}}$. Note that $\mathbf{S}_{max}'$ is the unique maximal element  and $\mathbf{S}_{min}'$ is the  unique minimal element in this order, and $\mathcal{T}_{\mathbf{S}_{max}}'=\mathcal{T}$ and $\mathcal{T}_{\mathbf{S}_{min}}'=\mathcal{T}'$. 

For an admissible path $\mathbf S,$ we let $\mathbf T_{\mathbf S}=\{\mathcal T_{\mathbf S}',\mathcal T_{\mathbf S}''\}$ if $\mathbf S$ is special, and let $\mathbf T_{\mathbf S}=\{\mathcal T_{\mathbf S}'\}$ otherwise. 

\begin{proposition}\label{prop:order}
For an admissible path $\mathbf{S}$ such that $\mathbf{S}\ne \mathbf{S}_{min}$, there exists an $\mathbf{S}^{*}$ such that: 
\begin{enumerate}[(1)]
\item $\mathbf{S}^{*}\prec \mathbf{S}$.
\item One can choose a triangulation $\mathcal T_{\mathbf S}$ from $\mathbf T_{\mathbf S}$ and a triangulation $\mathcal T_{\mathbf S^*}$ from $\mathbf T_{\mathbf S^*}$ so that they differ by a single $2$-$3$, $3$-$2$ or $4$-$4$ move.
\end{enumerate}
\end{proposition}

To prove Proposition \ref{prop:order}, we need the following lemmas.

\begin{lemma}\label{lem:fixpoint}
Let $[n]\doteq\{0, 1, \ldots, n\}$ and $f:[n]\to [n]$ be a map such that there is no $i\in [n]$ satisfying $f(i)>i $ and $f(i+1)< i+1$. Then $f$ has a fixed point. Moreover, if in addition there is some $k\in[n]$ such that $f(k)<k$, then there is a fixed point $j$ of $f$ with $j<k$.
\end{lemma}
\begin{proof}
We first observe that the second statement implies the first statement as follows: If $f(n)=n$, then $n$ is a fixed point. If otherwise, then  the second statement guarantees a fixed point in $[0, n-1]$. 

Therefore, it suffices to prove the second statement; and we use contradiction. Assume that $f$ has no fixed point in between $0$ and $k-1$. Then $f(k)< k$ implies that $f(k-1) < k$ due to the assumption of the lemma with $i=k-1$. Since $f(k-1)\neq k-1$ by assumption, we have $f(k-1)<k-1$. Applying this reasoning inductively, one obtains $f(j)<j$ for all $0\leqslant j\leqslant k$. In particular, $f(0)<0$, which is a contradiction.
\end{proof}

\begin{lemma}\label{lem:intersection}
Let $\mathbf{S}=\{S_{m}, S_{m+1},\ldots, S_{n}\}$ be a path in $\mathbf{B}$ satisfying the following conditions:
\begin{enumerate}[(1)]
    \item  $A\mathbf{S}$ divides $P_{u}$ into two or more disconnected components, and 
    \item   $AB\mathbf{S}$ consists of hyperideal tetrahedra whose interior are disjoint from each other.
\end{enumerate}
Suppose $B'$ is a vertex of $P_u$ lying in $\mathbf{B}$ such that $B$ and $B'$ lie in the different components of $P_u\setminus A\mathbf{S}$. Then the geodesic $BB'$ intersects $A\mathbf{S}$ exactly once in $P_u$.
\end{lemma}

\begin{proof}
Since $B$ and $B'$ lie in different connected components of $P_u\setminus A\mathbf{S}$, the geodesic segment $BB'$ intersects $A\mathbf{S}$ due to the convexity of $P_u$. On the other hand, by (2), $AB\mathbf{S}$ can be considered as a cone with base $A\mathbf{S}$ and apex $B$, and a geodesic segment passing through the apex of a cone can only intersect the base at most once.
\end{proof}

\begin{lemma}\label{lem:hypothesis}
Let $\mathbf{S}=\{S_{0}, \ldots, S_{k}\}\subset \mathbf{B}$ be an admissible path, and fix an $i\in\{0,\dots, k\}.$ Let $F$ be a totally geodesic plane that passes through the vertices $A$, $B$ and $S_{i+1}.$ We call the component of $P_u \setminus F$ containing $S_{0}$ the back component, and call the other component of $P_u \setminus F$ the front component. (See Figure \ref{fig:lemma518}.) Suppose $E_{i}$ and $E_{i+1}$ are two vertices of $\mathbf{B}_{L}^{\mathbf{S}}$ such that $E_{i}S_{i}S_{i+1}$ and $E_{i+1}S_{i+1}S_{i+2}$ are two triangles of $\mathbf{B}$.   Then the following two scenarios cannot occur simultaneously:
\begin{enumerate}[(1)]
\item $BE_{i}$ intersects $A\mathbf{S}$ in the front component.
\item $BE_{i+1}$ intersects $A\mathbf{S}$ in the back component.
\end{enumerate}
\end{lemma}

\begin{proof}
By the admissibility of $\mathbf{S}$, the totally geodesic plane $F$ also divides $A\mathbf{S}$ into exactly two components, because otherwise there will be a geodesic ray starting from $B$ and lying in $F$ that passes through $A\mathbf{S}$ at least twice, which contradicts the fact that a geodesic ray starting from the apex $B$ of the cone $AB\mathbf{S}$ intersects the base $A\mathbf{S}$ at most once.
We also have that the  vertices $\{S_{0},\ldots, S_{i}\}$ lie in the back component and $\{S_{i+2},\ldots, S_{k}\}$ lie in the front component. 
For the sake of contradiction, assume otherwise that $BE_{i}$ intersects $A\mathbf{S}$ in the front component, and $BE_{i+1}$ intersects $A\mathbf{S}$ in the back component. Then by the convexity of the front and the back components, $E_i$ lies in the front component and $E_{i+1}$ lies in the back component. As a consequence,  $E_i$ and $S_{i+2}$ are in one component, and $E_{i+1}$ and $S_{i}$ are in the other component of $P_u\setminus F.$ Now consider the tetrahedra $AE_iS_iS_{i+1}$ and $AE_{i+1}S_{i+1}S_{i+2}$. Since $E_i$ and $S_i$ are on different sides of $F$, the interior of the geodesic segment $E_iS_i$ intersects $F$, and hence the interior of the tetrahedron 
$AE_iS_iS_{i+1}$ intersects $F$ at a triangle. Similarly, the interior of the tetrahedron $AE_{i+1}S_{i+1}S_{i+2}$ also intersects $F$ at a triangle the two triangles share a common side $AS_{i+1}$, and lie on the same side of $AS_{i+1}$ that is opposite to $B$. Then they intersect, contradicting that the tetrahedra $AE_iS_iS_{i+1}$ and $AE_{i+1}S_{i+1}S_{i+2}$ only intersect at the boundary. 
\end{proof}

\begin{figure}[ht] 
\centering
$\vcenter{\hbox{
\begin{tikzpicture}
\coordinate (A) at (-1, 3);
\coordinate (B) at (1.5, 3);
\draw (A) node[above]{$A$};
\draw (B) node[above]{$B$};
\draw[line width=0.8 pt] (A)--(B);
\draw[red, line width=0.8 pt] (0, -2)node[below]{$\mathbf{S}$}node[right]{\tiny$S_k$}--(0.5, -1.5)node[right]{\tiny$S_{i+2}$}--(-1, -0.5) node[left]{\tiny$S_{i+1}$}--(0.7, 0.5)node[right]{\tiny$S_{i}$}--(-1.2, 1.5)--(0, 2)node[right]{\tiny$S_0$};
\coordinate (ei) at (-1.5, 0.3);
\draw[blue, line width=0.8 pt] (0.7, 0.5)--(ei)node[left]{\tiny{$E_{i}$}}--(-1, -0.5);
\coordinate (ei1) at (-1.5, -1);
\draw[blue, line width= 0.8 pt] (0.5, -1.5)--(ei1)node[left]{\tiny{$E_{i+1}$}}--(-1, -0.5);
\path[fill=green, opacity=0.3](A)--(-1, -0.5)--(1.5,-0.5)--(B);
\draw (0.5, 2.6) node[green!50!black]{$F$};
\draw (-3, 0) node {$\mathbf{B}_{L}^{\mathbf{S}}$};
\draw (2, 0) node {$\mathbf{B}_{R}^{\mathbf{S}}$};
\end{tikzpicture}
}}$
\caption{Illustration of Lemma~\ref{lem:hypothesis}. The edges $S_{i}S_{i+1}$ and $ S_{i+1}S_{i+2}$ are respectively adjacent to the triangles $E_{i}S_{i}S_{i+1}$ and $E_{i+1}S_{i+1}S_{i+2}$ in $\mathbf{B}_{L}^{\mathbf{S}}$. The geodesic plane $F$ can be the one passing through the points $A, B$ and $S_{i+1}$.}
\label{fig:lemma518}
\end{figure}

\begin{lemma}\label{prop:induction}
Let $\mathbf{S}=\{S_{0},\ldots S_{k}\}$ be an admissible path such that $\mathbf{S}\ne \mathbf{S}_{min}$. Then there exists a admissible path $\mathbf{S}^{*}\prec \mathbf{S}$ satisfying one of the following conditions:
\begin{enumerate}[(1)]
    \item $\mathbf{S}^{*}=\{S_{0},\ldots, S_{i-1}, S_{i+1},\ldots, S_{k}\}$ for some $0<i<k$. In addition, $S_{i-1}S_{i}S_{i+1}$ form a triangle in $\mathbf{B}_{L}^{\mathbf{S}}$. (See Figure \ref{fig:lemma519} (a).)
    \item $\mathbf{S}^{*}=\{S_{0},\ldots, S_{i}, E_i, S_{i+1},\ldots, S_{k}\}$ for some $i>0$ and some vertex $E_i$. In addition, $S_{i}S_{i+1}E_i$ form a triangle in $\mathbf{B}_{L}^{\mathbf{S}}$ and $BE_{i}$ passes the interior of the triangle $AS_{i}S_{i+1}$. (See Figure \ref{fig:lemma519} (b).)
    \item $\mathbf{S}^{*}=\{S_{0},\ldots, S_{i-1}, E_{i},S_{i+1},\ldots, S_{k}\}$  for some $i>0$ and some vertex $E_i$. In addition, $S_{i-1}S_{i}E_{i}$ and $S_{i+1}S_{i}E_{i}$ are triangles in $\mathbf{B}_{L}^{\mathbf{S}}$, and the four vertices $\{S_{i}, E_i, A, B\}$ lie in the same totally geodesic plane. (See Figure \ref{fig:lemma519} (c).)
    \item $\mathbf{S}^{*}=\{E_0, S_{1},\dots, S_{k}\}$ for some vertex $E_0$. In addition, $E_0S_{0}S_{1}$ form a triangle in $\mathbf{B}_{L}^{\mathbf{S}}$  that is adjacent to $\partial \mathbf{B}\cap\partial F_1$. (See Figure \ref{fig:lemma519} (d).)
    \item $\mathbf{S}^{*}=\{S_{0},\dots,S_{k-1}, E_{k-1}\}$ for some vertex $E_{k-1}$. In addition, $E_{k-1}S_{k}S_{k-1}$ form a triangle in $\mathbf{B}_{L}^{\mathbf{S}}$ that is adjacent to $\partial \mathbf{B}\cap\partial F_2$. (See Figure \ref{fig:lemma519} (e).)
\end{enumerate}
\end{lemma}
\begin{figure}[ht]
    \centering
\begin{subfigure}[b]{0.30\textwidth}
    \centering
    \scalebox{0.8}{$\vcenter{\hbox{
\begin{tikzpicture}
\draw[red, line width=2pt] (0, -2)node[right]{$\mathbf{S}$}--(0.5, -1.5)--(-1, -0.5) node[left]{\tiny$S_{i+1}$}--(0.7, 0.5)node[right]{\tiny$S_{i}$}--(-1.1, 1.5)node[left]{\tiny$S_{i-1}$}--(0, 2)node[above]{$\vdots$};
\draw[blue, line width=1pt] (0, -2)node[left]{$\mathbf{S}^{*}$}--(0.5, -1.5)--(-1, -0.5)--(-1.1, 1.5)--(0, 2);
\draw (-3, 0) node {$\mathbf{B}_{L}^{\mathbf{S}}$};
\draw (2, 0) node {$\mathbf{B}_{R}^{\mathbf{S}}$};
\end{tikzpicture}
}}$}
\caption{Condition (1)}
\label{fig:519case1}
\end{subfigure}
\hfill
\begin{subfigure}[b]{0.30\textwidth}
\centering
\scalebox{0.8}{$\vcenter{\hbox{
\begin{tikzpicture}
\draw[red, line width=2pt] (0, -2)node[right]{$\mathbf{S}$}--(0.5, -1.5)node[right]{\tiny$S_{i+2}$}--(-1, -0.5) node[left]{\tiny$S_{i+1}$}--(0.7, 0.5)node[right]{\tiny$S_{i}$}--(-1.1, 1.5)--(0, 2)node[above]{$\vdots$};
\coordinate (ei) at (-1.5, 0.3);
\draw(ei)node[left, blue]{\tiny{$E_{i}$}};
\coordinate (ei1) at (-1.5, -1);
\draw[blue, line width=1pt] (0, -2)node[left]{$\mathbf{S}^{*}$}--(0.5, -1.5)--(-1, -0.5)--(ei)--(0.7, 0.5)--(-1.1, 1.5)--(0, 2);
\draw (-3, 0) node {$\mathbf{B}_{L}^{\mathbf{S}}$};
\draw (2, 0) node {$\mathbf{B}_{R}^{\mathbf{S}}$};
\end{tikzpicture}
}}$}
        \caption{Condition (2)}
    \label{fig:519case2}
\end{subfigure}
\hfill
\begin{subfigure}[b]{0.30\textwidth}
\centering
    \scalebox{0.8}{$\vcenter{\hbox{
\begin{tikzpicture}
\draw[red, line width=2pt] (0, -2)node[right]{$\mathbf{S}$}--(0.5, -1.5)node[right]{\tiny$S_{i+2}$}--(-1, -0.5) node[left]{\tiny$S_{i+1}$}--(0.7, 0.5)node[right]{\tiny$S_{i}$}--(-1.1, 1.5)node[left]{\tiny$S_{i-1}$}--(0, 2)node[above]{$\vdots$};
\coordinate (ei) at (-1.5, 0.3);
\draw(ei)node[left, blue]{\tiny{$E_{i}$}};
\coordinate (ei1) at (-1.5, -1);
\draw[blue, line width=1pt] (0, -2)node[left]{$\mathbf{S}^{*}$}--(0.5, -1.5)--(-1, -0.5)--(ei)--(-1.1, 1.5)--(0, 2);
\draw (ei)--(0.7, 0.5);
\draw (-3, 0) node {$\mathbf{B}_{L}^{\mathbf{S}}$};
\draw (2, 0) node {$\mathbf{B}_{R}^{\mathbf{S}}$};
\end{tikzpicture}
}}$}
        \caption{Condition (3)}
    \label{fig:519case3}
\end{subfigure}
\begin{subfigure}[b]{0.48\textwidth}
\centering
\scalebox{0.8}{$\vcenter{\hbox{
\begin{tikzpicture}
\draw[red, line width=2pt] (0, -2)node[right]{$\mathbf{S}$}--(0.5, -1.5)--(-1, -0.5) node[left]{\tiny$S_{3}$}--(0.7, 0.5)node[right]{\tiny$S_{2}$}--(-1.1, 1.5)node[left]{\tiny$S_{1}$}--(0, 2)node[above]{\tiny$S_{0}$};
\draw[blue, line width=1pt] (0, -2)node[left]{$\mathbf{S}^{*}$}--(0.5, -1.5)--(-1, -0.5)--(0.7, 0.5)--(-1.1, 1.5)--(-1, 2)node[above]{\tiny$E_{0}$};
\draw[red] (0, -1.8)node[below]{$\vdots$};
\draw[green!60!black] (-3, 2)--(2, 2)node[right]{$\partial F_{1}$};
\draw (-3, 0) node {$\mathbf{B}_{L}^{\mathbf{S}}$};
\draw (2, 0) node {$\mathbf{B}_{R}^{\mathbf{S}}$};
\end{tikzpicture}
}}$
}
    \caption{Condition (4)}
    \label{fig:519case4} 
\end{subfigure}
\hfill
\begin{subfigure}[b]{0.48\textwidth}
\centering
\scalebox{0.8}{$\vcenter{\hbox{
\begin{tikzpicture}
\draw[red, line width=2pt] (0, -2)node[below]{\tiny$S_{k}$}--(0.5, -1.5)node[right]{\tiny$S_{k-1}$}--(-1, -0.5) node[left]{\tiny$S_{k-2}$}--(0.7, 0.5)node[right]{\tiny$S_{k-3}$}--(-1.1, 1.5)--(0, 2)node[right]{$\mathbf{S}$};
\draw[blue, line width=1pt] (-1, -2)node[below]{\tiny$E_{k-1}$}--(0.5, -1.5)--(-1, -0.5)--(0.7, 0.5)--(-1.1, 1.5)--(0, 2);
\draw[red] (0, 2)node[above]{$\vdots$};
\draw[green!60!black] (-3, -2)--(2, -2)node[right]{$\partial F_{2}$};
\draw (-3, 0) node {$\mathbf{B}_{L}^{\mathbf{S}}$};
\draw (2, 0) node {$\mathbf{B}_{R}^{\mathbf{S}}$};
\end{tikzpicture}
}}$}
    \caption{Condition (5)}
    \label{fig:519case5}
\end{subfigure}
    \caption{The relative positions of $\mathbf{S}$ and $\mathbf{S}^{*}$ described in Lemma \ref{prop:induction}: The red curve denotes the admissible path $\mathbf{S}$, the blue curve denotes the admissible path $\mathbf{S}^{*}$, and the green lines in (d) and (e) respectively denote the intersection $\partial \mathbf{B}\cap \partial F_{1}$ and $\partial \mathbf{B}\cap \partial F_{2}$.}
    \label{fig:lemma519} 
\end{figure}
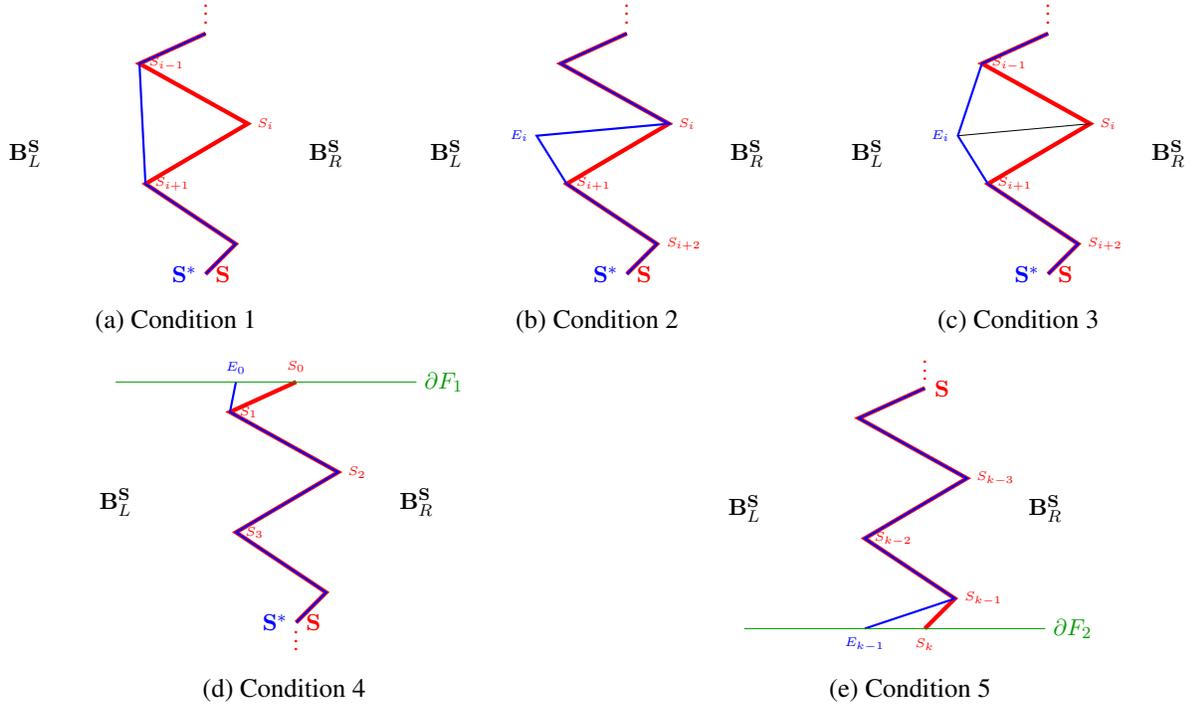
\begin{proof}

For a given $\mathbf S,$ we by a case-by-case discussion  provide an algorithm to find the path $\mathbf{S}^{*}$ satisfying  one of the conditions in the lemma, and prove that $\mathbf{S}^{*}$ is admissible. 

\begin{enumerate}
    \item  [(i)] If there exist three consecutive vertices $S_{i-1}, S_{i}, S_{i+1}$ in $\mathbf{S}$ that form a triangle in $\mathbf{B}_{L}^{\mathbf{S}}$, then the geodesic segment $AS_i$ must intersect $BS_{i-1}S_i,$ as $AS_i$ lies in the union of $AS_{i-1}S_iS_{i+1},$ $ABS_{i-1}S_i$ and $ABS_iS_{i+1}$. In this case, we set $\mathbf{S}^{*}\doteq\{S_{0},\ldots S_{i-1}, S_{i+1}, \ldots S_{k}\}$, which satisfies Condition (1). Since $A\mathbf{B}_{L}^{\mathbf{S}^{*}}\cup AB\mathbf{S}^{*} \cup B\mathbf{B}_{R}^{\mathbf{S}^{*}}$ 
is obtained from $A\mathbf{B}_{L}^{\mathbf{S}}\cup AB\mathbf{S} \cup B\mathbf{B}_{R}^{\mathbf{S}}$ by performing the $3$-$2$ move at above the union of the three tetrahedra around $AS_{i}$, the two new tetrahedra $ABS_{i-1}S_{i+1}$ and $BS_{i-1}S_iS_{i+1}$ created in this move are hyperideal. Therefore, the path $\mathbf{S}^*$ is admissible.
  \end{enumerate}

Now, suppose that there are no three consecutive vertices 
of $\mathbf{S}$ that form a triangle in $\mathbf{B}_{L}^{\mathbf{S}}$. 
Let $J=\{ 0\leqslant j< k\ |\ S_{j}S_{j+1} \textrm{ is not in } \mathbf{S}_{min} \}$.
Then for each $j\in J$, there exists a unique vertex $E_{j}\in \mathbf{B}_{L}^{\mathbf{S}}$ such that $E_{j}S_{j}S_{j+1}$ form  a triangle in $\partial P_{u}$.
We will define a function $f$ on $J$ so that Lemma \ref{lem:fixpoint} applies as follows: 
Write $J$ as the disjoint union $\sqcup_{1\leqslant s\leqslant r} J_s$ where each $J_s$ is a maximal subset of consecutive integers. Then each $J_s$ is of the form $\{m,\dots,n-1\}\cap J$ where $S_{m-1}S_{m} $ and $ S_{n}S_{n+1}$ belong to $\mathbf{S}_{min}$ while $S_{j}\notin \mathbf{S}_{min}$ for all $j$ with $m<j<n$. For $j\in J_s$, by applying Lemma \ref{lem:intersection} to the path $\mathbf S$ and $B'=E_j$, 
we see that  the geodesic segment $BE_{j}$ intersects $A\mathbf{S}$ exactly once.
If  $BE_j\cap A\mathbf{S}\in AS_{k}$, then we set $f(j)=k-1$. Otherwise, there exists a unique $k$ such that 
$BE_j\cap A\mathbf{S}\in AS_{k}S_{k+1}\setminus AS_{k+1}$; and in this case we set $f(j)=k$.
Now we claim that $f(j) \in J_s$, as a consequence of which we have  $f(J_s)\subset J_s$. Indeed, by applying Lemma \ref{lem:intersection} to $\{S_{m}, S_{m+1},\ldots, S_{n}\}$ and $B'=E_j$, we see that  $BE_j\cap A\mathbf{S}$ lies in $\cup_{m-1<j<n} AS_{j}S_{j+1}$.
Therefore $f(j)\in \{m,\dots,n\}$, and $f(j)=n$ if and only if $BE_j\cap A\mathbf{S}\in AS_{n}$. 
Since $AS_{n}$ lies on the boundary of $P_{u}$ and the interior of $BE_{j}$ lies in the interior of $P_{u}$, the edges $AS_{n}$ and $BE_{j}$ do  not intersect. Therefore $f(j)\in \{m,\dots,n-1\}=J_s$, and hence $f(J_s)\subset J_s$. Next, we show that on each $J_s$ the function $f$ satisfies the assumption in  Lemma \ref{lem:fixpoint}. 
Suppose $j$ and $j+1$ belong to $J_s$. 
Let $F$  be the geodesic plane in $\mathbb H^3$ passing  through the vertices $A, B$ and $S_{j+1}$. Then $F$ satisfies the condition in Lemma \ref{lem:hypothesis}. Then by applying Lemma \ref{lem:hypothesis} to $E_j$ and $E_{j+1}$, we have that the following two scenarios cannot occur simultaneously: (1)   $BE_{j}$ intersects $A\mathbf{S}$ at the front component,  and (2) $BE_{j+1}$ intersects $A\mathbf{S}$ at the back component. The first  scenario is equivalent to $f(j)>j$ and the second scenario is equivalent to $f(j+1)<j+1$. Therefore, $f$ satisfies the assumption in Lemma \ref{lem:fixpoint}.
Guaranteed to exist by Lemma \ref{lem:fixpoint}, we let $i$ be the smallest fixed point of $f.$

We are ready to construct the admissible path $\mathbf S^*$ in the following cases:
\begin{enumerate}
    \item  [(ii)]If $BE_{i}$ intersects $AS_{i}S_{i+1}$ in the interior, then we set $\mathbf{S}^{*}=\{S_{0},\ldots S_{i}, E_{i}, S_{i+1},\ldots S_{k}\}$, which satisfies Condition (2). 
   To show that $\mathbf{S}^{*}$ is admissible,  observe in this case that $A\mathbf{B}_{L}^{\mathbf{S}^{*}}\cup AB\mathbf{S}^{*} \cup B\mathbf{B}_{R}^{\mathbf{S}^{*}}$ is obtained from $A\mathbf{B}_{L}^{\mathbf{S}}\cup AB\mathbf{S} \cup B\mathbf{B}_{R}^{\mathbf{S}}$ by performing the  $2$-$3$ move to the union of $AE_iS_iS_{i+1}$ and $ABS_iS_{i+1}$ adding the edge $BE_{i}.$ Then the three new tetrahedra created in the move are hyperideal because $BE_{i}$ intersects in the interior of $AS_{i}S_{i+1}$. Therefore, the  path $\mathbf{S}^{*}$ is admissible.

  \item  [(iii)] If $0<i\leqslant k-1$ and the condition in (ii) above and the condition in (v) below do not occur, then $BE_{i}$ intersects $AS_{i}$.  We first prove that $BE_{i-1}$ also intersects $AS_{i}$. For the sake of contradiction, suppose that $BE_{i-1}$ intersects $A\mathbf{S}$ in $\cup_{0\leqslant j\leqslant i-1}AS_{j}S_{j+1} \setminus AS_{i}$, which means either $f(i-1)=i-1$ or $f(i-1)<i-1$. If $f(i-1)=i-1$, then it contradicts the hypothesis that $i$ is the first fixed point of $f$; and if $f(i-1)<i-1$, then by Lemma \ref{lem:fixpoint}, the function $f$ has  a fix point smaller than $i-1$, which contradicts the same hypothesis.
    For the sake of contradiction again,  suppose that $BE_{i-1}$ intersects $A\mathbf{S}$ in $\cup_{i\leqslant j\leqslant k}AS_{j}S_{j+1} \setminus AS_{i}$ at a point $X$. Then $X$ and $S_{i+1}$ lie in the same component of 
$A\mathbf{S}\setminus AS_i$. Let $F$ be the totally geodesic plane passing through the points $A, B$ and $S_{i}$. By the hypothesis, $F$ also passes the point $E_{i}$. As argued in the beginning of the proof of Lemma~\ref{lem:hypothesis}, $F$ divides $A\mathbf{S}$ into two components, with vertices 
    $\{S_{0},\dots, S_{i-1}\}$ lie in one component and $\{S_{i+1},\dots, S_{k}\}$ lie in the other component.
    As a consequence,  $S_{i-1}$ lies in one component of $P_u\setminus F,$ and $S_{i+1}$ and $X$ lie in the other component of  $P_u\setminus F.$ We call the component of $P_u\setminus F$ containing $S_{i+1}$ the front component, and call the other one the back component. Then the geodesic ray $BX\cap P_{u}$ lies in the front component, implying that $E_{i-1}\in BX$ also lies in the front component.
    Now, as $S_{i-1}$ lies on one side of $F$ (the back component) and $E_{i-1}$ lies on the other side of $F$ (the front component), the totally geodesic plane $F$ intersects the interior of the tetrahedron $AE_{i-1}S_{i-1}S_{i}.$ As both $E_i$ and the triangle $AE_{i-1}S_{i-1}$ lie in $A\mathbf{B}_{L}^{\mathbf{S}}$, the triangle $AE_iS_i\subset F$ and the intersection  $AE_{i-1}S_{i-1}S_{i}\cap F$ lie on the same side of $AS_i$ in $F,$ hence $AE_{i-1}S_{i-1}$ and the interior of $AE_{i-1}S_{i-1}S_{i}$ intersect. This further implies that the interior of the two tetrahedra $AE_{i-1}S_{i-1}S_{i}$ and $AE_{i}S_{i}S_{i+1}$ intersect, which is a contradiction. See Figure~\ref{fig:lemma518b}. Putting the two cases together, we have $BE_{i-1}$ intersects $A\mathbf{S}$ at $AS_{i}$ as claimed.
    
    \begin{figure}
\centering
$\vcenter{\hbox{
\begin{tikzpicture}
\coordinate (A) at (-1, 3);
\coordinate (B) at (1.5, 3);
\draw (A) node[above]{$A$};
\draw (B) node[above]{$B$};
\draw [line width=0.8 pt] (A)--(B);
\draw[red, line width=0.8 pt] (0, -2)node[below]{$\mathbf{S}$}node[right]{\tiny$S_k$}--(0.5, -1.5)node[right]{\tiny$S_{i+1}$}--(-1, -0.5) node[below]{\tiny$S_{i}$}--(0.7, 0.5)node[right]{\tiny$S_{i-1}$}--(-1.2, 1.5)--(0, 2)node[right]{\tiny$S_0$};
\coordinate (ei) at (-2, 0.5);
\draw[blue, line width=0.8 pt] (0.7, 0.5)--(ei)node[left]{\tiny{$E_{i-1}$}}--(-1, -0.5);
\draw[purple, line width=0.8 pt] (ei).. controls (-3, -2) and (1, -2) ..(B);
\draw[fill] (-0.15, -0.7) node[right]{\tiny$X$} circle(0.05);
\coordinate (ei1) at (-2, -0.5);
\draw[blue, line width=0.8 pt] (0.5, -1.5)--(ei1)node[left]{\tiny{$E_{i}$}}--(-1, -0.5);
\path[fill=green, opacity=0.3](A)--(ei1)--(-1, -0.5)--(1.5,-0.5)--(B);
\path[fill=purple, opacity=0.3](A)--(-1, -0.5)--(0.5, -1.5);
\draw (0.5, 2.6) node[green!50!black]{$F$};
\draw (-3, 0) node {$\mathbf{B}_{L}^{\mathbf{S}}$};
\draw (2, 0) node {$\mathbf{B}_{R}^{\mathbf{S}}$};
\end{tikzpicture}
}}$
\caption{Illustration of the proof of Lemma~\ref{prop:induction}, Case (iii).
The green face $F$ divides $P_{u}$ into two components, with $X=BE_{i-1}\cap A\mathbf{S}$ lying in the front component, and  $BE_{i}\cap AS_{i}$ lying in the back component. This contradicts Lemma~\ref{lem:hypothesis}, hence cannot occur.}
\label{fig:lemma518b}
\end{figure}

Next, since both  $BE_{i-1}$ and  $BE_{i}$ intersect $AS_{i}$, both $E_{i-1}$ and  $E_{i}$  lie on the totally geodesic plane passing through $A,$ $B$ and $S_{i}$; and within this plane, $E_{i-1}$ and  $E_{i}$ lie on the same side of the geodesic containing $AS_{i}$ as a segment.    
Therefore, this two faces  $AS_{i}E_{i-1}$ and $AS_{i}E_{i}$ coincide, with   $E_{i-1}=E_{i}$. We set $\mathbf{S}^{*}=\{S_{0},\ldots, S_{i-1},$ $ E_{i}, S_{i+1}, \ldots S_{k}\}$. 
   Since  $BE_i$ intersects $AS_{i}$ and $E_i=E_{i-1}$, we see that
    $\mathbf{S}^{*}$ satisfies Condition (3). Observe that $A\mathbf{B}_{L}^{\mathbf{S}^{*}}\cup AB\mathbf{S}^{*} \cup B\mathbf{B}_{R}^{\mathbf{S}^{*}}$ is obtained from $A\mathbf{B}_{L}^{\mathbf{S}}\cup AB\mathbf{S} \cup B\mathbf{B}_{R}^{\mathbf{S}}$ by performing the $4$-$4$ move  that removes the edge $AS_{i}$ and adds the edge $BE_{i}$. (See Figure~\ref{fig:lemma519} (c).) As $AS_{i}$ and $BE_{i}$ intersect in the interior of $P_u$, the $4$-$4$ move only involves hyperideal tetrahedra. Therefore, the path $\mathbf{S}^{*}$ is admissible.

\item [(iv)] If $i=0$ and the condition in (ii) does not occur,  then 
$BE_{0}$ intersects $AS_{0}$.  Hence $E_{0}$ is the vertex of $ F_{1}$ that is adjacent to $S_{0}$ and lies in $\mathbf{B}_{L}^{\mathbf{S}}$.  We set $\mathbf{S}^{*}\doteq\{E_{0}, S_{1}, \ldots, S_{k}\}$, which satisfies Condition (4).  Since $A\mathbf{B}_{L}^{\mathbf{S}^{*}}\cup AB\mathbf{S}^{*} \cup B\mathbf{B}_{R}^{\mathbf{S}^{*}}$ is obtained from $A\mathbf{B}_{L}^{\mathbf{S}}\cup AB\mathbf{S} \cup B\mathbf{B}_{R}^{\mathbf{S}}$  by first performing the $2$-$3$ move  that adds the edge $BE_{0},$ and then removing the only flat tetrahedron $ABE_{0}S_{0}$ in the three new tetrahedra created by the $2$-$3$ move, it is geometric. Therefore, the path $\mathbf{S}^{*}$ is admissible.

 \item [(v)] If $i=k-1$ and $BE_{k-1}$ intersects $AS_{k}$, then $E_{k-1}$ is the  vertex of $F_{2}$ that is adjacent to $S_{k}$ and lies in $\mathbf{B}_{L}^{\mathbf{S}}$.  We set $\mathbf{S}^{*}=\{S_{0},\ldots, S_{k-1},$
    $ E_{k-1}\}$, which satisfies Condition (5).  
    Since $A\mathbf{B}_{L}^{\mathbf{S}^{*}}\cup AB\mathbf{S}^{*} \cup B\mathbf{B}_{R}^{\mathbf{S}^{*}}$ is obtained from $A\mathbf{B}_{L}^{\mathbf{S}}\cup AB\mathbf{S} \cup B\mathbf{B}_{R}^{\mathbf{S}}$ by first performing the $2$-$3$ move   that adds the edge $BE_{k-1}$, and then removing the flat tetrahedron $ABE_{k-1}S_{k}$ in the three new tetrahedra created by the $2$-$3$ move,  it is geometric. Therefore, the path  $\mathbf{S}^{*}$ is admissible.\qedhere
\end{enumerate}
\end{proof}

\begin{proof}[Proof of Proposition \ref{prop:order}]
We first observe that in Cases (1), (2) and (3) of Lemma~\ref{prop:induction}, $\mathcal{T}_{\mathbf{S}^{*}}'$ and $\mathcal{T}_{\mathbf{S}}'$ 
share the same set of flat tetrahedra. Moreover,   $\mathcal{T}_{\mathbf{S}^{*}}'$ is obtained from $\mathcal{T}_{\mathbf{S}}'$ by a single $2$-$3$, $3$-$2$ or $4$-$4$ move with only hyperideal tetrahedra involved. Hence Proposition \ref{prop:order} holds for these Cases with $\mathcal{T}_{\mathbf{S}}=\mathcal{T}_{\mathbf{S}}'$ and $\mathcal{T}_{\mathbf{S}^{*}}=\mathcal{T}_{\mathbf{S}^{*}}'$.

We are left to consider the Cases (4) and (5) in Lemma \ref{prop:induction}.
Notice that these two cases are symmetric (see Figure~\ref{fig:lemma519} (d) and (e)); hence    without loss of generality, 
we can only focus on Case (4). 
In below Step I, we will choose the  triangulation $\mathcal T_{\mathbf S}$ from $\mathbf T_{\mathbf S}$; in Step II, we will according to the cases perform a $2$-$3$, $3$-$2$ or $4$-$4$ move to the triangulation $\mathcal T_{\mathbf S}$ to obtain a new triangulation $\mathcal T_{\mathbf S^*}$; and in Step III, we will show that the triangulation $\mathcal T_{\mathbf S^*}$ belongs to $\mathbf T_{\mathbf S^*}.$
\medskip

\noindent Step I. If $\mathbf{S}$ is not special, then we have no choice but let $\mathcal T_{\mathbf S}=\mathcal T_{\mathbf S}'.$

If $\mathbf{S}$ is special, then we first claim that  $\mathbf S^*$ is not special. Indeed, if $\mathbf S$ satisfies Condition (A) in Definition \ref{special}, then $S_{0}=S_{k}\ne E_{0}$, which implies that $\mathbf{S}^{*}$ does not satisfy condition  (A);  and $AB$ and $\bar{A}\bar{B}$ are not adjacent, which implies that $\mathbf{S}^{*}$ does not satisfy Condition (B). Also, if $\mathbf{S}$ satisfies Condition (B) in Definition \ref{special}, then $AB$ and $\bar{A}\bar{B}$ are adjacent, which implies that  $\mathbf{S}^{*}$ does not satisfy condition (A); and $S_{0}$ and $S_{k}$ are adjacent, in this case, $E_{0}$ can not be adjacent to $S_{k}$ as they both lie on $\partial F_1$ and are on different sides of $S_0$ (see Figure~\ref{fig:lemma519} (d)), which implies that  $\mathbf{S}^{*}$ does not satisfy Condition (B). Therefore, $\mathbf S^*$ is not special. In this case, there exists either zero or one edge in $\mathcal{T}_{\mathbf{S}^{*}}'$ that lies in between the polygons $F_{1}$ and $F_{2}$, and  belongs to neither $F_{1}$ nor $F_{2}.$ (Such edges are colored by orange in the illustrative diagrams in 
Subsection~\ref{subsec:step1}.)  By a case-by-case verification, we see that if $\mathbf{S}$ satisfies Condition (A), then $\mathcal{T}_{\mathbf{S}}'$ has one such orange edge and $\mathcal{T}_{\mathbf{S}}''$  has no orange edge, and if $\mathbf{S}$ satisfies Condition (B), then $\mathcal{T}_{\mathbf{S}}'$ has no orange edge and $\mathcal{T}_{\mathbf{S}}''$ has one orange edge.
In all scenarios, we choose from $\mathbf T_{\mathbf S}$ the unique triangulation $\mathcal T_{\mathbf S}$ that contains the same number of the orange edge as $\mathcal{T}_{\mathbf{S}^{*}}'$ does.
\medskip

\noindent  Step II.  We consider the following cases: 
\begin{enumerate}[(a)]
\item If the triangle $ABS_{0}$ is adjacent to two hyperideal tetrahedra, then we perform a $2$-$3$ move that adds the edge $BE_{0}$. See Figure \ref{fig:moveatboundary} (a).

\item If the triangle $ABS_{0}$ is adjacent to the tetrahedron $ABS_{0}E_{0}$, which is flat, then we perform a $3$-$2$ move that removes the edge $AS_{0}$. See Figure \ref{fig:moveatboundary} (b).

\item If the triangle $ABS_{0}$ is adjacent to a flat tetrahedron $ABS_{0}\bar{X}$ with $\bar{X}\in \{\bar{A}, \bar{B}\}$ and $\bar{X}\ne E_{0}$, then we perform a  $4$-$4$ move that removes the edge $AS_{0}$ and adds the edge  $BE_{0}$. See Figure \ref{fig:moveatboundary} (c).
\end{enumerate}

\begin{figure}[ht]
    \centering
\begin{subfigure}[b]{0.48\textwidth}
        \centering
\scalebox{0.7}{
$
\vcenter{\hbox{\begin{tikzpicture}[line width=0.8pt]
\coordinate (S') at (0, 0);
\coordinate (S) at (2, 0);
\coordinate (A) at (0, 4);
\coordinate (B) at (2, 4);
\coordinate (S1) at (-1.8, -1.4);
\draw (A)--(B);
\draw (S)--(B);
\draw[dashed] (A)--(S');
\draw[dashed] (S1)--(S');
\draw[dashed] (S')--(S);
\draw (S1)--(A);
\draw (S1)--(B);
\draw (S1)--(S);
\draw (S') node[left]{$E_0$};
\draw (S1) node[left]{$S_{1}$};
\draw (A) node[above]{$A$};
\draw (B) node[above]{$B$};
\draw (S) node[right]{$S_{0}$};
\draw[dashed] (A) to (S);
\end{tikzpicture}}}
\mathbf{\xrightarrow{2-3}}
\vcenter{\hbox{\begin{tikzpicture}[line width=0.8pt]
\coordinate (S') at (0, 0);
\coordinate (S) at (2, 0);
\coordinate (A) at (0, 4);
\coordinate (B) at (2, 4);
\coordinate (S1) at (-1.8, -1.4);
\draw (A)--(B);
\draw (S)--(B);
\draw[dashed] (A)--(S');
\draw[dashed] (S1)--(S');
\draw[dashed] (S')--(S);
\draw (S1)--(A);
\draw (S1)--(B);
\draw (S1)--(S);
\draw (S') node[left]{$E_0$};
\draw (S1) node[left]{$S_{1}$};
\draw (A) node[above]{$A$};
\draw (B) node[above]{$B$};
\draw (S) node[right]{$S_{0}$};
\draw[dashed](A) to (S);
\draw[dashed, purple, line width= 1.2pt] (B) to (S');
\end{tikzpicture}}}
$}
 \caption{{The} 2-3 move}
    \label{fig:23moveatboundary}
    \end{subfigure}
\hfill\quad
\begin{subfigure}[b]{0.48\textwidth}
        \centering
\scalebox{0.7}{$\vcenter{\hbox{\begin{tikzpicture}[line width=0.8pt]
\coordinate (S') at (0, 0);
\coordinate (S) at (2, 0);
\coordinate (A) at (0, 4);
\coordinate (B) at (2, 4);
\coordinate (S1) at (-1.8, -1.4);
\draw (A)--(B);
\draw (S)--(B);
\draw[dashed] (A)--(S');
\draw[dashed] (S1)--(S');
\draw[dashed] (S')--(S);
\draw (S1)--(A);
\draw (S1)--(B);
\draw (S1)--(S);
\draw (S') node[left]{$E_0$};
\draw (S1) node[left]{$S_{1}$};
\draw (A) node[above]{$A$};
\draw (B) node[above]{$B$};
\draw (S) node[right]{$S_{0}$};
\draw[dashed, purple, line width= 1.2pt] (A) to (S);
\draw[dashed] (B) to (S');
\end{tikzpicture}}}
\mathbf{\xrightarrow{3-2}}
\vcenter{\hbox{\begin{tikzpicture}[line width=0.8pt]
\coordinate (S') at (0, 0);
\coordinate (S) at (2, 0);
\coordinate (A) at (0, 4);
\coordinate (B) at (2, 4);
\coordinate (S1) at (-1.8, -1.4);
\draw (A)--(B);
\draw (S)--(B);
\draw[dashed] (A)--(S');
\draw[dashed] (S1)--(S');
\draw[dashed] (S')--(S);
\draw (S1)--(A);
\draw (S1)--(B);
\draw (S1)--(S);
\draw (S') node[left]{$E_0$};
\draw (S1) node[left]{$S_{1}$};
\draw (A) node[above]{$A$};
\draw (B) node[above]{$B$};
\draw (S) node[right]{$S_{0}$};
\draw[dashed] (B) to (S');
\end{tikzpicture}}}
$}
\caption{{The} 3-2 move}
\label{fig:32moveatboundary}
\end{subfigure}
\centering
\bigskip

\begin{subfigure}[b]{0.48\textwidth}
\hspace{-20mm}
\scalebox{0.7}{$
\vcenter{\hbox{\begin{tikzpicture}[line width=0.8pt]
\coordinate (S') at (0, 0);
\coordinate (S) at (2, 0);
\coordinate (A) at (0, 4);
\coordinate (B) at (2, 4);
\coordinate (S1) at (-1.8, -1.4);
\draw (A)--(B);
\draw (S)--(B);
\draw[dashed] (A)--(S');
\draw[dashed] (S1)--(S');
\draw[dashed] (S')--(S);
\draw (S1)--(A);
\draw (S1)--(B);
\draw (S1)--(S);
\draw (S') node[left]{$E_0$};
\draw (S1) node[left]{$S_{1}$};
\draw (A) node[above]{$A$};
\draw (B) node[above]{$B$};
\draw (S) node[right]{$S_{0}$};
\draw[dashed, purple, line width= 1.2 pt] (A) to (S);
\coordinate (X) at (3.5, 2);
\draw (X) node[right]{$\bar{X}=\bar A \text{ or }\bar B$};
\draw (X) to (S);
\draw[dashed] (X) to (A);
\draw[dashed] (X) to (S');
\draw (X) to (B);
\end{tikzpicture}}}
\mathbf{\xrightarrow{4-4}}
\vcenter{\hbox{\begin{tikzpicture}[line width=0.8pt]
\coordinate (S') at (0, 0);
\coordinate (S) at (2, 0);
\coordinate (A) at (0, 4);
\coordinate (B) at (2, 4);
\coordinate (S1) at (-1.8, -1.4);
\draw (A)--(B);
\draw (S)--(B);
\draw[dashed] (A)--(S');
\draw[dashed] (S1)--(S');
\draw[dashed] (S')--(S);
\draw (S1)--(A);
\draw (S1)--(B);
\draw (S1)--(S);
\draw (S') node[left]{$E_0$};
\draw (S1) node[left]{$S_{1}$};
\draw (A) node[above]{$A$};
\draw (B) node[above]{$B$};
\draw (S) node[right]{$S_{0}$};
\draw[dashed, purple, line width= 1.2 pt] (B) to (S');
\coordinate (X) at (3.5, 2);
\draw (X) node[right]{$\bar{X}=\bar A \text{ or } \bar B$};
\draw (X) to (S);
\draw[dashed] (X) to (A);
\draw[dashed] (X) to (S');
\draw  (X) to (B);
\end{tikzpicture}}}
$}
\caption{The $4$-$4$ move}
\label{fig:44moveatboundary}
\end{subfigure}
    \caption{In (a), the edge $BE_0$ is added and the flat tetrahedron $ABS_0E_0$ is added. In (b), the edge $AS_0$ is removed and the flat tetrahedron $ABS_0E_0$ is removed. In (c), the edge $AS_0$ is removed, the edge $BE_0$ is added, and the two flat tetrahedra $AB\bar XS_0$ and $A\bar XS_0E_0$ are replaced by the two flat tetrahedra $AB\bar XE_0$ and $B\bar XS_0E_0$.}
    \label{fig:moveatboundary}
\end{figure}
\medskip

\noindent  Step III.
By the construction of $\mathbf S^*$ in the proof of Lemma~\ref{prop:induction}, Case (4),  we see that $\mathcal{T}_{\mathbf{S}^{*}}'$ and $\mathcal{T}_{\mathbf{S}^{*}}''$ share the same set of hyperideal tetrahedra, which are obtained from the set of hyperideal tetrahedra in $\mathcal{T}_{\mathbf{S}}'$ by removing the tetrahedra $\{ABS_{0}S_{1}, AE_0S_{0}S_{1}\}$ and adding the tetrahedra $\{BE_0S_{0}S_{1}, ABE_0S_{1}\}$. Hence  the triangulation $\mathcal{T}_{\mathbf{S^*}}$ obtained in Step (II)  share the same set of hyperideal tetrahedra with $\mathcal{T}_{\mathbf{S}^{*}}'$ and $\mathcal{T}_{\mathbf{S}^{*}}''$. We also have that the triangulations $\mathcal{T}_{\mathbf{S^*}},$ $\mathcal{T}_{\mathbf{S}^{*}}'$ and $\mathcal{T}_{\mathbf{S}^{*}}''$ have the same set of flat tetrahedra that are not adjacent to the edge $S_{0}E_{0}$.

Therefore, it suffices to keep track of the change of at most four flat tetrahedra, including the central ones  adjacent to $S_{0}$, and the sided ones adjacent to $S_{0}E_{0}$. We will do a case-by-case verification in details for $\mathbf{S}$  in Case 1 and Case 2 of Step 1 in Subsection \ref{subsec:step1}; and the subcases of Case 3 can be verified in the  identical way as those of Case 2.

\begin{enumerate}
    \item [(1a)] The triangle $ABS_{0}$ is adjacent to the flat tetrahedron $ABS_{0}\bar{A}$ if $S_{0}\ne \bar{A}$, and is not adjacent to any flat tetrahedron if otherwise. 

    If $S_{0}=B_{t-1}$, then the $3$-$2$ move removing $AS_{0}$ is performed. Comparing $\mathcal{T}_{\mathbf{S}^{*}}$ with  $\mathcal{T}_{\mathbf{S}^{*}}'$ which belongs to the subcase (1a), we see that they agree.
        
         If $S_{0}=B_{t}$, then the $2$-$3$ move adding $BE_{0}$ is performed. 
        Comparing $\mathcal{T}_{\mathbf{S}^{*}}$ with  $\mathcal{T}_{\mathbf{S}^{*}}'$ which belongs to the subcase (1b), we see that they agree.
        
If otherwise, then the $4$-$4$ move removing $AS_{0}$ and adding $BE_{0}$ is performed. 
        Comparing $\mathcal{T}_{\mathbf{S}^{*}}$ with  $\mathcal{T}_{\mathbf{S}^{*}}'$ which belongs to the subcase (1a), we see that they agree.

   \item [(1b)] The triangle $ABS_{0}$ is adjacent to the flat tetrahedron $ABS_{0}\bar{A}$ and $\bar{A}\ne E_{0}$. Therefore,  the  $4$-$4$ move removing $AS_{0}$ and adding $BE_{0}$ is performed. 
   Comparing $\mathcal{T}_{\mathbf{S}^{*}}$ with  $\mathcal{T}_{\mathbf{S}^{*}}'$ which belongs to the subcase (1b), we see that they agree.
   
   \item [(2a)] The triangle $ABS_{0}$ is adjacent to the flat tetrahedron $ABS_{0}\bar{A}$ if $S_{0}\ne \bar{A}$, and is not adjacent to any flat tetrahedra if otherwise. 
   
   If $S_{0}=B_{t-1}$, then the $3$-$2$ move removing $AS_{0}$ is performed. 
        Comparing $\mathcal{T}_{\mathbf{S}^{*}}$ with  $\mathcal{T}_{\mathbf{S}^{*}}'$ which belongs to the subcase (2a), we see that they agree. 
        
   If $S_{0}=B_{t}$, then the $2$-$3$ move adding $BE_{0}$ if performed. 
        Comparing $\mathcal{T}_{\mathbf{S}^{*}}$ with  $\mathcal{T}_{\mathbf{S}^{*}}'$ which belongs to the subcase (2e), we see that they agree.

  If otherwise, then the $4$-$4$ move removing $AS_{0}$ and adding $BE_{0}$ is performed. 
        Comparing $\mathcal{T}_{\mathbf{S}^{*}}$ with  $\mathcal{T}_{\mathbf{S}^{*}}'$, which belongs to the subcase (2a), we see that they agree.

\item [(2b)] The triangle $ABS_{0}$ is adjacent to the flat tetrahedron $ABS_{0}\bar{B}$ if $S_{0}\ne \bar{B}$, and is not adjacent to any flat tetrahdra if otherwise. 
   
    If $S_{0}=\bar{B}$, then the $2$-$3$ move adding $BE_{0}$ is performed. By comparing  $\mathcal{T}_{\mathbf{S}^{*}}$ with $\mathcal{T}_{\mathbf{S}^{*}}'$ which still belongs to the subcase (2b), we see that they agree.
     
        If otherwise, then the $4$-$4$ move removing $AS_{0}$ and adding $BE_{0}$ is performed. 
        Comparing $\mathcal{T}_{\mathbf{S}^{*}}$ with  $\mathcal{T}_{\mathbf{S}^{*}}'$ which belongs to the subcase (2b). we see that they agree. 

\item [(2c)] If $S_{0}=S_{k}$ and $B\ne \bar{A}$, then $\mathbf{S}$ is special. 
In this case
$\mathbf{S}^{*}$ is always in the subcase (2c) and $\mathcal{T}_{\mathbf{S}}=\mathcal T_{\mathbf S}'$.
Note that the triangle $ABS_{0}$ is adjacent to the flat tetrahedron $ABS_{0}\bar{A}$ if $\bar{A}\ne B$, and is adjacent to $ABS_{0}B_{s+1}=ABS_{0}E_{0}$ if otherwise. 

If $\bar{A}=B$, then the $3$-$2$ move removing $AS_{0}$ is performed. Comparing $\mathcal{T}_{\mathbf{S}^{*}}$ with $\mathcal{T}_{\mathbf{S}^{*}}'$ which still belongs to the subcase (2c), we see that they agree.

If otherwise, then the $4$-$4$ move removing $AS_{0}$ and adding $BE_{0}$ is performed. Comparing $\mathcal{T}_{\mathbf{S}^{*}}$ with $\mathcal{T}_{\mathbf{S}^{*}}'$ which belongs to the subcase (2c), we see that they agree.

\item [(2d)] If $S_{0}=S_{k}$ and $A\ne \bar{B}$, then $\mathbf{S}$ is special. In this case,  $\mathcal{T}_{\mathbf{S}^{*}}'$ belongs to the subcase (2f) and $\mathcal{T}_{\mathbf{S}}=\mathcal T''_{\mathbf S}$.
Since the triangle $ABS_{0}$ is adjacent to the flat tetrahedron $ABS_{0}\bar{A}$ in $\mathcal{T}_{\mathbf{S}}''$, 
the $4$-$4$ move removing $AS_{0}$ and adding $BE_{0}$ is performed. Comparing $\mathcal{T}_{\mathbf{S}^{*}}$ with 
$\mathcal{T}_{\mathbf{S}^{*}}'$, we see that they agree.

If $A\ne \bar{B}$ and $S_{0}\ne S_{k}$, then the triangle $ABS_{0}$ is adjacent to the flat tetrahedron $ABS_{0}\bar{B}$ and the $4$-$4$ move removing $AS_{0}$ and adding $BE_{0}$ is performed. Comparing $\mathcal{T}_{\mathbf{S}^{*}}$ with  $\mathcal{T}_{\mathbf{S}^{*}}'$which belongs to the subcase (2d), we see that they agree.

    If $A=\bar{B}$ and $S_{0}\ne B_{3}$, then the triangle $ABS_{0}$ is adjacent to the flat tetrahedron $ABS_{0}B_{s-1}$, and  the $4$-$4$ move removing $AS_{0}$ and adding $BE_{0}$ is performed. In this case, if $S_{0}\ne S_{k}$, then $\mathcal{T}_{\mathbf{S}^{*}}$ agrees with $\mathcal{T}_{\mathbf{S}^{*}}'$ which belongs to the subcase (2d); and if otherwise, then $\mathbf{S}^{*}$ is special, and $\mathcal{T}_{\mathbf{S}^{*}}$ agrees with   $\mathcal{T}_{\mathbf{S}^{*}}''$ which belongs to the subcase (2f). 

    If $A=\bar{B}$ and $S_{0}= B_{3}$, then the triangle $ABS_{0}$
is not adjacent to any flat tetrahedra, and 
the $2$-$3$ move adding $BE_{0}$ is performed. In this case if $S_{0}\ne S_{k}$, then if $\mathcal{T}_{\mathbf{S}^{*}}$ agrees with $\mathcal{T}_{\mathbf{S}^{*}}'$  which belongs to the subcase (2d); and if otherwise, then $\mathbf{S}^{*}$ is special, and $\mathcal{T}_{\mathbf{S}^{*}}$ agrees with  $\mathcal{T}_{\mathbf{S}^{*}}''$ which belongs to the subcase (2f).

\item [(2e)] If $\bar {A}=B$ and $s'=s+1$, then $\mathbf{S}$ is special. In this case, $\mathcal{T}_{\mathbf{S}^{*}}'$  belongs to the subcase (2c), $\mathcal{T}_{\mathbf{S}}=\mathcal T'_{\mathbf S}$, and the triangle $ABS_{0}$ is adjacent to the flat tetrahedron $ABS_{0}S_{k}$ (see Figure \ref{fig:alternativetrig2}). Then the $3$-$2$ move removing $AS_{0}$ is performed. Comparing $\mathcal{T}_{\mathbf{S}^{*}}$ with $\mathcal{T}_{\mathbf{S}^{*}}'$, we see that they agree.

If $\bar{A}= B$ and $S_{0}=\bar{B}$, then the triangle $ABS_{0}$ is not adjacent to any flat tetrahedra, and the $2$-$3$ move adding $BE_{0}$ is performed. In this case, $\mathcal{T}_{\mathbf{S}^{*}}$ agrees with 
$\mathcal{T}_{\mathbf{S}^{*}}'$ which belongs to the subcase (2e) if $s\ne s'-1$,  and belongs to the subcase (2c) if $s= s'-1$.

    If $\bar{A}\ne B$ and  $S_{0}=\bar{B}$, then  the triangle $ABS_{0}$ is adjacent to $ABS_{0}\bar{A}=AB\bar{A}\bar{B}$; and if otherwise, then the triangle $ABS_{0}$ is adjacent to $ABS_{0}\bar{B}$. In both of the cases, the $4$-$4$ move removing $AS_{0}$ and adding $BE_{0}$ is performed. 
    If $s\ne s'-1$, then $\mathcal{T}_{\mathbf{S}^{*}}$ agrees with  $\mathcal{T}_{\mathbf{S}^{*}}'$ which belongs to then subcase (2e); and if otherwise, the $\mathcal{T}_{\mathbf{S}^{*}}$ agrees with $\mathcal{T}_{\mathbf{S}^{*}}''$ which belongs to the subcase (2c).

\item [(2f)] If $s=s'+1\ne t$ and $\bar{B}=A$, then $\mathbf{S}$ is special. In this case, $\mathbf{S}^{*}$ belongs to the subcase (2f) and $\mathcal{T}_{\mathbf{S}}=\mathcal{T}_{\mathbf{S}}'$.
The triangle $ABS_{0}$ is adjacent to $ABS_{0}\bar{A}$ if $S_{0}\ne \bar{A}$, and is adjacent to $ABS_{0}\bar{B}$ if otherwise.

If $s\notin \{t-1, t\}$, then the $4$-$4$ move removing $AS_{0}$ and adding $BE_{0}$ is performed. Comparing $\mathcal{T}_{\mathbf{S}^{*}}$ with $\mathcal{T}_{\mathbf{S}^{*}}'$ which belongs to the subcase (2f), we see that they agree.

If $s=t-1$, then the $3$-$2$ move removing $AS_{0}$ is performed. Comparing $\mathcal{T}_{\mathbf{S}^{*}}$ with $\mathcal{T}_{\mathbf{S}^{*}}'$ which belongs to the subcase (2f), we see that they agree.

If $s=t$, then the $3$-$2$ move removing $AS_{0}$ is performed. Comparing $\mathcal{T}_{\mathbf{S}^{*}}$ with $\mathcal{T}_{\mathbf{S}^{*}}'$ which belongs to the subcase (2b), we see that they agree. \qedhere
\end{enumerate}
\end{proof}

\bibliographystyle{abbrv}
\bibliography{cibib}


\end{document}